%% file: toricdivisors_for_memoirs.tex
\documentclass{amsart}[12pt]
\usepackage{graphicx,url,amsthm,xspace,enumitem,subfigure,amssymb}
\usepackage[dvipsnames]{xcolor}
\usepackage{amscd, amsmath}
\usepackage{multirow, overpic}
\usepackage{tikz}

\usepackage{xy}
\xyoption{all}

\usepackage{pinlabel}
\usepackage[linktocpage]{hyperref}

\usepackage[boxruled]{algorithm2e}

\hypersetup{
  colorlinks,
  citecolor=Purple,
  linkcolor=Black,
  urlcolor=Green}
  
\usepackage[margin=3cm, marginpar=2.5cm]{geometry}

\newcommand{\C}{\mathbb{C}}
\renewcommand{\P}{\mathbb{P}}
\newcommand{\R}{\mathbb{R}}
\newcommand{\D}{\mathbb{D}}
\newcommand{\Q}{\mathbb{Q}}
\newcommand{\Z}{\mathbb{Z}}
\newcommand{\cptwo}{\C\P^2}
\newcommand{\cpone}{\C\P^1}
\newcommand{\cptwobar}{\overline{\C\P}\,\!^2}
\newcommand{\rptwo}{\R\P^2}
\newcommand{\std}{std}

\newcommand{\sse}{\subseteq}

\DeclareMathOperator{\coker}{coker}

\theoremstyle{plain}
\newtheorem{theorem}{Theorem}[section]

\newtheorem{lemma}[theorem]{Lemma}
\newtheorem{prop}[theorem]{Proposition}
\newtheorem{cor}[theorem]{Corollary}

\theoremstyle{definition}

\newtheorem{definition}[theorem]{Definition}

\theoremstyle{remark}
\newtheorem{remark}[theorem]{Remark}
\newtheorem{example}[theorem]{Example}

\numberwithin{equation}{section}

\theoremstyle{plain}

\title{Weinstein handlebodies for complements of smoothed toric divisors}

\author[Acu]{Bahar Acu}
\address[B.\ Acu]{Department of Mathematics \\ ETH Z\"urich\\ Switzerland}
\email{bahar.acu@math.ethz.ch}

\author[Capovilla-Searle]{Orsola Capovilla-Searle}
\address[O.\ Capovilla-Searle]{Department of Mathematics \\ Uc Davis \\ Davis \\ CA \\ U.S.A.}
\email{ocapovillasearle@ucdavis.edu}

\author[Gadbled]{Agn\`{e}s Gadbled}
\address[A.\ Gadbled]{D\'epartement de Math\'ematiques \\ Universit\'e Paris-Saclay \\ Orsay \\ France}
\email{agnes.gadbled@universite-paris-saclay.fr}

\author[Marinkovi\'c]{Aleksandra Marinkovi\'c}
\address[A.\ Marinkovi\'c]{Matemati\v{c}ki Fakultet \\ Belgrade \\ Serbia}
\email{aleks@matf.bg.ac.rs}

\author[Murphy]{Emmy Murphy}
\address[E.\ Murphy]{Department of Mathematics \\ Northwestern University \\ Evanston \\ IL \\ U.S.A.}
\email{e\_murphy@math.northwestern.edu}

\author[Starkston]{Laura Starkston}
\address[L.\ Starkston]{Department of Mathematics \\ UC Davis\\ Davis \\ CA \\ U.S.A.}
\email{lstarkston@math.ucdavis.edu}

\author[Wu]{Angela Wu}
\address[A.\ Wu]{Department of Mathematics \\ Louisiana State University \\ Baton Rouge \\ U.S.A}
\email{awu@lsu.edu}

\begin{document}

\maketitle

\begin{abstract}
	We study the interactions between toric manifolds and Weinstein handlebodies. We define a partially-centeredness condition on a Delzant polytope, which we prove ensures that the complement of a corresponding partial smoothing of the total toric divisor supports an explicit Weinstein structure. Many examples which fail this condition also fail to have Weinstein (or even exact) complement to the partially smoothed divisor. We investigate the combinatorial possibilities of Delzant polytopes that realize such Weinstein domain complements.
	We also develop an algorithm to construct a Weinstein handlebody diagram in Gompf standard form for the complement of such a partially smoothed total toric divisor. The algorithm we develop more generally outputs a Weinstein handlebody diagram for any Weinstein 4-manifold constructed by attaching 2-handles to the disk cotangent bundle of any surface~$F$, where the 2-handles are attached along the co-oriented conormal lifts of curves on~$F$. We discuss how to use these diagrams to calculate invariants and provide numerous examples applying this procedure. For example, we provide Weinstein handlebody diagrams for the complements of the smooth and nodal cubics in $\cptwo$.
\end{abstract} 

\tableofcontents

\input{parts/introduction.tex}

\input{parts/weinstein.tex}
\input{parts/toric.tex}

\input{parts/handlesect.tex}

\input{parts/classification.tex}

\input{parts/obstruction-noncentered.tex}
\input{parts/cotangent.tex}

\input{parts/jetspace.tex}

\input{parts/applications.tex}
\input{parts/examples.tex}

\bibliography{references_memoirs}
\bibliographystyle{amsalpha}

%
%
%
%

\end{document}

%% file: parts/introduction.tex

\section{Introduction}
{
This article provides a novel connection between two areas that were not previously able to interact: Weinstein handlebody theory and toric geometry. We develop a systematic procedure to produce Weinstein handlebody diagrams in Gompf standard form for a large class of log Calabi-Yau divisor complements. This is done in two parts. First, we identify and explore a criterion which we prove ensures that a partial smoothing of the total toric divisor has a Weinstein domain as its neighborhood complement. This has been erroneously taken for granted in the past, with some suspecting that no hypothesis is required (which we show is false), and others assuming full monotonicity is required (which we show is much more restrictive than our condition except when the divisor is fully smoothed). In the second part, we show how to obtain standard Weinstein handlebody diagrams for Weinstein domains obtained by attaching $2$-handles along co-normal lifts to $D^*F$ (for $F=T^2$ this covers the complements of the divisors from the first part). Previously, the literature was lacking a proof that even the standard handlebody diagrams for $D^*F$ are Weinstein homotopic to the canonical Weinstein structure. We use this Weinstein homotopy to track the effect on attaching spheres for additional $2$-handles. We accomplish this via a new method which carefully tracks a contact isotopy, using a guiding singular Legendrian formed from the union of two overlapping smooth Legendrians that carry the key data on the Weinstein structure of $D^*F$. In the end, we are able to carry a collection of attaching spheres from a co-normal diagram to a handlebody diagram in standard form in an algorithmic manner. After simplifying these diagrams with equivalence moves, they provide key information about the symplectic invariants and geometric properties of these manifolds. Despite their usefulness, explicit Weinstein handlebody diagrams are not known for many fundamental examples of divisor complements. Translating an abstract description of a divisor to an explicit description of the Legendrian attaching spheres for its complementary handlebody is highly nontrivial in most examples. The aim of this article is to make progress on this important problem. 

Because this article brings together techniques and objects which have not typically been worked with together, we begin by introducing the main characters and their significance, and then explain our results connecting them.

Handlebody theory has been a key tool in smooth topology to encode complicated manifolds, calculate their invariants, and identify equivalences. The use of handlebodies in symplectic geometry first arose with Eliashberg's topological characterization of Stein manifolds \cite{Eliash}. To encode the symplectic structure ignoring the complex structure, Weinstein's simpler construction provides a model for symplectic handles \cite{Weinstein}. In dimension $4$, Gompf developed a standard form for \emph{Weinstein handlebody diagrams} which fully encode the Weinstein handle structure, and can be manipulated by equivalence moves \cite{Gompf}. Weinstein handlebody diagrams in dimension $4$ are front projections of Legendrian links in $(\#^k(S^1\times S^2), \xi_{std})$, encoding the Legendrian attaching spheres of the $2$-handles in the boundary of the $1$-handlebody.

Symplectic divisors are co-dimension $2$ symplectic submanifolds that may have controlled singularities. Donaldson proved that every closed integral symplectic manifold $(M,\omega)$ has a smooth symplectic divisor Poincar\'e dual to $k[\omega]$~\cite{Donaldson}, and Giroux proved such a divisor can be chosen so that the complement has a  Weinstein structure~\cite{GirouxICM, Giroux}. Another homology class of interest for a divisor is that which is Poincar\'e dual to the anti-canonical class (the first Chern class) of a symplectic manifold. A symplectic manifold together with a divisor in this homology class is called a \emph{log Calabi-Yau pair}. Such pairs gained importance through their significance in homological mirror symmetry~\cite{AurouxTduality, Abouzaid, GrossHackingKeel, HackingKeating}. This leads us to the setting of toric manifolds.

A \emph{toric manifold} is a symplectic manifold with an effective Hamiltonian action of the torus of the maximal dimension.  For a $4$-dimensional toric manifolds, the Hamiltonian action induces a moment map whose image is a 2-dimensional Delzant polytope, denoted by $\Delta$, which fully characterizes the toric manifold. The preimage under the moment map of all facets of the Delzant polytope is an invariant symplectic divisor which we call the \emph{total toric divisor.} 
Total toric divisors initially have normal crossing singularities, which in dimension $4$ are nodes. Nodes are precisely the fixed points of the toric action and they are mapped under the moment map to the vertices of the Delzant polytope. There is a one to one correspondence between the nodes of the total toric divisor and the vertices of the Delzant polytope $\Delta$, hence we refer to both a node and its image vertex with the same notation.  Any of the nodes can be smoothed so the divisor has fewer singularities. A total toric divisor and any of its smoothings are Poincar\'e dual to the anti-canonical class, so it is a log Calabi-Yau divisor.

The intersection of anti-canonical divisors and Donaldson divisors requires $c_1(M,\omega)$ to be a multiple of $[\omega]$, which means that the symplectic manifold is monotone. If we only consider smooth anti-canonical divisors, the complement can only be exact (a prerequisite for Weinstein) if we are in this monotone setting. However, if we consider divisors with nodes and multiple irreducible components, for example partially smoothed total toric divisors, the complement often supports a Weinstein structure even in the non-monotone setting. We identify a condition, which can be thought of as an analogue of the monotone condition for partially smoothed total toric divisors in toric $4$-manifolds.

\begin{definition}\label{defn:centered}
	For each vertex, $V$, of the Delzant polytope $\Delta$, we associate to it a ray $R$ generated by the sum of the edge vectors of $\Delta$ adjacent to $V$ and beginning at $V$. A Delzant polytope with a chosen subset $\{V_1, \ldots, V_k\}$ of the vertices is called \emph{$\{V_1, \ldots, V_k\}$-centered} if the corresponding rays $R_1, \ldots, R_k$ all intersect at a common single point in the interior of the Delzant polytope. The corresponding toric manifold is also called $\{V_1, \ldots, V_k\}$-centered.	If we denote by $r(V)\in\mathbb Z^2$ the direction of the ray in $V,$ then we define the \emph{slope} of a vertex $V$ in a Delzant polytope $\Delta$ as a vector $s(V)\in\mathbb Z^2$
given by a $-\pi/2$ rotation of the vector $r(V)$.
\end{definition}

The motivation for these definitions is that if we have a $\{V_1,\ldots, V_k \}$-centered toric manifold, then we show that the complement of the total toric divisor smoothed precisely at the nodes $V_1,\ldots,V_k$ admits a Weinstein structure. Moreover, the slope of a vertex generates the slope of a curve in the torus $T^2$ that determines the attaching sphere of the 2-handle.

\newtheorem*{thm:toricWein}{Theorem~\ref{thm:toricWein}}
\begin{thm:toricWein}
	Let $(M,\omega)$ be a toric $4$-manifold corresponding to Delzant polytope $\Delta$ that is $\{V_1,\ldots, V_k \}$-centered. Let $D$ denote the symplectic divisor obtained by smoothing the total toric divisor at the nodes $V_1,\ldots, V_k$. Then, there exists an arbitrarily small neighborhood $N$ of $D$ such that $M\setminus N$ admits the structure of a Weinstein domain.
	
Furthermore, $M\setminus N$ is Weinstein homotopic to the Weinstein domain $\mathcal{W}_{F,c}$ obtained by attaching Weinstein $2$-handles to the unit disk cotangent bundle of the torus, $D^*F$ with $F=T^2$, along the Legendrian co-normal lifts of co-oriented curves $c=\{\gamma_1,\cdots \gamma_k\}$ of slopes $s(V_1),\ldots, s(V_k)$. 
\end{thm:toricWein}

If $\{V_1,\dots, V_k \}$ is the set of all vertices of the polytope $\Delta$, then we show that the polytope is $\{V_1,\dots,V_k \}$-centered if and only if it is monotone. This corresponds to the intersection of the Donaldson-Giroux divisors with the log Calabi-Yau divisors. Note that there are only five monotone toric $4$-manifolds (up to a rescaling of a symplectic form), so the vast majority of examples come from partial smoothings where $\{V_1,\dots, V_k \}$ is a strict subset of the vertices of $\Delta$. In fact, we prove that these partially smoothed total toric divisor complements can realize many different manifolds.

\newtheorem*{thm:infinitemanifolds}{Theorem~\ref{thm:infinitemanifolds}}
\begin{thm:infinitemanifolds}
	There are infinitely many non-diffeomorphic Weinstein manifolds obtained by taking the completion of the complement of a neighborhood of a partially smoothed total toric divisor in a toric $4$-manifold.
\end{thm:infinitemanifolds}

The proof of Theorem~\ref{thm:infinitemanifolds} is based on an explicit construction of $\{V_1,\dots,V_k \}$-centered toric 4-manifolds, for any $k\in\mathbb N.$
We remark that we do not give a complete list of such toric manifolds, but we do further explore the question of existence of a partially centered toric 4-manifold for a given set of slopes. For instance, no toric 4-manifold is centered with respect to multiple vertices with the same slopes.
Additionally, we give one infinite family of distinct slopes that cannot be realized as slopes of the vertices of any partially centered Delzant polytope.
\newtheorem*{example non-centered}{Proposition~\ref{example non-centered}}
\begin{example non-centered}
 	For any $K\geq 2$, there is no $\{V_1,V_2,V_3,V_4\}$-centered Delzant polytope where 
 	$$s(V_1)=(1,1), s(V_2)=(1,2), s(V_3)=(-K,-1), s(V_4)=(0,-1).$$
\end{example non-centered}

By contrast, if one removes the partially centeredness requirement, any collection of slopes can be realised by a Delzant polytope. This shows that the centeredness criterion is a non-trivial constraint for toric manifolds, not only on the lengths of the edges of the polytope, but also on the combinatorial slope data.
 
\newtheorem*{proposition slopes}{Theorem \ref{proposition slopes}}
\begin{proposition slopes}
	For any choice of primitive vectors $\{(a_1, b_1), \ldots, (a_k,b_k)\}$
	there is a Delzant polytope with at least $k$ edges such that there are vertices $V_1,\ldots, V_k$ with the slopes $s(V_i)=(a_i,b_i),$ $i=1,\ldots, k.$ 
	
\end{proposition slopes}

As part of the proof of this theorem, we use the following lemmas, which may be of independent interest.

\newtheorem*{lemma normals}{Lemma \ref{lemma normals}}
\begin{lemma normals} For any choice of distinct primitive vectors $\{(a_1, b_1), \ldots, (a_k,b_k)\}$
	there is a Delzant polytope with at least $k$ edges, so that the inward normal vectors of $k$ edges are precisely $\{(a_1, b_1), \ldots, (a_k,b_k)\}.$
\end{lemma normals}

\newtheorem*{lemma blowup}{Lemma \ref{lemma blowup}}
\begin{lemma blowup}
	For any primitive vector $(a,b)\in\mathbb{Z}^2$ and any $n\in \mathbb{N}$ there is a Delzant polytope with $n$ vertices with the slope $(a,b).$ 
\end{lemma blowup}

One can slightly extend the list of realizable collections of slopes by considering almost toric fibrations \cite{Symington}. Our Theorem~\ref{thm:toricWein} also holds for some compact almost toric manifolds with an adapted centeredness condition, that is encoded in the base of the almost toric fibration (see for example Remark~\ref{remark one almost toric} and Section~\ref{section:CP2bu5}). Bases of almost toric fibrations need not be Delzant polytopes. Both the centeredness condition and Delzant condition impose constraints on the allowable polytopes or bases, and we hope future work will explore these conditions in greater depth.

Further justifying our centeredness condition, we show that in the non-centered case, the complement of a neighborhood of the divisor does not always support a Weinstein structure. In fact, for many non-centered cases we can prove that the complement is not even exact.

\newtheorem*{prop:notexact}{Proposition~\ref{prop:notexact}}
\begin{prop:notexact}
	Let $(M,\omega)$ be a symplectic toric manifold and $\Delta$ its Delzant polytope.
	Let $V_1,\dots, V_k$ be a subsets of the vertices of $\Delta$.
	Assume that $M$ fails to be $\{V_1,\dots, V_k\}$-centered because either 
	\begin{itemize}
		\item[(i)] Two rays associated to two of the vertices $\{V_1,\dots, V_k\}$ are parallel or anti-parallel but the lines extending them do not coincide, or;
		\item[(ii)] There exists three vertices, $\{V_{i_1}, V_{i_2}, V_{i_3}\}$ such that for the associated rays 
		$R_{i_1}, R_{i_2}, R_{i_3}$, $R_{i_1}$ intersects $R_{i_2}$ at a point $c_1 \in \mathring \Delta$ in the interior of the polytope $\Delta$ that does not belong to $R_{i_3}$.
	\end{itemize}
	Then the complement of its total toric divisor smoothed at 
	$\{V_1,\dots, V_k\}$ is not an exact symplectic manifold (and in particular cannot support a Weinstein handlebody structure).
\end{prop:notexact}

There are other situations where a toric manifold may fail to be $\{V_1,\dots, V_k \}$-centered which do not fall under the above cases. For example, the rays may all intersect at points outside the interior of the polytope, or the directed rays may fail to intersect at all because of the placement of the vertices and the directions of the rays. In certain examples, we prove that even though the complement of the $\{V_1,\dots,V_k \}$-smoothing is exact, it still does not admit a Weinstein structure. In these examples the obstruction comes from the fact that a Weinstein domain must have convex boundary. Correspondingly the neighborhood $N$ of the smoothed divisor must have concave boundary. We discuss the criterion that $N$ has a concave structure and provide these examples in Section~\ref{s:obstructions}.

Theorem~\ref{thm:toricWein} identifies the geometric significance from the toric perspective of Weinstein domains obtained from $D^*T^2$ by attaching $2$-handles to the co-oriented co-normal lifts of curves, we now turn to the question of how to study these Weinstein domains diagrammatically. Our next result develops a systematic procedure to take the Weinstein domains produced by Theorem~\ref{thm:toricWein} and convert them into Weinstein handlebody diagrams in the Gompf standard form. In fact, this procedure works more generally when we replace the torus $T^2$ with any closed surface. 

\newtheorem*{thm:cotangentdiagram}{Theorem~\ref{thm:cotangentdiagram}}
\begin{thm:cotangentdiagram}
Let $F$ be a closed surface and $c=\{\gamma_i\}_{i=1}^k$ a finite unordered collection of co-oriented curves in $F$. Let $\mathcal{W}_{F,c}$ denote the Weinstein domain obtained by attaching $2$-handles to $D^*F$ along the Legendrian co-normal lifts of the $\gamma_i$. Then, the Weinstein handle diagram in the standard form obtained by the procedure of Section~\ref{s:procedure}, represents a Weinstein domain  that is Weinstein homotopic to $\mathcal{W}_{F,c}$.
\end{thm:cotangentdiagram}

In this paper, we focus on applying this theorem to the output of Theorem~\ref{thm:toricWein} where $F$ is a torus, though this result applies more generally and we provide some example applications in the more general context in Section~\ref{section examples}. Weinstein manifolds obtained by attaching $2$-handles to cotangent bundles have appeared in other contexts such as~\cite{STW}, and we hope that this systematic diagrammatic procedure will find many future applications towards studying this interesting class of Weinstein domains. 

As part of the proof of Theorem~\ref{thm:cotangentdiagram} we first establish the base case: the Gompf standard handle diagram for $D^*F$ with no additional $2$-handles attached. In~\cite{Gompf}, Gompf produced Weinstein handlebody diagrams which are \emph{diffeomorphic} to $D^*F$. Although it was expected that these structures were Weinstein homotopic to the canonical (Morse-Bott) Weinstein structure on the cotangent bundle, the proof was lacking in the literature (see \cite{Ozbagci, OzbagciC} for the contact case). We fill this gap here.

\newtheorem*{thm:cot}{Theorem~\ref{thm:cot}}
\begin{thm:cot}
	Let $F$ be a closed surface. The Gompf handlebody diagram for $D^*F$ corresponds to a Weinstein structure which is Weinstein homotopic to the canonical Weinstein structure on the cotangent bundles of a surface.
\end{thm:cot}

Note that in the case that $F$ is a torus, this result follows from a classification of fillings by Wendl~\cite{Wendl}. Our proof is more direct and applies to all surfaces. Moreover the Weinstein homotopy we construct is an essential ingredient in the proof of Theorem~\ref{thm:cotangentdiagram}.

There are two primary advantages of working with a Weinstein handlebody diagram in standard form to study a symplectic manifold. The first is simplification by Weinstein Kirby calculus which helps identify equivalences between manifolds arising from distinct constructions and also reveals exact Lagrangian submanifolds built from the union of exact Lagrangian fillings of the Legendrian attaching spheres and the cores of the $2$-handles. 

The second advantage of a handlebody diagram, is the combinatorial computation of invariants. At the topological level, the fundamental group, homology, and intersection form of the $4$-manifold, and the homology of its boundary can be computed directly from the diagram. In Section~\ref{sec:applications} we provide the necessary formulas to easily perform these computations for some of the diagrams output by our algorithm when $F=T^2$. At the symplectic level, more subtle invariants like symplectic homology and the wrapped Fukaya category can be computed. Let $X_{\Lambda}$ be a Weinstein $4$-manifold constructed as the cylindrical completion of the Weinstein domain obtained by attaching $2$-handles to a $1$-handlebody along a Legendrian link $\Lambda \sse (\#^k(S^1 \times S^2), \xi_{std})$. The wrapped Fukaya category of the Weinstein manifold $X_{\Lambda}$ is given by modules over the Chekanov-Eliashberg dga of $\Lambda$~\cite{BEE, Ekholm, EkholmLekili}. Ekholm and Ng~\cite{EkholmNg} provide a combinatorial description of the Chekanov-Eliashberg dga for Legendrian links in $ \#^k(S^1 \times S^2)$,  allowing for a purely combinatorial way to compute the wrapped Fukaya category of $X_{\Lambda}$. Such an approach could provide insight towards mirror symmetry~\cite{EtguLekili, CM}, where these manifolds play an important role. 

We show that the symplectic invariants of the Weinstein manifolds appearing in Theorem~\ref{thm:cotangentdiagram} are non-trivial
using a result of the second author~\cite{orsola_thesis} (which is an extension of a result of Leverson~\cite{Leverson}).

\newtheorem*{prop:vanishingSH1}{Proposition~\ref{prop:vanishingSH}}
\begin{prop:vanishingSH1}
Let $F$ be an orientable surface and let $X$ be any Weinstein $4$-manifold constructed by attaching $1$ or $2$-handles to $D^*F$ and taking its cylindrical completion. Then $X$ has nonvanishing symplectic homology.
\end{prop:vanishingSH1}

\newtheorem*{cor:flexible}{Corollary~\ref{cor:flexible}}
\begin{cor:flexible}
	Any Weinstein $4$-manifold $X$ constructed by attaching $1$- or $2$-handles to $T^*F$ for $i=1,\ldots, k$ for any orientable closed surface $F$, is not a flexible Weinstein manifold.
\end{cor:flexible}

Although we have many toric examples which produce examples where we can distinguish the diffeomorphism types of the complements using topological invariants, there are also many cases where different toric manifolds and choices of smoothed nodes can lead to equivalent complements. For any toric $4$-manifold, the cylindrical completion of the complement of a neighborhood of the total toric divisor smoothed at a single node is Weinstein homotopic to the affine variety $\{ (x,y,z)\in \C^3~|~x(xy^2-1)=z^2\}\subset \C^2\}$ ~\cite{CM, EtguLekili, ACGMMSW1}. More generally, we show that if two complements of smoothings of divisors satisfy the following relation their cylindrical completions are symplectomorphic. We notationally identify a slope $(a_i,b_i)$ with a curve on $T^2$ of that slope.

\newtheorem*{propknodes}{Proposition~\ref{propknodes}}
\begin{propknodes}
	 Consider any $\{V_1, \ldots V_k\}$-centered toric $4$-manifold where $s(V_i)=(a_i,b_i),$ for all $i=1,\ldots, k$. If there is an $SL(2,\mathbb{Z})$ transformation mapping
                   the set $\{(a_1, b_1), \ldots, (a_k,b_k)\}$ to the set $\{(a'_1, b'_1), \ldots, (a'_k,b'_k)\}$    then the completions of Weinstein domains
$\mathcal{W}_{T^2, \{(a_1,b_1),\ldots, (a_k,b_k)\}}$ and $\mathcal{W}_{T^2, \{(a'_1,b'_1),\ldots, (a'_k,b'_k)\}}$
        are symplectomorphic. (The set of vertices and slopes that we consider are unordered.)
\end{propknodes}

Proposition \ref{propknodes} motivated Proposition~\ref{prop:1handles} where we show that an $SL(2,\mathbb{Z})$ transformation taking the set of slopes $\{(a_1, b_1), \ldots, (a_k,b_k)\}$ to the set $\{(a'_1, b'_1), \ldots, (a'_k,b'_k)\}$ corresponds to Legendrian $1$-handle slides from $\mathcal{W}_{T^2, \{ (a_1, b_1), \ldots, (a_k,b_k)\}}$  to
$\mathcal{W}_{T^2, \{ (a'_1, b'_1), \ldots, (a'_k,b'_k)\}}$. 
This result is stronger and more general, since it shows that these Weinstein domains are Weinstein homotopic without requiring a centeredness condition.
Although the existence of an $SL(2,\Z)$ transformation relating the slopes is sufficient to ensure the domains are Weinstein homotopic, it is not necessary. We provide an example of a pair of Weinstein homotopic Weinstein domains output by our procedure such that the pair of associated slopes are not related by an $SL(2,\Z)$ transformation.

\newtheorem*{remark:sl2q}{Theorem~\ref{remark:sl2q}}
\begin{remark:sl2q} The existence of an $SL(2, \Z)$ transformation between sets of slopes $c$ and $c'$ is a sufficient condition to guarantee that $\mathcal{W}_{T^2, c}$ and $\mathcal{W}_{T^2, c'}$ are Weinstein homotopic, but it is not necessary (even when $\mathcal{W}_{T^2,c}$ and $\mathcal{W}_{T^2,c'}$ are both realizable as complements of partially smoothed total toric divisors).
\end{remark:sl2q}

We expect to find further examples with equivalent Weinstein complements through an almost toric extension of Theorem~\ref{thm:toricWein} via mutations of almost toric manifolds. Example~\ref{remark counterexample} provides some evidence towards this expectation.

We conclude our paper by applying our procedure to specific examples. 
Beyond the base case, the Weinstein diagrams we produce with Theorem~\ref{thm:cotangentdiagram} are generally new. In Section~\ref{section examples}, we carry out the procedure of Theorem~\ref{thm:cotangentdiagram} to produce Weinstein handlebody diagrams for a variety of examples of complements of partially or fully smoothed total toric divisors. We then apply sequences of Weinstein Kirby calculus moves to simplify the diagrams. We analyze these simplified diagrams to construct Lagrangian submanifolds in these Weinstein $4$-manifolds, and compute their homology. For a gentler introduction to Weinstein Kirby calculus and two more examples of our algorithm see~\cite{ACGMMSW1}. The following are a selection of the results we obtain in Section~\ref{section examples} by applying our algorithm to different toric manifolds.

\newtheorem*{thm:blowup_example}{Theorem~\ref{thm:blowup_example}}
\begin{thm:blowup_example}
	The Weinstein handlebody diagram of the complement of a total toric divisor smoothed in adjacent nodes of a blow up in any toric manifold is a standard Legendrian trefoil (see Figure \ref{alg_blowup} for an example).
\end{thm:blowup_example}

\begin{figure}
	\begin{center}
		\begin{tikzpicture}
			\node[inner sep=0] at (0,0) {\includegraphics[width=8 cm]{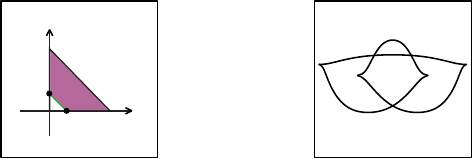}};
			\draw[->, line width=0.8mm] (0.9,0) -- (1.2,0);
			\draw[line width=0.8mm] (-1.2,0) -- (-1,0);
			\draw (-1,-0.4) rectangle (0.9,0.5) node[pos=.5] {Algorithm};
		\end{tikzpicture}
		\caption{The algorithm applied to the complement of the total toric divisor smoothed in adjacent nodes of a blow up of $\cptwo$.}
		\label{alg_blowup}
	\end{center}
\end{figure}
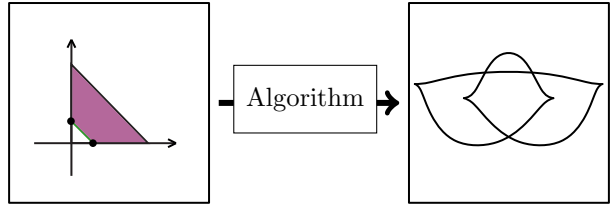

\newtheorem*{thm:cp1_example}{Theorem~\ref{thm:cp1_example}}
\begin{thm:cp1_example}
	The complement of the total toric divisor in $(\cpone\times\cpone, \omega_{a,a})$ smoothed in opposite nodes is Weinstein homotopic to the cyclic plumbing of two disk cotangent bundles of spheres, see Figure \ref{alg_cp1xcp1}.
\end{thm:cp1_example}

\begin{figure}
	\begin{center}
		\begin{tikzpicture}
			\node[inner sep=0] at (0,0) {\includegraphics[width=8 cm]{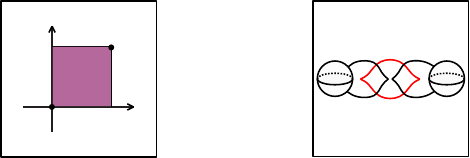}};
			\draw[->, line width=0.8mm] (0.9,0) -- (1.2,0);
			\draw[line width=0.8mm] (-1.2,0) -- (-1,0);
			\draw (-1,-0.4) rectangle (0.9,0.5) node[pos=.5] {Algorithm};
		\end{tikzpicture}
		\caption{The algorithm applied to the complement of the total toric divisor of $\cpone\times\cpone$ smoothed in opposite nodes, the Gompf Weinstein handlebody diagram is also depicted in Figure \ref{CP1xCP1simple}.}
		\label{alg_cp1xcp1}
	\end{center}
\end{figure}

\newtheorem*{thm:cubic}{Theorem~\ref{thm:cubic}}
\begin{thm:cubic}
 	The Weinstein handlebody diagram of the complement of a smooth cubic in $\cptwo$ is shown in Figure \ref{alg_cubic}.
\end{thm:cubic}

\begin{figure}
	\begin{center}
		\begin{tikzpicture}
			\node[inner sep=0] at (0,0) {\includegraphics[width=8 cm]{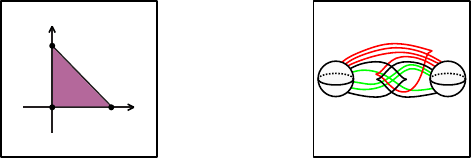}};
			\draw[->, line width=0.8mm] (0.9,0) -- (1.2,0);
			\draw[line width=0.8mm] (-1.2,0) -- (-1,0);
			\draw (-1,-0.4) rectangle (0.9,0.5) node[pos=.5] {Algorithm};
		\end{tikzpicture}
		\caption{The algorithm applied to the complement of the total toric divisor of $\cptwo$ smoothed in all nodes. The Gompf Weinstein handlebody diagram is also depicted in Figure \ref{CP23final}.}
		\label{alg_cubic}
	\end{center}
\end{figure}

\subsection*{Outline} This paper is organized as follows. In Section~\ref{section weinstein} we provide definitions and background material on Weinstein domains and Kirby calculus, as well as a review of Weinstein structures on cotangent bundles of smooth closed manifolds. In Section~\ref{section toric} we introduce the relevant background on toric manifolds. In Section~\ref{section smoothing} we construct the explicit Weinstein structure on the complement of a $\{V_1, \ldots, V_k\}$-smoothing of a total toric divisor in a $\{V_1, \ldots, V_k\}$-centered toric $4$-manifold, proving our first main result. We also discuss some extensions to almost toric manifolds. In Section~\ref{s: classifying} we prove numerous results about the variety of Weinstein $4$-manifolds arising as such complements using combinatorial properties of Delzant polytopes satisfying the centeredness condition. In Section~\ref{s:obstructions} we explain the importance of the centeredness condition in allowing the complement of the smoothed divisor to admit a Weinstein structure. Next we turn our focus to understanding Weinstein handlebody diagrams for these complements, and more general Weinstein domains of the form $\mathcal{W}_{F,c}$. We start with the base case in Section~\ref{s:cotangent}, proving that Gompf's Weinstein handle diagram for $D^*F$ is Weinstein homotopic to the canonical Weinstein structure. In Section~\ref{sec:curves} we state the algorithm to construct the fronts of co-normal lifts of co-oriented curves $\gamma_i\subset F$, and prove that the resulting standard Weinstein diagram is Weinstein homotopic to $\mathcal{W}_{F,c}$. Section~\ref{sec:applications} is dedicated to computing topological and symplectic invariants from the explicit Weinstein handlebodies of $\mathcal{W}_{F, c}$. Finally in Section~\ref{section examples} we provide numerous examples applying this procedure to obtain Weinstein handlebody diagrams for complements of partially smoothed divisors in toric manifolds as well as almost toric manifolds. We end with two basic examples of Weinstein handlebodies of $\mathcal{W}_{F, c}$ for higher genus, and non-orientable surfaces $F$.

\subsection*{Acknowledgments} This project grew out of our collaboration at the \href{https://icerm.brown.edu/topical_workshops/tw19-4-wiscon/}{2019 Research Collaboration Conference for Women in Symplectic and Contact Geometry and Topology (WiSCon)} that took place on July 22-26, 2019 at ICERM. The authors would like to extend their gratitude to the WiSCon organizers and the hosting institution ICERM and their staff for their hospitality. The authors would like to thank Jes\'us de Loera, Jonathan Evans, Yanki Lekili, Klaus Niederkrüger, Lenhard Ng, Steven Sivek, and Renato Vianna for useful conversations. We would like to thank Igor Uljarevi\'c  who helped us to prove Lemma \ref{lemma normals}. BA and AM would like to thank the Oberwolfach Research Institute for Mathematics for hosting them during an earlier stage of this project. OC-S is supported by NSF Graduate Research Fellowship under grant no. DGE-1644868. AG was partially supported by Wallenberg grant no. KAW 2016-0440 and the Fondation Math\'ematique Jacques Hadamard. AM is partially supported by Ministry of Education and Science of Republic of Serbia, project ON174034.
LS is supported by NSF grant no. DMS 1904074 and DMS 2042345. AW is supported by the Engineering and Physical Sciences Research Council [EP/L015234/1], the EPSRC Centre for Doctoral Training in Geometry and Number Theory (The London School of Geometry and Number Theory), University College London.

%% file: parts/weinstein.tex

\section{Weinstein Domains and Kirby Calculus} \label{section weinstein}

Here we provide a brief introduction to Weinstein structures and diagrams. In Section~\ref{sec:weinstein_handle} we review some fundamental definitions and properties of Weinstein structures. In Section~\ref{s:kirby} we give a brief introduction to Weinstein handlebody diagrams and their Kirby calculus moves. In section~\ref{s:cotangentbackground} we review the example of the canonical Weinstein structure on the cotangent bundle and related definitions which are crucial throughout the article. For further background, we refer to \cite{Weinstein, CE, ACGMMSW1}. 

\subsection{Background on Weinstein domains}~\label{sec:weinstein_handle}

A \emph{Liouville vector field} $Z$ for a symplectic manifold $(W,\omega)$ is a vector field satisfying $\mathcal{L}_Z \omega = \omega$, and the $1$-form $\lambda = \iota_Z\omega$ which satisfies $d\lambda=\omega$ is called the \emph{Liouville form}. When the Liouville vector field on a manifold with boundary is transverse to the boundary it defines a contact structure on the boundary manifold and can be used to identify a collared neighborhood of the boundary with a piece of the symplectization of that contact boundary. This allows symplectic manifolds with boundary that are equipped with such Liouville vector fields to be glued together along their contact boundaries. One of the two pieces being glued must have the Liouville vector field pointing outward from the boundary (i.e. the boundary is \emph{convex}) and the other must have the Liouville vector field pointing inward from the boundary (i.e. the boundary is \emph{concave}). 

A symplectic manifold with boundary, equipped with a Liouville vector field pointing transversally outward from the boundary is called a \emph{Liouville domain}. Suppose $(W,\omega, Z)$ is a Liouville domain, with induced contact boundary $(\partial W, \xi=\ker(\alpha))$ where $\alpha = i^*\lambda$ and $\lambda = \iota_Z\omega$. Then we obtain the \emph{cylindrical completion} of $(W,\omega, Z)$ by gluing onto $\partial W$ the portion of the symplectization $(\R\times \partial W, d(e^t\alpha), \partial_t)$ where $t>0$. The cylindrical completion of a Liouville domain is a non-compact \emph{Liouville manifold} with a complete Liouville vector field.

Weinstein defined a model of a handle decomposition for symplectic manifolds in which he equipped the handle with a Liouville vector field so that the gluing of the handle attachment could be performed using only contact information on the boundary. The model Weinstein handle $H$ of index $k$ in dimension $2n$ for $k\leq n$ is a subset of $\R^{2n}$ with coordinates $(x_1,y_1,\cdots, x_n,y_n)$, with the standard symplectic structure $\omega = \sum_j dx_j\wedge dy_j$ and Liouville vector field
$$Z_k = \sum_{j=1}^k  \left( -x_j\partial_{x_j}+2y_j\partial_{y_j} \right) + \sum_{j={k+1}}^n  \left(\frac{1}{2}x_j\partial_{x_j}+\frac{1}{2} y_j\partial_{y_j}\right).$$ 
This Liouville vector field is \emph{gradient-like} for the Morse function 
$$\phi_k = \sum_{j=1}^k \left(-\frac{1}{2} x_j^2 + y_j^2\right) + \sum_{j={k+1}}^n \left(  \frac{1}{4}x_j^2+\frac{1}{4}y_j^2 \right).$$

\begin{figure}
	\centering
	\includegraphics[scale=.5]{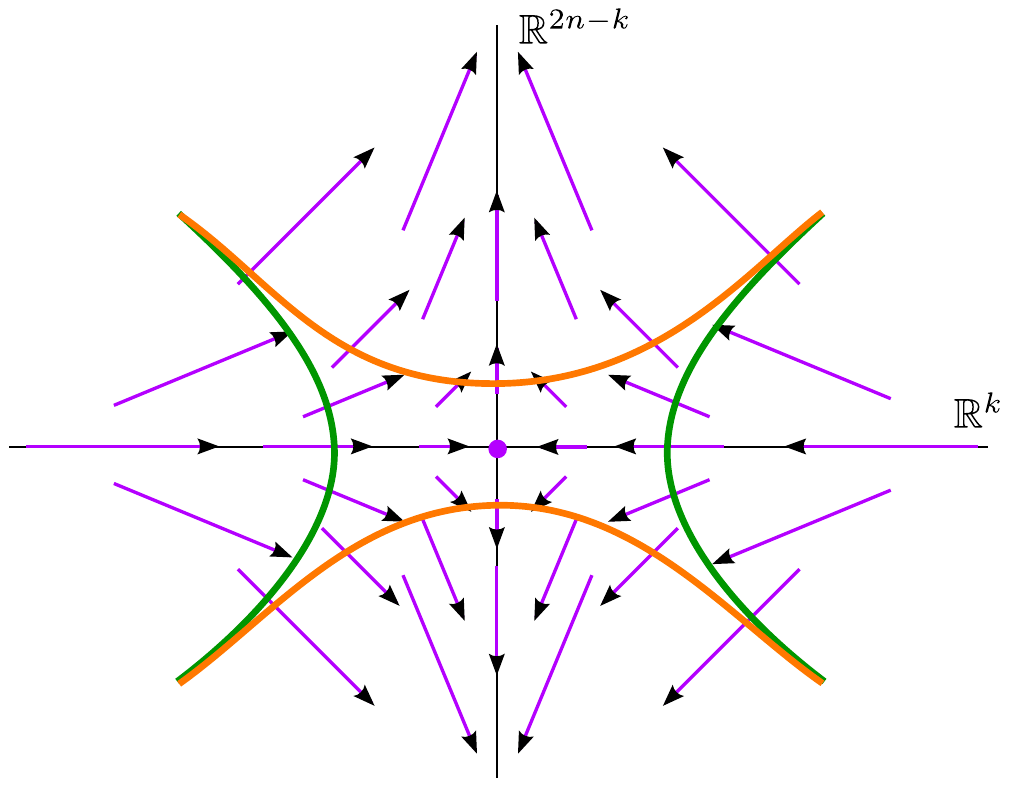}
	\caption{The model for the Weinstein handle $H$ of index $k$, with the associated gradient-like Liouville vector field $Z_k$.}
	\label{fig:model}
\end{figure}

The subset $H$ of $\R^{2n}$ is as in Figure~\ref{fig:model}, such that $\partial H$ decomposes into two pieces $\partial_- H$ and $\partial_+ H$ where $Z_k$ points inward along $\partial_-H$ and outward along $\partial_+H$ (with corners where $\partial_-H$ meets $\partial_+ H$). Then, $\partial_-H$ serves as the \emph{attaching region} for the handle $H$, which has concave boundary. This attaching region contains at its core, the \emph{attaching sphere} $L=\R^k_{(x_1,\dots, x_k)}\cap \partial_- H$. In fact, using the contact structure induced by the Liouville form on $\partial_-H$, the attaching sphere $L$ is contact isotropic (Legendrian in the case $k=n$), and $\partial_-H$ is a standard contact neighborhood of $L$. If we have a symplectic manifold $W$ with boundary and a Liouville vector field defined near $\partial W$ pointing outward (convexly), then we can attach the handle $H$ by gluing $\partial_-H$ to a standard neighborhood of a contact isotropic/Legendrian sphere in $\partial W$ (with its induced contact structure). The result is a new symplectic manifold with boundary $W'$, which also has convex boundary. 

In particular, we can build a manifold entirely from a union of Weinstein's model handles, starting with $0$-handles (which have empty attaching region) and gluing on handles sequentially (typically in increasing order of index $k$). After each handle attachment, the resulting manifold has a globally defined Liouville vector field, $Z$, pointing outward along the resulting boundary. Furthermore, the vector field is gradient-like for a global function $\phi:W\to \R$ which has exactly one critical point of index $k$ inside each $k$-handle. 

A manifold that is built from sequentially gluing on Weinstein handles is called a \emph{Weinstein domain}. We can also perform a cylindrical completion obtaining a \emph{Weinstein manifold}. Each Weinstein domain/manifold can be encoded by a \emph{Weinstein structure} which is a tuple $(W,\omega, Z,\phi)$ where $W$ is the compact manifold with boundary, $\omega$ is the global symplectic structure, $Z$ is the Liouville vector field which points transversally outward to $\partial W$ and $\phi$ is the function for which $Z$ is gradient-like. Sometimes $\phi$ is assumed to be Morse, but in this article we also allow $\phi$ to be Morse-Bott. Note that $Z$ and $\omega$ induce a contact structure on $\partial W$ (positively co-oriented with respect to the boundary orientation).

Given $(W,\omega_0, Z_0,\phi_0)$ and $(W,\omega_1, Z_1,\phi_1)$, we say that these Weinstein domains/manifolds are \emph{Weinstein homotopic} if there is a family of Weinstein structures $(W,\omega_t,Z_t,\phi_t)$ connecting them. Note that $(W,\omega_t)$ gives a symplectic deformation relating $(W,\omega_0)$ and $(W,\omega_1)$. If $W$ is a Weinstein domain, it may be that $(W,\omega_0)$ and $(W,\omega_1)$ are not symplectomorphic even if they are Weinstein homotopic (e.g. their symplectic volumes may differ). However, if $(W,\omega_0)$ and $(W,\omega_1)$ are Weinstein manifolds (with the cylindrical completion), then if they are Weinstein homotopic, they are actually symplectomorphic. For this reason, we often compare Weinstein domains by adding the cylindrical completion to make them Weinstein manifolds.

\subsection{Weinstein Kirby calculus} \label{s:kirby}
For Weinstein domains of dimension $4$, each of the handles has index $0$, $1$, or $2$. We can assume there is a single $0$-handle when the domain is connected. The attaching region of a $1$-handle is $S^0\times \{0\} \subset S^0 \times \D^3$, and is drawn as a pair of $3$-balls in $S^3$. Note that $\D^4$ with $n$ $1$-handles attached has boundary $\#^n( S^1\times S^2)$ with its standard contact structure. The attaching sphere of a $2$-handle is a Legendrian embedded circle (a Legendrian knot). The $4$-dimensional $2$-handle attachment is determined by the knot together with a framing, but in the Weinstein case, the framing is determined by the contact structure. More specifically, the contact planes along a Legendrian knot determine a framing by taking a vector field transverse to the contact planes. The contactomorphism gluing the attaching region of the $2$-handle to the neighborhood of the Legendrian identifies the product framing in the $2$-handle with the $tb-1$ framing. Here $tb$ denotes the contact framing (Thurston-Bennequin number), which is identified with an integer by looking at the difference between the contact framing and the Seifert framing (this must be appropriately interpreted when the Weinstein domain contains $1$-handles--see \cite{Gompf}). 

The Weinstein handlebody diagram we draw should specify the Legendrian attaching knots in $S^3$ along with the pairs of $3$-balls indicating the attachments of the $1$-handles. By removing a point away from these attachments, we reduce the picture in $S^3$ to a picture in $\R^3$. After a contactomorphism, the contact structure on $\R^3$ is $\ker(dz-ydx)$ in coordinates $(x,y,z)$. The \emph{front projection} is the map $\Pi: \R^3\to \R^2$ with $\Pi(x,y,z)=(x,z)$. A Legendrian curve in this contact structure is tangent to the contact planes, which happens precisely when the $y$-coordinate is equal to the slope $\frac{dz}{dx}$ of the front projection. Therefore, a Legendrian knot can be recovered from the diagram of its front projection with the requirement that the diagram has no vertical tangencies (instead it has cusp singularities where the knot is tangent to the fibers of the projection) and the crossings are always resolved so that the over-strand is the strand with the more negative slope (we orient the $y$-axis into the page to maintain the standard orientation convention for $\R^3$ so the over-strand is the strand with a more negative $y$-coordinate). In these front projections, the contact framing $tb$ can be computed combinatorially in terms of the oriented crossings and cusps of the diagram, when the diagram is placed in Gompf standard form, where the pairs of $3$-balls giving the attaching regions of $1$-handles related by a reflection across a vertical axis. Namely, $tb$ of a Legendrian knot $L$ in the front projection, $\Pi(L)$ can be computed by the formula
$$tb(L)=\text{writhe}(\Pi(L))-(\text{the number of right cusps in }\Pi(L)).$$

The set of moves that relate Weinstein handlebody diagrams in Gompf standard form for Weinstein homotopic Weinstein domains include Legendrian isotopy moves listed in~\cite{Gompf} which are the usual three Legendrian Reidemeister moves shown in Figure~\ref{Reid}, and the three moves we call \emph{Gompf moves $4,5$, and $6$} shown in Figure~\ref{Gompf456}. Weinstein handlebody diagrams are also equivalent if they are related by handle slides, handle pair cancellations/additions~\cite{DingGeiges}. Given two $k$-handles, $h_1$ and $h_2$, a \emph{handle slide} of $h_1$ over $h_2$ is given by isotoping the attaching sphere of $h_1$, and pushing it through the belt sphere of $h_2$. In the Weinstein context, the isotopy must be a Legendrian (or isotropic) isotopy. Diagrammatically, to perform a $2$-handle slide of $h_1$ over $h_2$ one takes the attaching spheres $S_1$ of $h_1$ and $S_2$ of $h_2$ and replaces $S_1$ by a connected sum of $S_1$ and $S_2$ with cusps as in Figure~\ref{2hslide}. See~\cite{ACGMMSW1} or~\cite{GompfStipsicz} for more information on Weinstein $1$-handle slides. A 1-handle $h_1$ and a 2-handle $h_2$ can be deleted or inserted together, provided that the attaching sphere of $h_2$ intersects the belt sphere of $h_1$ transversely in a single point. We call this a \emph{handle cancellation} or adding a \emph{canceling pair}. See~\cite{DingGeiges} for further details.

\begin{figure}
	\begin{center}
		\includegraphics[width=6cm]{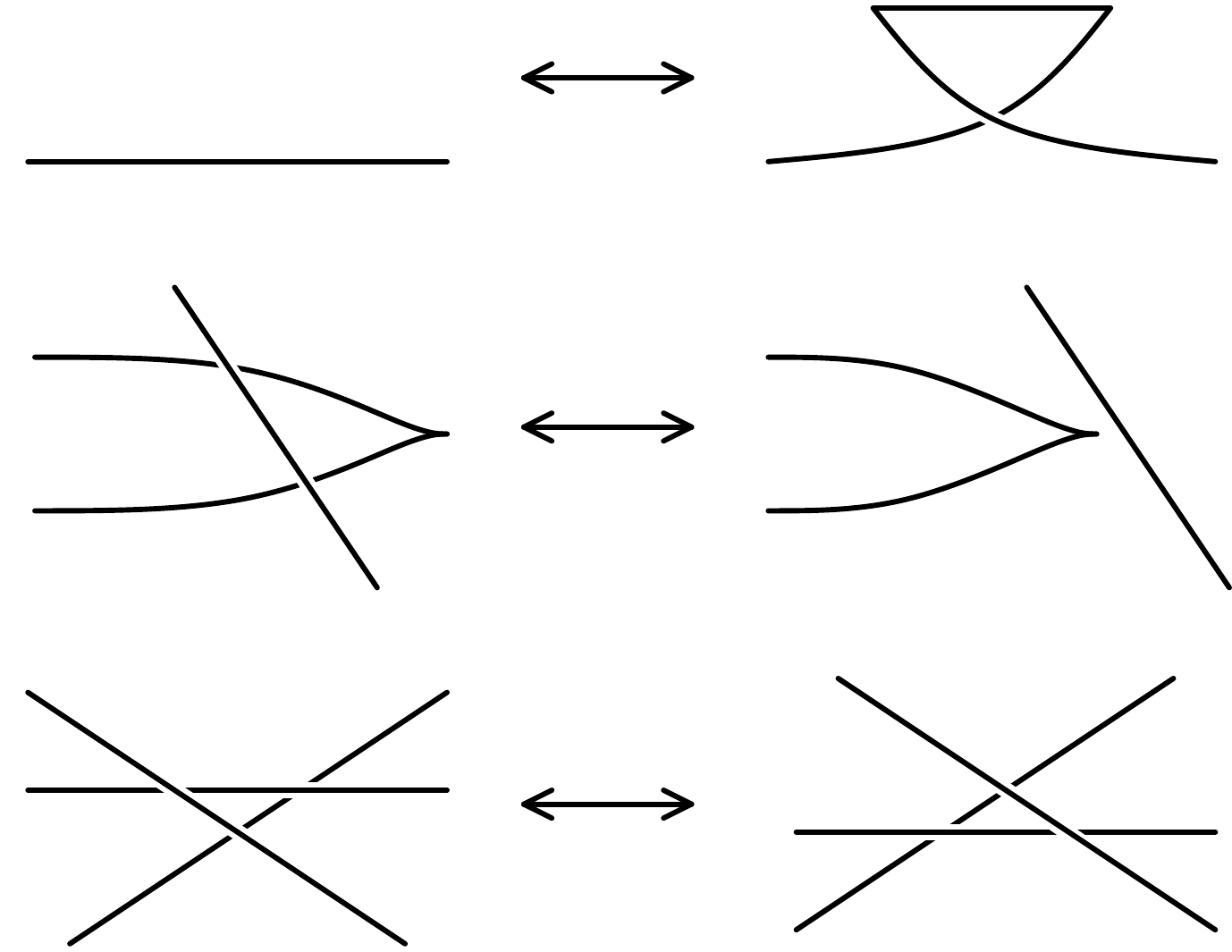}
		\caption{The Legendrian Reidemeister moves up to 180 degree rotation about each axis.}
		\label{Reid}
	\end{center}
\end{figure}

\begin{figure}
	\begin{center}
		\includegraphics[width=12cm]{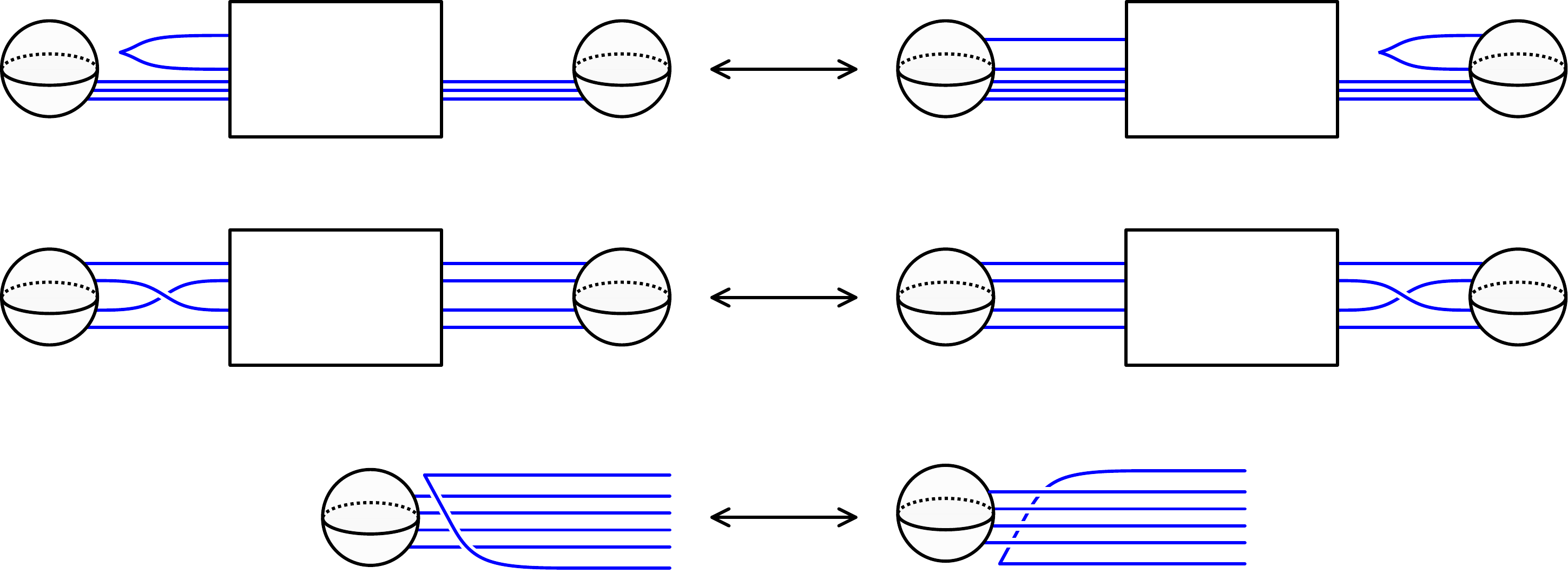}
		\caption{Gompf's three additional isotopic moves, up to 180 degree rotation about each axis.}
		\label{Gompf456}
	\end{center}
\end{figure}

\begin{figure}
	\begin{center}
		\includegraphics[width=5cm]{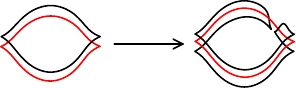}
		\caption{A 2-handle slide of the black unknot over the red unknot.}
		\label{2hslide}
	\end{center}
\end{figure}

\subsection{Background on Weinstein structures on the cotangent bundle} \label{s:cotangentbackground}

A key example of a Weinstein manifold is the cotangent bundle of a smooth manifold $F$. Since we are working in dimension $4$, we focus on the case where $F$ is a surface. With respect to local coordinates $(q_1,q_2)\in F$ and dual coordinates $(p_1,p_2)\in T_q^{*}F$, the canonical Liouville $1$-form is defined to be
$$\lambda_{can} = p_1dq_1+p_2dq_2.$$
(This can also be defined in a coordinate free manner.) The canonical symplectic structure is $\omega_{can} = d\lambda_{can}$ which in local coordinates is $dp_1\wedge dq_1 +dp_2\wedge dq_2$. The Liouville vector field $Z_{can}$ associated with $\lambda_{can}$ is determined by $\iota_{Z_{can}}\omega_{can} = \lambda_{can}$. In local coordinates, this canonical Liouville vector field is given by
$$Z_{can} = p_1\partial_{p_1} + p_2\partial_{p_2}.$$
This Liouville vector field is gradient-like for the radial function $\phi:T^*F\to \R$ defined in coordinates by $\phi(q_1,q_2,p_1,p_2)=p_1^2+p_2^2$. This function is Morse-Bott rather than Morse, but we allow this within our notion of Weinstein structures.

Observe that $Z_{can}$ vanishes along the zero section where $p_1=p_2=0$, and points radially outward in each cotangent fiber. The cotangent disk bundle is a Weinstein subdomain $D^*F =\{p_1^2+p_2^2\leq 1\}$. The boundary is the unit cotangent bundle $S^*F=\{p_1^2+p_2^2=1\}$, which inherits a contact structure because $Z_{can}= p_1\partial_{p_1}+p_2\partial_{p_2}$ is outwardly transverse to this boundary. More explicitly, we can describe this unit cotangent bundle where $p_1^2+p_2^2=1$ by coordinates $(q_1,q_2,\theta)$ where $p_1 = \cos\theta$ and $p_2 = \sin\theta$. Then the contact form is obtained by restricting the Liouville form $\lambda_{can}$ to the boundary, which in coordinates is:
$$\alpha = \cos\theta~dq_1 +\sin\theta~dq_2.$$

\begin{remark}
	Note that we did not necessarily need to make the boundary so regular. In fact, as long as the boundary of our domain is described by equations $p_1 = f(q_1,q_2,\theta)\cos\theta$ and $p_2 = f(q_1,q_2,\theta) \sin\theta$ for a positive real valued function $f$, the boundary remains transverse to the canonical Liouville vector field, so the Liouville form induces a contact form
	$$\alpha = f(q_1,q_2,\theta)\cos\theta~dq_1 +f(q_1,q_2,\theta) \sin\theta~dq_2$$
	inducing the same contact structure $\xi=\ker(\alpha)$.
\end{remark}

Given a co-oriented curve $\gamma \sse F$, we can define an associated Legendrian curve in $\mathcal S^*F$ as follows. 
\begin{definition} \label{def conormal} If $\gamma$ is a co-oriented curve on a surface $F$, with co-orientation corresponding to a choice of a normal vector $n$, then the oriented conormal Legendrian lift of $\gamma$ is
	$$\Lambda_{\gamma}=\{ (q,p) \in \mathcal S^{*}F~|~ q\in \gamma,\; p(v)=0 \;\forall v\in T_{q}\gamma,~\text{and}~p(n)>0\}.$$
\end{definition}

%% file: parts/toric.tex

\section{Toric transformations}\label{section toric}

The purpose of this section is to introduce the basics of toric geometry, with an emphasis on dimension 4. We also give a brief introduction to almost toric manifolds. For more details on toric manifolds from the symplectic point of view, we refer to~\cite{CdS-toric, Audin} and for the algebraic geometry point of view, we refer to~\cite{Cox}. For almost toric manifolds, see~\cite{Symington, Evans}.

\subsection{The symplectic viewpoint.}\label{subsection 3.1} A \textit{toric manifold} is a symplectic manifold $(M^{2n},\omega)$ equipped with an effective Hamiltonian  $n$-dimensional torus action. This Hamiltonian action induces the existence of a moment map which plays a crucial role in the study of toric manifolds. 
After identifying the Lie algebra dual of the $n$-torus with $\mathbb R^n,$ the moment map is given by  $$\Phi=(\Phi_1,\ldots,\Phi_n):M\rightarrow\mathbb{R}^n$$
where  $\Phi_k,$ $k=1,\ldots, n,$ are Hamiltonian functions for the infinitesimal generators of the action. That is, they are defined up to constants by equations:
$$\iota_{X_k}\omega=-d\Phi_k,$$
where the vector fields $X_k,$ for $k=1,\ldots,n,$ generate the action of the $k$-th coordinate circle in $T^n.$

The standard example of a toric manifold is $(\mathbb{C}^n,\omega_{st}=\frac{i}{2}dz\wedge d\bar{z})$ with the $T^n$-action
$$(t_1,\ldots,t_n)*(z_1,\ldots,z_n)\mapsto(t_1z_1,\ldots,t_nz_n).$$
The generators of this action are Hamiltonian vector fields $X_k=i\left(z_k\frac{\partial}{\partial z_k}-\bar{z}_k\frac{\partial}{\partial \bar{z}_k}\right),$ for $k=1,\ldots,n,$
and the moment map is  $\Phi(z_1,\ldots, z_n)=\frac{1}{2}(|z_1|^2,\ldots,|z_n|^2)+const.$ The image of the moment map is a cone spanned by edges given by directions $e_i=(0,\ldots, 1,\ldots, 0),$ 
for any $i\in\{1,\ldots, n\}$, and the vertex of the cone is the moment map image of the fixed point of the action $(0,\ldots,0)\in\mathbb C^n.$ We call it \emph{a standard cone} (see Figure \ref{both sl} on the left).
If we reparametrise the torus $T^n$ acting on $\mathbb C^n$, then the corresponding moment map image of $\mathbb C^n$ is a new cone related to the standard cone by a transformation $G\in SL(n,\mathbb Z)$, where  $G^T\in SL(n,\mathbb Z)$ denotes the reparametrisation of the torus. More precisely, the edges of the new cone is given by directions $Ge_i,$ for any $i\in \{1,\ldots, n\}$ and, consequently, the inward normal vectors of the facets (a facet is an $n$-dimensional face of the cone) of the new cone are obtained by $(G^{-1})^T\in SL(n,\mathbb Z)$ transformation of the inward normal vectors of the standard cone. This can be seen using the properties of the inner product $\langle (G^{-1})^Tv,Ge\rangle=\langle v,G^{-1}Ge\rangle=\langle v,e\rangle,$ where $v\in\mathbb Z^2$ is an inward normal vector to an edge given by a direction $e\in\mathbb Z^2.$ We illustrate these transformations in Figure \ref{both sl} where $n=2$ and the action of $T^2$ on $\mathbb C^2$ is given by a matrix $G^T = \begin{pmatrix}
		a & b \\
		c & d
	\end{pmatrix}\in SL(2,\mathbb Z),$ i.e. the toric action is precisely $(t_1,t_2)*(z_1,z_2)\mapsto(t_1^at_2^bz_1,t_1^ct_2^dz_2).$

\begin{figure}
	\centering
	\includegraphics[width=14cm]{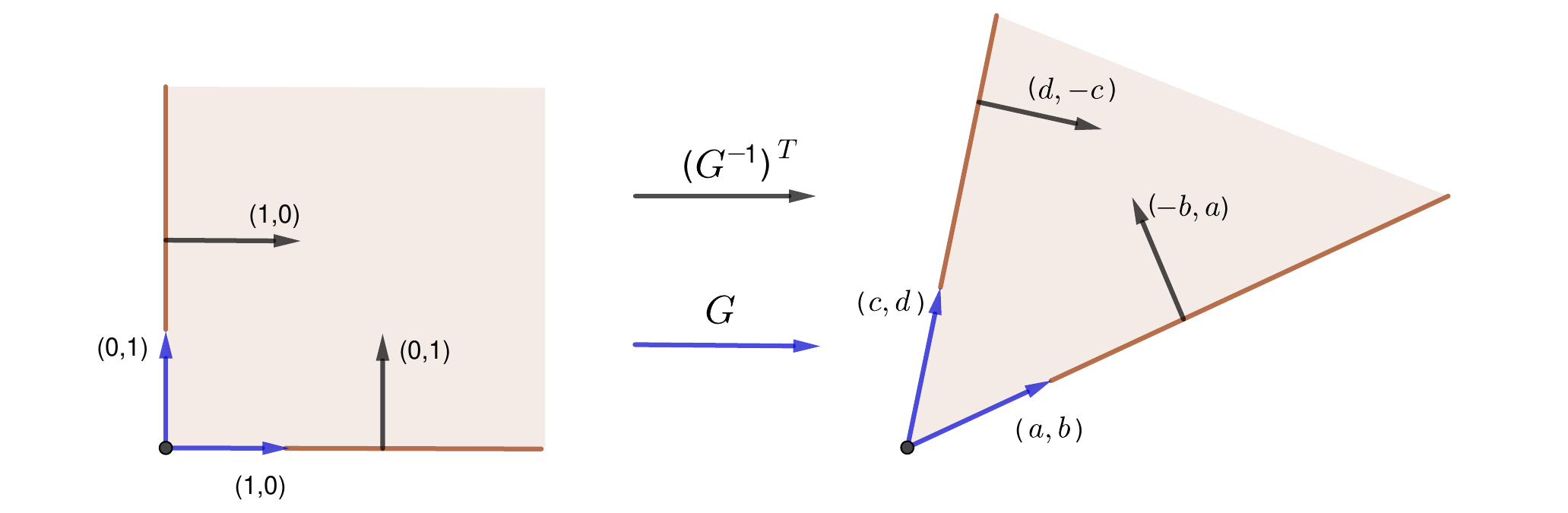}
	\caption{A moment map image of $\mathbb C^2:$ under the standard toric action- on the left,  under the 
	$G^T$-reparametrised toric action-on the right.}
	\label{both sl}
\end{figure}   
	
The equivariant Darboux theorem states that the neighborhood of any fixed point in a toric manifold $(M,\omega)$ is equivariantly symplectomorphic to $(\mathbb{C}^n,\omega_{st}=\Sigma_{j=1}^n\frac{i}{2}dz_j\wedge d\bar{z_j})$ with the standard toric action. That is, there is a  symplectomorphism that intertwines the standard toric action on $\mathbb{C}^n$ and the toric action on the neighborhood of a fixed point in $M.$ There is also a $G\in SL(n,\mathbb{Z})$ transformation between the images of the moment maps and $(G^{-1})^T$
transformation of the corresponding inward normal vectors of the facets of the cones. Therefore, the moment map image of a neighborhood of any fixed point is a convex cone and since an $SL(n,\mathbb{Z})$ transformation preserves a determinant, the inward normal vectors of the facets of the cone form a $\mathbb Z^n$-basis.

If $M$ is compact, then the moment map image $\Phi(M)\subset\mathbb{R}^{n}$ is a closed convex polytope (\cite{Ati}, \cite{GS}) whose vertices are the images of fixed points. 
Moreover, this polytope is
\begin{itemize}
	\item simple (there are $n$ edges meeting in every vertex),
	\item rational (all inward normal vectors to the facets are primitive vectors in $\mathbb Z^n$), and
	\item smooth (inward normal vectors of the facets meeting in a vertex form a $\mathbb Z^n$-basis).
\end{itemize}
According to the Delzant classification, every such polytope is a moment map image of a unique toric manifold, up to equivariant symplectomorphism. For that reason such a polytope is called a \emph{Delzant polytope}.

\begin{example}\label{example cp2} The complex projective space $\mathbb{CP}^n$ with the Fubini-Studi symplectic form $\omega_{FS}$ and the $T^n$ action induced from the standard $T^{n+1}$ action on $\mathbb{C}^{n+1}$ is a toric manifold. More precisely, since $(\mathbb{CP}^n,\omega_{FS})$ is a symplectic reduction of $\mathbb{C}^{n+1}$ with respect to the diagonal circle action that commutes with the toric action on $\mathbb{C}^{n+1}$, there is an induced $T^n$ action on the reduced space. The moment map image in $n=2$ case is depicted on the left in Figure \ref{fig2}.
\end{example}
 
\begin{figure}
	\centering
	\includegraphics[width=9cm]{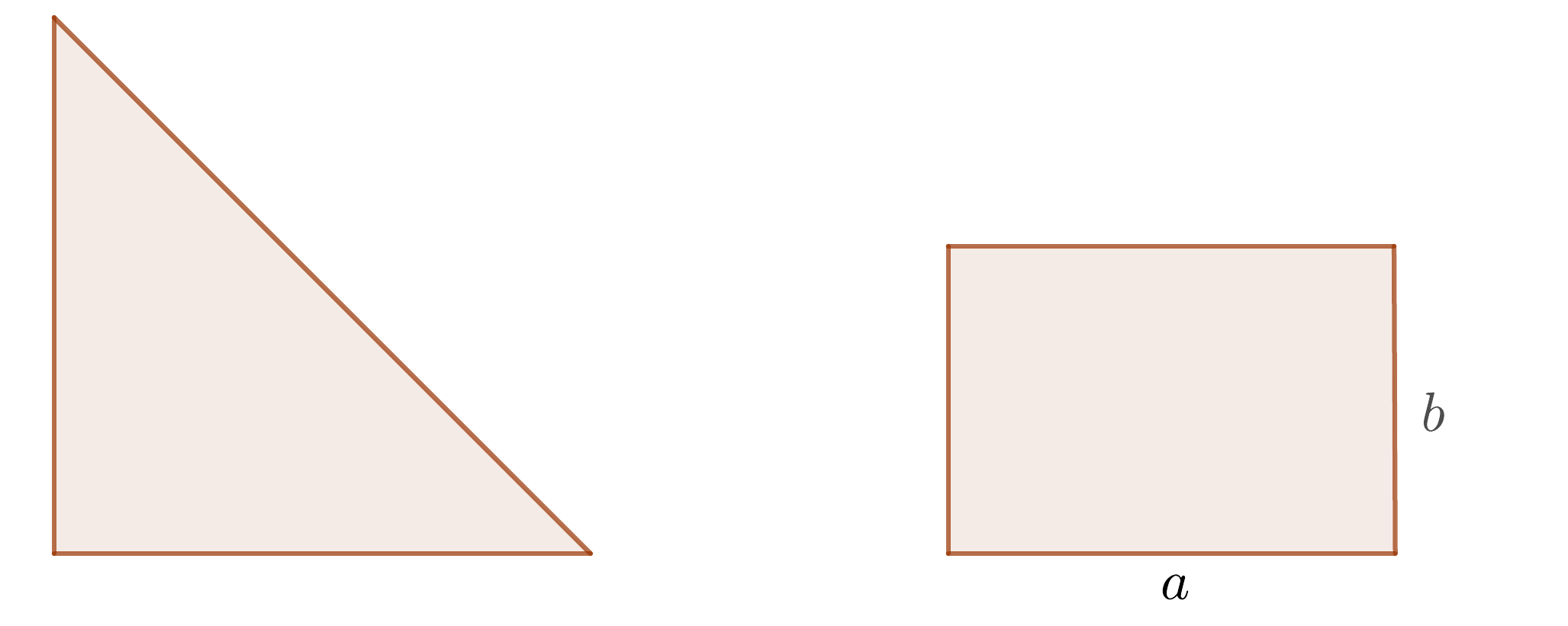}
	\caption{On the left: The moment map image of $(\mathbb{CP}^2,\omega_{FS})$. On the right: The moment map image of $(\mathbb{CP}^1\times\mathbb{CP}^1,\omega_{a,b}).$}
	\label{fig2}
\end{figure}

\begin{example}\label{example cp1xcp1} Consider $\mathbb{CP}^1\times\mathbb{CP}^1$ with the symplectic form $\omega_{a,b}=a\pi_1^*\omega_{st}+b\pi_2^*\omega_{st},$ where $ab>0,$ 
$\pi_k:\mathbb{CP}^1\times\mathbb{CP}^1\rightarrow \mathbb{CP}^1$ is the projection of the $k$-th component and $\omega_{st}$ is the standard area form on $\mathbb{CP}^1\cong S^2$
(recall that $\omega_{st}=d\theta\wedge dh$ in cylindrical coordinates away from the poles and that $\omega_{FS}=\frac{1}{4}\omega_{st}$). The $T^2$-action on $(\mathbb{CP}^1\times\mathbb{CP}^1,\omega_{a,b})$ given as the rotation of each sphere with the same speed is a toric action. The moment map in cylindrical coordinate is given by  $\Phi(\theta_1,h_1,\theta_2,h_2)=(ah_1,bh_2)$ and the moment polytope is depicted on the right in Figure \ref{fig2}.
\end{example}

\begin{example}\label{example blow up} A symplectic blow up of a toric manifold at a fixed point is also a toric manifold. The moment map image of a symplectic blow up is obtained by chopping off the corner of the original polytope at the vertex which is the image of the fixed point that we are blowing up. The inward normal vector of the new facet obtained after the chop is the sum of inward normal vectors of the facets meeting at the vertex where we performed the blow up. The distance of the new facet from the chopped vertex is given by the size of the blow up. See Figure \ref{fig3} for two examples of Delzant polytopes of symplectic blow ups of toric manifolds. Under the moment map, the exceptional divisor of the blow-up maps to the new facet.

\begin{figure}
	\centering
	\includegraphics[width=9cm]{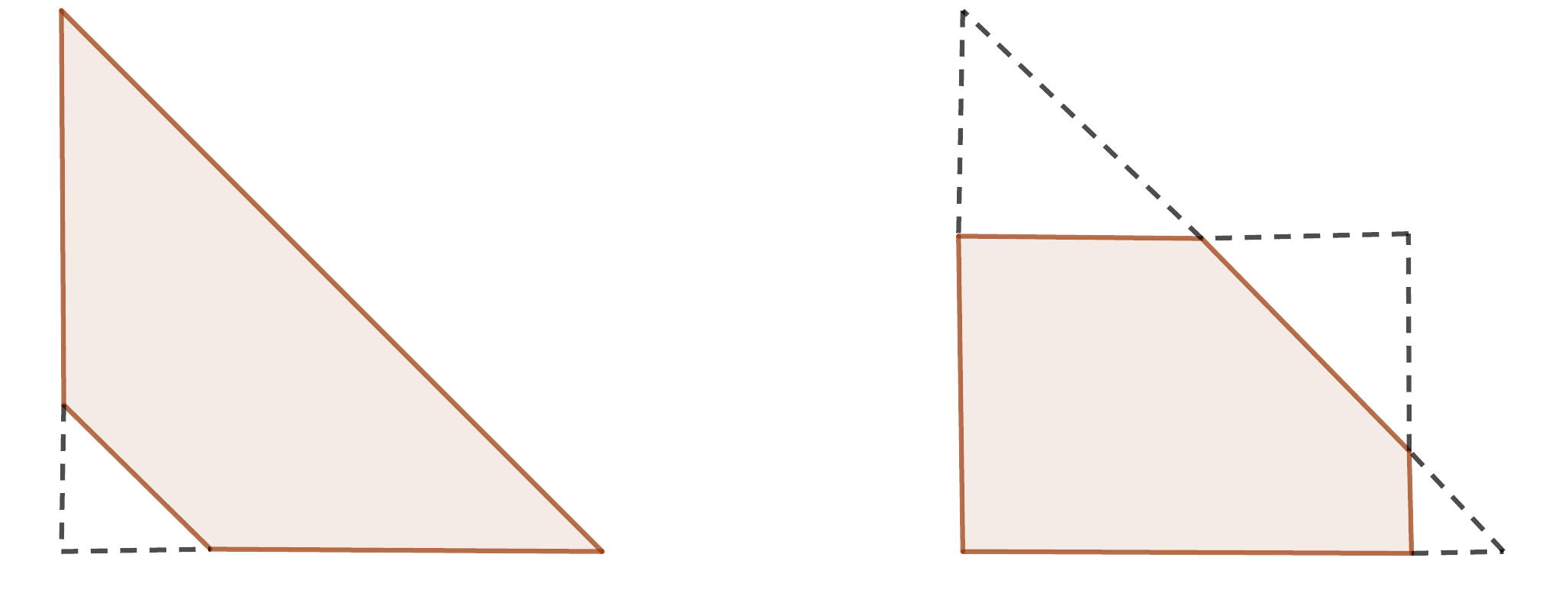}
	\caption{On the left: The moment map image of $\mathbb{CP}^2\sharp\overline{\mathbb{CP}^2}$. 
		On the right: The moment map image of $\mathbb{CP}^1\times\mathbb{CP}^1\sharp\overline{\mathbb{CP}^2},$ or $\mathbb{CP}^2\sharp\overline{2\mathbb{CP}^2}$.}
	\label{fig3}
\end{figure}       
\end{example}

Vector fields generating the toric action commute, i.e. they form a completely integrable system. Thus, a moment map is constant on every orbit and every orbit is isotropic with respect to $\omega.$ If $T^m$ is an effective Hamiltonian torus action on a symplectic manifold $(M^{2n},\omega)$, then $m\leq n$. Furthermore, every orbit with a $k$-dimensional stabilizer is mapped, under the moment map, to a point in a $k$-dimensional face of the moment polytope. In particular, free points are mapped to the interior of the polytope. Since all level sets of the moment map are connected \cite{Ati}, the preimage of a point in the polytope is exactly one orbit in $M$. Thus, in action-angle coordinates 
$$(p=(p_1,\ldots, p_n)\in\mathbb{R}^n,q=(q_1,\ldots, q_n)\in T^n),$$
 free orbits form a subset of $\mathbb{R}^n\times T^n\subset M$ where the moment map is simply given by $\Phi(p,q)=p$ and the symplectic form is $dp \wedge dq$. The pre-image of the interior of the polytope is symplectomorphic to a subset of the cotangent bundle $T^*T^n$ with the standard symplectic form. Similarly, the preimage of any $k$-dimensional face is symplectomorphic to a neighborhood of the zero section in the cotangent bundle $T^*T^k$ with the standard symplectic form.

\subsection{The algebraic viewpoint} Every compact toric manifold is a toric variety. Starting from the corresponding Delzant polytope $\Delta,$ mark each of the $k$ points in 
$\mathbb{Z}^n\cap \Delta$. Consider the toric variety defined as the closure\footnote{with respect to the Zariski topology} of the orbit of a point $[1:\cdots:1]\in \mathbb{CP}^{k-1}$ with respect to the action of the algebraic torus $(\mathbb{C}^*)^n$ with weights given by points in $\mathbb{Z}^n\cap \Delta.$ Being equivariantly embedded in $\mathbb{CP}^{k-1},$ this toric variety inherits a toric structure and a moment map image is precisely $\Delta.$ For more details, see \cite{CdS-toric}. \smallskip

Since toric manifolds are toric varieties we know that they admit prime divisors,  i.e. irreducible subvarieties of complex codimension one. A divisor in a toric variety is called a \emph{toric divisor} if it is invariant under the toric action. From the discussion above, we conclude that the preimage of any facet of the polytope is a toric divisor. Since it is a codimension $2$ symplectic submanifolds it is also a \emph{symplectic divisor}. The union of all toric divisors is the set of all points with nontrivial stabilizers and it is a connected singular subvariety.
The singular points are the intersection points of adjacent toric divisors and they are precisely the fixed points of the toric action. Throughout this article, we call this subvariety the \emph{total toric divisor}. As previously explained, the completion of the complement of a neighborhood of the total toric divisor is symplectomorphic to the cotangent bundle $T^*T^n$ with the standard symplectic form and a Weinstein structure.
On the other hand, the complement of a closed toric divisor may not admit a Weinstein structure. Namely, the complement of a toric divisor may still contain another closed toric divisor. Since a closed toric divisor has positive symplectic area, by Stokes' theorem the symplectic form on such a complement cannot be exact. Consequently, in this article we consider only total toric divisors.
\smallskip

An important property of toric varieties is their correspondence with fans. A \emph{fan} in $\mathbb R^n$ is a collection of cones of dimension $0,1,\ldots, n,$ where the $k$-dimensional cone is a convex span of $k$ linearly independent rays given by vectors in $\mathbb Z^n.$ Moreover, every face of a cone is again a cone in a fan, and the intersection of any two cones is a face of both cones. From an $n$-dimensional toric manifold one can obtain a fan in $\mathbb R^n$ by considering all the inward normal vectors to the facets of the polytope. The convex span of the vectors of the facets meeting in a vertex is precisely the $n$-dimensional cone of the fan. Moreover, from every fan one can construct a toric variety. A fan is \emph{complete} if the union of all cones is equal to the whole space $\mathbb{R}^n.$ A complete fan corresponds to a closed polytope and a compact toric variety.
A fan is \emph{regular} if all the vectors spanning $n$-dimensional cone form a $\mathbb{Z}^n$-basis and all the vectors spanning a lower dimensional cone can be completed to a 
$\mathbb Z^n$-basis. A regular, complete fan corresponds to a smooth compact polytope. See~\cite{Cox} for more information on fans and their associated polytopes.

\subsection{Toric 4-manifolds} Given any Delzant 2-dimensional polytope, there is a corresponding $SL(2,\mathbb{Z})$ transformation mapping that polytope to one of the two Delzant polytopes shown in Figure \ref{fig2} after some sequence of corner chops, possibly none. 
Note that chopping one corner in the rectangle is equivalent to chopping two corners in the triangle, as shown in Figure~\ref{fig3}. By the Delzant classification, every compact toric 4-manifold is therefore equivariantly symplectomorphic to $(\mathbb{CP}^2,\omega_{FS})$, 
 $(\mathbb{CP}^1\times\mathbb{CP}^1,\omega_{a,b})$, or to a sequence of symplectic blow ups of $(\mathbb{CP}^2,\omega_{FS})$. 
 
The pre-image of any edge of a 2-dimensional Delzant polytope is $\mathbb{CP}^1$. A total toric divisor of any toric 4-manifold is the union of transversally intersecting toric embeddings of $\cpone$ (with varying normal bundle data). For example, the total toric divisor of $(\mathbb{CP}^2,\omega_{FS})$ is a union of three $\cpone$'s intersecting at three transverse intersection points, which we can think of as a three-nodal cubic. In general, the number of singular points can be reduced by smoothing the divisor. If one smooths every singular point of the total toric divisor, it becomes a torus $T^2$.

\subsection{Almost toric manifolds}\label{intro to almost toric}
We now take a brief foray into the world of almost toric $4$-manifolds. This subsection can be skipped on a first reading. Almost toric manifolds generalize the notion of toric manifolds and also provide interesting examples of log-Calabi Yau pairs which we hope to further explore in future work. In the present article, almost toric manifolds make an appearance in Remark~\ref{remark one almost toric}, Example~\ref{remark counterexample}  and Section~\ref{section:CP2bu5}. 
	
An almost toric manifold is a symplectic manifold with an almost toric fibration, first introduced by Symington in~\cite{Symington} and classified in~\cite{LeungSymington}. We refer the reader to these articles and Evans' lecture notes~\cite{Evans}	for a thorough introduction.
	
\begin{definition}\label{defn:almosttoric}
	An \emph{almost toric fibration} of a compact symplectic $4$-manifold $(M,\omega)$ is a singular Lagrangian fibration $\pi$ over a surface $\mathcal{B}$  where the fibration either		
	\begin{itemize}
		\item is locally the moment map of a toric manifold with Lagrangian tori as regular 
		fibers and isotropic circles or points as singular fibers, or
		\item contains nodal singularities for which the singular fiber is a pinched torus. These singular fibers can be tori with multiple cycles pinched so long as these cycles are in the same homology class. (Nodal singularities are also called \emph{focus-focus} or \emph{of Lagrangian Lefschetz type}).
	\end{itemize}
\end{definition}

Similarly to the case of toric manifolds, the total space of an almost toric fibration can be reconstructed from its base $\mathcal{B}$ and some additional data. We consider only the almost toric fibrations constructed from toric manifolds where the base $\mathcal{B}$ of the almost toric fibration is always homeomorphic to a disk. In the presence of a nodal singularity, the local action-angle coordinates at a regular point in the neighborhood of the nodal singularity are not preserved when going around the singular value. The \emph{topological monodromy} encodes how the angle coordinates are transformed. It is described by the action of a matrix $A$ of $GL(2,\Z)$ on $H_1(F,\Z)$ of a regular fiber $F$ in the neighborhood of the singularity.	The \emph{affine monodromy} encodes the change of the action coordinates (more precisely the transformation of a so-called integral affine structure on the base). Because the change of coordinates is a symplectic transformation, the affine monodromy on the base is the transpose of the inverse $(A^{-1})^T$ of the topological monodromy. The study of integrable systems shows that if the nodal fiber is a torus where the circle of slope $(a,b)$ is pinched, then the topological monodromy along a clockwise oriented loop encircling the nodal singular value is
	$A_{(a,b)} = \begin{pmatrix}
		1-ab & a^2 \\
		-b^2 & 1+ab
	\end{pmatrix}$, see~\cite[Lemma 4.13]{Symington}.
On the base, we mark the change of coordinates by ``cutting'' the surface along a dotted segment from the nodal singular value to the boundary of the base.  When crossing these \emph{branch cuts} the local angle coordinates are equal up to applying the topological monodromy (or its inverse if we cross the cut along an anti-clockwise oriented loop), and the action coordinates are related by the affine monodromy. For a singular fiber, one dotted segment realizes the cut along an eigenline of the affine monodromy. In the example above, if the nodal fiber corresponds to pinching the circle of slope $(a,b)$, the affine monodromy is given by the matrix $(A_{(a,b)}^{-1})^T = A_{(-b,a)}$ with eigenline spanned by the vector $(-b,a)$.

\emph{Mutations} are important operations in the almost toric setting. Mutations are roughly the result of pushing a branch cut to the other side of the base so that the branch cut connects the nodal point to the other intersection point of the eigenline and the boundary of the base. The mutated base diagram encodes the same total space but gives a change in the choice of associated local coordinates. 	In practice, the mutated base diagram is obtained by cutting the base along the eigenline, keeping one half of the diagram invariant and gluing the other half after applying the affine monodromy (or its inverse depending on the choice of the half surface with respect to the cut). See the proof of Lemma~\ref{AThandle} and Section~\ref{section examples} for examples of mutations.

%% file: parts/handlesect.tex

\section{Weinstein structure on the complement of a partially smoothed total toric divisor}\label{section smoothing}

The main goal of this section is to show that there exists a Weinstein structure on the complement of the neighborhood of a  partially smoothed total toric divisor and to describe this Weinstein structure in terms of Weinstein handle attachments in an explicit way. In section~\ref{section:Smoothing}, we describe the smoothing locally in a standard model neighborhood of a node at a topological level. In Section~\ref{section:Topologycomplements}, we compare the smooth topology of the complement of the smoothing with the complement of the unsmoothed divisor, and see that they differ by a smooth $2$-handle for each node which is smoothed. In Section~\ref{section:2handleWeinstein} we consider how to view this smooth $2$-handle symplectically, by identifying Lagrangian disks which can serve as the core and co-core of the $2$-handle. We focus on a local model, and explain how to transform this local model to any vertex using the Delzant properties of the corresponding polytope. This allows us to describe the attaching sphere of the Weinstein $2$-handle in terms of the slope we associate to the vertex in the moment polytope corresponding to the chosen node. Finally in Section~\ref{section:$2$-handle attachment}, we prove our main result, Theorem~\ref{thm:toricWein} by showing how to put together these local models to a global Weinstein structure. The $\{V_1,\dots, V_k\}$-centeredness condition ensures that the Weinstein structure we impose on the model complement of the smoothed node can be glued in compatibly at all of the nodes associated with vertices $V_1,\dots, V_k$ simultaneously. In Section~\ref{section:almosttoric}, we discuss an extension of our main theorem to the almost toric setting.

\subsection{A local model of smoothing one node.} \label{section:Smoothing}

In a $4$-dimensional manifold~$M$, the smoothing of a ($\omega$-orthogonal) node with two symplectic branches, $\Sigma_1$ and~$\Sigma_2$, has the following local model. There exists a local Darboux chart~$(\C^2, \omega_{\std})$ at the node such that in this chart the image of the intersection is the origin, the divisor~$\Sigma_1$ corresponds to $\left \{{(z_1,z_2) \in \C^2 \, | \, z_2 = 0} \right\}$, and the divisor~$\Sigma_2$ corresponds to $\left \{{(z_1,z_2) \in \C^2 \, | \, z_1 = 0} \right\}$. The smoothing of the node corresponds to locally substituting the model nodal divisor $\left \{{(z_1,z_2) \in \C^2  \, |  \, z_1 \cdot z_2 = 0} \right\}$ with the hypersurface $\Sigma = \left \{{(z_1,z_2) \in \C^2  \, | \,  z_1 \cdot z_2 = \varepsilon^2} \right\}$ for some $\varepsilon > 0$.

In order to patch this model in, it is useful to observe that this smoothing is the orbit space under the circle action $e^{i\theta}\cdot(z_1,z_2)=(e^{i\theta}z_1,e^{-i\theta}z_2)$ of the real curve $\{(x_1,x_2)\mid x_1,x_2\in \R, \; x_1x_2=\varepsilon^2 \}$. In toric coordinates, the smoothing is described by the equations $\{p_1p_2=\frac{1}{4}\varepsilon^4 \textrm{ and } q_1=-q_2\}$. Using this toric perspective, we modify the smoothing so that it is supported in an arbitrarily small neighborhood of the node as follows.

Let $\gamma(t)=(\gamma_1(t),\gamma_2(t))$ be a smooth curve in $\R^2$ as in Figure~\ref{fig:smooth} with the following properties
\begin{itemize}
	\item $\gamma(t)=(0,-t)$ for $t\leq -3\varepsilon$,
	\item $\gamma(t)=(t,0)$ for $t\geq 3\varepsilon$,
	\item $\gamma_1(t)\gamma_2(t)=\varepsilon^2$ for $-2\varepsilon<t<2\varepsilon$, and
	\item $\gamma_1'(t)>0$ for $t>-3\varepsilon$ and $\gamma_2'(t)<0$ for $t<3\varepsilon$.
\end{itemize}

\begin{figure}
	\centering
	\includegraphics[scale=.35]{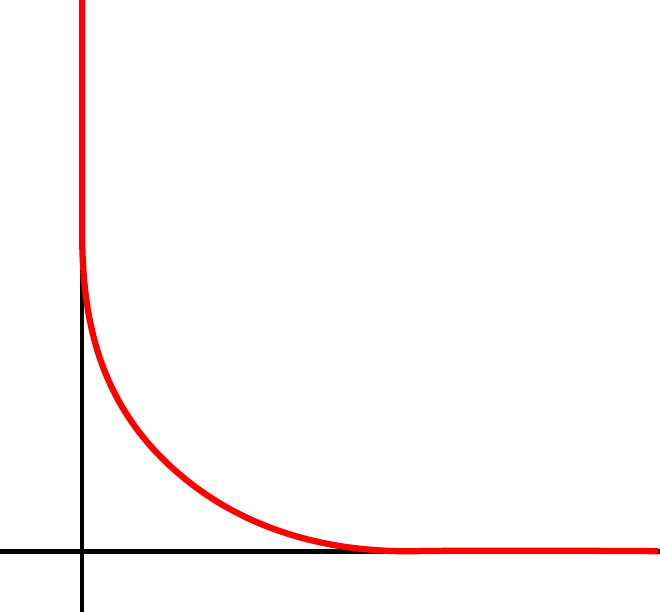}
	\caption{A curve $\gamma(t)$ in $\R^2\subset \C^2$ which is used in the construction of the model smoothing. (We are not viewing this as a moment map image.)}
	\label{fig:smooth}
\end{figure}

Let $\Sigma$ be the orbit space of $\gamma$ under the circle action. In toric coordinates, we can think of $\Sigma$ as the set of points where $p_1=\gamma_1(t)$, $p_2=\gamma_2(t)$, and $q_1=-q_2$. Note that outside of a $3\varepsilon$ neighborhood of the origin, $\Sigma$ agrees with the original model nodal divisor, and inside a smaller neighborhood it agrees with the complex smoothing $\{z_1z_2=\varepsilon^2\}$.

At the topological level, the smoothing replaces two transversally intersecting disks with an annulus. As $\varepsilon \to 0$, the core circle of the annulus shrinks to a point which is the intersection point $0$ of the transverse disks in the model nodal divisor.

\subsection{The topology of the complements.} \label{section:Topologycomplements}

We want to compare the complement of a neighborhood of the nodal divisor to the complement of a neighborhood of the smoothing. The nodal divisor can be viewed as the preimage of the two axes under the toric moment map. Therefore, a regular neighborhood of the nodal divisor is locally given by the preimage of a regular neighborhood of the axes under the toric moment map. In particular, a regular neighborhood of the nodal divisor can be chosen so that its complement is the preimage of the shaded region in Figure~\ref{fig:compl} under the moment map. Denote this preimage by $U$. Then, when $\varepsilon$ is sufficiently small, $U$ is disjoint from both the nodal divisor and its $\varepsilon$-smoothing. 

\begin{figure}
	\centering
	\includegraphics[scale=.35]{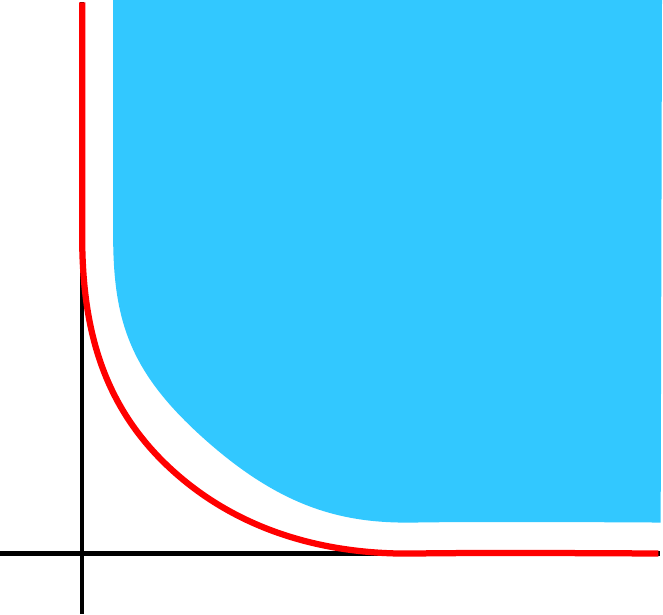}
	\caption{The complement $U$ of a neighborhood of the nodal divisor is the preimage of the blue shaded region under the moment map image.}
	\label{fig:compl}
\end{figure}

The regular neighborhood of the smoothing that we are considering has a different shape in this local neighborhood, because the smoothing does not fill the entire preimage of the curve $\gamma(t)$ under the toric action, but only the portions where $q_1=-q_2$. Therefore, this regular neighborhood can be thought of as the union of the neighborhoods of the curve $q_1=-q_2$ in the tori that lie over a neighborhood of $\gamma(t)$. As we move over points further from $\gamma(t)$, the neighborhood of the curve $q_1=-q_2$ in the fibers of these points shrinks and eventually disappears.

\begin{figure}
	\centering
	\includegraphics[scale=.5]{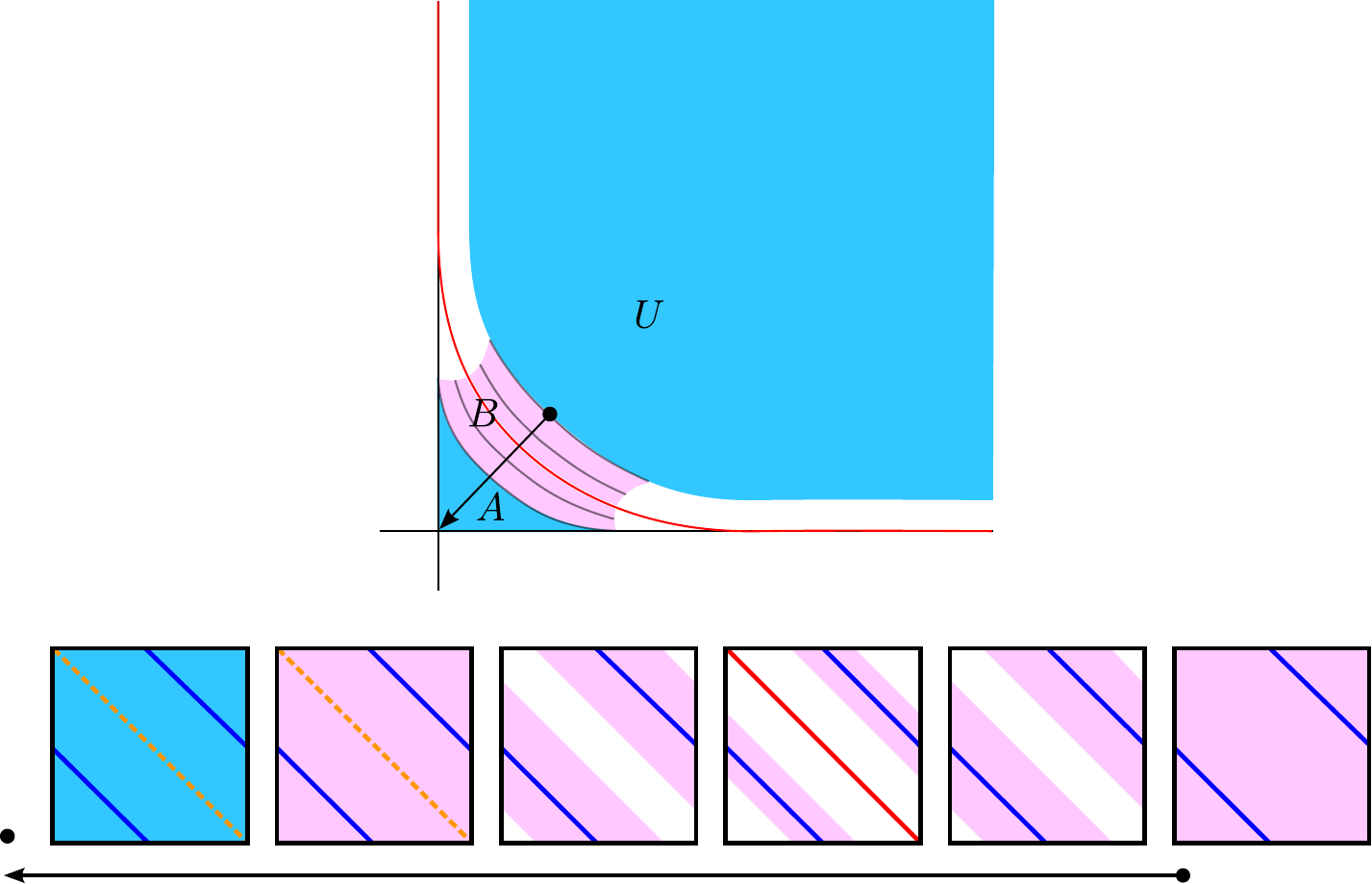}
	\caption{This figure represents the complement of a neighborhood of the smoothing. This complement is made up of three parts: $U$, $A$ and $B$. $U$ is the preimage of the unbounded blue shaded region in the upper figure which agrees with Figure~\ref{fig:compl}. Both $U$ and $A$ are the full preimages of the blue shaded regions. $B$ is a subset of the preimage of the connecting region shaded in pink. The part of the fiber above points in this region included in $B$ depends on distance from the curve $\gamma$ as indicated by the arrow and the squares at the bottom of the figure representing torus fibers over a slice. The pink shaded region indicates the portion of each fiber in this slice which is included in $B$. The core of the attaching $2$-handle is given by the union of the dark blue curves in the $1$-parameter family of torus fibers lying over the arrow, and the co-core is the union of the the dashed orange curves lying in the $A$ region with boundary lying on the boundary of the complement of the neighborhood of the smoothing where the $B$ region opens a gap.}
	\label{fig:complsm}
\end{figure}

 Instead of removing an entire neighborhood of the origin in $\mathbb{C}^2$, in the complement of the smoothing we include a smaller neighborhood of the origin. We obtain the complement of a regular neighborhood of the smoothing by gluing together three pieces:
 \begin{enumerate}
 	\item The region $U$ which is the preimage of a region complementary to a regular neighborhood of the axes in  the moment map image $\R_{\geq 0}\times \R_{\geq 0}$. This is the complement of a regular neighborhood of the nodal divisor.
 	\item A smaller neighborhood of the origin $A$, which can also be taken as the preimage under the moment map of a neighborhood of the origin in $\R_{\geq 0}\times \R_{\geq 0}$ as shown shaded in blue in Figure~\ref{fig:complsm}.
 	\item A piece $B$ diffeomorphic to $S^1\times D^3$. With respect to the moment map, $B$ lies above the  connecting region shaded in pink in Figure~\ref{fig:complsm}. $B$ is given by an annulus in each torus fiber above points in this connecting region. The width of the annuli increases as the fibers move away from the curve $\gamma$ and towards the image of $U$ and $A$.
 \end{enumerate}
 
 Observe that $B$ is smoothly attached to $A$ along a solid torus $S^1\times D^2$ since they are glued along the annuli in the torus fibers in the interval where the moment image of $A$ and $B$ meet. Similarly $B$ is attached to $U$ along a solid torus as well. The union $A\cup B$ is still diffeomorphic to a $4$-ball since $B$ just further thickens a collar of a circle in the boundary of the $4$-ball $A$. Moreover, $A\cup B$ is glued to $U$ (the complement of the nodal divisor) along a solid torus $S^1\times D^2$. Therefore, gluing $A\cup B$ to $U$ amounts to attaching a smooth $2$-handle to $U$. Observe that the attaching circle, $S^1\times \{0\}\subset S^1\times D^2$, inside the attaching region on the boundary of $A\cup B$ is a curve of slope $(1,-1)$ in a single torus fiber in the boundary of $U$. 

This completes the topological picture of the complement of the smoothing. Next, we discuss the symplectic and Weinstein aspects of this 2-handle attachment.

\subsection{The Weinstein structure on the $2$-handle} \label{section:2handleWeinstein}

	To understand the $2$-handle attachment symplectically, we need to identify it with Weinstein's model handle. In this section we first discuss how we insert Weinstein's handle into our local model of the complement of the smoothing of the total toric divisor at a node. The core and co-core of a Weinstein handle are Lagrangian disks that intersect transversally, and we identify such disks in our toric local model to motivate our embedding of Weinstein's handle. Then we use a symplectomorphism arising from an $SL(2,\Z)$ transformation to send our local model to the neighborhood of any toric node, and we examine the image of the core and co-core of the Weinstein $2$-handle under this symplectomorphism. In section~\ref{section:$2$-handle attachment} we establish the global result by showing how to glue the Weinstein structures from all of these $2$-handles compatibly with a Weinstein structure on the complement of a neighborhood of the nodal total toric divisor. Our goal in this section is to establish a solid understanding of where the Weinstein core and co-core are embedded, which we see in section~\ref{section:$2$-handle attachment} is the key behind establishing a global gluing theorem utilizing the centeredness condition.
	
	We focus on the model $2$-handle $A\cup B$ from the previous subsection. We noted that in this model, the attaching sphere is a $(1,-1)$ curve lying in a single torus fiber in the boundary of the $2$-handle. A potential Weinstein core of the $2$-handle is provided by a Lagrangian disk with boundary on this $(1,-1)$ curve. The Weinstein co-core is another Lagrangian disk in the $2$-handle that transversally intersects the core at a single point. In order to identify the symplectomorphism that carries Weinstein's model structure for a $2$-handle (see section~\ref{sec:weinstein_handle}) onto $A\cup B$ (to endow the complement of the smoothing with a Weinstein structure), we first identify the core and co-core as Lagrangian disks.
	
	In action-angle coordinates (defined away from the total toric divisor), a special class of surfaces is given by parametrizations of the form
	\begin{equation}\label{eq:param}
	\{(p_1(t), p_2(t), q_1(s), q_2(s)) | 0\leq t\leq 1, 0\leq s \leq 2\pi\}\subset (D^*T^2, dp_1\wedge dq_1+dp_2\wedge dq_2).
	\end{equation}
	
	A Lagrangian disc $D$ whose inclusion map into the toric manifold is denoted by $i$, must satisfy $i^*\omega = \sum_{k=1}^2i^*(dp_k\wedge dq_k)=0$. If we assume $D$ has a parameterization as in (\ref{eq:param}) the Lagrangian condition translates to: 
	\begin{equation}\label{Lagrangian equation}
	p_1'(t) q_1'(s)+p_2'(t)q_2'(s)=0.
	\end{equation}

	Recall that the attaching sphere for the $2$-handle in the smooth model was a curve above the point $p_1=p_2=2\varepsilon$, which is at the core of the interval where the image of $B$ meets the image of $U$. This curve has $q_1(s) =s$, and $q_2(s)=-s+\pi$. Observe that the disk
	\begin{equation}\label{core in action-angle coordinates}
	D_c = \{ (2\varepsilon t, 2\varepsilon t, s, -s+\pi) \mid 0\leq t \leq 1, 0\leq s \leq 2\pi \}
	\end{equation}
	is Lagrangian because it satisfies equation (\ref{Lagrangian equation}), and it lies completely inside of $A\cup B$. The intersection with $B$ is shown in Figure~\ref{fig:complsm} with a dark blue curve in the torus fibers. From this picture, we can see that the $2$-handle deformation retracts onto $D_c$, so this can serve as a core of the handle.
	
	Note that $D_c$ lies in the moment map preimage of the ray of the direction $(1,1)$ emanating from the vertex of the polytope. In this model, this is precisely the ray $R$ associated to the vertex $V$ as in Definition \ref{defn:centered}.
	
Similarly, the co-core can be represented by the Lagrangian disk
\begin{equation}\label{toric co-core in the standard model}
	D_{cc} = \{ (\varepsilon' t, \varepsilon' t, s, -s) \mid 0\leq t \leq 1, 0\leq s \leq 2\pi \}.
\end{equation}
where $\varepsilon'<\varepsilon$ is chosen such that belt sphere $\{(\varepsilon,\varepsilon, s, -s)\}$, the boundary of $D_{cc}$, lies on the boundary of $A\cup B\cup U$. The co-core $D_{cc}$ is denoted in Figure~\ref{fig:complsm} by the dotted orange curves ($q_1=-q_2$) lying in the fibers over the segment where $p_1=p_2$ and $p_1,p_2\leq \varepsilon$. In complex coordinates $(z_1,z_2)$, these two disks are described by 
	
\begin{equation}\label{chart core in the standard model}
	D_c = \{ (2\varepsilon t e^{is}, 2\varepsilon t e^{-i(s+\pi)})\} = \{(z_1, -\overline{z_1}) \mid |z_1|<2\varepsilon \}
\end{equation}
	and
\begin{equation}\label{chart co-core in the standard model}
	D_{cc} = \{(\varepsilon' t e^{is}, \varepsilon' t e^{-is} )\} = \{ (z_1, \overline{z_1}) \mid |z_1|<\varepsilon' \}.
\end{equation}
	
This description also shows that these disks are smooth at the origin. Moreover, one can check that $D_{cc}$ and $D_c$ intersect transversally only at the origin (see~\cite{ACGMMSW1} for the explicit computation). Observe in this model that the slope of the attaching sphere in the torus is $(1,-1)$. 

Next we discuss how to send this model to the neighborhood of any vertex via an invariant symplectomorphism induced by an $SL(2,\mathbb Z)$ transformation of the moment map images, and observe how this symplectomorphism acts on the core and attaching sphere of the handle. In order to proceed we make use of the combinatorial properties of the Delzant polytope close to any vertex where we perform the smoothing. As introduced in  Definition \ref{defn:centered}, a \emph{ray} $R_V$ of a vertex $V$ is a ray starting in $V$ that is generated by the vector $r(V)\in\mathbb Z^2$ being equal to the sum of the edge vectors of $\Delta$ adjacent to $V$ and a \emph{slope} of a vertex $V$ is a vector $s(V)\in\mathbb Z^2$
given by a $-\pi/2$ rotation of the vector $r(V)$. In particular, the ray of the vertex in the standard cone is given by a vector $(1,1)$ and a slope of this vertex is a vector $(1,-1).$ The key role of a ray and a slope is that the core $D_c$ and the co-core $D_{cc}$ of the 2-handle  just defined project under the moment map precisely to the ray given by direction $(1,1)$ and the slope of the attaching sphere in the torus is precisely $(1,-1)$.

Let us now use these results when smoothing an arbitrary node $V$ of the total toric divisor.
Denote by $G\in SL(2,\mathbb Z)$  the transformation mapping the standard cone to the moment map image of an equivariant Darboux neighborhood $U$ of the fixed point that maps to a vertex $V$ (Figure \ref{both sl}).  The transformation $G$ uniquely defines an equivariant symplectomorphism $\psi$ between $\mathbb C^2$ and the neighborhood $U$. Being equivariant, $\psi$ maps orbits in $\mathbb C^2$ to the orbits in $U.$ In action-angle coordinates (defined away from the total toric divisor), any surface given by parametrizations (\ref{eq:param})
is mapped under $\psi$ to the surface
	\begin{equation}\label{eq:paramnew}
	\{(G(p_1(t), p_2(t)),(G^{-1})^T (q_1(s), q_2(s)) | 0\leq t\leq 1, 0\leq s \leq 2\pi\}\subset (D^*T^2, dp_1\wedge dq_1+dp_2\wedge dq_2).
	\end{equation}	
In particular, using the description of the core (\ref{core in action-angle coordinates}) and the co-core (\ref{toric co-core in the standard model}) of the standard model, we obtain an explicit model of the core $\psi(D_c)$ and the co-core $\psi(D_{cc})$ of the desired Weinstein 2-handle. Since the moment map in action-angle coordinates is simply a projection onto the $(p_1,p_2)$-plane, both the core $\psi(D_c)$ and the co-core $\psi(D_{cc})$ project  to the ray given by the direction $G(1,1)$. From the polytope perspective, the  transformation $G$ is linear and preserves edges. Thus, it maps the ray of the standard cone, given by the direction $(1,1),$ to the ray $R_V$ of the vertex $V$. Therefore, both the core $\psi(D_c)$ and the co-core $\psi(D_{cc})$ project under the moment map precisely to the ray $R_V.$ Further, according to the parametrization (\ref{eq:paramnew}) the slope of the curve in $T^2$ representing the attaching sphere of the 2-handle  is given by the vector $(G^{-1})^T(1,-1)$. According to Proposition \ref{prop def of slopes}, a slope $s(V)$ of the vertex $V$ is  equal to the  difference of inward normal vectors of the edges meeting in $V,$ taken in the counter clock-wise direction. As explained in Section 3, when $G$ is the transformation of the polytopes, then $(G^{-1})^T$ is the transformation of the inward normal vectors. Thus, being linear, the transformation $(G^{-1})^T$  maps the slope $(1,-1)$ of the vertex of the standard cone to the slope $s(V)$ of the vertex $V$ and this is precisely the slope of the attaching sphere. This proves the following important property.

\begin{lemma}\label{lemma slope}
	Let $(M,\omega)$ be a toric symplectic 4-manifold and let $V$ be the vertex in the moment polytope that is the moment map image of the fixed point where we smooth the total toric divisor. Then the slope of the curve in the torus $T^2$ representing the corresponding attaching sphere of a 2-handle is equal to $s(V).$ 
\end{lemma}
	
\subsection{The global Weinstein structure on the complement} \label{section:$2$-handle attachment}

We are now ready to prove our first main result. We show how to glue the canonical Weinstein structure on the complement of a neighborhood of the total toric divisor with a Weinstein structure on the $2$-handles near the nodes $V_i$. Note that the centeredness condition is important to ensure that our model structure on the $2$-handles glues nicely with the chosen Weinstein structure on the subdomain of $T^*T^2$.

\begin{theorem}\label{thm:toricWein}
Let $(M,\omega)$ be a toric $4$-manifold corresponding to Delzant polytope $\Delta$ that is $\{V_1,\ldots, V_k \}$-centered. Let $D$ denote the symplectic divisor obtained by smoothing the total toric divisor at the nodes $V_1,\ldots, V_k$. Then there exists an arbitrarily small neighborhood $N$ of $D$ such that $M\setminus N$ admits the structure of a Weinstein domain.
	
Furthermore, $M\setminus N$ is Weinstein homotopic to the Weinstein domain obtained by attaching $k$ Weinstein $2$-handles to the unit disk cotangent bundle of the torus, $D^*T^2$, along the Legendrian co-normal lifts of co-oriented curves of slope $s(V_1),\ldots, s(V_k)$. 
\end{theorem}

\begin{proof}[Proof of Theorem \ref{thm:toricWein}]
We begin with an outline of the proof and then provide detailed formulas to fill in the three main steps. The first main step is to determine the Weinstein structure we want on $E$, the complement of a small neighborhood of the total toric divisor. For this we use the canonical (Morse-Bott) Weinstein structure on $T^*T^2$, together with the centeredness condition to designate the position of the zero section. The second main step is to embed the model Weinstein structure on a $2$-handle into our model neighborhood of the node. Then we use the equivariant symplectomorphism coming from the appropriate $SL(2,\Z)$ transformation to send this model to a neighborhood of the node $V_i$. 
The third main step is to show how to glue the Weinstein structure on the handle to the Weinstein structure on $E$ in a neighborhood of the attaching region in a way that avoids creating any additional critical points. The gluing we perform occurs in a local neighborhood of the node $V_i$, and the Weinstein structure outside of this neighborhood agrees with the canonical structure on $E$. Therefore we can repeat this gluing at each of the nodes $V_1,\dots, V_n$ independently to obtain the global Weinstein structure.

\textbf{Step 1: placing the center.} The complement of a sufficiently small neighborhood of the total toric divisor is the preimage under the moment map of a closed disk in the interior of $\Delta$  (topologically a disk, not necessarily geometrically a disk). This is a natural Weinstein domain $E$ whose completion is $T^*T^2$.  We call the point in the polytope where the rays associated to $\{V_1,\ldots, V_n \}$ intersect the \emph{$\{V_1,\dots, V_n\}$-center}. Note that this point lies in the interior of a $\{V_1,\dots, V_n\}$-centered polytope. We use the canonical cotangent Liouville structure on the domain $E$, by identifying the Lagrangian torus fiber over the point in $\Delta$ given as the $\{V_1,\dots, V_n\}$-center with the zero section of $T^*T^2$. The $(p_1,p_2)$ coordinates on $\Delta$ (after translating $\Delta$ so that its $\{V_1,\dots, V_n\}$-center occurs at the origin) is identified with the $(p_1,p_2)$ coordinates in the cotangent model, and the toric $(q_1,q_2)$ coordinates on the $T^2$ fibers correspond to the cotangent coordinates on the zero section $T^2$. Because $\Delta$ is star-shaped with respect to the $\{V_1,\dots, V_n\}$-center, if we consider the complement of a small neighborhood of its boundary, $D=\Delta\setminus \nu(\partial \Delta)$, the vector field $p_1\partial_{p_1}+p_2\partial_{p_2}$ is transverse to $\partial D$. Therefore the canonical Liouville vector field is transverse to $\partial E$, and so it induces a contact structure on $\partial E$. The induced contact form on $\partial E$ is  $i^*(\iota_{Z_{can}}\omega) = i^*(p_1dq_1+p_2dq_2)$ where $i:\partial E\to E$ is the inclusion.

The attaching sphere of a Weinstein $2$-handle is a Legendrian knot in the contact boundary of the complement of the neighborhood of the nodal divisor. Since we want to be able to attach $2$-handles for all the vertices in the chosen subset $\{V_1,\ldots, V_n \}$, to a fixed Liouville structure on $E$, we need the attaching curves for each of these $2$-handles to be Legendrian in $\partial E$. Here we begin to see the purpose of the $\{V_1,\ldots, V_n \}$-centeredness condition. As explained in Lemma~\ref{lemma slope} and the discussion before it, the Lagrangian disk we use as the core of our Weinstein $2$-handle at $V_i$ projects to the ray of the direction $r(V_i)$ and the attaching sphere of this core is a curve in a torus fiber over a point on this ray of slope $s(V_i)$. This attaching sphere is Legendrian precisely when $s(V_i)$ is in the kernel of $p_1dq_1+p_2dq_2$ when $(p_1,p_2)$ are the coordinates in the polytope of the point on the ray $r(V_i)$ which the attaching sphere maps to. Since we designated that $(p_1,p_2)=(0,0)$ at the $\{V_1,\dots, V_n\}$-center, a point on the ray has $(p_1,p_2)$ coordinates which are a scalar multiple of $r(V)$. Since $r(V_i)$ and $s(V_i)$ are orthogonal, this ensures that $s(V_i)$ is in the kernel of $p_1dq_1+p_2dq_2$ when $(p_1,p_2)$ lies on the ray. Thus the centeredness condition ensures that the attaching spheres of the Weinstein $2$-handles are Legendrian. This perspective gives the first hint of why this condition is a natural requirement. In step 3, we see how the $\{V_1,\dots, V_n\}$-centeredness condition is used in gluing embedded Weinstein structures.

\textbf{Step 2: embedding the model Weinstein structure on a $2$-handle into our standard model.}	
The natural Liouville structure on a Weinstein $2$-handle in $(\R^4, d\widetilde x_1\wedge d\widetilde y_1+d\widetilde x_2\wedge d\widetilde y_2)$ has Liouville form 
$$\lambda_2=-\widetilde x_1d\widetilde y_1-\widetilde x_2d\widetilde y_2-2\widetilde y_1d\widetilde x_1-2\widetilde y_2d\widetilde x_2$$ 
which corresponds to the Liouville vector field 
$$Z_2=-\widetilde x_1\partial_{\widetilde x_1}-\widetilde x_2\partial_{\widetilde x_2}+2\widetilde y_1\partial_{\widetilde y_1}+2\widetilde y_2\partial_{\widetilde y_2}.$$
Notice that this vector field vanishes at the origin. The core of the $2$-handle is the stable manifold of the zero which is the $(\widetilde x_1,\widetilde x_2)$-plane: $\{\widetilde y_1=\widetilde y_2=0\}$. The co-core of the handle is the unstable manifold of the zero which is the $(\widetilde y_1,\widetilde y_2)$-plane: $\{\widetilde x_1=\widetilde x_2=0\}$.

As in the standard model from section~\ref{section:2handleWeinstein}, we want to align the core of Weinstein's model $2$-handle with the Lagrangian disk 
\begin{equation}\label{eq core}
D_c = \{(r \varepsilon e^{i\theta}, -r \varepsilon e^{-i\theta})\mid 0\leq r\leq 1,\; 0\leq \theta \leq 2\pi\}.
\end{equation}
This disk lies in the plane where $z_2=-\overline{z_1}$. In real coordinates this means $x_2=-x_1$ and $y_2=y_1$. Similarly we want the co-core of the 2-handle to align with the Lagrangian disk
\begin{equation}\label{eq co-core}
D_{cc} = \{(r \varepsilon' e^{i\theta}, r \varepsilon' e^{-i\theta})\mid0\leq r\leq 1,\; 0\leq \theta \leq 2\pi \}.
\end{equation}
This disk lies in the plane where $z_2=\overline{z_1}$. In real coordinates this means $x_2=x_1$ and $y_2=-y_1$.

We define a (linear) symplectomorphism of $(\R^4,\omega_{std})$ which sends the core of the standard handle to $D_c$ and the co-core of the standard handle to $D_{cc}$ as follows.
$$\Psi(\widetilde x_1,\widetilde y_1, \widetilde x_1, \widetilde y_2) = \left(\frac{\sqrt{2}}{2}(\widetilde x_1-\widetilde y_2),\frac{\sqrt{2}}{2}(\widetilde x_2+\widetilde y_1), -\frac{\sqrt{2}}{2}(\widetilde x_1+\widetilde y_2),\frac{\sqrt{2}}{2}(\widetilde x_2-\widetilde y_1)   \right)$$
with inverse map
$$\Psi^{-1}(x_1,y_1,x_2,y_2) = \left(\frac{\sqrt{2}}{2}(x_1-x_2),\frac{\sqrt{2}}{2}(y_1-y_2),\frac{\sqrt{2}}{2}(y_1+y_2),-\frac{\sqrt{2}}{2}(x_1+x_2) \right)$$

Notice that the plane $\{\widetilde y_1=\widetilde y_2=0\}$ is sent under $\Psi$ to the plane $\{x_2=-x_1,\; y_2=y_1\}$ and the plane $\{\widetilde x_1=\widetilde x_2=0\}$ is sent to $\{x_1=x_2,\; y_2=-y_1 \}$. Pulling back the Liouville form $\lambda_2$ under $\Psi^{-1}$ gives us the Liouville form on the $2$-handle in our model:
$$\lambda_h=(\Psi^{-1})^*(\lambda_2) = \left(-\frac{1}{2}y_1+\frac{3}{2}y_2 \right)dx_1+\left(\frac{1}{2}x_1+\frac{3}{2}x_2 \right)dy_1+\left(\frac{3}{2}y_1-\frac{1}{2}y_2 \right)dx_2+\left(\frac{3}{2}x_1+\frac{1}{2}x_2 \right)dy_2$$
The Liouville vector field $\omega_{std}$-dual to the form $\lambda_h$ is
$$Z_h = \left(\frac{1}{2}y_1-\frac{3}{2}y_2 \right)\partial_{y_1}+\left(\frac{1}{2}x_1+\frac{3}{2}x_2 \right)\partial_{x_1}+\left(-\frac{3}{2}y_1+\frac{1}{2}y_2 \right)\partial_{y_2}+\left(\frac{3}{2}x_1+\frac{1}{2}x_2 \right)\partial_{x_2}$$
Thus, $\Psi$ provides the embedding into our model neighborhood of a node. After performing the $T^2$-equivariant symplectomorphism induced by the appropriate $SL(2,\Z)$ transformation, this provides a Liouville structure (which is gradientlike for a function with a single index $2$ critical point) on the complement of the smoothing in a neighborhood of a node $V_i$.

\textbf{Step 3: glue the Liouville structure from the $2$-handle to the canonical Liouville structure on $U$.} 

Now we have proposed Liouville vector fields and forms on the pieces of the complement of the smoothing near the nodes $V_1,\dots, V_n$ and on $E$, the complement of the nodal total toric divisor. If we compare these two Liouville structures on the overlap of these pieces, they do not fully agree. Therefore, to glue them together to get a global Weinstein structure, we need to interpolate between these two Liouville structures on their overlap. To show that we get a global Weinstein structure on the complement of the smoothing with our handles, we need to ensure that these Liouville structures can be glued together in a way that does not introduce new zeros of the vector field in the interpolating region. 

To explain this interpolation more explicitly, we work in our standard model neighborhood of the node. The Liouville form $\lambda_h$ on the $2$-handle is defined on this neighborhood. The second Liouville form is defined on $E$ by $\lambda_{can}=p_1dq_1+p_2dq_2$ with Liouville vector field $Z_{can}=p_1\partial_{p_1}+p_2\partial_{p_2}$, where the $(p_1,p_2)$ origin is placed at the $\{V_1,\dots, V_n\}$-center of the polytope. We intersect $E$ with the neighborhood $N_i$ of a vertex $V_i$ to find the overlapping region where both forms are defined. To get from $E\cap N_i$ to our standard local model we use the equivariant symplectomorphism $\psi_i$ induced by the appropriate $SL(2,\Z)$ transformation as explained in Section~\ref{section:2handleWeinstein}. Under this symplectomorphism $\psi_i^{-1}(E\cap N_i)$ is the region denoted by $U$ in Section~\ref{section:Topologycomplements}. Pulling back, the Liouville form is $\lambda_T:=\psi_i^*\lambda_{can}$ and the Liouville vector field is $Z_T:=d\psi_i^{-1}(Z_{can})$. By the $\{V_1,\dots, V_n\}$-centeredness assumption, $Z_{can}$ is tangent to the ray $r(V_i)$. Because $\psi_i$ sends the direction of the ray in the standard model $(1,1)$ to $r(V_i)$, $Z_T=d\psi_i^{-1}(Z_{can})$ is tangent to the ray $(1,1)$ in our standard model. Because the core $D_c$ and co-core $D_{cc}$ project to the ray $(1,1)$, $Z_T$ is tangent to $D_c$ and $D_{cc}$. The other Liouville vector field $Z_h$, arising from Weinstein's model on the $2$-handle also is tangent to the core $D_c$ and co-core $D_{cc}$. Let $i_1: D_{cc}\to \C^2$ be the inclusion of the co-core $D_{cc}$ and $i_2:D_c\to \C^2$  the inclusion of the core $D_c$. Since both $D_c$ and $D_{cc}$ are Lagrangian, the tangency of the Liouville vector fields implies that $i_1^*\lambda_h = i_1^*\lambda_T=0$ and $i_2^*\lambda_h=i_2^*\lambda_T=0$. This is the first key property which we use to get a good interpolation between $\lambda_T$ and $\lambda_h$. Note that this key property relies crucially on the $\{V_1,\dots, V_n\}$-centeredness assumption.

Recall from Section~\ref{section:Topologycomplements} that in our standard local model, we decompose the complement of the smoothing into three pieces $A$, $B$, and $U$. We perform the interpolation between $\lambda_h$ and $\lambda_T$ in the region $B$. Note that $B$ is a solid torus which deformation retracts onto the attaching sphere $C$. This is the second property we need to construct a good interpolation.

Next we define a function $\rho:B\to [0,1]$ which depends only on the momentum coordinates, which is identically $0$ near $A$ and is identically $1$ near $U$. We assume that $\rho$ has level sets as depicted in Figure~\ref{fig:complsm}, and is monotonically decreasing from $1$ to $0$ across the level sets from $U$ to $A$. Observe that the Liouville vector fields $Z_h$ and $Z_T$ are both inwardly transverse to the level sets of $\rho$. This is the last property we need to prove we can obtain a good interpolation.

Since $B$ deformation retracts onto the attaching circle $C$, the inclusion $i:C\to B$ induces an isomorphism on cohomology $i^*: H^1(B)\to H^1(C)$. Because $\lambda_h$ and $\lambda_T$ both pull back to the $0$-form on $C$, we must have that their restrictions to $B$ are co-homologous so the difference is exact: 
$$\lambda_h|_B - \lambda_T|_B = dH.$$
The restriction to the core vanishes, ($i_2^*(\lambda_h - \lambda_T)=0$), so $H$ must be constant along the core. After possibly changing $H$ by a constant, we can assume that $H$ vanishes along the core (and thus also along the co-core where $H$ must also be constant for similar reasons). 

Now we define the interpolated Liouville structure.
$$\lambda = \lambda_h+d(\rho H) = \lambda_h+\rho dH + Hd\rho = (1-\rho)\lambda_h+\rho\lambda_T+Hd\rho.$$

This form is Liouville for the standard symplectic form, agrees with $\lambda_h$ near $\rho^{-1}(0)$, and agrees with $\lambda_T$ near $\rho^{-1}(1)$. Moreover, at points $p$ in the core or co-core $H=0$ so $\lambda(p) = (1-\rho)\lambda_h(p)+\rho\lambda_T(p)$ is exactly the convex interpolation between the two forms. In particular, $i_1^*\lambda$ and $i_2^*\lambda$ vanish. Therefore, the dual Liouville vector field is tangent to the chosen core and co-core. This implies that the stable and unstable manifolds of the Liouville vector field $Z$ corresponding to $\lambda$ agree with those of $\lambda_h$, so the core and co-core of this modified Liouville structure on the $2$-handle agree with the model. 

The final thing to check is that the only point where $\lambda$ (or equivalently $Z$) vanishes is the origin. The significance of this is to ensure that the Weinstein structure we construct only has a single handle in the neighborhood of the node $V_i$ (the $2$-handle we understand). Because $Z_h$ and $Z_T$ are both inwardly transverse to the level sets of $\rho$, the convex combination $(1-\rho)Z_h+\rho Z_T$ is inwardly transverse to the level sets of $\rho$. The Liouville vector field is given by $Z=(1-\rho)Z_h+\rho Z_T+Z_{err}$ where $Z_{err}$ is $\omega_{std}$ dual to $Hd\rho$.

Because $\omega(Z_{err},\cdot) = Hd\rho$, we get that $\iota_{Z_{err}}\omega|_{\operatorname{ker}d\rho} = 0$. This implies that $Z_{err}$ is actually tangent to the level sets of $\rho$. Since $(1-\rho)Z_h+\rho Z_T$ is transverse to the level sets of $\rho$ and $Z_{err}$ is tangent, we conclude that $Z=(1-\rho)Z_h+\rho Z_T+Z_{err}$ is transverse to the level sets of $\rho$ and in particular is non-vanishing in $B$. Therefore, the only point in the model neighborhood where $\lambda$ or $Z$ vanishes is the node point $V_i$.

This shows that the complement of an arbitrarily small neighborhood of the $\{V_1,\dots, V_n \}$-smoothing of the total toric divisor supports the structure of a Weinstein domain. The Weinstein structure has a single $2$-handle for each node $V_i$, which is attached to the complement of the total toric divisor along a curve of slope $s(V_i)$ in the torus fiber over a point on the ray $R_{V_i}$. 

The last technical point is that the complement of the neighborhood of the total toric divisor is not precisely $D^*T^2$ for a round disk $D$, but there is a Weinstein homotopy which relates these two by shrinking the boundary to a round disk. If we perform this boundary shrinking homotopy away from the attaching regions of the $2$-handles (away from the rays), we get a Weinstein homotopy from the Weinstein structure we have constructed on the smoothing to the Weinstein manifold obtained from $D^*T^2$ by attaching $2$-handles along the co-oriented co-normal lifts of curves of slope $s(V_i)$ for $i=1,\dots, n$. (Note that the co-oriented co-normal lift of a curve of slope $s(V_i)$ lives in a torus fiber over a point in the ray $R_{V_i}$, which is precisely the attaching sphere of the $2$-handle we attach near $V_i$.)
\end{proof}

\subsection{An almost toric point of view} \label{section:almosttoric}

In this subsection we take a brief detour to describe smoothing a node and the corresponding Weinstein $2$-handle attachment in the complement from the almost toric viewpoint.

Recall from Section~\ref{intro to almost toric} that almost toric manifolds generalize toric manifolds by allowing torus fibers to have nodal singularities. We restrict ourselves to almost toric manifolds whose base is a disk. For any almost toric $4$-manifold, the preimage of the (undotted) boundary of the almost toric base is a symplectic hypersurface we call the \emph{total almost toric divisor}.

\subsubsection{An almost toric model via nodal trade.} 
First we discuss how the smoothed model toric divisor as in Section~\ref{section:Smoothing} can be  understood as a total almost toric divisor via operations called nodal trades. For any toric manifold, a \emph{nodal trade} transforms the node of the total toric divisor at a vertex of the moment polytope into a nodal singularity of an almost toric fibration above a point in the interior of the base. In order to understand the relation between the two models, we follow \cite[Section 7]{Evans}. We consider the \emph{Auroux system} described by the Lagrangian fibration $H = (H_1,H_2): \C^2 \rightarrow \R^2$ with 
$H_1 (z_1, z_2)=\frac{1}{2} |z_1 z_2 - \varepsilon^2|^2$ and $H_2(z_1,z_2) = \frac{1}{2} ( |z_2|^2-|z_1|^2)$. The image of this Lagrangian fibration is the closed half plane of points with non-negative first coordinates as on the left of Figure~\ref{fig:nodaltradecorner}. 
There is one singular fiber at the point $p=(\frac{1}{2} |\varepsilon^2|^2, 0)$. The preimage of the boundary is the smooth conic $\{z_1 z_2 = \varepsilon^2\}$. $H$ is an almost toric fibration with a branch cut along $\{(x,0) | x > \frac{1}{2} |\varepsilon^2|^2\}$ (see~\cite[Lemma 7.2]{Evans}). Notice that in this description, $\{z_1z_2=0\}$ lies in $H_1^{-1}(\frac{1}{2} |\varepsilon^2|^2)$, the node lies in the singular fiber at $p$, and the total almost toric divisor corresponds locally to the local model of smoothing (away from the interpolation locus). We are therefore describing the same smoothing as in Section~\ref{section:Smoothing} but with another Lagrangian fibration than the standard moment map. In order to recover a similar projection to the first quadrant, we perform a mutation with respect to the $(-1,0)$-eigenray to get an almost toric fibration of the same space whose base is shown in the central image of Figure~\ref{fig:nodaltradecorner}. This mutated base is obtained by keeping the quadrant above the $(-1,0)$-eigenline unchanged and applying the inverse of the affine monodromy $\begin{pmatrix}
	1 & 1 \\
	0 & 1
\end{pmatrix}^{-1} = 
\begin{pmatrix}
	1 & -1 \\
	0 & 1
\end{pmatrix}$
to the quadrant below the eigenline (as we cross the cut in the anti-clockwise direction when gluing the two halves), so that the boundary edge spanned by the vector $(0,-1)$ is mapped to the edge spanned by the vector $(1,-1)$. We apply an $SL(2,\Z)$ transformation to obtain the standard nodal trade corner. The $SL(2,\Z)$ transformation fixes the $(0,1)$-edge and transforms the $(1,-1)$-edge to a $(1,0)$-edge. The $(1,0)$-eigenray is mapped to the $(1,1)$-eigenray as on the right of Figure~\ref{fig:nodaltradecorner}.

\begin{figure}
	\centering
	\includegraphics[width=3cm]{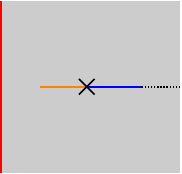}
	\hspace{1.5cm}
	\includegraphics[width=3cm]{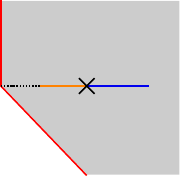}
	\hspace{1.5cm}
	\includegraphics[width=3cm]{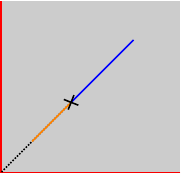}
	\caption{The different steps to get the nodal trade picture for the standard Delzant corner: on the left the image of the Auroux system, in the center the base after a mutation, on the right the base after a $SL(2,\Z)$-transformation.}
	\label{fig:nodaltradecorner}
\end{figure}

By this procedure, the local model of $\C^2$ with the standard torus action and Delzant polytope $\Delta$ given by the upper right quadrant in the plane (as on the left of Figure~\ref{both sl}) perturbs the standard moment map to an almost toric fibration the base of which consists of the upper right quadrant of the plane, and one dotted segment in the direction~$(1,1)$ joining the origin to an interior point $x$, as shown on the right of Figure~\ref{fig:nodaltradecorner}.

We denote this local model of $\C^2$ after a nodal trade $(\mathcal{B},\pi)$. The base $\mathcal{B}$ encodes the following information:
\begin{itemize}
	\item the preimage under $\pi$ of any interior point excluding $x$ is a smooth Lagrangian torus,
	\item the preimage under $\pi$ of any point on the edge, including the origin, is an isotropic circle,
	\item the preimage under $\pi$ of $x$ is a singular Lagrangian torus with a single (Lagrangian) node,
	\item the dotted segment lies here on a ray called the \emph{eigenray} of the node and encodes the branch cut.
\end{itemize}
Note that in~\cite{Symington} the eigenray begins at the marked point, while in our setting it starts at the vertex, in order to match the rays used in the centeredness condition. The line prolonging the eigenray is called an \emph{eigenline}.

\subsubsection{The topology of the complement.} 
Consider the local model $(\mathcal{B}, \pi)$ of a nodal trade as above. Let $\Sigma'$ denote the total almost toric divisor in $(\mathcal{B},\pi)$, ie. the preimage of the boundary of the quadrant. Note that the preimage of the origin is a circle because of the nodal trade. The preimages under $\pi$ of the complement of both a neighborhood of the boundary of the quadrant and a neighborhood of the branch cut correspond to a perturbation of the open subset $U$ defined in Section~\ref{section:Topologycomplements}. The complement of $\Sigma'$ is obtained by adding in a neighborhood of the nodal fiber. Topologically this amounts to adding a $2$-handle attached along the vanishing cycle which collapses to the node. In the Auroux system model, the vanishing cycle corresponds to the attaching sphere of Section~\ref{section:Topologycomplements} and is an orbit of the Hamiltonian vector field associated to $H_2$. In the local angle coordinates, the vanishing cycle corresponds to a circle of 
slope $(0,1)$. To get the base of the nodal trade, we perform a mutation and an $SL(2,\Z)$ transformation, see Figure \ref{fig:nodaltradecorner}}. Under our conventions, a mutation reverses the orientation of the eigenray along which we perform the mutation. This reverses the corresponding orientation of the slope of the attaching sphere, so under a mutation, our circle of slope $(0,1)$ is preserved but is oppositely oriented with a slope of $(0,-1)$. The $SL(2,\Z)$ transformation mapping the base to the standard corner then maps in the fibers the $(0,-1)$ curve to the $(1,-1)$ curve which is the same attaching sphere as in the standard toric picture of Section~\ref{section:2handleWeinstein}. See Section $4.1$ in~\cite{LeungSymington} for more details on nodal trades and the topology of the complement of such a total almost toric divisor via a handle attachment.  

\begin{figure}
	\centering
	\includegraphics[scale=.5]{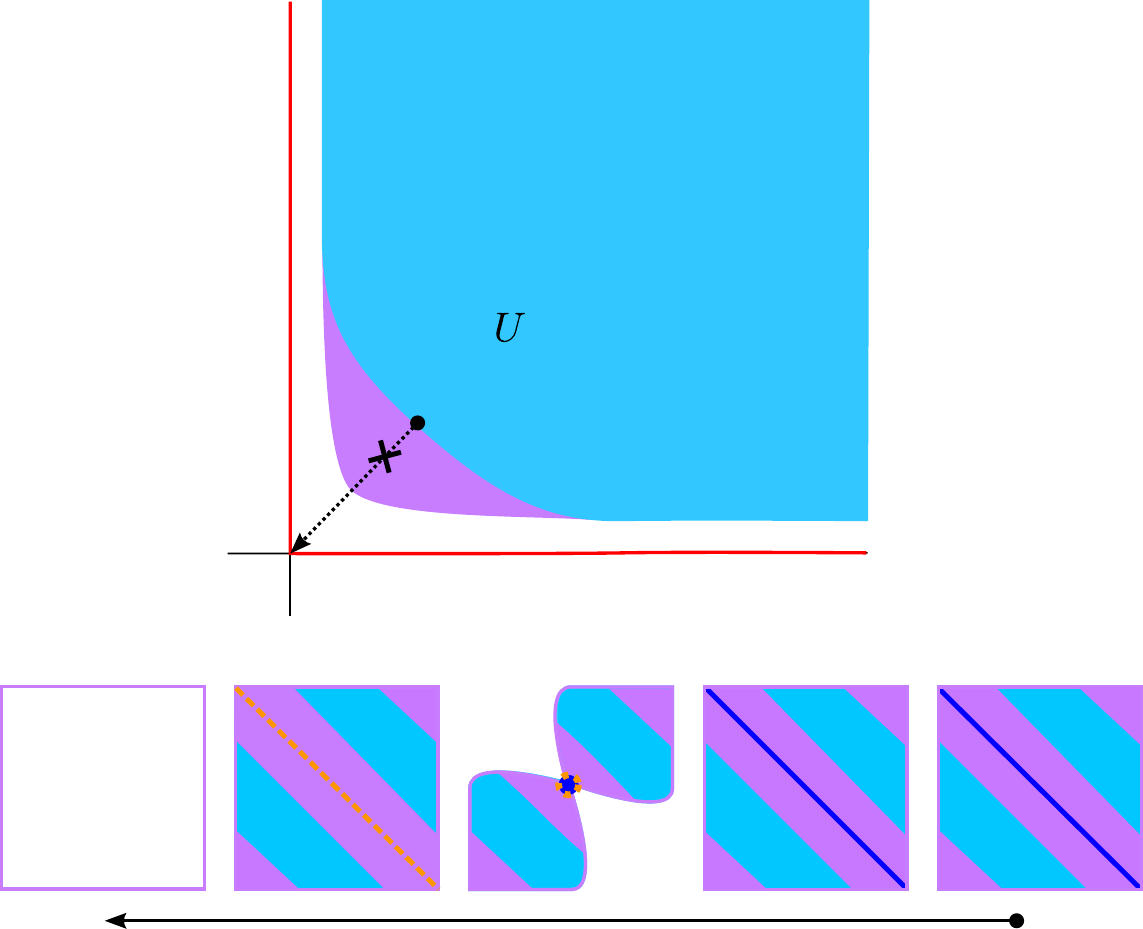}
	\caption{The handle attachment picture in the almost toric point of view. The core of the $2$-handle is the union of the blue curves in the tori over the points in the moment map image over the dotted line between the dot and the x together with the collapsed singular point over the x. A neighborhood of this core gives the $2$-handle, which is shown in purple.}
	\label{fig:divisorcomplementalmosttoric}
\end{figure}

\subsubsection{The Weinstein structure on the 2-handle.}

\begin{lemma}\label{AThandle}
In the local model $(\mathcal{B}, \pi)$ of a nodal trade, the core $D_c$ and co-core $D_{cc}$ of the $2$-handle of Section~\ref{section:2handleWeinstein} correspond to disks projecting via the almost toric fibration $\pi$ on the ray starting at the corner and with the direction $(1,1)$. The slope of the attaching sphere is equal to $(1,-1)$.
\end{lemma}

\begin{proof}
Let us show that in the local model of a nodal trade $(\mathcal{B}, \pi)$ , the core $D_c$ and co-core $D_{cc}$ of the $2$-handle project through $\pi$ onto half segments lying along the eigenline as shown in Figure~\ref{fig:divisorcomplementalmosttoric}. The co-core $D_{cc}$ projects onto the half segment from the marked point x to a point near the vertex $(0,0)$ 
denoted by the orange segment on the right of Figure~\ref{fig:nodaltradecorner}. The core $D_c$ projects onto a half segment from x to a point further in the interior of the original polytope denoted by the blue segment the right of Figure~\ref{fig:nodaltradecorner}.

To verify these projections of the core and co-core in an almost toric manifold, let us see where the disks of $\C^2$ described by the formulas~\ref{chart core in the standard model} and \ref{chart co-core in the standard model} project in the Auroux system. Under the Lagrangian fibration $H = (H_1,H_2)$, both of the images of the disks $D_{c}$ and $D_{cc}$ vanish under $H_2$. The value of $H_1$ is greater than $\frac{1}{2} \varepsilon^4$ along $D_{c}$. The image of the disk $D_{cc}$ under $H_1$ is positive and less than $\frac{1}{2} \varepsilon^4$. As the eigenline is fixed under the mutation with respect to the $(-1,0)$-eigenray, the projections of the disks $D_{cc}$ and $D_c$ do not change. The $SL(2,\Z)$ transformation to the nodal trade corner maps the $(1,0)$-eigenray  to the $(1,1)$-eigenray. The disks $D_{cc}$ and $D_c$ now project to half segments along this diagonal eigenray, as on the right of Figure~\ref{fig:nodaltradecorner}. 
\end{proof}

Let $V$ be a vertex of the Delzant polytope of a toric manifold with slope $s(V)$. 
There is a symplectomorphism induced by an $SL(2,\Z)$ transformation as in the proof of Lemma~\ref{lemma slope}, which carries a neighborhood of the standard model to the neighborhood of the vertex equivariantly. Applying this symplectomorphism after performing the nodal trade in the model yields a nodal trade at the vertex. This symplectomorphism takes the slope $(1,-1)$ to $s(V)$ exactly as in the toric case.

\subsubsection{More general handle attachment in the almost-toric setting} \label{sss:ATgeneral}
In general, a neighborhood of any fiber with a single singularity is described by the local model where one attaches a $2$-handle to the complement of a neighborhood of the union of the total almost toric divisor and the branch cuts. The core disk of the attached $2$-handle passes through the nodal point of the divisor and projects to a half segment along the corresponding branch cut. If the branch cut is along an eigenray spanned by the vector $(a,b)$ then the attaching sphere is the $(b,-a)$-circle. However, whenever the branch cut and eigenray of a nodal singularity agree, we can find Lagrangian disks near the nodal singularity which can serve as the core and co-core for a Weinstein 2-handle attachment. The core $D_c$ projects to the portion of the eigenray starting at the singular value x and continues in the direction away from the branch cut. The intersection of $D_c$ and each fiber above these points consists of a circle whose slope is the vanishing cycle of the nodal singularity. The co-core $D_{cc}$ projects to the portion of the eigenray starting at the singular value x and continues along the branch cut. The intersection of $D_{cc}$ and each fiber above these points is again a circle given by the vanishing cycle of the singularity. The eigenlines of the topological and affine monodromies are orthogonal, thus when the branch cut direction agrees with the eigenline of the affine monodromy, these disks are Lagrangian as in Section~\ref{section:2handleWeinstein}. Note that a complement of an total almost toric divisor does not always have a natural description as a smoothing of a toric divisor.

Almost toric fibrations are also allowed to have several distinct singularities in the same fiber or equivalently to have various fibers each with a single singularity projecting to distinct points along an eigenray. For an almost toric fibration with $k$ singularities in the same fiber, $k \geq 2$, the complement of the total almost toric divisor has a Weinstein structure built by attaching $2$-handles near each singular point to $D^*T^2$. The core and co-core disks have attaching spheres with parallel slopes, translated by different amounts in the torus so that the core and co-core disks of distinct handles are disjoint from each other.

\subsubsection{The global Weinstein structure on the complement}

Consider an almost toric fibration as above where all branch cuts are colinear to eigenlines.
Choose the subset of standard toric vertices where one wants to consider smoothing and perform a nodal trade at these vertices. The centeredness condition can then be adapted to almost toric manifolds in the following way. An almost toric polytope is centered if all its eigenlines intersect at a common single point in the interior of the polytope without crossing any branch cuts. 

The proof of Theorem~\ref{thm:toricWein} as in Section~\ref{section:$2$-handle attachment} can be carried out similarly in the almost toric case. Because the eigenray in the nodal trade picture plays a similar role as the ray at a vertex, Step 1 of the proof is carried out in exactly the same way using the centeredness notion above. For Step 2, using the Lagrangian disks $D_c$ and $D_{cc}$ described in Section~\ref{sss:ATgeneral}, we can find a similar formula for a linear symplectomorphism which carries Weinstein's standard model of a $4$-dimensional $2$-handle to the neighborhood of the almost toric nodal singularity and maps the core and co-core in Weinstein's model to $D_c$ and $D_{cc}$. As the core Lagrangian disk projects on the eigenray via the almost toric fibration, the arguments of Step 3 hold in this almost toric setting. Therefore, Theorem~\ref{thm:toricWein} holds when substituting~$D$ with the total almost toric divisor. 

The original assumption and statement of Theorem~\ref{thm:toricWein} is recovered from the almost toric assumption and statement by performing nodal trades at each of the nodes that are to be smoothed. Note that our centeredness condition as stated above may not cover all of the base diagrams that can arise in the almost toric setting, however this version suffices for the examples we describe in Section~\ref{section examples}. We hope to extend this condition to more general almost toric base diagrams in future work.

%% file: parts/classification.tex
\section{Combinatorial Possibilities for Delzant Polytopes and resulting Weinstein divisor complements}\label{s: classifying}

In this section, we explore the possibilities for the combinatorial data in the statement of Theorem~\ref{thm:toricWein} given by the set of slopes of the nodes we smooth and the centeredness condition. In Section \ref{subsection Delzant} we show that  for any set of primitive vectors there is a toric 4-manifold with a subset of vertices that have the desired slopes (Theorem~\ref{proposition slopes}). However, it is not always possible to find a toric 4-manifold that is also centered with respect to these vertices and we deeply explore the centeredness condition in Section \ref{subsection centeredness}. We also address when our construction of Weinstein domains yields repetition. Namely, equivariantly symplectomorphic toric 4-manifolds give rise to symplectomorphic completions of corresponding Weinstein domains, as we explain in detail in Section \ref{subsection repetitions}.

We first recall the definitions of a ray and a slope of a vertex, which are significantly used throughout this section. As in Definition \ref{defn:centered}, 
	a ray $R_V$ of a vertex $V$ of a Delzant polytope $\Delta$ is generated by the vector $r(V)\in\mathbb Z^2$ that is equal to the sum of the edge vectors of $\Delta$ adjacent to $V$ and beginning at $V$ (see Figure \ref{fig ray}). Further, the slope of a vertex $V$ is a vector $s(V)\in\mathbb Z^2$
given by a $-\pi/2$ rotation of the vector $r(V)$.

\begin{figure}
	\centering
	\includegraphics[width=10cm]{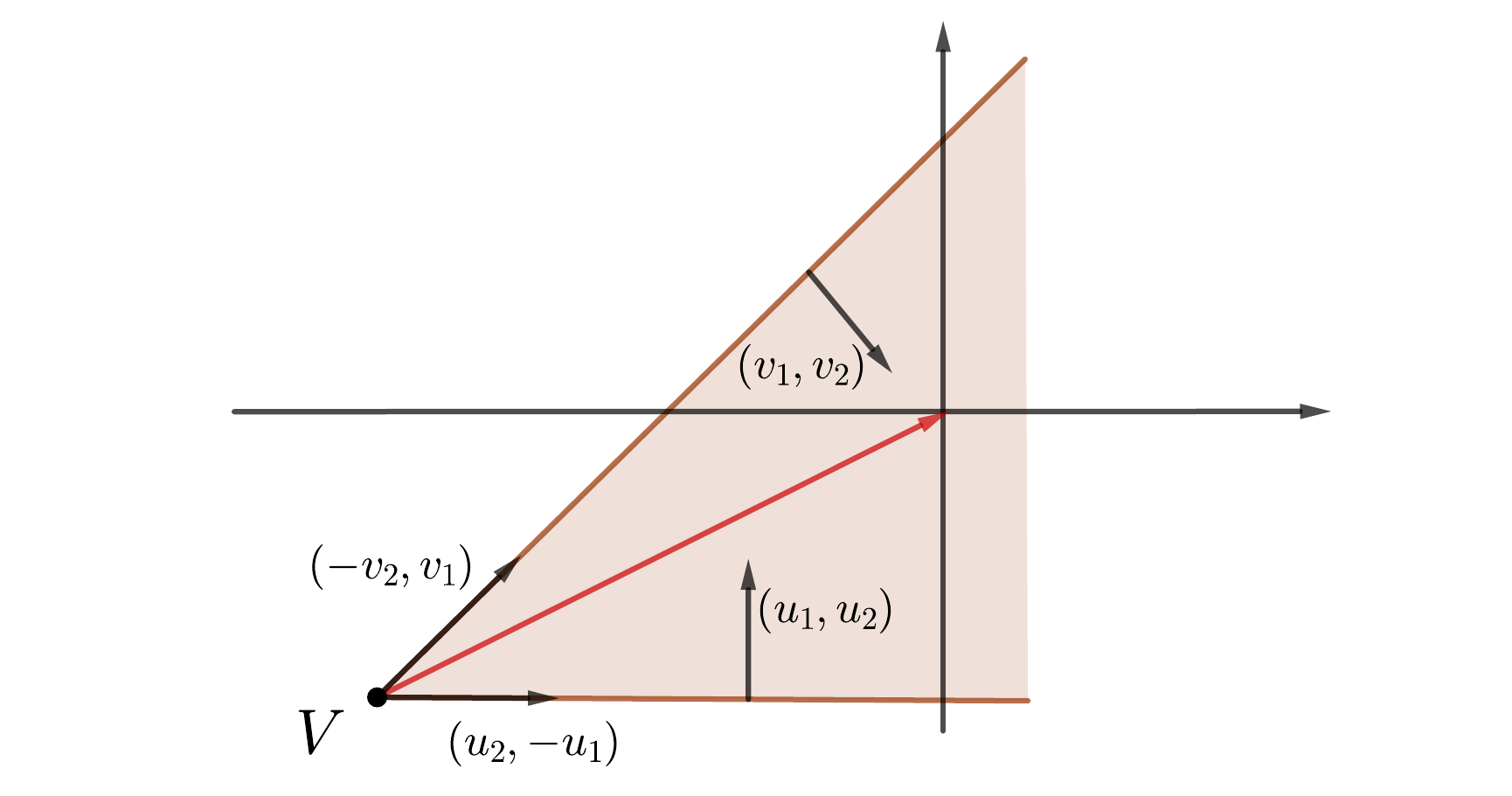}
	\caption{A vertex with a ray (depicted with a red arrow) generated by $r(V)=(u_2-v_2,v_1-u_1)$ and  the slope $s(V)=(v_1-u_1,v_2-u_2)$.}
	\label{fig ray}
\end{figure}      

An equivalent characterisation of a slope of a vertex is given by the following proposition.

\begin{prop}\label{prop def of slopes} A slope $s(V)\in\mathbb Z^2$ of a vertex $V$ is equal to the difference of the inward normal vectors of the edges meeting in $V$, taken in the counterclockwise direction.
\end{prop}
\begin{proof}
Denote by $v=(v_1,v_2)$ and $u=(u_1,u_2)$ inward normal vectors of the edges meeting in $V,$ taken in the counterclockwise direction (Figure \ref{fig ray}). Then, the vectors of these edges are $(-v_2,v_1)$ and $(u_2,-u_1)$ and, consequently,  the direction of the ray in $V$ is 
$r(V)=(-v_2+u_2,v_1-u_1)$. According to Definition \ref{defn:centered} the slope of the vertex $V$ is equal to $s(V)=\begin{pmatrix}
   0 & 1 \\
    -1 & 0
  \end{pmatrix}r(V)=(v_1-u_1,v_2-u_2)=v-u.$
  \end{proof}
  
Observe that vectors $r(V)$ and $s(V)$ are primitive vectors in $\mathbb Z^2$, meaning they are not an integer multiply of any other vector in $\mathbb Z^2.$ Being related by a $\pi/2$ rotation, it is enough to check that $s(V)$ is primitive. Assume $v_1-u_1=ka$ and $v_2-u_2=kb,$ for some $a,b,k\in\mathbb{Z}$. Since the vectors $u$ and $v$ form a $\Z^2$-basis we have $v_1u_2-u_1v_2=1.$ Thus, we obtain $k(au_2-bu_1)=1$ and $-k(v_1b+v_2a)=1$ and the observation follows. 

\subsection{On the richness of Delzant family}\label{subsection Delzant}
The main purpose of this section is to prove the following theorem.

\begin{theorem}\label{proposition slopes}
For any choice of primitive vectors $\{(a_1, b_1), \ldots, (a_k,b_k)\}$ there is a Delzant polytope with at least $k$ edges such that there are vertices $V_1,\ldots, V_k$ with the slopes $s(V_i)=(a_i,b_i),$ $i=1,\ldots, k.$ 
\end{theorem}
Throughout this section, we use the characterisation of a slope given by Proposition \ref{prop def of slopes}.
Note that, at this moment we don't impose the centeredness condition. In particular, this result demonstrates some of the richness of the Delzant family, and may be of independent interest. The properties of $2$-dimensional Delzant  polytopes listed in Lemma \ref{lemma normals} and Lemma \ref{lemma blowup} are used to prove Theorem~\ref{proposition slopes}. Indeed, the proof of Theorem~\ref{proposition slopes} only works for dimension $4$, and it is not known whether an analogous result holds for higher dimensional toric manifolds. 

\begin{lemma} \label{lemma normals} For any choice of distinct primitive vectors $(a_i,b_i),$ $i=1,\ldots, k,$ where $k\geq1,$ there is a Delzant polytope with at least $k$ edges, so that the inward normal vectors of $k$ edges are precisely $(a_i,b_i),$ $i=1,\ldots, k.$
\end{lemma}

\begin{proof} \emph{The $k=1$ case.} We first show that for any primitive vector $(a,b)\in\mathbb{Z}^2$ there is a Delzant polytope with an edge whose inward normal vector is $(a,b)$. Let $p,q\in\mathbb{Z}$ be one of the solutions of Bezout's equation $ap-bq=1.$
Then, the $SL(2,\mathbb{Z})$ transformation 
$G=\begin{pmatrix}
   a & q \\
    b & p
  \end{pmatrix},$
maps the vector $(1,0)$ to the vector
$(a,b).$
The moment map image of the complex projective space $\mathbb{CP}^2$ with the standard toric action is a triangle with one inward normal $(1,0)$. If $G$ is the transformation of the inward normal vectors then
$(G^{-1})^T=\begin{pmatrix}
   p & -b \\
    -q & a
  \end{pmatrix}$ 
is the transformation of the polytopes, that is, $(G^{-1})^T$ maps the standard triangle to the triangle with one inward normal vector $(a,b)$ as shown in Figure \ref{trougao}.  The triangle we obtain is the moment map image of 
$\mathbb{CP}^2$ with a reparametrized toric action. In general, after an appropriate $SL(2,\mathbb{Z})$ transformation, every Delzant polytope can be transformed into one with the desired inward normal vector.

\begin{figure}
	\centering
	\includegraphics[width=10cm]{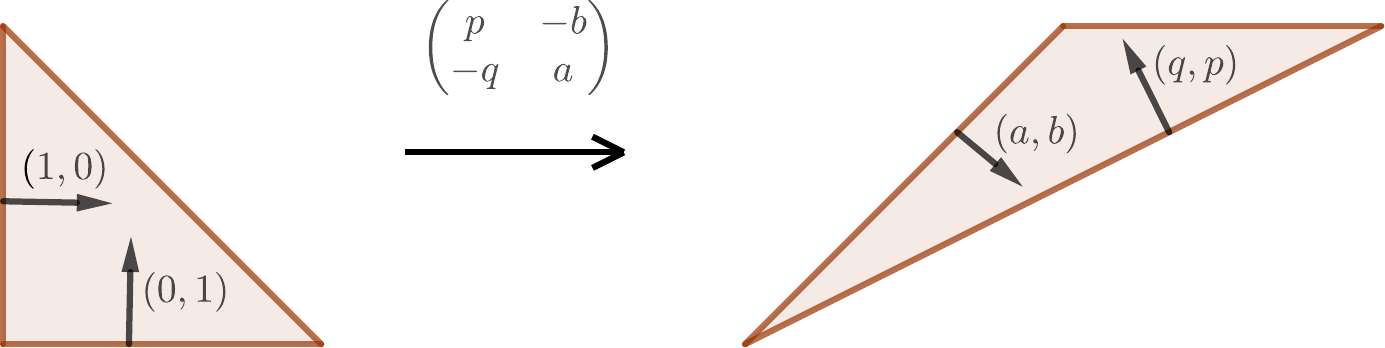}
	\caption{$SL(2,\mathbb Z)$ transformation of the polytope of $\mathbb{CP}^2$.}
	\label{trougao}
\end{figure}  

\begin{figure}
	\centering
	\includegraphics[width=14cm]{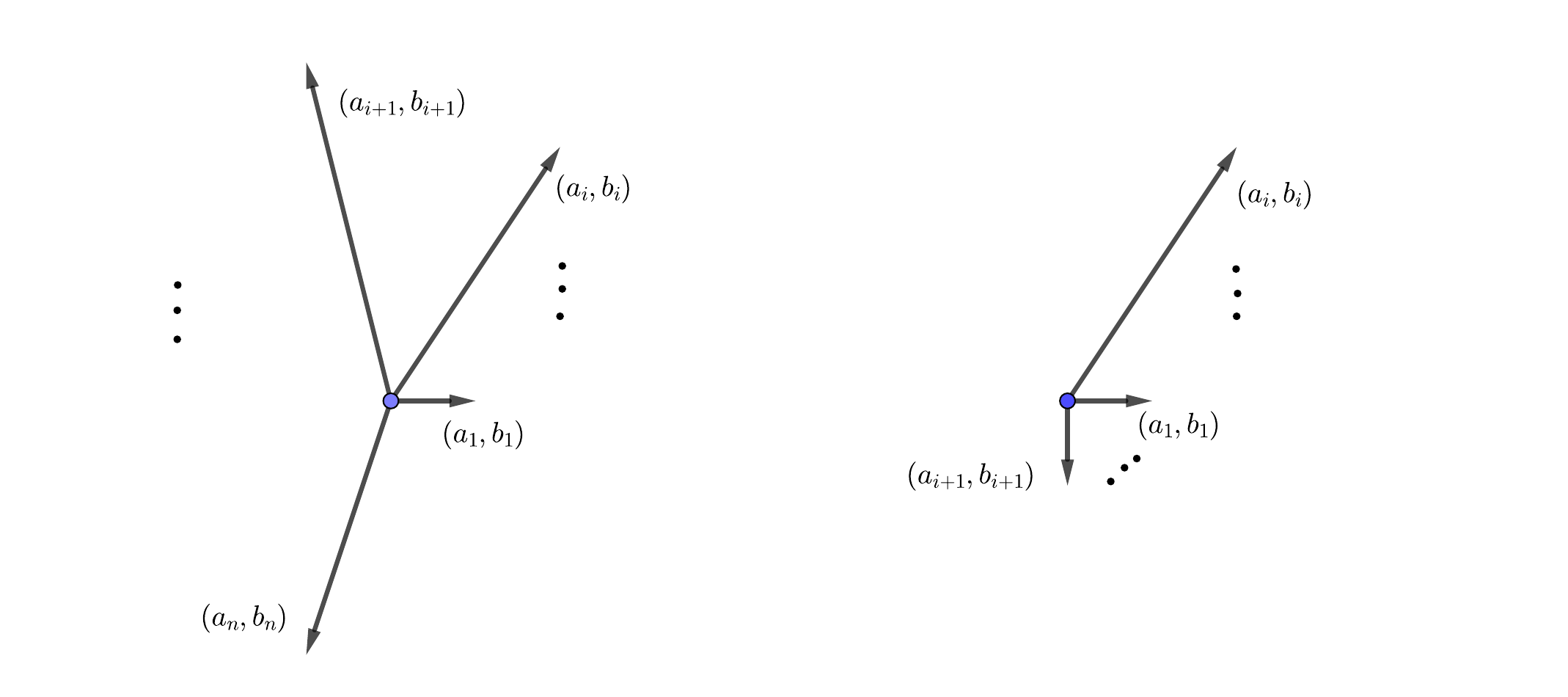}
	\caption{Two fans with vectors $(a_i,b_i),$ $i=1,\ldots, n$. The left one is complete, the right one is not.}
	\label{fan}
\end{figure}  

\emph{The $k\geq 2$ case.} Since the position of vertices and the length of edges are irrelevant, it is enough to find a complete and regular fan in $\mathbb R^2$ that is a collection of cones where some of the 2-dimensional cones are spanned by the desired vectors. 
Draw all the vectors $(a_i,b_i),$ $i=1,\ldots, k,$ in the $\mathbb{R}^2$-plane where the length of each vector is $\sqrt{a_i^2+b_i^2}$.  Up to relabeling the list of slopes, we can assume that the vector adjacent to the vector $(a_i,b_i)$ in the counter clockwise direction is precisely $(a_{i+1},b_{i+1})$ for $i=1,\ldots, k-1$, and the vector adjacent to the vector $(a_k,b_k)$ in the counter clockwise direction is $(a_1,b_1)$ (see Figure \ref{fan}). The collection of all these vectors and of the convex cones spanned by the pairs of adjacent vectors is a fan. If every two adjacent vectors span a convex cone, then this fan is complete. Moreover, if any two adjacent vectors form a $\mathbb{Z}^2$ basis then this fan is also regular. If the fan is not both regular and complete, we have to add more vectors to produce new cones and obtain desired fan. We will now explain this procedure. 

Start with a vector $(a_i,b_i)$ and the adjacent vector $(a_{i+1},b_{i+1})$ in the counter clockwise direction. We will consider two cases.

\emph{Case 1: the span of the two vectors  $(a_i,b_i), (a_{i+1},b_{i+1})$ (taken in the counter clockwise direction) is convex and they do not form a $\mathbb Z^2$-basis (Figure \ref{steps}).} Note that the determinant of two vectors is equal to the area of the parallelogram spanned by these vectors, so the area of the triangle spanned by these vectors is equal to the half of the determinant. Pick's theorem~\cite{Pick} tells us that if $\Delta$ is a triangle whose vertices lie in $\mathbb{Z}^2$, then
$$Area(\Delta)=I+\frac{B}{2}-1,$$
where $I$ is the number of lattice points contained in the interior of the triangle and $B$ is the number of lattice points contained on the boundary of $\Delta$.
Since the vertices of every such triangle are in $\mathbb{Z}^2$, it follows that $B\geq3.$ Two vectors form a $\mathbb{Z}^2$-basis if and only if the area of the spanned triangle is $\frac{1}{2}$. In that case, $B=3$ and $I=0$, and therefore there are no $\Z^2$ lattice points in the triangle except at its vertices. Note also that the two edges from the origin are primitive vectors, and so there are no $\Z^2$ lattice points along these edges, except at its vertices.

Consider now a triangle obtained from a vector $(a_i,b_i)$ and the adjacent vertex $(a_{i+1},b_{i+1})$ and mark all the interior and boundary lattice points in this triangle.
\begin{figure}
	\centering
	\includegraphics[width=14cm]{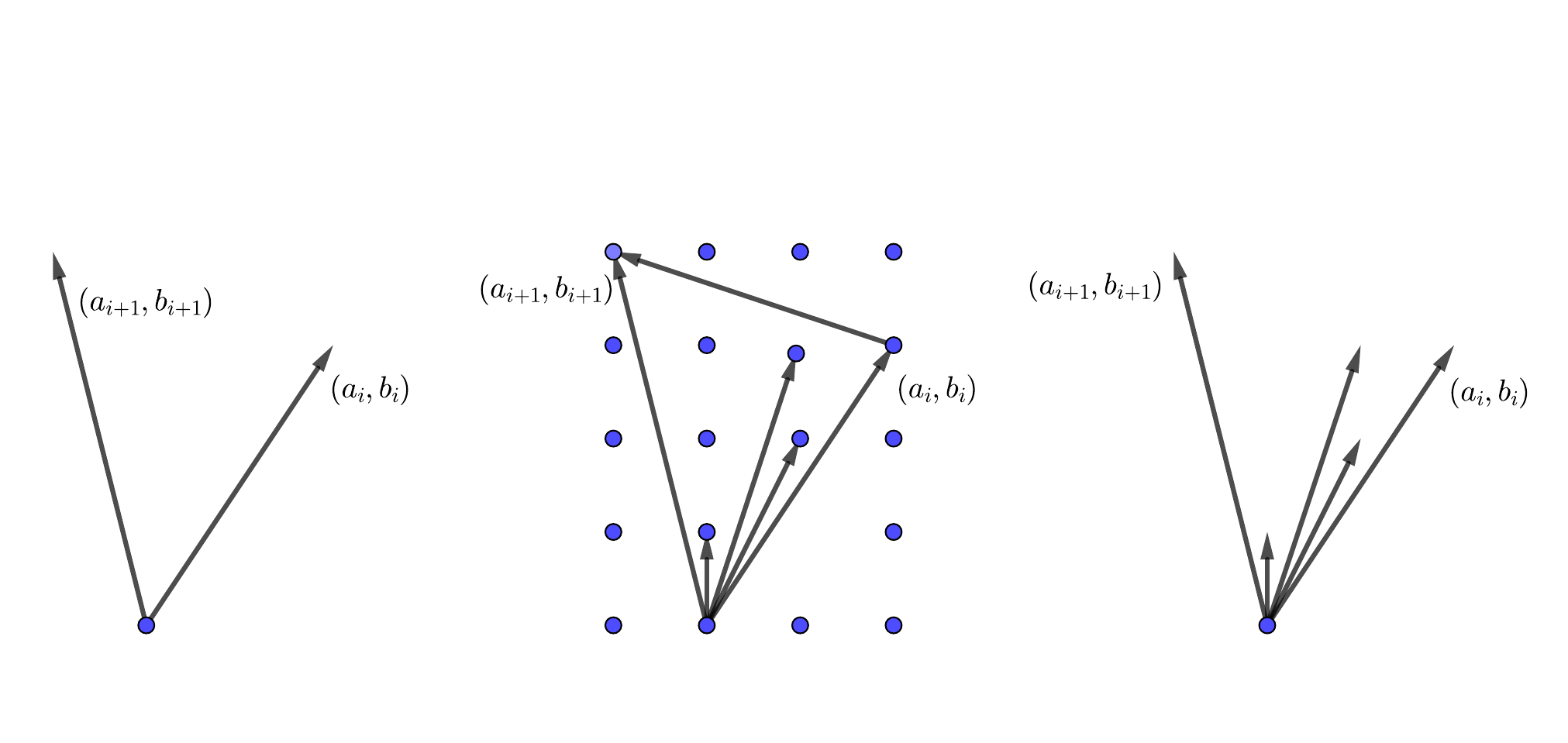}
	\caption{Adding new vectors to make the fan that contains the vectors $(a_i,b_i)$ and $(a_{i+1},b_{i+1})$ regular.}
	\label{steps}
\end{figure}  
Starting from the vector $(a_i,b_i)$ move in the counter clockwise direction. When we encounter a lattice point we add to the fan a new primitive vector $u$ that contains this lattice point. Since there are no interior and boundary points in the triangle spanned by $u$ and $(a_i,b_i)$, we know that they form $\mathbb{Z}^2$-basis. We ignore points that are multiples of $u,$ even if they are contained in the triangle spanned by $(a_i,b_i)$ and $(a_{i+1},b_{i+1})$. We continue from the vector $u$ in the counter clockwise direction to the next lattice point and its associated primitive vector $v$. If $v=(a_{i+1},b_{i+1})$, then $u$ and $(a_{i+1},b_{i+1})$ form a $\mathbb{Z}^2$-basis and we are done. If $v\neq (a_{i+1},b_{i+1})$, then it will be a new vector in our fan. We proceed in the same manner until we get to the vector $(a_{i+1},b_{i+1}).$ The procedure ends because there is a finite number of lattice points in the triangle spanned by $(a_i,b_i)$ and $(a_{i+1},b_{i+1})$.

\emph{Case $2$: the span of the two vectors $(a_i,b_i)$ and $(a_{i+1},b_{i+1})$ (taken in the counter clockwise direction) is concave (see for example Figure  \ref{steps2})} If this is the case, we add a new vector $v$ in between $(a_i,b_i)$ and $(a_{i+1},b_{i+1})$  so that we obtain two convex cones, the first one spanned by $(a_i,b_i)$ and $v$ and the second one spanned by $v$ and $(a_{i+1},b_{i+1}).$ We consider these two convex two triangles, track all the integer lattice points in both of them and, as in case $1$, add all the necessary vectors to our fan.

 \begin{figure}
	\centering
	\includegraphics[width=14cm]{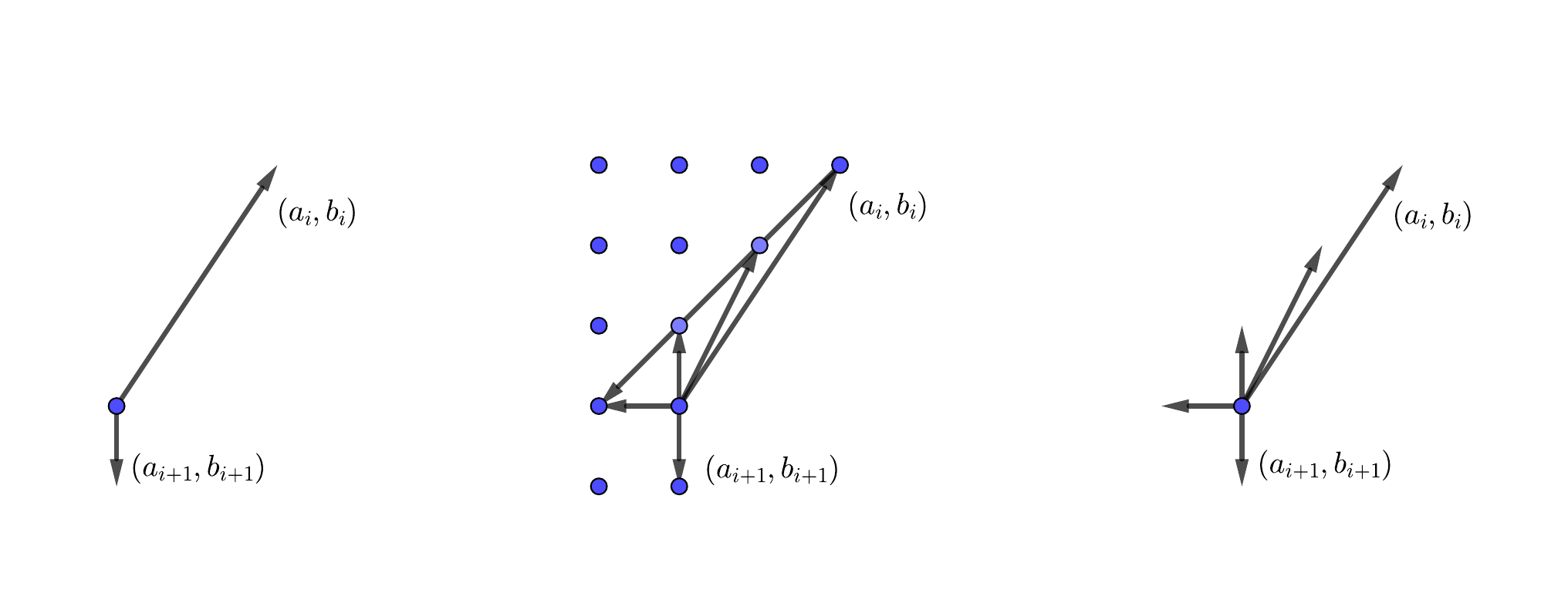}
	\caption{Completing the fan with vectors $(a_i,b_i)$ and $(a_{i+1},b_{i+1})$ by adding vectors so that pairs of vectors span convex cones.}
	\label{steps2}
\end{figure}

We apply the procedure described above to each pair of adjacent vectors of the original fan and obtain a fan that is both regular and complete.  All rational polytopes associated to this fan are closed, convex, simple, and smooth. In other words, they are Delzant. 
\end{proof}

\begin{example} Let us find a Delzant polytope with inward normal vectors $(3,2)$ and $(1,3).$ Start with the fan with two vectors $(3,2)$ and $(1,3)$, and two cones, one convex and one concave, as shown on the left in Figure \ref{primer}. Consider the triangle spanned by the vectors $(3,2)$ and $(1,3)$ and find all the interior and boundary integral lattice points. The first point moving in the counter clockwise direction from the vector $(3,2)$ is $(1,1).$ The next one is $(1,2)$ followed by $(1,3).$ Next, the span of vectors $(1,3)$ and $(3,2)$ in the counter clockwise direction is concave, so we first add the vector $(-1,0).$ In the convex triangle spanned by the vectors $(1,3)$ and $(-1,0)$ we find the vector $(0,1).$ The span by vectors $(-1,0)$ and $(3,2)$ is concave, so we add the vector $(0,-1).$ In the convex triangle spanned by $(-1,0)$ and $(0,-1)$ we don't add any vector, since these two already from a $\Z^2$-basis. Next, in the convex triangle spanned by the vectors $(0,-1)$ and $(3,2)$ we find the vectors $(1,0)$ and $(2,1).$ Finally, the convex hulls of the set of vectors $\{(3,2),(1,1),(1,2),(1,3),(0,1),(-1,0),(0,-1),(1,0),(2,1)\}$ produces a regular and complete fan.

\begin{figure}
	\centering
	\includegraphics[width=16cm]{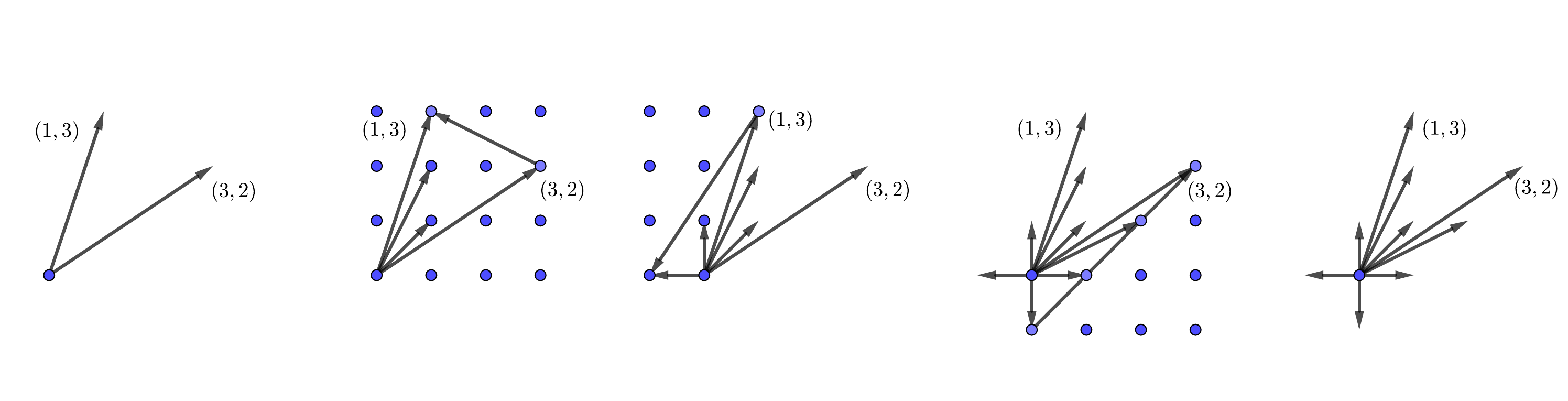}
	\caption{An example of the steps taken to find a complete and regular fan containing the vectors $(3,2)$ and $(1,3)$.}
	\label{primer}
\end{figure}   
\end{example}

\begin{lemma}\label{lemma blowup}
For any primitive vector $(a,b)\in\mathbb{Z}^2$ and any $n\in \mathbb{N}$ there is a Delzant polytope with $n$ vertices with the slope $(a,b).$ 
\end{lemma}
\begin{proof}
According to Lemma \ref{lemma normals} there is a Delzant polytope $\Delta$ with one edge $e$ given by inward normal vector $(a,b).$ We show that desired polytope can be obtained from $\Delta$ by the procedure of chopping off the corners precisely $n$ times. As $\Delta$ is closed, our resulting polytope is also closed. Recall that chopping off a corner of a Delzant polytope $\Delta$ corresponds to blowing up a toric manifold whose moment map image is $\Delta,$ as seen in Example \ref{example blow up}.  Denote by $f$ the edge of $\Delta$ adjacent to the edge $e$ in the counter clockwise direction and by $u$ its inward normal vector. We now chop the corner of $\Delta$ given by the vertex where the edges $e$ and $f$ meet. We obtain a new edge $e_1$ with inward normal $(a,b)+u.$ The slope of the vertex where $e_1$ and $f$ meet is precisely $(a,b)$ (see the leftmost image in Figure \ref{blowups}). We proceed in the same way in the case $n\geq2$. For clarity, we explain the next step. By chopping off the corner given by the vertex where $e$ and $e_1$ meet we obtain a new edge $e_2$ with the inward normal vector $2(a,b)+u.$ The vertex where the edges $e_1$ and $e_2$ meet has the desired slope $(a,b)$ (see the image on the right in Figure \ref{blowups}). In general, we chop off precisely $n$ corners to obtain the desired Delzant polytope and the corresponding fan gets $n$ new vectors $j(a,b)+u,$ $j=1,\ldots n$ (Figure \ref{blowupsfan}).

\begin{figure}
	\centering
 	\includegraphics[width=12cm]{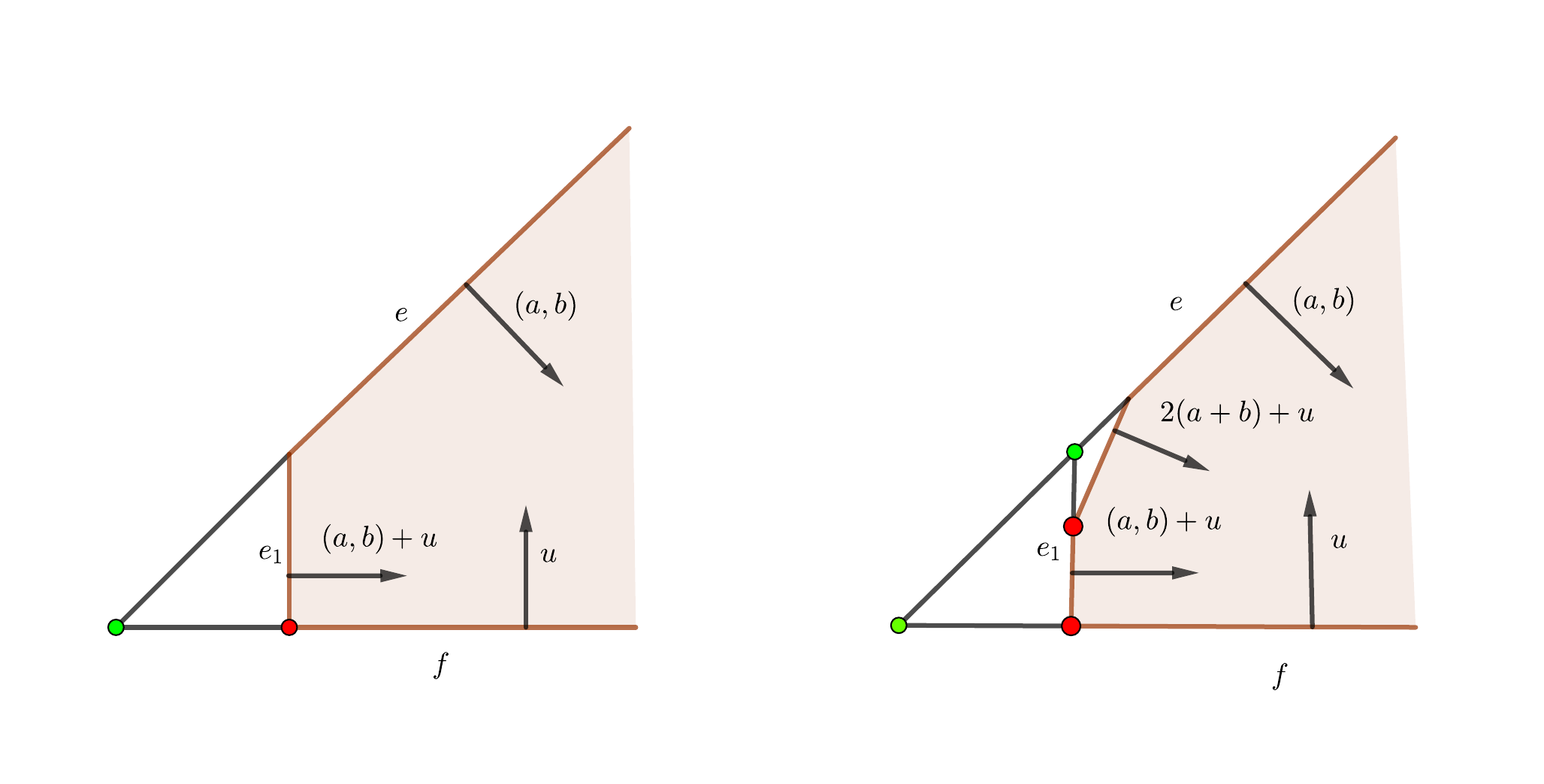}
 	\caption{Chopping the corners given by the green vertices; the slope at the red vertices is $(a,b).$}
 	\label{blowups}
\end{figure}  

\begin{figure}
	\centering
	\includegraphics[width=12cm]{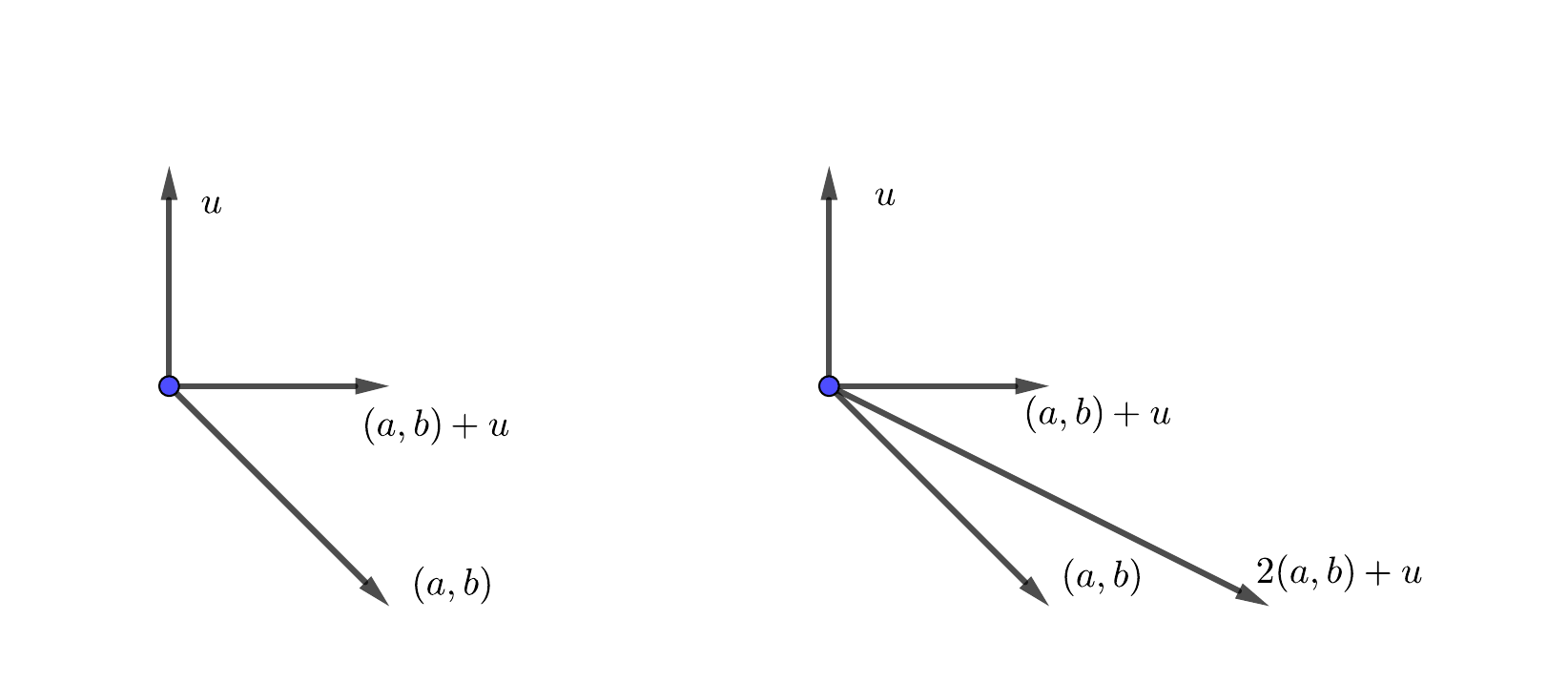}
	\caption{Associated fans after the chopping of the corners as seen in Figure~\ref{blowups}.}
	\label{blowupsfan}
\end{figure}  
\end{proof}

We are now ready to prove the main result of this subsection.

\begin{proof}[Proof of Theorem~\ref{proposition slopes}.]
From the collection of vectors $(a_i,b_i),$ $i=1,\ldots, k$, we take the largest possible pairwise distinct subset and use these to make a fan $F$. As explained in the proof of Lemma \ref{lemma normals} we can complete the fan $F$ to a regular fan $F'$ by adding new vectors so that every two adjacent vectors form a $\mathbb{Z}^2$-basis. Next, for every chosen vector $(a_i,b_i),$ if $n_i\geq 1$ is the number of times this vector appears in the given collection $(a_i,b_i),$  $i=1,\ldots, k,$ we add $n_i$ vectors $j(a_i,b_i)+u_i,$ $j=1,\ldots n_i$ to the fan $F'$ as explained in the proof of Lemma \ref{lemma blowup}. Here $u_i$ denotes the vector in $F'$ adjacent to the vector $(a_i,b_i)$ in the counter clockwise direction. Note that even if the vector $(a_i,b_i)$ appears only once in the given collection, we have to add an additional vector $(a_i,b_i)+u,$ in order to obtain the vertex with the slope $(a_i,b_i)$ in the corresponding polytope. This way we obtain a complete and regular fan $F''$, and any Delzant polytope given by the fan $F''$ is a polytope with desired slopes. 
\end{proof}

\subsection{Non-uniqueness of Weinstein manifolds from Delzant polytopes}\label{subsection repetitions}
In this section, we explore when the smoothing of the vertices of two Delzant polytopes leads to symplectomorphic completions of the corresponding Weinstein domains. For example, the Weinstein complements of any two total toric divisors each with a single node smoothed, will always be Weinstein homotopic so their completions will be symplectomorphic. This example was seen in \cite{ACGMMSW1} and was proven there using $1$-handle slides from Weinstein Kirby calculus. Here we give a much more general statement of when two complements  may end up being symplectomorphic  (Proposition~\ref{propknodes}), as well as some more computationally useful corollaries and examples. 
 
Further we show that the symplectomorphisms provided by Proposition~\ref{propknodes} do not make a complete list of identifications--there exist complements of smoothed divisors which are not related by the condition given in Proposition~\ref{propknodes} but are symplectomorphic (see Remark \ref{remark counterexample0} and Theorem \ref{remark counterexample}). Nevertheless, in practice, many specific examples of Weinstein domains arising as the complements of partially smoothed total toric divisors which are not related by Proposition~\ref{propknodes} can be distinguished by their topological invariants such as their second homology and intersection form. The computation of these invariants is explained in Section~\ref{sec:applications}, and explicitly computed for the examples in Section \ref{section examples}.

Recall that $\mathcal{W}_{T^2, \{(a_1,b_1)\ldots (a_k,b_k)\}}$ denotes the Weinstein domain constructed by attaching $k$ Weinstein 2-handles along the Legendrian co-normal lifts of the curves $(a_1,b_1),\ldots,(a_k,b_k)$ in the torus $T^2$.
      
\begin{prop} \label{propknodes} 
 Consider any $\{V_1, \ldots V_k\}$-centered toric $4$-manifold where $s(V_i)=(a_i,b_i),$ for all $i=1,\ldots, k$. If there is an $SL(2,\mathbb{Z})$ transformation mapping the set $\{(a_1, b_1), \ldots, (a_k,b_k)\}$ to the set $\{(a'_1, b'_1), \ldots, (a'_k,b'_k)\}$ then the completions of Weinstein domains $\mathcal{W}_{T^2, \{(a_1,b_1),\ldots, (a_k,b_k)\}}$ and $\mathcal{W}_{T^2, \{(a'_1,b'_1),\ldots, (a'_k,b'_k)\}}$ are symplectomorphic. (The set of vertices and slopes that we consider are unordered.)
\end{prop}

\begin{proof}
Take any toric symplectic 4-manifold $M$ so that there are vertices $V_1,\ldots, V_k$ in the corresponding moment polytope $\Delta$ with $s(V_i)=(a_i,b_i),$ $i=1,\ldots, k,$ and such that $\Delta$ is $\{V_1, \ldots V_k\}$-centered. According to Theorem \ref{thm:toricWein} and Lemma \ref{lemma slope}, $\mathcal{W}_{T^2, \{(a_1,b_1),\ldots, (a_k,b_k)\}}$  is Weinstein homotopic to the complement of the total toric divisor of $M$ smoothed at the fixed points of the toric action that map under the moment map to the vertices $V_1,\ldots, V_k$. If $G\in SL(2,\mathbb{Z})$ is the given transformation of pairs, then $(G^{-1})^T\in SL(2,\mathbb{Z})$ is the transformation that maps the polytope $\Delta$ to some Delzant polytope $\Delta'$ with the vertices $V_1', \ldots, V_k'$ such that  $s(V_i')=(a_i',b_i'),$ $i=1,\ldots, k.$ Note that $\Delta'$ will be $\{V'_1, \ldots V'_k\}$-centered, since linear transformations preserve the centeredness condition. Thus, we obtain another Delzant polytope and the corresponding toric symplectic manifold $M'$ is equivariantly symplectomorphic to $M.$
Therefore, the completions of the complement of the total toric divisor of $M$ smoothed at the vertices $V_1,\ldots, V_k$ and the complement of the total toric divisor of $M'$ smoothed at the vertices $V_1',\ldots, V_k'$ are symplectomorphic.     
\end{proof}
          
\begin{remark} The result of Proposition \ref{propknodes} cannot be extended to $-SL(2,\mathbb Z)$ transformations in a way that is consistent with our conventions relating the orientation of the slope and the ray. Namely, every $-SL(2,\mathbb Z)$ transformation of the 2-torus is a composition of the reflection
        $G=\begin{pmatrix}
   -1 & 0 \\
    0 & 1
  \end{pmatrix}$      
and an $SL(2,\mathbb Z)$ transformation and therefore it is enough to observe only the reflection. Since $G$ only changes the orientation of the acting torus, the corresponding   moment map image changes by a reflection over $x$-axis. Thus, the smoothing of the node that we perform after the change of the orientation of the 2-torus is related to the smoothing of the standard node also with a reflection over $x$-axis. Since the ray given by the direction $(1,1)$ reflects to the ray given by the direction $G(1,1)=(-1,1),$ a core and a co-core of the reflected handle will project to the ray given by the direction $(-1,1)$. Further, the slope of the attaching sphere is given by a vector orthogonal to the ray and by our convention it is a $-\pi/2$ rotation of the vector of the ray, hence the slope is $(1,1).$  
However, according to Section \ref{subsection 3.1}, when $G$ is the transformation of the rays, then $(G^{-1})^T$ is the transformation of the slopes. Thus, the slope of the attaching sphere of the reflected handle should be   $(G^{-1})^T(1,-1)=G(1,-1)=(-1,-1).$ Therefore, we will not consider  $-SL(2,\mathbb Z)$ transformations. 
\end{remark}
         
We now focus on the set of two slopes. Note first that if $(a_1,b_1),(a_2,b_2)\in \Z^2$ and if $\det\begin{pmatrix}
   a_1 & a_2 \\
    b_1 & b_2
  \end{pmatrix}=1$, 
then according to Proposition \ref{proposition 2 good slopes} there is a toric 4-manifold with two vertices with the slopes $(a_1,b_1)$ and $(a_2,b_2)$ and that is centered with respect to these two vertices.
     
\begin{cor} \label{corollary det=1}
Let $(a_i,b_i)\in \Z^2$ for $i=1,2$. If   $\det\begin{pmatrix}
   a_1 & a_2 \\
    b_1 & b_2
  \end{pmatrix}=\det\begin{pmatrix}
   a_1' & a_2' \\
    b_1' & b_2'
  \end{pmatrix}=\pm1$ then the completions of $\mathcal{W}_{T^2, \{(a_1,b_1), (a_2,b_2)\}}$ and $\mathcal{W}_{T^2, \{(a'_1,b'_1),(a'_2,b'_2)\}}$ are symplectomorphic.
\end{cor}   
 
\begin{proof}  In order to apply Proposition \ref{propknodes} we only have to notice that the matrix
 $\begin{pmatrix}
   a_1' & a_2' \\
    b_1' & b_2'
  \end{pmatrix}\cdot \begin{pmatrix}
   a_1 & a_2 \\
    b_1 & b_2
  \end{pmatrix}^{-1}$ belongs to $SL(2,\mathbb Z)$ and it maps $\begin{pmatrix}
   a_1 & a_2 \\
    b_1 & b_2
  \end{pmatrix}$ to $\begin{pmatrix}
   a_1' & a_2' \\
    b_1' & b_2'
  \end{pmatrix}.$
\end{proof}     
       
\begin{cor} Let $(a_1,b_1)$ and $(a_2,b_2)$ be the primitive vectors 
for which there is a $\{V_1,V_2\}$-centered Delzant polytope with $s(V_i)=(a_i,b_i),$ $i=1,2.$    
If
   $ \begin{pmatrix}
   a_1' & a_2' \\
    b_1' & b_2'
  \end{pmatrix}\cdot \begin{pmatrix}
   a_1 & a_2 \\
    b_1 & b_2
  \end{pmatrix}^{-1}\in SL(2,\mathbb{Z})$ 
then the completions of $\mathcal{W}_{T^2, \{(a_1,b_1), (a_2,b_2)\}}$  and  $\mathcal{W}_{T^2, \{(a'_1,b'_1),(a'_2,b'_2)\}}$ are symplectomorphic.      
\end{cor}   
 
\begin{proof} We apply Proposition \ref{propknodes} directly since the given $SL(2,\mathbb Z)$ transformation maps 
$\begin{pmatrix}
   a_1 & a_2 \\
    b_1 & b_2
  \end{pmatrix}$ to $\begin{pmatrix}
   a'_1 & a'_2 \\
    b'_1 & b'_2
  \end{pmatrix}$
\end{proof}   

\begin{remark} \label{remark counterexample0} 
Proposition~\ref{propknodes} motivated the stronger result given in Proposition~\ref{prop:1handles}, where we show that an $SL(2,\mathbb{Z})$ transformation of the sets of slopes corresponds to 1-handle slides, and therefore leads to Weinstein homotopic domains. Note that Proposition~\ref{prop:1handles}  does not require the existence of a
$\{V_1, \ldots V_k\}$-centered toric $4$-manifold with the desired slopes. Since centeredness is a restrictive condition,  Proposition \ref{propknodes} does not give the complete list of identifications. For instance, the pairs of slopes
  $\begin{pmatrix}
   1 & 1 \\
    0 & 0
  \end{pmatrix}$ and    $\begin{pmatrix}
   1 & -1 \\
    0 & 0
  \end{pmatrix},$ 
cannot be realized by our method, because there does not exist $\{V_1,V_2\}$-centered polytope with $s(V_1)=s(V_2)$. However, the Weinstein domains $\mathcal{W}_{T^2, \{(1,0),(1,0)\}}$ and $\mathcal{W}_{T^2, \{(1,0),(-1,0)\}}$ are Weinstein homotopic, as we show in Example \ref{remark counterexample}. 
  
Further, even when centeredness is realizable, there are symplectomorphic examples not covered by Proposition \ref{propknodes} or Proposition \ref{prop:1handles}. In particular, the pairs of slopes
  $\begin{pmatrix}
   1 & 3 \\
    -1 & 1
  \end{pmatrix}$ and    $\begin{pmatrix}
   4 & 0 \\
    1 & 1
  \end{pmatrix},$ can be realized by $\{V_1,V_2\}$-centered Delzant polytopes (see Figure \ref{two possible}), but are related only by an $SL(2,\mathbb Q)$ transformation
  $\begin{pmatrix}
  1 & -3 \\
   \frac{1}{2} &  -\frac{1}{2}
  \end{pmatrix}$  
and not by an $SL(2,\mathbb Z)$ transformation. Still, the Weinstein domains $\mathcal{W}_{T^2, \{(1,-1),(3,1)\}}$ and $\mathcal{W}_{T^2, \{(0,1),(4,1)\}}$ are Weinstein homotopic, as we show in Theorem \ref{remark:sl2q}.
\end{remark}

\begin{figure}
	\centering
	\includegraphics[width=13cm]{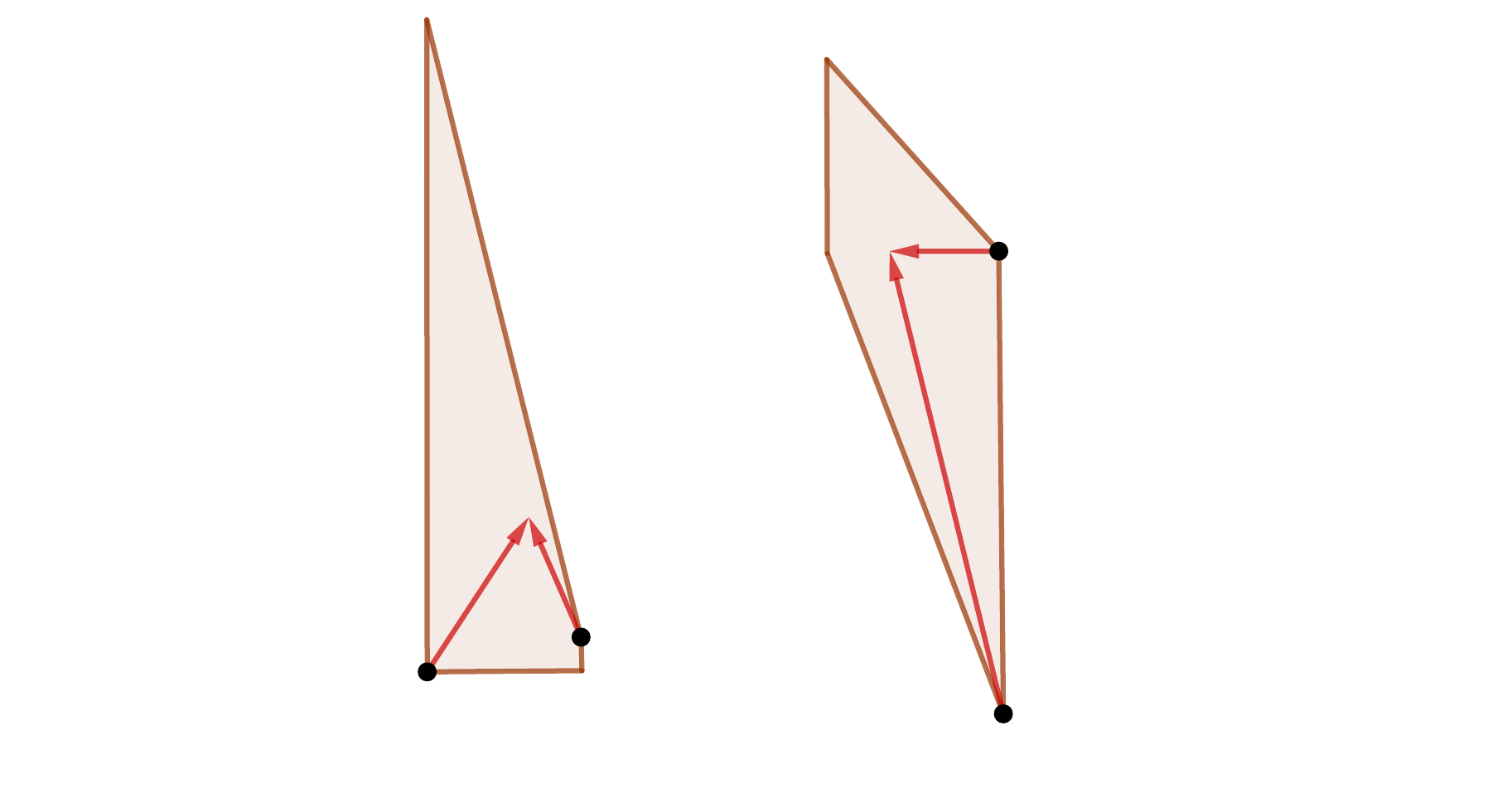}
	\caption{$\{V_1,V_2\}$-centered Delzant polytopes, with the slopes $(1,-1), (3,1)$  (the left one) and $(4,1),(0,1)$  (the right one). 
		They are not related by $SL(2,\mathbb Z)$ transformation.}
	\label{two possible}
\end{figure}  
    
We now give a list of examples where we can apply our results.

\begin{example} \label{example CP2} \emph{Complements of smoothed total toric divisors in $\C P^2$.} The complement of the total toric divisor of $\mathbb{CP}^2$ smoothed in two nodes does not depend on the choice of which two nodes we smooth.
First note that $\mathbb{CP}^2$ is centered (see Corollary \ref{monotone=centered}), so we are allowed to smooth all three vertices.
Denote the vertices of the polytope by $V_1,$ $V_2$, and $V_3$ where $s(V_1)=(1,-1),$ $s(V_2)=(1,2)$ and $s(V_3)=(-2,-1)$. 
The $SL(2,\mathbb{Z})$ transformation 
$\begin{pmatrix}
   0 & -1 \\
    1 & -1
  \end{pmatrix}$ shows that the completions of
$\mathcal{W}_{T^2, \{s(V_1), s(V_2)\}}$ and $\mathcal{W}_{T^2, \{s(V_2), s(V_3)\}}$ are symplectomorphic.
The $SL(2,\mathbb{Z})$ transformation 
$\begin{pmatrix}
   -1 & 1 \\
    -1 & 0
  \end{pmatrix}$
 shows that the completions of $\mathcal{W}_{T^2, \{s(V_1), s(V_2)\}}$ and $\mathcal{W}_{T^2, \{s(V_3), s(V_1)\}}$ are symplectomorphic.
\end{example}
       
\begin{example} \label{example rectangle} \emph{Complements of smoothed total toric divisors in $\mathbb{CP}^1 \times \mathbb{CP}^1$.}
\begin{enumerate}
	\item The complement of the total toric divisor of $(\mathbb{CP}^1\times\mathbb{CP}^1, \omega_{a,a})$ smoothed in two opposite nodes does not depend on the choice of opposite nodes. First note that the rays of opposite vertices in a rectangle intersect in a point if and only if the rectangle is the square. Thus, we are imposing here the condition that both $\mathbb{CP}^1$'s have the same symplectic area. The result now follows from the fact that the $\pi/2$ rotation is an $SL(2,\mathbb Z)$ transformation $J=\begin{pmatrix}
	   0 & -1 \\
	    1 & 0
	  \end{pmatrix}$ 
	and it maps one pair of opposite vertices to the other pair of opposite vertices in the moment map image. Then, $(J^{-1})^T$ will be the transformation of the desired inward normal vectors. Note that in this very special case, both matrices coincide. 
	\item The complement of the total toric divisor of $(\mathbb{CP}^1\times\mathbb{CP}^1,\omega_{a,b}),$ where $b/2<a<2b,$ smoothed in any two adjacent nodes is independent of the choice of adjacent nodes. The condition $b/2<a<2b$ ensures that the rays of any two adjacent vertices intersect in the interior of the moment map image. The result follows from that fact that any two adjacent vertices in the rectangle can be $SL(2,\mathbb{Z})$-rotated to any other two adjacent vertices. 
	 
	\item One could also compare the complements after smoothing opposite nodes and the complements after smoothing adjacent nodes in $(\mathbb{CP}^1\times\mathbb{CP}^1, \omega_{a,a})$. Since the determinant of the slopes for opposite vertices is zero and the determinant for adjacent vertices is $\pm2$, there is no $SL(2,\mathbb{Z})$ transformation between the sets of slopes, so our method does not give an answer. However, having different determinants is enough to conclude that the corresponding homologies of the Weinstein manifolds are different, as we will show later in Remark \ref{remark two nodes}.
	Thus, the corresponding Weinstein manifolds are not  even homeomorphic. 
	 
	\item The complement of the total toric divisor of $(\mathbb{CP}^1\times\mathbb{CP}^1,\omega_{a,a})$ smoothed at any three nodes does not depend on the choice of three nodes, which follows after finding appropriate rotations.
\end{enumerate}
\end{example}
 
\begin{example}\emph{Complements of smoothed total toric divisors in $\mathbb{CP}^2\sharp3\overline{\mathbb{CP}^2}$}. The complement of the total toric divisor of $\mathbb{CP}^2\sharp3\overline{\mathbb{CP}^2}$ smoothed at any three alternating nodes does not depend on the choice of the nodes. Also, the complements of the total toric divisor of $\mathbb{CP}^2\sharp3\overline{\mathbb{CP}^2}$ smoothed at any three adjacent nodes does not depend on the choice of the nodes.  
\end{example}      
        
\subsection{The centeredness condition}\label{subsection centeredness}

In this section we discuss a restrictiveness of the centeredness condition. We have already seen that there is a Delzant polytope with any given set of slopes (Theorem \ref{proposition slopes}). However, the centeredness condition may not follow. For instance, if two vertices have the same slope, then, since the slope and the direction of the ray are related by a $\pi/2$ rotation, we conclude that these vertices have parallel rays that will never intersect in a point. We are able to overcome this particular obstacle using almost toric manifolds (introduced in Section~\ref{intro to almost toric}) instead of  toric manifolds. Unfortunately, there is not yet a systematic characterization of the sets of slopes for which there exists a toric $4$-manifold that is centered with respect to those slopes. Even in the case of just two slopes $\{s(V_1), s(V_2)\}$ the existence of $\{V_1,V_2\}$-centered polytope with those two slopes is not in general obvious. We are able to show that if two slopes form a $\mathbb{Z}^2$-basis, then we can find such a polytope (Proposition \ref{proposition 2 good slopes}). We expect the same is true for arbitrary sets of two or three distinct slopes. However, the case of four slopes includes more obstacles. We give an example of four (and more) slopes that cannot be realised with the centeredness condition (Proposition \ref{example non-centered}). We also show how the centeredness condition relates to monotone toric symplectic $4$-manifolds, and produce infinite families of examples of partially centered toric $4$-manifolds (Theorem \ref{thm:infinitemanifolds}).

\begin{prop} \label{proposition 2 good slopes}  Let $(a_1,b_1)$ and $(a_2,b_2)$ be any two primitive vectors that span a positively oriented $\mathbb Z^2$-basis (i.e. the determinant of the vectors is $1$). Then, there is a $\{V_1,V_2\}$-centered Delzant polytope such that $s(V_i)=(a_i,b_i),$ $i=1,2.$
\end{prop}

\begin{proof} 
The transformation $G=\begin{pmatrix}
   a_2 & -a_1 \\
    b_2 & -b_1
  \end{pmatrix}\in SL(2,\mathbb Z)$ maps  the vectors $(0,-1)$ and $(1,0)$ to the vectors $s(V_1)$ and $s(V_2)$ respectively.
One can easily find a Delzant polytope so that the slopes of two vertices are $s(V_1')=(0,-1)$ and $s(V_2')=(1,0)$,
 for example, the moment map image of
$\mathbb{CP}^2\sharp\overline{\mathbb{CP}}^2$ shown in Figure \ref{two slopes} on the left.
If $G$ is the transformation of the inward normal vectors, then, as explained in Lemma \ref{lemma slope}, $(G^{-1})^T$ is the transformation of the corresponding polytopes. 
The polytope on the right in Figure \ref{two slopes}, obtained as $(G^{-1})^T$-transformation of the polytope on the left in Figure \ref{two slopes}, is a polytope with the desired slopes.

\begin{figure}
	\centering
	\includegraphics[width=15cm]{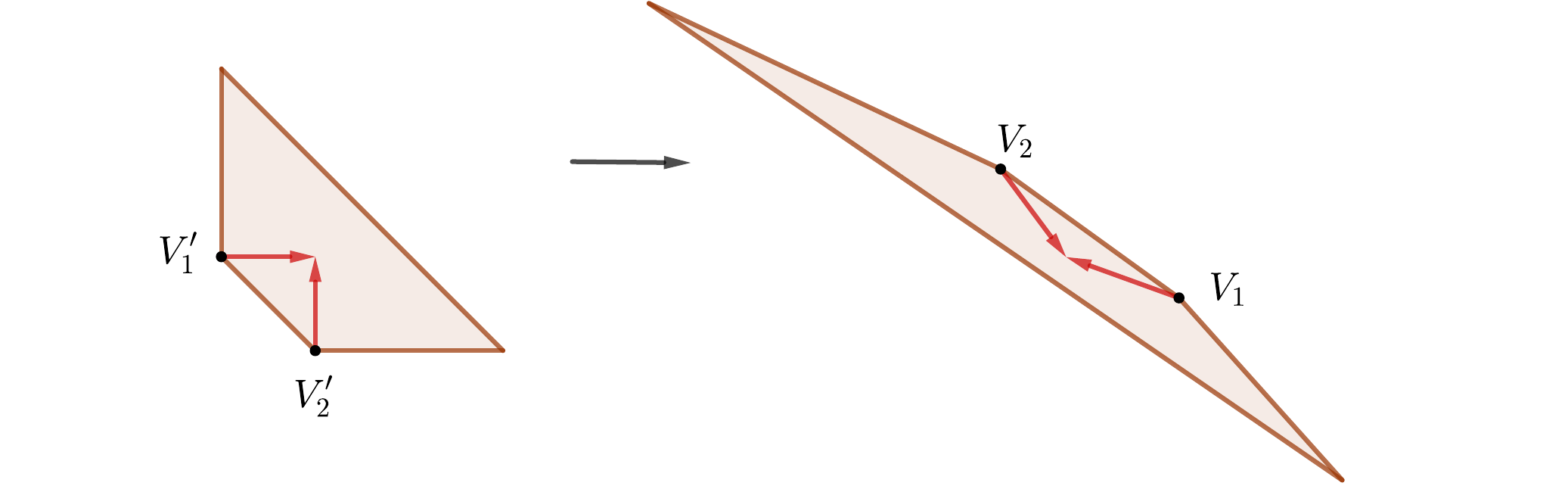}
	\caption{$SL(2,\mathbb Z)$ transformation of the polytope of $\mathbb{CP}^2\sharp\overline{\mathbb{CP}}^2$.}
	\label{two slopes}
\end{figure}   
\end{proof}

Since the vectors used in the proof of previous proposition are obtained by blowing up a standard Delzant corner, Proposition \ref{proposition 2 good slopes}
can be reformulated as follows.

\begin{prop}\label{prop centered blow} If we blow up a Delzant polytope at a vertex $V$ and the size of the blow up is sufficiently small, we obtain  two new vertices $V_1$ and $V_2$ 
and a $\{V_1,V_2\}$-centered Delzant polytope.
\end{prop}
\begin{proof} After an appropriate $SL(2,\mathbb Z)$ transformation, we can assume that we are blowing up the standard corner and therefore obtain two new vertices with rays $(1,0)$
and $(0,1).$ The importance of the size of the blow up is depicted in Figure \ref{fig blow up size}. In particular, a small enough blow up allows the polytope to be centered with respect to the two new vertices.

\begin{figure}
	\centering
	\includegraphics[width=12cm]{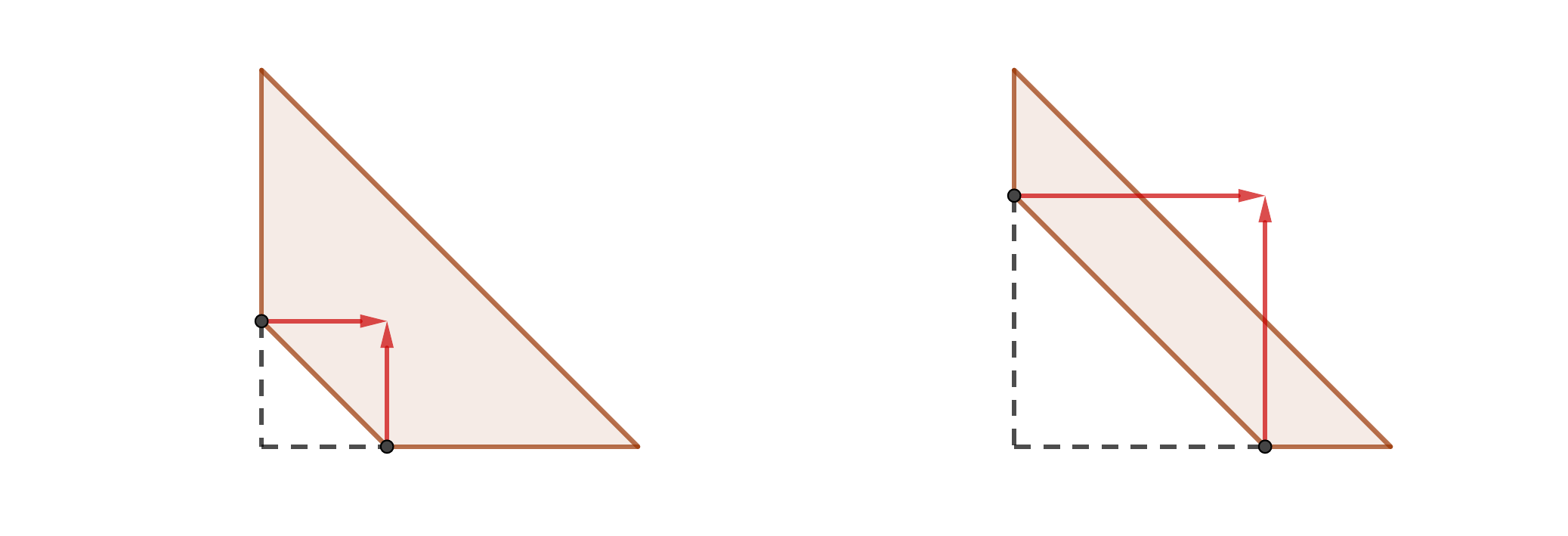}
	\caption{Small size of the blow up-on the left, large size of the blow up-on the right.}
	\label{fig blow up size}
\end{figure}  
\end{proof}

Without loss of generality, we can require that the desired rays for a centeredness condition intersect at the origin. Indeed, if the rays intersect in one point, then an appropriate translation of the polytope, that corresponds to adding a constant to a moment map, moves the intersection point to the origin.  If $\nu_1,\ldots, \nu_n\in\mathbb Z^2$ denote the vectors of a given fan, then every corresponding polytope is given by 
\begin{equation}\label{eq:delzant}
	\Delta=\bigcap_{k=1}^n\{x\in\mathbb R^2|<x,\nu_k>\hskip2mm\geq\lambda_k\},
\end{equation}
where $\lambda_1,\ldots,\lambda_n$ are any real numbers so that the intersection of half-spaces is not an empty set. Translation of the polytope is obtained by an appropriate change of constants $\lambda_i$'s. The condition $\lambda_1<0,\ldots, \lambda_n<0$ guarantees that the origin is contained in the interior of $\Delta.$

We now give a necessary and sufficient condition so that a ray contains the origin.

\begin{lemma}\label{lemma same lambdas} 
Let $\nu_i$ and $\nu_{i+1}$ be inward normal vectors of the edges meeting in a vertex $V.$
The ray from the vertex $V$ contains the origin if and only if $\lambda_i=\lambda_{i+1}$ and $\lambda_i<0.$
\end{lemma}

\begin{proof}  Let $\nu_i=(v_1,v_2)$ and $\nu_{i+1}=(u_1,u_2)$ be given inward normal vectors as in Figure \ref{fig ray}. Then, the corresponding edges are given by equations  $<(x,y),(v_1,v_2)>=\lambda_{i}$ and  $<(x,y),(u_1,u_2)>=\lambda_{i+1}.$ Since the vertex $V$ is the common point of these edges and the polytope is regular, that is $v_1u_2-v_2u_1=1,$ we obtain
$$V=(\lambda_iu_2-\lambda_{i+1}v_2,\lambda_{i+1}v_1-\lambda_iu_1).$$
In order to choose $\lambda_i$ and $\lambda_{i+1}$ so that the ray in $V$ goes through $(0,0)$ we will first find the equation of the ray. The direction of the ray is $r(V)=(-(v_2-u_2),v_1-u_1)$ and therefore the normal vector to the ray is  $(v_1-u_1,v_2-u_2)$ and the equation of the line containing the ray is
$$<(v_1-u_1,v_2-u_2),(x,y)>=0.$$
We substitute $V$ into the equation of the line and obtain  $(\lambda_i-\lambda_{i+1})(u_1v_2-v_1u_2)=0.$ Because of the regularity of the polytope we find $\lambda_i=\lambda_{i+1}.$ We conclude that the ray of the vertex $V$ goes through $(0,0)$ if and only if $\lambda_i=\lambda_{i+1}.$ The condition $\lambda_i<0$ is required to have the origin in the intersection of the given two half-spaces.
\end{proof}

The result of Lemma \ref{lemma same lambdas} leads us to recall the definition of a monotone Delzant polytope. A Delzant polytope given by Equation (\ref{eq:delzant}) where $\lambda_1,\ldots,\lambda_n$ are real numbers, is called \emph{monotone} if $\lambda_1=\cdots=\lambda_n.$ Such a polytope corresponds to a monotone symplectic (toric) manifold. Thus, we conclude the following.

\begin{figure}
	\centering
	\includegraphics[width=15cm]{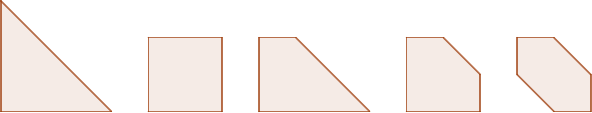}
	\caption{Monotone Delzant 2-dimensional polytopes}
	\label{fig monotone}
\end{figure} 

\begin{prop} \label{monotone=centered} A Delzant polytope is centered with respect to all of its vertices if and only if it is monotone.
\end{prop}

Note that, up to a rescaling of the symplectic form, there are exactly five monotone Delzant 2-dimensonal polytopes  (see Figure \ref{fig monotone}). These correspond to
$\mathbb{CP}^2,$ $(\mathbb{CP}^1\times\mathbb{CP}^1,\omega_{a,a}),$ $\mathbb{CP}^2\sharp\overline{\mathbb{CP}^2},$ $\mathbb{CP}^2\sharp\overline{2\mathbb{CP}^2}$, and 
$\mathbb{CP}^2\sharp\overline{3\mathbb{CP}^2}$. In particular, there is no fully centered polytope with more than $6$ vertices. Thus, our primary source of examples comes from partially centered polytopes.
 
In the remainder of the section we explore the range of partially centered polytopes. On the one hand, we provide an infinite list of slopes for which there is no partially centered polytope realizing those slopes. On the other hand, we find an infinite list of slopes for which there is a polytope that is partially centered with respect to vertices realizing those slopes. These examples initiate the study of the existence of partially centered Delzant polytopes, but we hope to understand these objects more fully in the future.

\begin{prop}\label{example non-centered}
There is no $\{V_1,V_2,V_3,V_4\}$-centered Delzant polytope where 
$$s(V_1)=(1,1), s(V_2)=(1,2), s(V_3)=(-K,-1), s(V_4)=(0,-1),$$
for any $K\geq2.$
\end{prop}

\begin{proof}
Suppose there were such a polytope. Denote by $u^k$ and $v^k,$ $k=1,2,3,4$ the inward normal vectors of the edges meeting in these vertices. Then, according to Proposition \ref{prop def of slopes}, it follows that the corresponding slopes are $s(V_k)=v^k-u^k$. Due to the smoothness of the polytope, represented by the equalities $\det(v^k,u^k)=1,$  we obtain
\begin{align*}
	u^1&=(a-2,a-1),&~v^1&=(a-1,a),\\
	u^2&=(b-1,2b-1),&~v^2&=(b,1+2b),\\
	u^3&=(Kc+1+K,c+1),&~v^3&=(Kc+1,c),\\
	u^4&=(1,d+1),&~v^4&=(1,d),
\end{align*}
for some integers $a,b,c,d.$ Without loss of generality (by an $SL(2,\Z)$ transformation) we may assume that $d=0.$ Then, the inward normal vectors of the edges meeting in $V_4$ are $(1,1)$ and $(1,0).$ 

The corresponding rays of the vertices are
$$r(V_1)=(-1,1), r(V_2)=(-2,1), r(V_3)=(1,-K), r(V_4)=(1,0).$$ If the rays intersect in the origin, then the vertices  are  
$$V_1=(-\lambda_1,\lambda_1), V_2=(-2\lambda_2,\lambda_2), V_3=(\lambda_3,-K\lambda_3),~\text{and}~V_4=(\lambda_4,0),$$
where $\lambda_k<0,$ for all $k=1,2,3,4,$ are defined in (\ref{eq:delzant}).

Since the polytope is defined as an intersection of half-spaces, we have to check that all the vertices belong to each half-space. That is, we have to check
$$\langle V_j,u^k\rangle\geq\lambda_k~\text{and}~\langle V_j,v^k\rangle\geq\lambda_k,$$
for all $j, k=1,2,3,4.$ Since the cases $j=k$ are already satisfied, we have to check
 the following system of inequalities

\begin{align*}
 \langle(-\lambda_1,\lambda_1),u^k\rangle&\geq\lambda_k,~&~\langle(-\lambda_1,\lambda_1),v^k\rangle&\geq\lambda_k,~\text{for}~k=2,3,4;\\
  \langle(-2\lambda_2,\lambda_2),u^k\rangle&\geq\lambda_k,~&~\langle(-2\lambda_2,\lambda_2),v^k\rangle&\geq\lambda_k,~\text{for}~k=1,3,4;\\
    \langle(\lambda_3,-K\lambda_3),u^k\rangle&\geq\lambda_k, ~&~\langle(\lambda_3,-K\lambda_3),v^k\rangle&\geq\lambda_k,~\text{for}~k=1,2,4; \\
  \langle(\lambda_4,0),u^k\rangle&\geq\lambda_k, ~&~\langle(\lambda_4,0),v^k\rangle&\geq\lambda_k,~\text{for}~k=1,2,3.
\end{align*}

From the inequalities $\langle(\lambda_3,-K\lambda_3),v^4\rangle\geq\lambda_4$ and $\langle(\lambda_4,0),u^3\rangle\geq\lambda_3$, it follows that  
$\lambda_3\geq\lambda_4 $ and $(Kc+1+K)\lambda_4\geq\lambda_3.$
Thus, $c\leq-1.$

\begin{figure}
	\centering
	\includegraphics[width=8cm]{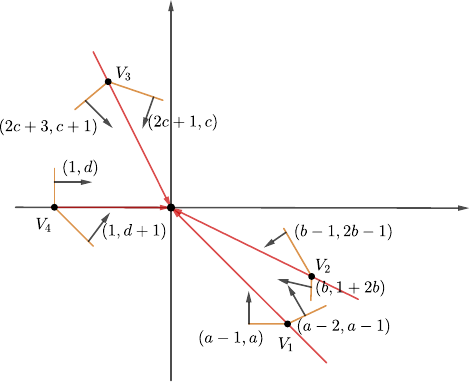}
	\caption{Vertices on the centered rays and the adjacent inward normal vectors.}
	\label{fig1}
\end{figure}    
   
Next, from the inequality $\langle(-2\lambda_2,\lambda_2),v^3\rangle\geq\lambda_3$ it follows that $-((2K-1)c+2)\lambda_2\geq\lambda_3.$ Since $(2K-1)c+2\leq-1$ we obtain     $\lambda_2\geq\lambda_3(\geq\lambda_4).$ From the inequality $\langle(\lambda_4,0),v^2\rangle\geq\lambda_2$ it follows that $b\lambda_4\geq\lambda_2$,    so $b\leq1.$ On the other hand, since    $\langle(\lambda_3,-K\lambda_3),u^2\rangle\geq\lambda_2$, it follows that    $((1-2K)b+K-1)\lambda_3\geq\lambda_2.$ Thus,  we conclude that     $((1-2K)b+K-1)\lambda_3\geq\lambda_3.$ This is possible only if $b\geq \frac{K-2}{2K-1}.$ We have two possibilities:  $b=1$ and $K\geq2$, or $b=0$ and $K=2.$

\emph{Case $1$: $b=1$ and $K\geq2$ }. If $b=1$ then $\lambda_2=\lambda_3=\lambda_4.$ We now use inequalities $\langle(-2\lambda_2,\lambda_2),u^1\rangle\geq\lambda_1$ and   
$\langle(-\lambda_1,\lambda_1),v^2\rangle\geq\lambda_2$ from which we conclude that     $(3-a)\lambda_2\geq\lambda_1$ and  $2\lambda_1\geq\lambda_2.$ Therefore,     $a\geq2\frac{1}{2}.$ Next, from the inequality $\langle(\lambda_4,0),v^1\rangle\geq\lambda_1$ it follows that $(a-1)\lambda_4\geq\lambda_1.$ However, since $\lambda_1\geq\frac{1}{2}\lambda_2=\frac{1}{2}\lambda_4$ we conclude that $a\leq 1\frac{1}{2}$ and arrive at a contradiction. Thus, it is not possible that $b=1.$ This completes the proof for $K\geq 3$.

\emph{Case $2$: $b=0$, and $K=2.$} Since $((1-2K)b+K-1)\lambda_3\geq\lambda_2$ and $\lambda_2 \geq \lambda_3$, it follows that $\lambda_2=\lambda_3.$ From the inequalities $\langle(-\lambda_1,\lambda_1),v^2\rangle\geq\lambda_2$ and $\langle(-2\lambda_2,\lambda_2),u^1\rangle\geq\lambda_1$ it follows that $\lambda_1\geq\lambda_2$ and  $(3-a)\lambda_2\geq\lambda_1.$ We conclude that $a\geq2.$ Next, from the inequality $\langle(-2\lambda_2,\lambda_2), v^3\rangle\geq\lambda_3$ it follows that $(-3c-2)\lambda_2\geq\lambda_3$, and since $\lambda_2=\lambda_3$ we conclude that $c\geq-1.$ We have already shown that $c\leq-1.$ Therefore, $c=-1.$ Furthermore, from the inequality     $\langle(\lambda_4,0),u^1\rangle\geq\lambda_1$ it follows that            $(a-2)\lambda_4\geq\lambda_1$. Since $\lambda_1\geq\lambda_2=\lambda_3\geq\lambda_4$ it follows that $a\leq 2.$ Therefore, $a=2$ and $\lambda_1=\lambda_2=\lambda_3=\lambda_4.$
However, if we compare  the inequality $\langle V_4,v^1\rangle\geq\lambda_1$ with $\lambda_1=\lambda_4$ we find that $V_4$ belongs to the edge given by vector $v^1=(1,2)$ coming from $V_1$. We again arrive at a contradiction since the edges coming from $V_4$ are given by $u^4=(1,1)$ and $v^4=(1,0).$
\end{proof}

\begin{remark} \label{remark one almost toric}
The Delzant polytope corresponding to $\mathbb{CP}^2\sharp\overline{2\mathbb{CP}^2}$ shown in Figure \ref{f-cp} on the left contains 4 vertices with the slopes 
$$s(V_1)=(1,1), s(V_2)=(1,2), s(V_3)=(-2,-1), s(V_4)=(0,-1).$$
According to Proposition \ref{example non-centered}, no matter what size the allowable toric blow ups are,  the polytope cannot be $\{V_1,V_2,V_3,V_4\}$-centered. However, we are able to obtain a Weinstein domain 
$$\mathcal{W}_{T^2, \{ (1, 1), (1,2),(-2,-1), (0,-1)\}}$$ 
as a divisor complement using almost toric manifolds introduced in Section~\ref{intro to almost toric}. Namely, consider $\mathbb{CP}^2\sharp\overline{\mathbb{CP}^2}$ constructed by taking the monotone size blow up of $\mathbb{CP}^2$ at the node $A$ (the blow up which creates nodes at $B_1$ and $B_2$ of Figure \ref{f-cp} on the right).  Then, perform a nodal trade of the vertex $B_2$ and produce an almost toric manifold with a Lagrangian fibration such that the pre-image of the singular point marked by x is a pinched torus and the pre-image of $B_2$ is an isotropic circle. Note that we then have a $(1,0)$ eigenray extending from $B_2$ to x, hence the vanishing class of the singular fiber is $(0,-1)$. As explained in Section \ref{section:almosttoric} a neighborhood of the singular fiber is filled with a 2-handle attached to the set of regular points along the vanishing circle.

We now  blow up this almost toric manifold at the node $B_1$. We ensure that the size of the blow up is such that the preimage of $B_2$ is a new node of the blow up.  This is allowed because the fiber over $B_2$ is a circle. This way we obtain an almost toric manifold with base the polytope with a singular fiber in the interior shown on the right of Figure~\ref{f-cp}. Consider the divisor given as the pre-image under this Lagrangian fibration of the boundary of the polytope and then smooth the divisor at the nodes $V_1$, $V_2$, and $V_3$.  The complement of this smoothed divisor is Weinstein homotopic to $\mathcal{W}_{T^2, \{ (1, 1), (1,2),(-2,-1), (0,-1)\}}.$ 

\begin{figure}
	\centering
	\includegraphics[width=9cm]{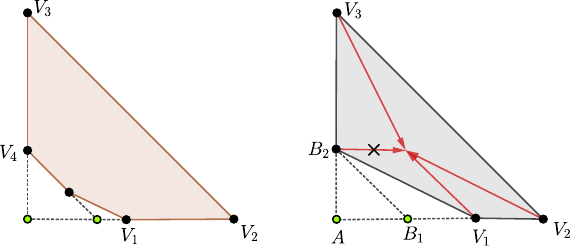}
	\caption{A Delzant polytope of $\mathbb{CP}^2\sharp\overline{2\mathbb{CP}^2}$ on the left, an almost toric base of $\mathbb{CP}^2\sharp\overline{2\mathbb{CP}^2}$ on the right.}
	\label{f-cp}
\end{figure}   
\end{remark}

We finish this section by presenting infinitely many distinct Weinstein manifolds obtained by our construction as complements of smoothed total toric divisors in toric $4$-manifolds. That is, we find an infinite list of slopes for each of which there is a partially centered Delzant polytope. As we noted earlier, this can only arise with partially smoothed total toric divisors since there are only finitely many monotone Delzant polytopes.

\begin{theorem} \label{thm:infinitemanifolds} 
There are infinitely many non-diffeomorphic Weinstein domains obtained by taking the complement of a neighborhood of a partially smoothed total toric divisor in a toric $4$-manifold.
\end{theorem}

\begin{proof}
We will show that there is a $\{V_0,\ldots, V_{k+1}\}$-centered Delzant polytope, for any $k\geq1,$ where 
$$s(V_0)=(0,-1), s(V_1)=(1,1), s(V_2)=(1,2), \ldots, s(V_k)=(1,k),~\text{and}~s(V_{k+1})=(-1,-1).$$ 
Then, the rest of the proof follows from Theorem \ref{thm:toricWein} and the fact that the diffeomorphism types (in fact homotopy types) of these manifolds are distinguished by their Euler characteristic (since they are built with increasing numbers of $2$-handles). Start with the half-spaces:
$$x\geq\lambda_0 ~\text{and} ~ x+y\geq \lambda_0,$$
for any $\lambda_0<0.$ The line $x+y=\lambda_0$ and $x=\lambda_0$ intersect at a vertex $V_0=(\lambda_0,0).$ Then, take any $\lambda_1<0$ such that \emph{$\lambda_1<2\lambda_0$} and the half-spaces  
$$2x+3y\geq\lambda_1 ~\text{and}~x+2y\geq\lambda_1.$$ 
The lines  $x+y=\lambda_0$ and $2x+3y=\lambda_1$ intersect at the vertex $A_1=(3\lambda_0-\lambda_1,\lambda_1-2\lambda_0)$  (see Figure \ref{nova fig}). The lines $2x+3y=\lambda_1$ and $x+2y=\lambda_1$ intersect at the vertex  $V_1=(-\lambda_1,\lambda_1).$ Then, take any $\lambda_2<0$ such that \emph{$\lambda_2<3\lambda_1$}, and the half-spaces  $$2x+5y\geq\lambda_2 ~\text{and}~ x+3y\geq\lambda_2.$$
The lines $x+2y=\lambda_1$ and $2x+5y=\lambda_2$ intersect at the vertex $A_2=(5\lambda_1-2\lambda_2,\lambda_2-2\lambda_1).$  The lines $2x+5y=\lambda_2$ and $x+3y=\lambda_2$ intersect at the vertex  $V_2=(-2\lambda_2,\lambda_2).$ Then, take any $\lambda_3<0$ such that \emph{$\lambda_3<3\lambda_2$} and take the half-spaces
$$2x+7y\geq\lambda_3~\text{and}~ x+4y\geq\lambda_3.$$  
The lines $x+3y=\lambda_2$ and $2x+7y=\lambda_3$ intersect at the vertex $A_3=(7\lambda_2-3\lambda_3,\lambda_3-2\lambda_2).$ The lines $2x+7y=\lambda_3$ and $x+4y=\lambda_3$ intersect at the vertex $V_3=(-3\lambda_3,\lambda_3).$ 

We continue in this manner until we need to consider the half-spaces $$2x+(2k-1)y\geq\lambda_{k-1} ~\text{and}~ x+ky\geq\lambda_{k-1},$$
where \emph{$\lambda_{k-1}<3\lambda_{k-2}$}. We obtain the vertex $A_{k-1}=((2k-1)\lambda_{k-2}-(k-1)\lambda_{k-1},\lambda_{k-1}-2\lambda_{k-2})$  as the intersection of the lines $2x+(2k-3)y=\lambda_{k-2}$ and  $2x+(2k-1)y=\lambda_{k-1}$. We also obtain the vertex $V_{k-1}=(-(k-1)\lambda_{k-1},\lambda_{k-1})$ as the intersection of the lines $2x+(2k-1)y=\lambda_{k-1}$ and $x+ky=\lambda_{k-1}$.

For the vertex $V_k$ consider the half-spaces
$$x+(k+1)y\geq\lambda_k~\text{and}~y\geq\lambda_k,$$
where \emph{$\lambda_k<2\lambda_{k-1}$}. The lines $x+ky=\lambda_{k-1}$ and $x+(k+1)y=\lambda_k$ intersect at the vertex $A_k=((k+1)\lambda_{k-1}-k\lambda_k,\lambda_k-\lambda_{k-1}).$ The lines $x+(k+1)y=\lambda_k$
and $y=\lambda_k$ intersect at the vertex $V_k=(-k\lambda_k,\lambda_k).$
  
Consider now the half space
$$-x\geq \lambda,$$
for any $\lambda<k\lambda_k.$ The lines $y=\lambda_k$ and $-x=\lambda$ intersect at the vertex $B=(-\lambda,\lambda_k).$ We finish this circle of half-spaces by adding the half-space
$$-y\geq\lambda_0.$$
The lines $-x=\lambda$ and $-y=\lambda_0$ intersect at the vertex $C=(-\lambda,-\lambda_0).$ The lines $-y=\lambda_0$ and $x=\lambda_0$ intersect at the vertex  $V_{k+1}=(\lambda_0,-\lambda_0).$ The polytope that we have constructed is a moment map image of $(\mathbb{CP}^1\times\mathbb{CP}^1)\sharp\overline{2k\mathbb{CP}^2}$, and all the inequalities between the $\lambda_i$ for $i=1,\ldots, k$ are imposed to get the sufficiently small sizes of blow ups so that the desired rays intersect inside the polytope. See Figure~\ref{nova fig} for an example for $k=3$.
\end{proof}

\begin{figure}
	\centering
	\includegraphics[width=15cm]{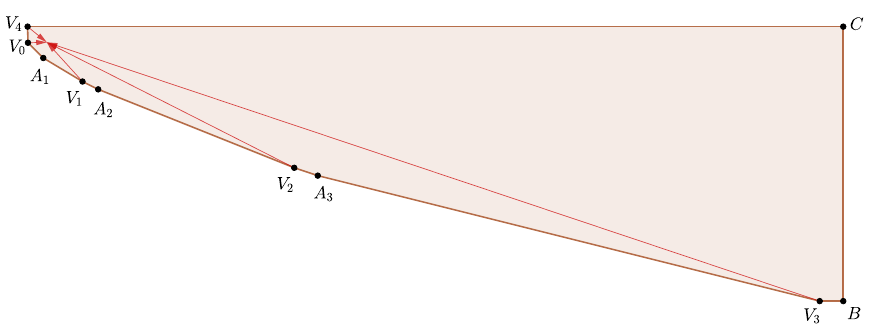}
	\caption{$\{V_0,\ldots, V_4\}$-centered polytope}
	\label{nova fig}
\end{figure}

%% file: parts/obstruction-noncentered.tex

\section{Non-centered toric manifolds and obstructions} \label{s:obstructions}

The assumption on the toric manifold to be $\{V_1, \ldots , V_k\}$-centered in our main theorem and algorithm is actually crucial. Indeed, we prove in this section that there are many toric manifolds which are not $\{V_1, \ldots, V_k \}$-centered such that the complement of the $\{V_1, \ldots, V_k \}$-smoothing of the total toric divisor does not support a Weinstein structure. In fact, in many cases the complement is not even exact.

\subsection{Exactness obstructions}
We first consider two cases that cause a toric manifold to fail to be $\{V_1,\ldots, V_n \}$-centered, and show that in these cases the symplectic form is not exact on the complement of a regular neighborhood of the $\{V_1,\ldots, V_k \}$-smoothed divisor.

\begin{prop}\label{prop:notexact}
Let $(M,\omega)$ be a symplectic toric manifold and $\Delta$ its Delzant polytope.
Let $V_1,\dots, V_k$ be a subsets of the vertices of $\Delta$.
Assume that $M$ fails to be $\{V_1,\dots, V_k\}$-centered because either 
\begin{itemize}
\item[(i)] Two rays associated to two of the vertices $V_{i_1}, V_{i_2}\in \{V_1,\dots, V_k\}$ are parallel or anti-parallel but the lines extending them do not coincide, or;
\item[(ii)] There exist three vertices, $V_{i_1},V_{i_2}, V_{i_3} \in \{V_1, \ldots, V_k\}$ such that for the associated rays 
$R_{i_1}, R_{i_2}, R_{i_3}$, $R_{i_2}$ intersects $R_{i_3}$ at a point $c_2 \in \mathring \Delta$ in the interior of the polytope $\Delta$ that does 
not belong to $R_{i_1}$.
\end{itemize}
Then the complement of its total toric divisor smoothed at 
$\{V_1,\dots, V_k\}$ is not an exact symplectic manifold (and in particular cannot support a Weinstein structure).
\end{prop}

\begin{proof}
	To prove $(M,\omega)$ is not exact we construct a 2-cycle $S$ built as the union of piecewise smooth surfaces such that $\langle \omega, [S]\rangle \neq 0$.
	
	In case (i), let $c_2$ be any point in $\mathring \Delta$ lying on $R_{i_2}$. In case (ii), $c_2$ is the intersection of $R_{i_2}$ with $R_{i_3}$. In both cases, let $c_1$ be any point in $\mathring \Delta$ lying on $R_{i_1}$. Let $\gamma(t)=(\alpha t+x_0, \beta t+y_0)$, $t\in [0,1]$ be the straight line segment from $c_1$ to $c_2$. See Figure~\ref{fig:nonexact}.
	
\begin{figure}
	\centering
	\includegraphics[width=11cm]{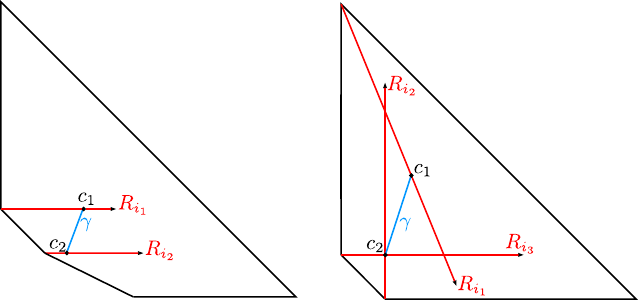}
	\caption{Non-centered examples, with case (i) on the left and case (ii) on the right.}
	\label{fig:nonexact}
\end{figure}
	
	Let $(a_j,b_j)$ denote the slope $s(V_{i_j})$ for $j=1,2,3$ (or $j=1,2$ in case (i)). Then $R_{i_j}$ is a ray with the direction $r(V_{i_j})=(-b_j,a_j)$. Recall that the complement of the divisor which has been smoothed at a vertex $V_{i_j}$ contains a Lagrangian disk $D_j$ with the node $V_{i_j}$ at its center such that $D_j$ projects to the ray $R_{i_j}$. Then, parameterized in toric coordinates $(p_1(t), p_2(t), q_1(t), q_2(t))$,
	$$D_j = \{(-b_jt+x_j, a_j t+y_j, a_js+\theta_j, b_js+\phi_j) \mid t\in [0,\delta_j], s\in [0,2\pi] \}$$
	where $\delta_j>0$, $(x_j,y_j)$ are the coordinates in $\R^2$ for the vertex $V_{i_j}$, and $(\theta_j,\phi_j)\in T^2$ is a constant point chosen so that $D_j$ lies in the complement of the small neighborhood of the smoothed divisor. Choose $\delta_1$ so that $(-b_1\delta_1+x_1, a_1\delta_1+y_1)=c_1$ in order to have $D_1$ project to the segment of $R_{i_1}$ ending at the point $c_1$. Choose $\delta_2$ and $\delta_3$ so that $(-b_2\delta_2+x_2,a_2\delta_2+y_2)=(-b_3\delta_3+x_3,a_3\delta_3+y_3)=c_2$, in order to have $D_2$ project to the segment of $R_{i_2}$ ending at $c_2$ and $D_3$ project to the segment of $R_{i_3}$ ending at $c_2$.
	
	Let $A$ be the annulus lying over the line segment $\gamma$ parameterized as
	$$A = \{ (\alpha t+x_0, \beta t+y_0, a_1 s+\theta_1, b_1 s+\phi_1) \mid t\in [0,1], s\in[0,2\pi] \}.$$
	
	Note that because the directions of the rays are orthogonal to the slopes $s(V_{i_j})=(a_j,b_j)$, the slopes $\{(a_j,b_j)\}$ are linearly dependent if and only if the directions of the rays $\{(-b_j,a_j)\}$ are. If two of these slopes are parallel or anti-parallel, we are in case (i) and there is a linear dependence relation $k_1(a_1,b_1)+k_2(a_2,b_2) =0$ with $k_1\neq 0$, $k_1,k_2\in \Q$. In case (ii), we have three distinct slopes in $\Z^2$ which must be linearly dependent in $\Q^2$, so there exist coefficients $k_1,k_2,k_3\in \Q$ such that $\sum k_j (a_j,b_j) = 0$ with $k_1\neq 0$. After multiplying by an integer, we may assume that all $k_j\in \Z$.
	
	For $j=2,3$ let $C_j=\partial D_j$, and let $C_1 = \partial (A\cup D_1)$. Then, for all $j$, $C_j$ is a curve in the $T^2$ fiber over $c_2$ of slope $(a_j,b_j)$ and this slope records the homology class of $C_j$ in $H_1(T^2)$. Since $\sum k_j (a_j,b_j) = 0$, there is a 2-chain $S'$ of $T^2$ with boundary given by $|k_1|$ copies of $C_1$ and $|k_2|$ copies of $C_2$ (in case (ii) we also have $|k_3|$ copies of $C_3$), with orientations specified by the signs of the $k_j$. We can form a $2$-cycle $S$ in $M$ by 
	$$S = S'+k_1A +\sum k_j D_j.$$
	
	Then 
	$$\langle [\omega], [S] \rangle = \int_{S'}\omega + k_1\int_A \omega + \sum k_j\int_{D_j} \omega.$$
	Since $S'$ is a sum of Lagrangian subsurfaces of the $T^2$ fiber over $c_2$, the first term in the sum is zero. Similarly, since each $D_j$ is Lagrangian, the last term is also zero. Therefore
	$$\langle [\omega], [S] \rangle = k_1\int_A \omega.$$
	Recall that $k_1\neq 0$, so it suffices to show that the symplectic area of $A$ is non-zero.
	
	Because $(a_1,b_1)$ and $(-b_1,a_1)$ are an orthogonal basis for $\R^2$, we can write $(\alpha,\beta) = \zeta_1(a_1,b_1)+\zeta_2(-b_1,a_1)$. Notice that because the line extending $R_{i_1}$ of direction $(-b_1,a_1)$ contains $c_1$ and does not contain $c_2$, we must have that $\zeta_1\neq 0$ (the vector connecting $c_1$ to $c_2$ has a non-zero component in the direction orthogonal to $R_{i_1}$). Using the parametrization of $A$ and the standard form of $\omega$ in toric coordinates we find that
	$$\int_A \omega = \int_{0 \leq t \leq 1} \int_{0 \leq s \leq 2 \pi} (\alpha,\beta) \cdot (a_1,b_1)\, ds\, dt = 2\pi \zeta_1 (a_1^2+b_1^2)\neq 0. $$
	Therefore $\langle [\omega], [S]\rangle\neq 0$, so $\omega$ is not exact on the complement of a small neighborhood of the divisor smoothed at $\{V_1,\ldots, V_k \}$.
\end{proof}

Although these cases cover many examples of non-centered toric manifolds, there are some other reasons why a toric manifold may fail to be $\{V_1,\ldots, V_k \}$-centered. If neither of the hypotheses (i) or (ii) of Proposition~\ref{prop:notexact} are satisfied, and the polytope is not $\{V_1,\ldots, V_k \}$-centered, then we must have that at least one of the following is true:
\begin{itemize}
	\item[(iii)] there is a pair of rays $R_{j_1}$ and $R_{j_2}$ which do not intersect, but are not parallel or anti-parallel, or
	\item[(iv)] the common intersections of rays lie outside the interior of the polytope $\mathring \Delta$.
\end{itemize}

Examples of these situations are shown in Figures~\ref{fig:disorder} and~\ref{fig:centareas}. In these scenarios, we are not currently able to generally rule out the existence of a Weinstein structure on the complement in all examples. However, we do have some additional obstructions which can rule out certain examples. In particular, we construct examples of case (iv) where the complement of a neighborhood of the $\{V_1,\ldots, V_k\}$-smoothed divisor is exact, but does not support a Weinstein structure because there is no compatible \emph{convex} Liouville structure.

\begin{figure}
	\centering
	\includegraphics[width = 6cm]{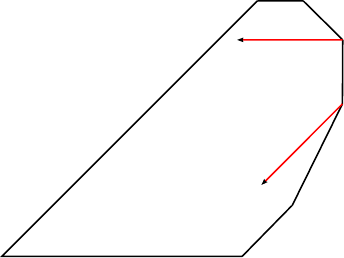}
	\caption{A non-centered example where two rays at the chosen vertices never intersect.}
	\label{fig:disorder}
\end{figure}

\subsection{Convexity obstructions}
The existence of a convex Liouville structure on the complement of a regular neighborhood $N$ of the $\{V_1,\ldots, V_k \}$-smoothed divisor is equivalent to the existence of a concave Liouville structure on the neighborhood $N$ of the divisor (defined near $\partial N$). Here we show that there exist smoothings of total toric divisors which do not admit concave neighborhoods. 

As a basic example, consider a toric representation of $\cptwo\# 9\cptwobar$. The boundary total toric divisor represents the homology class $3h-e_1-\cdots-e_9$. After smoothing all nodes, the smoothed divisor becomes a smoothly embedded torus representing the homology class $3h-e_1-\cdots-e_9$. This torus is a symplectic submanifold. The homology class shows that the self-intersection number of the torus is $0$, so its normal bundle is trivial. The standard neighborhood theorem shows that a regular neighborhood of this torus is symplectomorphic to $T^2\times D^2$ with a product symplectic structure $\omega_T\oplus \omega_D$, where $\omega_T$ is some symplectic area form on $T^2$ and $\omega_D$ is some symplectic area form on $D^2$. The restriction of $\omega_T\oplus \omega_D$ to the boundary $T^2\times S^1$ is not exact because $\omega_T\oplus \omega_D$ evaluates positively on $T^2\times \{\theta_0\}$. If the restriction of the symplectic form to the boundary of a neighborhood is not exact, the neighborhood cannot be concave. Therefore the complement cannot have convex boundary.

For a toric representation of $\cptwo\#\ell\cptwobar$ with $\ell>9$, the analogous smoothing of all vertices gives a divisor which is a smooth torus of negative self-intersection number. Because this is a negative definite intersection form, the divisor admits convex neighborhoods, but it cannot admit concave neighborhoods by the following \emph{relative} cohomology obstruction.

\begin{lemma}[c.f. \cite{LiMak}]
	Let $\Sigma$ denote a normal crossing divisor with $k$ components in a symplectic $4$-manifold $(X,\omega)$. Let $(N,\omega)$ be a standard regular neighborhood of $\Sigma$ and let $Q$ denote its intersection form. Let $\mathbf{a}\in \left(\R_{>0}\right)^k$ denote the vector of symplectic areas of the components of $\Sigma$ (ordered consistently with $Q$). Suppose there does not exist any $z\in \R_{>0}^k$ such that $Qz=\mathbf{a}$. Then there is no inward pointing (concave) Liouville vector field for $\omega$ defined along a neighborhood of the boundary of $N$.
\end{lemma}

\begin{proof}
	We can interpret the convex/concave conditions (where the Liouville vector field points transversally outward/inward along the boundary) in terms of the Liouville and contact forms as follows. A Liouville vector field $Z$ for a symplectic form $\omega$ is dual to the Liouville 1-form $\lambda = \iota_Z\omega$. If $Z$ is defined in a neighborhood of the boundary of $N$, then so is $\lambda$. Let $i:\partial N\to N$ denote the inclusion and let $\alpha=i^*\lambda$. Because of the Liouville condition, $d\alpha = i^*\omega$. 
	
	To say that $Z$ points transversally outward (respectively inward) from the boundary is equivalent to saying that $\omega\wedge \omega(Z,\zeta_1,\zeta_2,\zeta_3)>0$ (respectively $<0$) where $(\zeta_1,\zeta_2,\zeta_3)$ is a positively oriented basis frame for $T(\partial N)$ (oriented using the standard ``outward normal first'' convention for the boundary orientation). Now using the Liouville condition, 
	$$\omega\wedge \omega (Z,\zeta_1,\zeta_2,\zeta_3)=\lambda\wedge \omega(\zeta_1,\zeta_2,\zeta_3)=\alpha\wedge d\alpha(\zeta_1,\zeta_2,\zeta_3).$$
	Therefore, the condition that the Liouville vector field $Z$ points outward (respectively inward) to $\partial N$ is equivalent to asking that $\alpha\wedge d\alpha>0$ (respectively $<0$) with respect to the boundary orientation on $\partial N$.

	Now we relate this to the linear algebra condition, and explain the meaning of $z$. Consider the following portion of the long exact sequence of a pair in cohomology:
	$$\xymatrix{
		H^1(\partial N) \ar[r]^{\rho} & H^2(N,\partial N) \ar[r]^{\pi} & H^2(N) \ar[r]^{i^*} & H^2(\partial N)
	}.$$
	A class in $H^2(N,\partial N)$ is represented by a pair $(\tau,\alpha)$ where $\tau \in \Omega^2(N)$ is closed and $\alpha\in \Omega^1(\partial N)$ satisfies $d\alpha = i^*\tau$ (see \cite[p. 78]{BottTu}). The maps act as $\rho([\alpha]) = [(0,\alpha)]$,  $\pi([\tau,\alpha])= [\tau]$, and $i^*$ is the pull-back by the inclusion. Observe under this formulation, a Liouville form for the symplectic form $\omega$ precisely determines a class $[(\omega, \alpha)]\in H^2(N,\partial N)$ such that $\pi([(\omega,\alpha)]) = [\omega]$. By exactness, such classes exist precisely when $i^*[\omega]=0$. We assume this basic pre-requisite to finding a convex/concave Liouville structure. Then we want to understand when such a class $[(\omega,\alpha)]$ exists with the additional convex/concave condition that $\alpha\wedge d\alpha$ is a positive/negative area form on $\partial N$.
	
	Note that by Poincar\'e-Lefschetz duality, $H^2(N,\partial N)\cong H_2(N)$. Since $H_2(N)$ is freely generated by the components of $\Sigma$, an element of $H_2(N;\R)$ is represented by a vector $z\in \R^k$. For each class $[(\omega,\alpha)]\in H^2(N,\partial N)$, there is a corresponding Poincar\'e-Lefschetz dual class $z\in H_2(N)$. In general, Poincar\'e duality transforms the evaluation pairing between cohomology and homology into the intersection pairing on homology. Since the evaluation pairing on a component $\Sigma_i$ of $\Sigma$ is given by
	$$\langle [(\omega,\alpha)], \Sigma_i\rangle = \int_{\Sigma_i} \omega = a_i,$$
	we obtain the condition that $z$ is dual to an element $[(\omega,\alpha)]$ such that $\pi([(\omega,\alpha)]) = [\omega]$ precisely when $Qz=\mathbf a$.
	
	The standard neighborhood $N$ of $\Sigma$ is symplectically modeled as in~\cite{GayStipsicz}, built from gluing together pieces of two different types. First, for every irreducible component $\Sigma_i$ of $\Sigma$ of genus $g_i$, which intersects $n_i$ other divisors at nodes, we include a piece of the form $S_i\times D^2$ where $S_i$ is a surface of genus $g_i$ with $n_i$ boundary components. These pieces each come with a product symplectic form $\beta_i+rdr\wedge d\theta$, where $\beta_i$ is an area form on $S_i$ tuned to have a specific form near the boundary components. The other type of piece is a Darboux ball containing a neighborhood of a node. The specific choices of these pieces and the symplectomorphisms gluing them together to give the model of $N$ are specified in~\cite{GayStipsicz}. For our purposes, what is important is that, away from the nodes, the symplectic structure near $\Sigma$ is identified using a standard neighborhood theorem with a product symplectic structure on $\Sigma_i\times D^2$. For each component $\Sigma_i$ of $\Sigma$, let $T_i$ denote an orthogonal symplectic disk $\{p\}\times D^2$ under this identification.
	
	Now suppose that $N$ has a concave Liouville structure defined near $\partial N$. Then as above, there is a $1$-form $\alpha$ on $\partial N$ such that $d\alpha = i^*\omega$, and $\alpha\wedge d\alpha<0$ on $\partial N$ (with respect to the boundary orientation). We consider the quantities
	$$t_i:=\int_{\partial T_i} \alpha.$$
	Choose a frame for $T\partial N$ along $\partial T_i$ of the form $(\partial_\theta, \zeta_1,\zeta_2)$ where $\zeta_1$ and $\zeta_2$ are symplectically orthogonal to $T_i$ and $\partial_\theta$ is tangent to $\partial T_i$. Then if $(\zeta_1,\zeta_2)$ are ordered positively with respect to the form $\beta_i$ on $\Sigma_i$, we have that $(\partial_\theta,\zeta_1,\zeta_2)$ is a positively oriented frame with respect to the boundary orientation on $\partial N$. Therefore,
	$$\alpha(\partial_\theta)\beta_i(\zeta_1,\zeta_2) = \alpha\wedge \omega(\partial_\theta,\zeta_1,\zeta_2)=\alpha\wedge d\alpha(\partial_\theta,\zeta_1,\zeta_2)<0,$$
	so $\alpha(\partial_\theta)<0$. As a consequence, when $\alpha$ is the contact form induced on the concave boundary of $N$, $t_i<0$ for each $i$.
	
	Finally, we relate the vector $t=(t_1,\dots, t_k)$ to the class $z$. The vector $t$ is \emph{not} an invariant of the relative cohomology class $[(\omega,\alpha)]$. Also note that $\alpha$ is not defined over the interior of $N$, and it generally does not extend. To identify $z$ in terms of integrals of forms which represent the class $[(\omega,\alpha)]$ we write
	$$z_i=\int_{T_i}\omega - \int_{\partial T_i}\alpha.$$
	Note that this quantity can be checked to be independent of the choice of reprentative of the relative class $[(\omega,\alpha)]$ using Stokes' theorem.
	
	Depending on the size of the neighborhood we take, the symplectic area of $T_i$ varies, but it is always be positive. Therefore if a component of $z_i$ is negative:
	$$z_i = \int_{T_i}\omega - \int_{\partial T_i}\alpha <0$$
	then $t_i>0$. Since we know that a concave neighborhood has all $t_i<0$, all of the $z_i$ components must be strictly positive. Thus the condition that $z\in (\R_{>0})^k$ is necessary. That this condition is sufficient follows from the Gay-Stipsicz construction of the Liouville structure in the concave setting as in~\cite{LiMak}.
\end{proof}

\begin{example}
	If we vary a toric polytope by adjusting the symplectic areas of the divisor components, we may end up with a polytope which is not $\{V_1,\cdots, V_k\}$-centered, even if a different toric polytope with the same inward normal vectors (but different symplectic areas) is centered with respect to the chosen vertices.  In some cases, this can be explained by the fact that, the areas of the divisor in the uncentered case do not allow for a positive solution to the equation $Qz=\mathbf{a}$. For example, consider $(\cpone\times \cpone, \omega_{a,b})$, where the symplectic area of the first factor is $a$ and the symplectic area of the second factor is $b$. We use the moment map projection whose image is a rectangle as in Figure~\ref{fig:centareas}, such that the preimage of the horizontal boundary component on the bottom (or top) is the $\cpone$ of area $a$ and the preimage of the vertical boundary component on the left (or right) is the $\cpone$ of area $b$. 
	
\begin{figure}
	\centering
	\includegraphics[width=8cm]{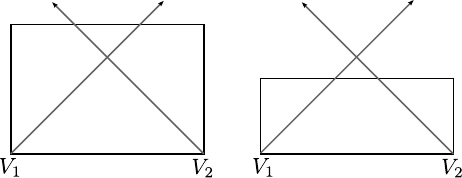}
	\caption{The polytope on the left is $\{V_1,V_2\}$-centered, while the polytope on the right is not.}
	\label{fig:centareas}
\end{figure}
	
	Let's consider smoothing the two bottom vertices $V_1$ and $V_2$. If $b>\frac{1}{2}a$, then the two corresponding rays meet at a point on the interior of the polytope, so the polytope is $\{V_1,V_2\}$-centered. However, if $b\leq \frac{1}{2}a$, the rays meet on the boundary or outside the polytope, making the polytope not centered with respect to $\{V_1,V_2\}$. Let's consider the symplectic divisor that results from smoothing $V_1$ and $V_2$. It has two irreducible components: the preimage of the top edge, and the smoothed union of the three other edges. The two components intersect at two points positively and transversally. The top edge component has self-intersection number $0$, and the other component has self-intersection number $4 = ([\cpone\times\{*\}]+2[\{*\}\times \cpone])^2$, so the intersection form is
	$$Q = \begin{pmatrix} 0&2\\2&4  \end{pmatrix}.$$
	The symplectic area of the first component is $a$, and of the second component is $a+2b$. Therefore a solution $z$ to $Qz = \mathbf a$ corresponds to
	$$\begin{pmatrix} 0&2\\2&4\end{pmatrix}  \begin{pmatrix} z_1\\ z_2 \end{pmatrix} = \begin{pmatrix} a\\ a+2b \end{pmatrix}$$
	which satisfies $z_2= a/2$, $2z_1+2a = a+2b$. Therefore, $z_1 = (2b-a)/2$ and $z_2=a/2$. If both quantities are required to be positive, we get that $2b-a> 0$, so $b>\frac{1}{2}a$. Therefore, in the $\{V_1,V_2\}$-uncentered case where $b\leq \frac{1}{2}a$, we find that the $\{V_1,V_2\}$-smoothing does not admit any concave neighborhood, so its complement cannot support a Weinstein structure.
\end{example}

%% file: parts/cotangent.tex

\section{Weinstein handlebody diagrams and cotangent bundles} \label{s:cotangent}

We now embark on our diagrammatic results. Our goal in this second half of the paper is to produce and analyze Weinstein handlebody diagrams for the complements of partially smoothed total toric divisors and more general Weinstein domains $\mathcal{W}_{F,c}$. Before approaching our general goal, we need the base case where we do not smooth any nodes, where the complement is Weinstein homotopic to $D^*T^2$. More generally for any surface $F$, we prove we have Weinstein handlebody diagrams which are Weinstein homotopic to the canonical Weinstein structure on the disk cotangent bundle $D^*F$. 

The Weinstein handlebody diagrams that we use for $D^*F$ were originally given in~\cite{Gompf} in the orientable case and in~\cite{Ozbagci} in the non-orientable case. They are shown in Figure~\ref{fig:cotorus}. It is well understood using standard smooth handlebody theory and a computation of framing that these diagrams smoothly represent $D^*F$. Although it has been generally accepted that this handle structure is likely Weinstein homotopic to the canonical Weinstein structure on $D^*F$, a proof of this was previously lacking in the literature. An attempt to prove this was given in \cite{Ozbagci}, but then modified in the erratum to the weaker statement that the contact structures on $S^*F$ agree between the handlebody and canonical structures \cite{OzbagciC}. For $D^*T^2$ specifically, \cite{Wendl} shows that there is a unique Stein filling of the boundary $T^3$. This provides an indirect proof that (after completion) the Gompf handlebody is symplectomorphic to the canonical Weinstein filling. In section~\ref{s:cotangentdiag} we give a much more direct proof of a more general result: that for any closed surface $F$, the canonical Weinstein structure on $D^*F$ is Weinstein homotopic to the structure corresponding to the Gompf handlebody diagram on $D^*F$. Our direct proof fills the gap in the literature for a general surface, and we use the Weinstein homotopy constructed in our proof to prove our algorithm in Section~\ref{sec:curves}.

\subsection{Morsification of the canonical Weinstein structure}
\label{s:morsify}

Recall from section~\ref{s:cotangentbackground} the canonical Weinstein structure on the cotangent bundle of a surface. Here we see another Weinstein structure on the cotangent bundle which is Weinstein homotopic to the canonical one.

Let $F$ be a surface. Given a Morse function $f:F\to \R$, one obtains a perturbed Liouville structure on its cotangent bundle as follows. Let $H_f: T^*F \to \R$ be the Hamiltonian function defined by 
$$H_f(q_1,q_2,p_1,p_2) = p_1\frac{\partial f}{\partial q_1} + p_2\frac{\partial f}{\partial q_2}.$$
We get a ``Morsified'' Liouville structure from the canonical one by setting
$$\lambda_f = \lambda_{can}-dH_f.$$
Note that this Liouville structure is connected to the canonical Liouville structure through a family of Liouville structures $\lambda_t = \lambda_{can}-tdH_f$. The key property of this new Liouville form comes from looking at the Liouville vector field along the zero section. Recall that the Liouville vector field for $\lambda_{can}$ vanishes along the zero section. For $\lambda_f$ the corresponding Liouville vector field $Z_f$ looks like the gradient of $f$ along the zero section. To see this, remember that $Z_f$ is defined by the equation $\omega(Z_f,\cdot) = \lambda_f$. In other words,
$$(dp_1\wedge dq_1+dp_2\wedge dq_2)(Z_f,\cdot) = p_1dq_1+p_2dq_2 - dH_{f}.$$
We compute
$$dH_{f} = \frac{\partial f}{\partial q_1} dp_1 + \frac{\partial f}{\partial q_2} dp_2 + \left(p_1\frac{\partial^2 f}{\partial q_1^2}+p_2\frac{\partial^2 f}{\partial q_1\partial q_2}\right)dq_1+\left(p_1\frac{\partial^2 f}{\partial q_1\partial q_2}+p_2\frac{\partial^2 f}{\partial q_2^2}\right)dq_2.$$
Therefore the corresponding Liouville vector field is
$$Z_f= \frac{\partial f}{\partial q_1} \partial_{q_1}+ \frac{\partial f}{\partial q_2}\partial_{q_2}+\left(p_1-p_1\frac{\partial^2 f}{\partial q_1^2}-p_2\frac{\partial^2 f}{\partial q_1\partial q_2} \right)\partial_{p_1} + \left(p_2-p_1\frac{\partial^2 f}{\partial q_1\partial q_2}-p_2\frac{\partial^2 f}{\partial q_2^2}  \right)\partial_{p_2}.$$
Indeed, we see that when $p_1=p_2=0$ the last two components vanish and the first two components are the gradient vector field of $f$. In particular, $Z_f$ is tangent to the $0$-section, so the $0$-section is invariant under this Liouville flow. To get a Weinstein structure, observe that the Liouville vector field $Z_{f}$ is gradient-like for the function
$$\Phi(q_1,q_2,p_1,p_2) = f(q_1,q_2)+\frac{1}{2}(p_1^2+p_2^2).$$

\subsection{Weinstein handlebody diagrams for the canonical cotangent structure} \label{s:cotangentdiag}

Here we prove our main result of this section. The Weinstein handlebody diagrams that we use for $D^*F$ are shown in Figure~\ref{fig:cotorus}. Our goal here is to construct a Weinstein homotopy between the Weinstein structure underlying this diagram and the canonical Weinstein structure on the cotangent bundle.

\begin{figure}
	\centering
	\includegraphics[width=6cm]{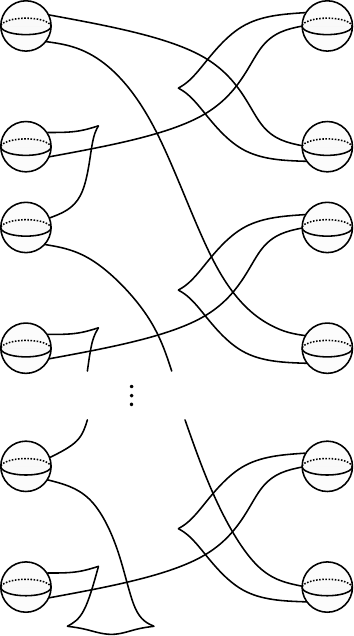}  \hspace{1cm}	\includegraphics[width=6cm]{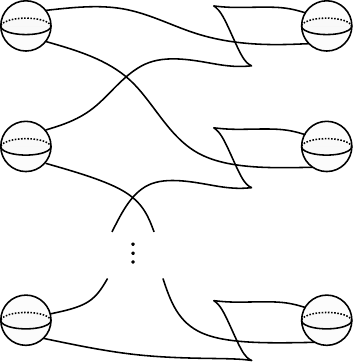}
	
	\caption{On the left see the Gompf handlebody diagram for the cotangent bundle of the orientable genus $g$ surface (the number of $1$-handles is $2g$, so the diagram is repeated $g$ times). On the right see the analogous handlebody for a non-orientable surface (when the surface is $\#_n \rptwo$ there are $n$ $1$-handles).}
	\label{fig:cotorus}
\end{figure}

Typically, one works with Weinstein structures $(W,\omega, Z,\phi)$ which are \emph{self-indexing} in the sense that if $p_1$ and $p_2$ are critical points of $\phi$ and the index of $p_2$ is greater than the index of $p_1$ then $\phi(p_2)>\phi(p_1)$. Then we can choose regular values $c_0$ and $c_1$ so that for any critical point $p$ of index $0$, $\phi(p)<c_0$ and for any critical point $q$ of index $1$, $c_0<\phi(q)<c_1$. The subdomain $W_0 = \{p\in W\mid \phi(p)\leq c_0\}$ gives the union of the $0$-handles of $W$, and $W_1 = \{p\in W\mid \phi(p)\leq c_1\}$ is the union of the $0$- and $1$-handles. The level sets $Y_0:=\partial W_0 = \phi^{-1}(c_0)$ and $Y_1:=\partial W_1 = \phi^{-1}(c_1)$ inherit contact structures by restricting the Liouville form. If we start with a Weinstein handle diagram, we assume that the handles are attached in order of increasing index so the naturally associated Weinstein structure is self-indexing.

Now we recall the definition of the skeleton of a compact Weinstein domain.

\begin{definition} Let $(W,\omega, Z,\phi)$ be a compact Weinstein domain. Then the \emph{skeleton} is
	$$Skel(W,\omega, Z,\phi) = \int_{t>0} Z^{-t}(W).$$
	Equivalently, $Skel(W,\omega, Z,\phi)$ is the union of the stable manifolds (with respect to the flow of $Z$) of the zeros of $Z$ (critical points of $\phi$).
\end{definition}

The skeleton of a Weinstein manifold is a stratified subset, where the strata are the different stable manifolds whose dimension is determined by the index of the critical point. The attaching spheres for handles are seen by intersecting this skeleton with regular level sets of $\phi$.

For the canonical Weinstein structure on $D^*F$, the zeros of $Z_{can}$ are precisely the points in the zero section $F$, and each such point is its own stable manifold. Therefore $Skel(D^*F,\omega_{can},Z_{can},\phi_{can}) = F$.

The last piece of background we need is the notion of a \emph{holonomy homotopy}. Given a Weinstein domain/manifold $(W,\omega, Z,\phi)$, let $c<d$ be two regular values of $\phi$ such that there are no critical values of $\phi$ between $c$ and $d$. Then $Y_c:=\phi^{-1}(c)$ and $Y_d:=\phi^{-1}(d)$ inherit contact structures $\xi_c$ and $\xi_d$ from the restriction of the Liouville form. Because there are no critical points between $c$ and $d$, the flow of $Z$ determines a contactomorphism $\Psi_0:(Y_d,\xi_d)\to (Y_c,\xi_c)$ called the \emph{holonomy}.

\begin{lemma}[{\cite[Lemma 12.5]{CE}}] \label{l:holonomy}
 Let $\Psi_0:(Y_d,\xi_d)\to (Y_c,\xi_c)$ be the holonomy from level $d$ to $c$ of a Weinstein domain as above. Let $\Psi_t$ be a contact isotopy starting at $\Psi_0$. Then there exists a Weinstein homotopy $(W,\omega,Z_t,\phi)$ (which keeps $\omega$ and $\phi$ constant), such that the holonomy from $Y_d$ to $Y_c$ with respect to $Z_t$ is $\Psi_t$. 
\end{lemma} 

Note that any Legendrian isotopy can be realized by a contact isotopy. Thus as a particular case of Lemma~\ref{l:holonomy}, given any Legendrian isotopy of attaching spheres, there exists a Weinstein homotopy which modifies the holonomy to realize that Legendrian isotopy. 

Now we are ready to prove the main theorem of this section. 

\begin{theorem}\label{thm:cot}
	Let $F$ be a closed surface. The Gompf handlebody diagram for $D^*F$ in Figure~\ref{fig:cotorus} corresponds to a Weinstein structure which is Weinstein homotopic to the canonical Weinstein structure on $D^*F$.
\end{theorem}

\begin{proof}
	Let $W=D^*F$. Starting with the canonical Weinstein structure $(W,\omega_{can}, Z_{can}, \phi_{can})$, we first perform a Weinstein homotopy to a Morsification as in section~\ref{s:morsify}.
	
	Let $f$ be a perfect Morse function on $F$, i.e. a Morse function with a single index $0$ critical point, a single index $2$ critical point, and $2g$ index $1$ critical points in the orientable case, where $g$ is a genus of $F$ (or $n$ index $1$ critical points if $F$ is diffeomorphic to $\#_n \rptwo$). After globally rescaling, we assume that $f$ is $C^{\infty}$ small (since $F$ is compact).
	
	In local coordinates, the Morsifying Weinstein homotopy $(W,\omega_t, Z_t, \phi_t)$ is specified by setting $\omega_t = \omega_{can}$ for all $t\in[0,1]$,
	$$Z_t = p_1\partial_{p_1}+p_2\partial_{p_2} + t\left(\frac{\partial f}{\partial q_1} \partial_{q_1}+ \frac{\partial f}{\partial q_2}\partial_{q_2}+\left(-p_1\frac{\partial^2 f}{\partial q_1^2}-p_2\frac{\partial^2 f}{\partial q_1\partial q_2} \right)\partial_{p_1} + \left(-p_1\frac{\partial^2 f}{\partial q_1\partial q_2}-p_2\frac{\partial^2 f}{\partial q_2^2}  \right)\partial_{p_2} \right)$$
	and $\phi_t = \phi_{can}+tf$. $Z_t$ is gradient-like for $\phi_t$ by \cite[Lemma 12.8]{CE}, and $Z_t$ is Liouville for $\omega_t$ because the term that contains a factor of $t$ is a Hamiltonian vector field, so $(W,\omega_t,Z_t,\phi_t)$ is a Weinstein homotopy. 	Let $(W_{M}, \omega_{M}, Z_{M}, \phi_{M}) = (W,\omega_1,Z_1,\phi_1)$ be the Morsified cotangent Weinstein domain. Let $(W_{G}, \omega_G, Z_G, \phi_G)$ be the Weinstein domain obtained by attaching Weinstein handles according to the Gompf standard diagram for $D^*F$ as in Figure~\ref{fig:cotorus}.
	
	Observe that $\phi_M$ and $\phi_G$ have the same number of critical points of each index because the Gompf handle diagram has a single $0$-handle implicitly, $2g$ $1$-handles (or $n$ $1$-handles in the non-orientable case), and a single $2$-handle. 	
	Furthermore, both $\phi_M$ and $\phi_G$ are self-indexing. Thus it remains to check that the attaching sphere data is equivalent (Legendrian/isotropically isotopic). 
	
	To do this, in each domain we consider a regular contact type level set, $Y^0_M$ or $Y^0_G$, above the index $0$ critical value and below the index $1$ critical values, and a regular contact level set, $Y^1_M$ or $Y^1_G$, above the index $1$ critical values and below the index $2$ critical value. Let $W_M^0$ (resp. $W_G^0$) denote the subdomains with boundary $Y^0_M$ (resp. $Y^0_G$) which give the $0$-handles.
	
	The diagrams of Figure~\ref{fig:cotorus} contain the information of the Gompf handlebody attaching data of the $1$-handles and the $2$-handle in $Y^0_G$ (the boundary of the $0$-handle). The attaching data of the $2$-handle in $Y^1_G$ can also be understood from these diagrams through the standard convention for $1$-handles.
	
	In the Morsified Weinstein manifold $(W_M,\omega_M,Z_M,\phi_M)$, we understand the skeleton because it agrees with the skeleton before Morsification: it is precisely the zero section $F$ in $T^*F$. This is because the Morsified Liouville vector field $Z_f$ is tangent to the zero section, and every point in the zero section is in the stable manifold of one of the critical points of $f$, and thus one of the zeros of $Z_f$. The intersection of $Skel(W_M,\omega_M,Z_M,\phi_M)$ with $W^0_M$ is the $0$-handle of $F$ with respect to the handle decomposition associated to the Morse function $f$. In particular, $L_0:=Skel(W_M,\omega_M,Z_M,\phi_M)\cap W_M^0$ is a Lagrangian disk in the $4$-ball $W_M^0$. Moreover, the Morse function $\phi_M$ has a unique critical point when restricted to $L_0$ and that critical point has index $0$. It follows that $L_0$ is an unknotted Lagrangian disk, so its boundary $\partial L_0$ in $Y^0_M$ is a maximal Thurston-Bennequin number Legendrian unknot.
	
	Although $Skel(W_M,\omega_M,Z_M,\phi_M)$ is smooth, it is still stratified by the index $0$ point, $2g$ (or $n$) $1$-dimensional open disks limiting to the $0$-dimensional point in both directions, and the remaining $2$-dimensional open subset of the skeleton is the stable manifold of the index $2$ critical point. The index $0$ point lies in the interior of $L_0$. Each $1$-dimensional stratum intersects $Y^0_M$ at two points in $\partial L_0$. The complement of these $4g$ (or $2n$) points in the standard Legendrian unknot $\partial L_0$ comes from the $2$-handle.
	
\begin{figure}
	\centering
	\includegraphics[width=14cm]{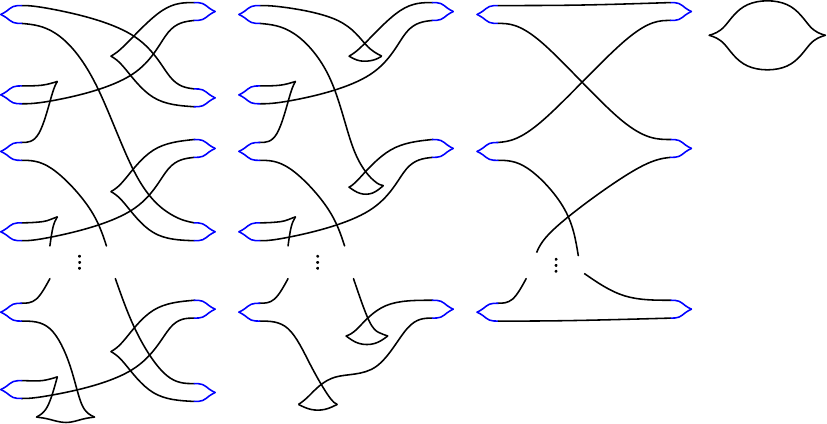}
	\caption{A Legendrian isotopy to the standard unknot via Legendrian Reidemeister moves.}
	\label{fig:unknot}
\end{figure}

\begin{figure}
	\centering
	\includegraphics[width=10cm]{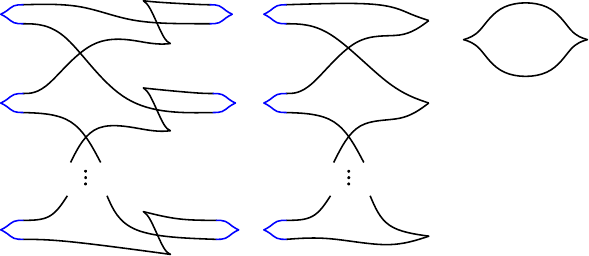}
	\caption{A Legendrian isotopy to the standard unknot via Legendrian Reidemeister moves.}
	\label{fig:unknot_nonorientable}
\end{figure}

	In the case of the Gompf Weinstein structure, the intersection of $Skel(W_G,\omega_G,Z_G,\phi_G)$ with $Y_0^G$ consists of: the portion of the attaching sphere of the $2$-handle shown in Figure~\ref{fig:cotorus}, the $S^0$ centers of the attaching regions of each of the $1$-handles, and the intersection of the stable manifold of the index $2$ critical point inside the attaching regions for the $1$-handles. In each $3$-ball in the attaching region for a $1$-handle, this last portion consists of two arcs meeting at the center attaching sphere point. This is because the $2$-handle passes over each $1$-handle exactly two times (either with opposite or the same orientation depending on whether $F$ is orientable or not). Through a Weinstein homotopy that adjusts the holonomy using~\cite[Lemma 12.5]{CE} in between two regular level sets which lie above all the index $1$ critical points and below $Y_G^1$, we can assume that these two arcs, together with the center of the $3$-ball form a single smoothly embedded unknotted Legendrian arc in each $3$-ball (unknotted because the restriction of Morse function has no critical points). Therefore $Skel(W_G,\omega_G,Z_G,\phi_G)\cap Y_G^0$ has front projection as in the leftmost images of Figure~\ref{fig:unknot} or Figure~\ref{fig:unknot_nonorientable}. The sequence of Legendrian Reidemeister moves shown in Figures~\ref{fig:unknot} and~\ref{fig:unknot_nonorientable} demonstrate that this intersection is Legendrian isotopic to the standard maximal Thurston-Bennequin number unknot. Therefore using~\cite[Lemma 12.5]{CE} again, now in the $0$-handle between two regular level sets which lie above the the index $0$ critical points and below $Y^0$, we obtain a Weinstein homotopy connecting $(W_G,\omega_G,Z_G,\phi_G)$ to $(W_M,\omega_M,Z_M,\phi_M)$. Thus to the canonical Weinstein structure on $D^*F$ (up to attaching a collar piece of the cylindrical completion which also can be achieved through a Weinstein homotopy that gradually grows/removes a collar). 
	
\end{proof}

%% file: parts/jetspace.tex

\section{Lifting co-normals to Kirby diagrams}\label{sec:curves}


Theorem~\ref{thm:toricWein} demonstrated that we are interested in Weinstein manifolds obtained by attaching $2$-handles to $D^*T^2$ along the Legendrian co-normal lifts $\Lambda_{(a,b)}$ of linearly embedded curves $\gamma_{(a,b)}\sse T^2$ in the $(a,b)$-homology class. More generally, we can consider the Weinstein domains $\mathcal{W}_{F,c}$ obtained by attaching $2$-handles along $\Lambda_{\gamma_i}$ to $D^*F$, where $c=\{\gamma_i\}$ is a collection of co-oriented curves on $F$. 

Our goal in this section is to find standard Weinstein handlebody \emph{diagrams} for $\mathcal{W}_{F,c}$.
	
	\begin{theorem}\label{thm:cotangentdiagram}
		Let $F$ be a closed surface and $c=\{\gamma_i\}_{i=1}^k$ a finite unordered collection of co-oriented curves in $F$. Let $\mathcal{W}_{F,c}$ denote the Weinstein domain obtained by attaching $2$-handles to $D^*F$ along the Legendrian co-normal lifts of the $\gamma_i$. Then, the Weinstein handle diagram in the standard form obtained by the procedure of Section~\ref{s:procedure}, represents a Weinstein manifold $\mathcal{W}_2$ that is Weinstein homotopic to $\mathcal{W}_{F,c}$.
	\end{theorem}

In section~\ref{s:procedure} we state the step-by-step procedure to produce the handlebody diagrams. Sections~\ref{section:step1} through~\ref{section:step4} prove that the result of this procedure is a diagram for a Weinstein domain which is Weinstein homotopic to $\mathcal{W}_{F,c}$. Finally in section~\ref{section:1handleslides} we show that if $c$ and $c'$ are related by an orientation preserving homeomorphism of $F$, then $\mathcal{W}_{F,c}$ and $\mathcal{W}_{F,c'}$ are related by a Weinstein homotopy via $1$-handle slides.

\subsection{The procedure to obtain a standard Weinstein handlebody diagram} \label{s:procedure}
We begin by presenting the sequence of steps that produces the Weinstein handlebody diagram, and then justify why these steps produce a diagram representing a Weinstein manifold that is Weinstein homotopic to the complement of the partially smoothed toric divisor (or more generally, a Weinstein manifold obtained by attaching $2$-handles to $D^*F$ along co-normal lifts of co-oriented curves $\gamma_1,\dots, \gamma_k$).

We begin with the ingredients, in this case:
\begin{itemize}
	\item A symplectic toric manifold with $\Delta$, its Delzant polytope, 
	\item $\mathcal{V} = \{V_1, \ldots V_k \}$ a list of centered vertices corresponding to a smoothing of the total toric divisor,
	\item $\mathcal{S}$ the square representing the torus $T^2$ with sides appropriately identified, shown in Figure \ref{torussquare}. For each vertex $V_i$, for $i=1,\ldots, k$, let $(v_i,u_i)$ be the positively oriented basis of primitive inward normal vectors to the edges meeting at $V_i$. Let $s_i=s(V_i)= v_i-u_i$, as in Proposition \ref{prop def of slopes}. Let $n_i=r(V_i)=Js_i,$ where $J=\begin{pmatrix}
	0 & -1 \\
	1 & 0
	\end{pmatrix}$ is a $\pi/2$ rotation. Draw on the square $\mathcal{S}$ the lines of slope $s_i$ with normals $n_i$. 
	\item $G$ the Gompf handlebody diagram of $D^*T^2$, shown in Figure \ref{fig:torusdecomp}.
\end{itemize}

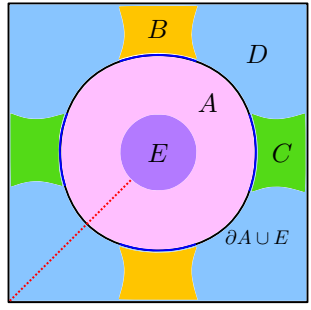
\begin{figure}
	\centering
	\begin{tikzpicture}[scale=0.66]
		\node[inner sep=0] at (0,0) {\includegraphics[width=4 cm]{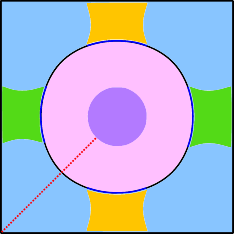}};
		\node at (0,0){$E$};
		\node at (1,1){$A$};
		\node at (0,2.5){$B$};
		\node at (2.5,0){$C$};
		\node at (2,2){$D$};
		\node[scale=0.75] at (2,-1.7){$\partial A\cup E$};
	\end{tikzpicture}
	\caption{A square $\mathcal{S}$ representing the torus with regions corresponding to: $A\cup E$ the 0-handle, $B$ and $C$ the 1-handles, $D$ the 2-handle. Additional curves can be drawn in the region $A\cup B\cup C$. We can remove the region $E$ and cut along the red dotted line to obtain the ``rectangle'' of Figure \ref{JetspaceRegions}.} 
	\label{torussquare}
\end{figure}

\begin{figure}
	\centering
	\begin{tikzpicture}
		\node[inner sep=0] at (0,0) {\includegraphics[width=10 cm]{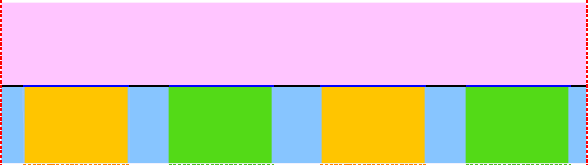}};
		\node at (0,0.75){$A$};
		\node at (1.25,-0.75){$B$};
		\node at (-1.25,-0.75){$C$};
		\node at (0,-0.75){$D$};
		\node at (4,0.25){$\partial A$};
	\end{tikzpicture}
	\caption{Regions in the jetspace corresponding to regions in the torus as seen in Figure \ref{torussquare}. Additional curves can be drawn in the region $A\cup B\cup C$.}
	\label{JetspaceRegions}
\end{figure}
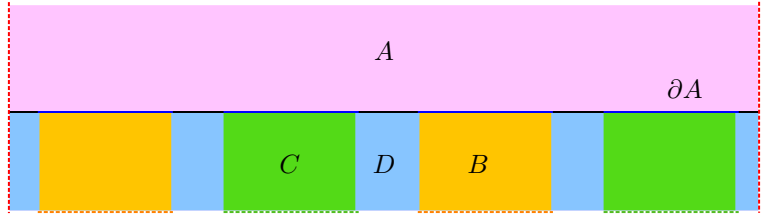

\begin{figure}
	\centering
	\includegraphics[width=7cm]{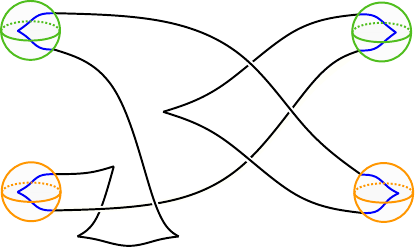}
	\caption{The Legendrian unknot $\partial A$ in the boundary of the $4$-dimensional $0$-handle of the Gompf diagram, indicating the intersection of this boundary with the Lagrangian torus ($0$-section of $D^*T^2$). The black portions coincide with segments of the attaching circle of the $2$-handle, and the blue portions give the attaching arcs of the $2$-dimensional $1$-handles of the torus.}
	\label{fig:torusdecomp}
\end{figure}

The steps are as follows (further details for these steps comprise the remainder of this section):

\begin{enumerate}
	\item \label{step:isotopy} Isotope the curves of slope $s_i$ in $\mathcal{S}$ such that at the end of the isotopy all the curves lie in the cross shaped region $A\cup B\cup C$ of Figure \ref{torussquare}, the curves only intersect in the interior of the annular region $A$, and bigons are removed according to the rule of Remark~\ref{rk:isotopy} and Figure \ref{bigons}. In this paper we choose the convention of always isotoping in the direction opposite the co-orientation, so that the resulting curves traverse the annulus always in the clockwise direction. See Figure \ref{cwccw}. 
	\item \label{step:jet} Cut $\mathcal{S}$ along a line from the bottom left corner of the square to the center (the red dotted line in Figure \ref{torussquare}), remove both the region $E$ and the boundary of the square, and ``unfold'' the annulus $A\cup B\cup C\cup D$ into a rectangle, representing $J^1(S^1)$, the 1-jet space of $S^1$, as in Figure \ref{JetspaceRegions}. Transfer the curves in the annulus $A$ to the rectangle $A$ via the unfolding diffeomorphism. Likewise, transfer curves in the regions $B$ and $C$ such that the bottom orange $1$-handle corresponds to the first orange $1$-handle on the left, etc.
	\item \label{step:satellite} Satellite the $J^1(S^1)$ picture onto the union of $\Lambda_0$ (the attaching sphere of the $2$-handle in the $D^*T^2$-handlebody diagram $G$) with the Legendrian obtained by cusping off instead of passing through the $1$-handles as in Figure~\ref{fig:torusdecomp}. Arrange the satellite such that the left $1$-handle of Figure \ref{JetspaceRegions} corresponds to the bottom-left $1$-handle in Figure \ref{fig:torusdecomp}. When satelliting near the $B$ and $C$ regions where $\Lambda_0$ and its cusped off version differ, follow the convention specified by Figure \ref{curvecorner}, taking care to preserve the relative Reeb heights. The result is a Weinstein handlebody diagram in standard form.
	\item \label{step:moves1} The Weinstein handlebody diagram can now be simplified using Reidemeister moves, Legendrian handle slides and cancellations. Some standard simplifications are illustrated in Figure \ref{fig:standard_simplifications}.
\end{enumerate}	

\begin{remark} 
	For a general surface $F$, the recipe to find the lift of a set of co-oriented curves $c=\{\gamma_{1}, \ldots, \gamma_k\} \sse F$, and obtain a Weinstein handlebody diagram of $\mathcal{W}_{F,c}$, start instead with a polygonal representation for the surface $F$ instead of the square $\mathcal{S}$ together with the curves $\gamma_i$ and their co-normal vectors replacing the lines of slopes $s_i$ and their normal vectors $n_i$ respectively. The polygonal representation similarly decomposes into an annulus $A$ surrounding a central disk $E$, together with $2g$ regions $B_1,C_1,\dots, B_g,C_g$ for the $2g$ handles of an orientable genus $g$ surface (or $g$ corresponding regions for $g$ $1$-handles if $F=\#_g \rptwo$) and a single region $D$ which is a neighborhood of the vertex of the polygonal representation. In step (\ref{step:satellite}), the satellite operation is onto the Legendrian $\Lambda_0^{F}$ in the Gompf diagram for $D^*F$. Otherwise, the conventions and proof are the same.
\end{remark}

\begin{figure}
	\begin{center}
		\includegraphics[width=7cm]{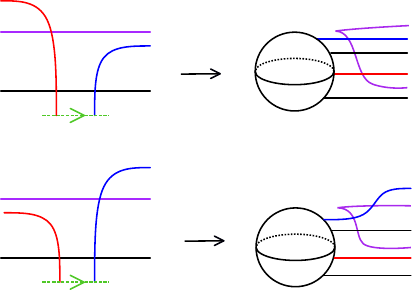}
		\caption{Mapping the red, blue and purple curves from $J^1(S^1)$ to $D^*T^2$. The order of the set of curves is preserved.}
		\label{curvecorner}
	\end{center}
\end{figure}

Our convention is to draw $T^2$ via its polygonal identification as the square with opposite sides identified. A curve $\gamma$ of slope $(a,b)$ is drawn in the standard way, where a horizontal curve oriented left to right is a $(1,0)$ curve, and a vertical curve oriented upward is a $(0,1)$ curve, see Figure \ref{vhcurves}. We assume the integers $(a,b)$ are relatively prime, but we allow any combination of signs. The signs determine an orientation on the curve $\gamma_{(a,b)}$, which in turn determines a co-orientation as follows.

We use the standard convention that the normal vector followed by the positive tangent vector should give the positive orientation on $T^2$ (we orient $T^2$, viewed as the square or equivalently $\R^2/\Z^2$ using the standard orientation on $\R^2$). Let $\Lambda_{(a,b)}$ be the corresponding Legendrian co-oriented co-normal lift of $\gamma_{(a,b)}$. For an orientable surface $F$ distinct from the torus we also assume that the normal vector followed by the positive tangent vector gives the positive orientation on $F$ inherited from the standard orientation on $\R^2$. For a non-orientable surface $F$, we restrict to co-oriented closed curves in $F$.

The square $\mathcal{S}$ with curves $s_i$ gives a description of a Weinstein handlebody given by the union of $D^*T^2$ with Weinstein $2$-handles attached along the co-oriented co-normal lifts of the curves of slope $s_1=s(V_1),\dots, s_k=s(V_k)$. By Theorem~\ref{thm:toricWein}, this is equivalent to the Weinstein complement of the $\{V_1,\dots, V_k \}$-smoothed toric divisor in our given $\{V_1,\dots, V_k \}$-centered toric manifold. Moreover, it is well known that the Stein handle calculus moves of step (\ref{step:moves1}) correspond to Weinstein homotopies. Thus, in order to prove Theorem \ref{thm:cotangentdiagram}, it suffices to show that the procedure of steps (\ref{step:isotopy})-(\ref{step:satellite}) yields a diagram corresponding to a Weinstein homotopic manifold. The proof of this theorem spans the next four subsections.

\subsection{Step (\ref{step:isotopy}): Legendrian isotopy in the co-sphere bundle} \label{section:step1}

First we note that an isotopy of a curve $\gamma \sse F$ gives a Legendrian isotopy of the co-normal lift $\Lambda_{\gamma}$. Similarly, an ambient isotopy in $F$ of a collection of curves in $F$ gives a Legendrian isotopy of the Legendrian link formed by their co-normal lifts. 

Recall that by Lemma~\ref{l:holonomy}, if two Weinstein manifolds differ by a Legendrian isotopy of the attaching link for the handles in a regular level set, then the two Weinstein manifolds are Weinstein homotopic.

We note that some Legendrian isotopies of links described as co-normal lifts can change the curves $\gamma_1,\cdots, \gamma_k$ in $F$ by more than just an ambient isotopy of $F$. Along with localized standard Legendrian Reidemeister moves, we can also include certain Reidemeister II moves which would not have been allowed in the usual front projection of $(\R^3, \xi= \ker(dz-ydx))$ to the $(x,z)$ plane. These Reidemeister II moves are allowed if the co-orientations of the two strands passing by each other are opposite as described in the following remark.
	
\begin{remark} \label{rk:isotopy} Pairs of co-oriented curves forming a bigon can be isotoped past each other if they have opposite co-normal orientations as in Figure~\ref{bigons} since they lift to disjoint co-normal lifts in $S^*F$ at every point of the isotopy. However, if the curves have the same co-normal orientation we cannot isotope them past each other because their lifts forms a clasp. 
\end{remark}

\begin{figure}
	\begin{center}
		\includegraphics[width=5cm]{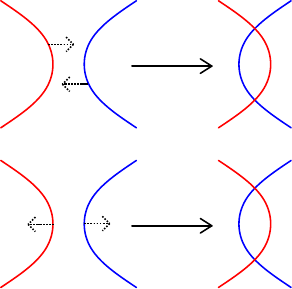}
		\caption{Curves with opposite co-normal orientations can be isotoped past each other without clasps forming in their lifts.}
		\label{bigons}
	\end{center}
\end{figure}

Note that the boundary of $A\cup E$ is a circle in $F$, which we can think of as the boundary of the $0$-handle in the handle decomposition of $F$. Let $\Gamma_0$ be the connected co-normal lift of $\partial(A\cup E)$ defined by the \emph{inward} normal co-orientation. Similarly, the boundary of $D$ is another circle in $F$, which we can think of as the attaching circle of the $2$-handle in the handle decomposition of $F$. Let $\Gamma_1$ be the connected co-normal lift of $\partial(D)$ co-oriented in the direction pointing outwards from $D$ (and towards $A\cup B\cup C\cup E$).

Our goal in step (\ref{step:isotopy}) is to perform a Legendrian isotopy of the link $\Lambda:=\Lambda_{\gamma_1}\cup\cdots\cup \Lambda_{\gamma_k}$ so that $\Lambda$ lies in a $C^0$-small neighborhood of $\Gamma_0\cup \Gamma_1$. Note that two Legendrian co-normal lifts in $S^*F$ are $C^k$-close if their images in $F$ are $C^{k+1}$ close. Therefore we want to isotope the curves $\gamma_1,\dots, \gamma_k$ so that they are contained in a $C^1$-neighbourhood of $(\partial (A\cup E))\cup(\partial D)$, \emph{with matching co-orientations}. An example of such an isotopy is pictured in Figure \ref{fig:removing_bigon}. For simplicity, we also ask that all crossings remain in $A$ in the portion that is $C^1$ close to $\partial(A\cup E)$, and that as many crossings as possible are removed. To achieve this, we use planar isotopies of the curves $\gamma_{1}\cup\cdots \cup \gamma_{k}$ in $F$ as well as the allowable co-normal bigon moves shown in Figure~\ref{bigons}. Note that in order for the co-orientations to match up with the inward normal of $\partial(A\cup E)$, we choose to isotope curves across $E$ so that they end up oriented clockwise in the final picture. The opposite convention could have been achieved by a relabelling of the 1-handles, see Figure \ref{cwccw}, though we find our convention is natural in the steps that follow.

\begin{figure}
	\begin{center}
		\includegraphics[width=13cm]{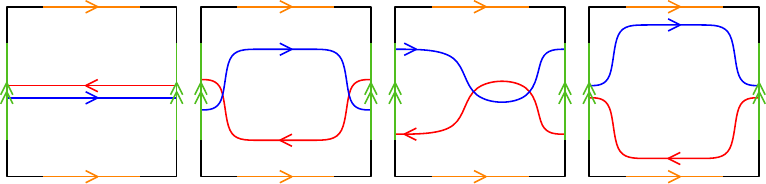}
		\caption{An example of an isotopy of two curves so that they move to a $C^1$-neighbourhood of $(\partial (A\cup E))\cup(\partial D)$ with matching co-orientations. This isotopy involves the removal of a bigon.}
		\label{fig:removing_bigon}
	\end{center}
\end{figure}

\begin{figure}
	\centering
	\includegraphics[width=10cm]{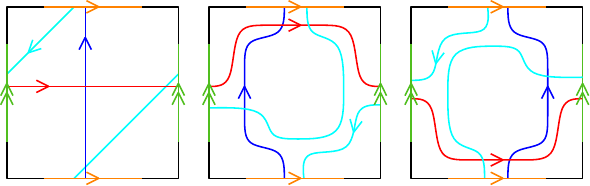}
	\caption{A set of curves in $T^2$, before and after isotoping to lie in $A\cup B\cup C$. In the second picture, curves are isotoped in the opposite direction of the co-orientation so that they are oriented clockwise, which is the convention used in this paper. In the third picture, curves are isotoped in the direction of the co-orientation. The second and third pictures yield the same diagram up to a 180 degree rotation, except the orientations of the curves are opposite.}
	\label{cwccw}
\end{figure}

\subsection{Step (\ref{step:jet}): Standard Legendrian neighborhoods}

Any Legendrian circle has a standard neighborhood that is contactomorphic to a neighborhood of the $0$-section in the $1$-jet space $J^1(S^1)=(S^1_{\theta}\times \R_y\times\R_z, \ker(dz-yd\theta))$. {This contactomorphism takes the Legendrian circle to the $0$-section and takes the Reeb vector field along the Legendrian to $\partial_z$, the Reeb vector field of $J^1(S^1)$.} Since we isotoped $\Lambda$ to be $C^0$ close to $\Gamma_0\cup \Gamma_1$, we can draw part of $\Lambda$ inside a $J^1(S^1)$ model representing a standard neighborhood of $\Gamma_0$, which we denote by $N(\Gamma_0)$. Moreover, a part of $\Lambda$ can be drawn inside a $J^1(S^1)$ model representing a standard neighborhood of $\Gamma_1$, which we denote by $N(\Gamma_1)$. Note that $\Gamma_0$ and $\Gamma_1$ coincide away from the preimages of the $B$ and $C$ regions. Furthermore, we have set our conventions so that all of the ``interesting'' parts of $\Lambda$ (the crossings in the projection) occur in $N(\Gamma_0)$. Our isotopy ensures that the portions of $\Lambda$ which are in $N(\Gamma_1)$, but not $N(\Gamma_0)$, are parallel positive Reeb push-offs of $\Gamma_1$. Therefore, it suffices to draw the portion of $\Lambda$ in $N(\Gamma_0)$, since this completely determines $\Lambda \cap N(\Gamma_1)$.

The front projection on $J^1(S^1)$ is to the $(\theta ,z)$ coordinates. We draw this projection as a rectangle, where $ \theta \in S^1$ is the horizontal coordinate, and where the left and right sides of the rectangle are identified. The unfolding of the polygon to a rectangle provides the front projection of $\Lambda$ in $N(\Gamma_0)$. Here it is important that $\Gamma_0$ is co-oriented by the inward normal direction, which gets sent by the unfolding to the upward $\partial_z$ direction. This corresponds to the fact that the Reeb push-off of $\Gamma_0$ is the co-normal lift of an inward normal push-off of $\partial(A\cup F)$, and the Reeb pushoff in the $J^1(S^1)$ picture is an upward $\partial_z$ push-off.

\subsection{Transitioning from co-normal diagrams to the Gompf diagram}

Our next goal is to relate the standard neighborhoods $N(\Gamma_0)$ and $N(\Gamma_1)$ to Legendrian curves in the Gompf diagram. Note that so far, we have been working with Legendrians in $S^*F$ with the canonical contact structure induced by the canonical Liouville form on $D^*F$. To go to the Gompf diagram, we need to perform a Weinstein homotopy of this canonical structure as in Theorem~\ref{thm:cot}, which induces a contact isotopy on the boundary $S^*F$ where we attach the $2$-handles. 

Recall from Section~\ref{s:cotangent} that the Weinstein structure compatible with the Gompf handlebody on $D^*F$ differs from the canonical structure on $D^*F$ by two steps. The first step is a Morsification of the canonical structure, as described in section~\ref{s:morsify}, using a perfect Morse function $f$ on $F$ such that the index $0$ critical point occurs in $E$, the index $1$ critical points occur in the $B$ and $C$ regions along the edges of the polygonal decomposition, and the index $2$ critical point occurs in $D$ at the vertex of the polygonal decomposition. The second step is a holonomy homotopy supported near the index $0$ and index $1$ critical points (which accounts for the difference between the smooth skeleton in the Morsification and the singular skeleton from the Gompf handle structure). Since this second step does not affect the Weinstein structure near the boundary of $D^*F$ where we attach the $2$-handles, we use the Morsified Liouville form $\lambda_f$ and its induced contact structure to relate the Legendrians $\Lambda$ in $S^*F$ with the canonical contact structure to their images under the induced contact isotopy in the Gompf Weinstein structure.

If our Morse function $f$ is sufficiently $C^\infty$-small (which can be achieved by a global rescaling by a small positive number), then the Morsified Liouville vector field is $C^\infty$-close to the original canonical Liouville vector field. There is a family of Liouville structures $\lambda_t = \lambda_{can}+ tdH_f$ for $t\in[0,1]$ interpolating between the canonical form $\lambda_{can}$ and the Morsified one $\lambda_f$. There are corresponding Liouville vector fields $Z_t$ interpolating linearly between $Z_{can}$ and $Z_f$. (For $t>0$, $Z_t$ looks qualitatively like $Z_f$ since $Z_t = Z_{tf}$.) If we fix a compact subdomain $W$ of the cotangent bundle such that $Z_t$ is outwardly transverse to the boundary for all $t\in[0,1]$, then we obtain a family of contact structures $\alpha_t$ induced on the boundary. By Gray's stability theorem, there is a corresponding contact isotopy $\Psi_t: \partial W \to \partial W$ relating $\xi_t = \ker(\alpha_t)$ to $\xi_0$ by $(\Psi_t)_*(\xi_0)=\xi_t$. If $f$ is $C^\infty$ small, this contact isotopy is $C^\infty$ small. In particular, the image of a Legendrian in $(\partial W, \xi_0)$ under $\Psi_1$ is a Legendrian in $(\partial W,\xi_1)$ which is $C^{\infty}$ close to the original curve. We  denote the contactomorphism at the end of this isotopy by $\Psi_f:=\Psi_1$.

In the Morsified structure, the boundary of the $0$-handle intersected with the $0$-section is a Legendrian $L_0$. Similarly, the boundary of the union of the $0$- and $1$-handles intersected with the $0$-section is a Legendrian $L_1$. From our diagrammatic conventions, $L_0 = \partial (A\cup F)$ and $L_1 = \partial D$, where we think of $F$ as the $0$-section of $D^*F$. In the Gompf diagram, the Legendrian knots $L_0$ and $L_1$ are well-understood. $L_1$ is simply the attaching sphere $\Lambda_0$ of the unique $2$-handle in the Gompf diagram. $L_0$ agrees with $L_1$ except in the attaching regions of the $1$-handles, where it gets cusped off to form a closed Legendrian unknot in $S^3$ (the boundary of the $0$-handle), as in Figure~\ref{fig:cotorus}.

Our next goal is to relate $L_0$ and $L_1$ (which we understand in the Gompf diagram) with $\Psi_f(\Gamma_0)$ and $\Psi_f(\Gamma_1)$ (which carry the information of the attaching data $\Lambda$ through their standard neighborhoods $N(\Gamma_0)$ and $N(\Gamma_1)$). Here $L_0$ is viewed as a Legendrian in the boundary of the $0$-handle for $D^*F$ and $L_1$ is a Legendrian in the boundary of the union of the $0$- and $1$-handles. Although we cannot directly compare $L_0$ and $L_1$ in the Morsified Liouville structure with analogues in the canonical Liouville structure, because the canonical Liouville form fails to be contact along $L_0$ and $L_1$, we can study Legendrian push-offs in both structures.

\begin{lemma}\label{l:reeb}
	Let $\phi$ be a sufficiently small Morse function as described above.
	Let $W_0$ denote a fixed $0$-handle in $D^*F$, and let $W_1$ denote the union of the fixed $0$- and $1$-handles. Let $\xi_0$ and $\xi_1$ denote the contact structures on $\partial W_0$ and $\partial W_1$ induced by the restriction of $\lambda_\phi$. Let $\Psi_i$ denote the end of the contact isotopy on $\partial W_i$ induced by $\lambda_t$ as described above. Then, $\Psi_i(\Gamma_i)$ is a Reeb push-off of $L_i$ for $i=0,1$ in $\partial(W_i,\xi_i)$ up to a $C^\infty$ small Legendrian isotopy.
\end{lemma}

\begin{proof}
Let $\alpha_f^i$ denote the restriction of $\lambda_f$ to $\partial W_i$.	
The Liouville condition implies that $d\lambda_f = \omega = dp_1\wedge dq_1+dp_2\wedge dq_2$, and therefore $d\alpha_f^i$ is the restriction of $\omega$ to the boundary. We can assume that $L_i =f^{-1}(c_i)$, and that the tangent space to $W_i$ along $L_i$ agrees with the tangent space to the set
$$\{(q_1,q_2,p_1,p_2) \mid f(q_1,q_2)\leq c, \textrm{ and } p_1^2+p_2^2\leq \varepsilon  \}.$$
Therefore along $L_i$ the vector field given by
$$\tilde R = \frac{\partial f}{\partial q_1}\partial_{p_1}+\frac{\partial f}{\partial q_2}\partial_{p_2}$$
is in the kernel of $d\alpha_f^i$ because
$$d\alpha_{f}(\tilde R, \cdot) = \frac{\partial f}{\partial q_1} dq_1 + \frac{\partial f}{\partial q_2}dq_2=df $$
and all vectors in the tangent space to $\{ f(q_1,q_2)=c \}$ are in the kernel of this 1-form.
To determine the positive direction, we check the sign of $\alpha_f(\tilde R) = \lambda_f(\tilde R)$
$$\lambda_f(\tilde R) = -\left(\frac{\partial f}{\partial q_1}\right)^2-\left(\frac{\partial f}{\partial q_1}\right)^2,$$
and find that it is negative. This shows that the \emph{actual Reeb vector field points in the co-normal direction co-oriented by the negative gradient of $f$}.
$$R = -\left(\frac{\partial f/\partial q_1}{(\partial f/\partial q_1)^2+(\partial f/\partial q_2)^2}  \right)\partial_{p_1}-\left(\frac{\partial f/\partial q_2}{(\partial f/\partial q_1)^2+(\partial f/\partial q_2)^2}  \right)\partial_{p_2}.$$
Therefore, the Reeb push-off $L_i^+$ of $L_i$ has the property that it projects to $f^{-1}(c)$ in $F$ under the standard co-normal projection $\pi:D^*F \to F$. Moreover, $L_i^+$ lies in the position specified by the co-orientation direction given by $-\nabla f$, which is the \emph{inward} normal direction to $A\cup E$. Since our initial choice of $f$ was made sufficiently $C^{\infty}$-small, $\Psi_f^{-1}(L_i^+)$ is $C^{\infty}$ close to $L_i^+$. Therefore the co-normal projections $\pi(\Psi_f^{-1}(L_i^+))$ and $\pi(L_i^+)=f^{-1}(c)\subset T^2$ are $C^{\infty}$ close to each other. Since $\Psi_f^{-1}(L_i^+)$ is Legendrian with respect to the canonical contact structure on $S^*F$, it must be the co-normal lift of its image $\pi(\Psi_f^{-1}(L_i^+))$. Because $\pi(L_i^+)$ and $\pi(\Psi_f^{-1}(L_i^+))$ are $C^{\infty}$-close, their co-normal lifts are Legendrian isotopic and $C^{\infty}$-close to each other. Therefore $\Psi_f^{-1}(L_i^+)$ must be Legendrian isotopic and $C^{\infty}$ close to the co-normal lift of $f^{-1}(c)$, co-oriented by $-\nabla f$, which is precisely $\Gamma_i$.
\end{proof}

\subsection{Step (\ref{step:satellite}): Satellite procedure}

Recall again that in the Gompf diagram, we can see $L_1$ as the attaching sphere of the unique $2$-handle. $L_0$ agrees with $L_1$ except in the attaching regions of the $1$-handles, where it gets cusped off to form a closed Legendrian unknot in $S^3$ (the boundary of the $0$-handle), as in Figure~\ref{fig:cotorus}. Since we can draw $L_0$ and $L_1$ in the Gompf handle diagram, we can similarly draw their positive Reeb push-offs in the Gompf diagram by a vertical push-off. 

Our final step is to satellite the $J^1(S^1)$ picture of $\Lambda$ into the Gompf diagram by identifying the core $0$-section $S^1$ in $J^1(S^1)$ with $\Gamma_0$ as the Reeb push-off of $L_0$ by Lemma~\ref{l:reeb}. We can similarly satellite a corresponding $J^1(S^1)$ model for the portion of $\Lambda$ in the standard neighborhood $N(\Gamma_1)$, aligning these two neighborhoods $N(\Gamma_0)$ and $N(\Gamma_1)$ on their overlap. Diagrammatically, this produces a standard Reeb-framed satellite, where some care is taken near the $1$-handles to account for the parts of $\Lambda$ in $N(\Gamma_1)$ that depart from $\Gamma_0$. The diagrammatic conventions for the satellite procedure here near the $1$-handles is shown in Figure~\ref{curvecorner}, which shows an example of three curves near a $1$ handle, and the corresponding front projections in the Gompf diagram. 

This completes the proof of Theorem~\ref{thm:cotangentdiagram}. See Figures~\ref{vhcurves} and~\ref{fig:basics} for basic examples applying the algorithm

\begin{figure}
	\centering
	\includegraphics[width=15cm]{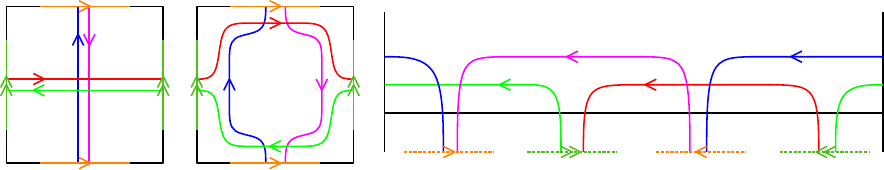}
	\caption{Pictures depicting the intermediate steps in obtaining the Legendrian lifts of the curves $(\pm1,0)$ and $(0,\pm1)$ in $T^2$. Here they are pictured before and after a planar isotopy to lie in $A\cup B\cup C$, and then redrawn in $J^1(S^1)$.}
	\label{vhcurves}
\end{figure}

\begin{figure}
	\centering
	\subfigure[$\mathcal{W}_{T^2, \{(1,0)\}}$]
	{\includegraphics[width=7cm]{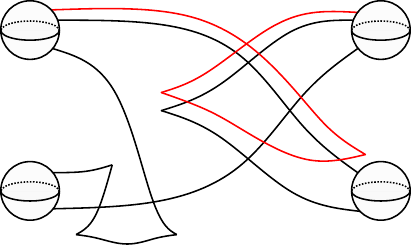}}
	\hspace{1cm}
	\subfigure[$\mathcal{W}_{T^2, \{(-1,0)\}}$]
	{\includegraphics[width=7cm]{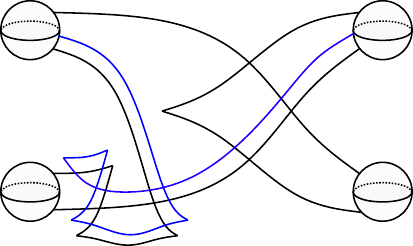}}
	
	\vspace{1cm}
	\subfigure[$\mathcal{W}_{T^2, \{(0,1)\}}$]
	{\includegraphics[width=7cm]{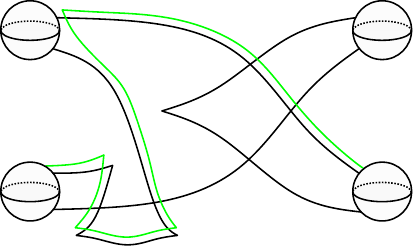}}
	\hspace{1cm}
	\subfigure[$\mathcal{W}_{T^2, \{(0,-1)\}}$]
	{\includegraphics[width=7cm]{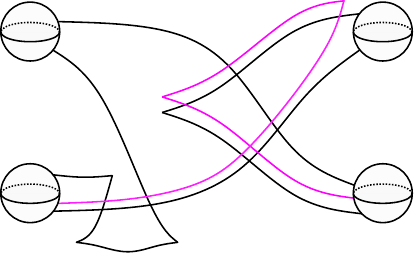}}
	\caption{The Legendrian lifts of the curves $(\pm1,0)$ and $(0,\pm1)\subset T^2$, which allow us to find Weinstein handlebody diagrams of $\mathcal{W}_{T^2, (\pm1,0) }$ and $\mathcal{W}_{T^2, (0, \pm1)}.$}
	\label{fig:basics}
\end{figure}

\subsection{Step (\ref{step:moves1}): Simplifying the Gompf diagram} \label{section:step4}

Having drawn additional curves in the Gompf diagram, we can now simplify using a series of Reidemeister moves, handles slides, and handle cancellations, as described in Section \ref{s:kirby}. Figure \ref{fig:standard_simplifications} shows some simplifications that we use in this paper and that can often be performed in the diagrams obtained through this algorithm. Although this figure shows the moves being performed on a single strand, the same moves can be performed with many parallel Reeb push-off copies of the same strand, a scenario that occurs with complicated slopes or multiple slopes. We also perform handle cancellation whenever possible, because diagrams without $1$-handles can be simpler to use for the calculation of invariants and for identifying Lagrangian submanifolds.

\begin{figure}
	\begin{center}
		\includegraphics[width=15cm]{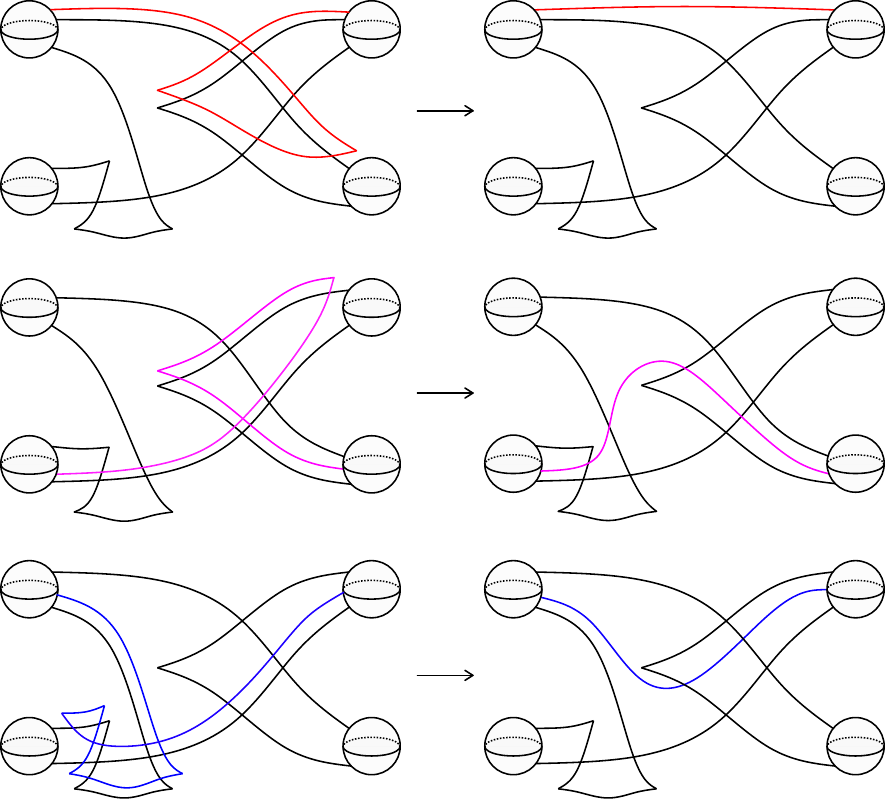}
		\caption{Some standard simplifications of curves that can be performed in step~\ref{step:moves1}. Note that the blue curve is simplified using a Gompf move 6.}
		\label{fig:standard_simplifications}
	\end{center}
\end{figure}

\subsection{Legendrian $1$-handle slides} \label{section:1handleslides}

With this algorithm in hand we can now translate the action of Weinstein $1$-handle slides in the output diagram $\mathcal{W}_2$ in terms of the attaching curves and disk cotangent bundle structure of $\mathcal{W}_{F,c}$. See~\cite{GompfStipsicz,ACGMMSW1} for an introduction to $1$-handle slides for Weinstein $4$-manifolds.

\begin{prop}\label{prop:1handles}
	Let $F$ be a closed orientable surface. The group of orientation preserving homeomorphisms of $F$ correspond to $1$-handle slides in the Gompf handlebody of $D^*F$, and to $1$-handle slides in the Weinstein handlebody diagram of $\mathcal{W}_{F,c}$, where $c$ is a set of co-oriented curves on $F$.
\end{prop}

\begin{proof}
	First recall that the group of orientation preserving homeomorphisms of $T^2$ is $SL(2,\mathbb{Z})$ and is generated by Dehn twists about the meridian and longitude:
	$$D_{a}=\begin{pmatrix}
		1&0\\
		1&1
	\end{pmatrix}~~ \text{and} ~~D_b=\begin{pmatrix}1&-1\\0&1\end{pmatrix}.$$
	More generally, for an orientable surface $F$, the group of orientation preserving homeomorphisms of $F$ is generated by Dehn twists which can be interpreted as $1$-handle slides of $F$ for the handle decomposition of $F$ constructed from a polygon (where sides are appropiately identified). \\
	In section~\ref{s:cotangent} we construct a Weinstein structure on $T^*F$ homotopic to the canonical Gompf Weinstein structure on $T^*F$ for a surface $F$. In particular we construct a Weinstein homotopy from the canonical Stein structure on $T^*F$ where we have a smooth Lagrangian skeleton $F$ to the Gompf handlebody of $T^*F$ where the Lagrangian skeleton is now singular. In between, we saw a Morsified Weinstein structure where the Lagrangian skeleton is the smooth zero section $F$, and the Weinstein structure restricts to a standard Morse function on $F$. The Weinstein structure with the singular skeleton differs from this one only by a holonomy Weinstein homotopy (which preserves the critical points and values). A $1$-handle slide on the standard Morse function of $F$ corresponds to an orientation preserving homeomorphism of $F$. By our construction of $D^*F$, $1$-handle slides of $F$ correspond to $1$-handle slides of $D^*F$. After performing the $1$-handle slides on the Morse Weinstein structure with smooth skeleton, we can perform the same holonomy Weinstein homotopy to align with the Morsified Weinstein structure with the Gompf Weinstein structure. 
\end{proof}

%% file: parts/applications.tex

\section{Invariants}\label{sec:applications}

Certain invariants become algorithmic to compute when given an explicit Weinstein diagram such as the usual homology and the Chekanov Eliashberg dga of the attaching spheres (which in turn is useful for understanding the symplectic homology or wrapped Fukaya category of the Weinstein manifold). In this section, we explain how to compute such invariants of the completions of the Weinstein domains $\mathcal{W}_{F,c}$ utilizing the Weinstein handlebodies that we produced in section~\ref{sec:curves}. We give formulas to compute the standard homology in section~\ref{s:homology}, and give a result about the symplectic homology in section~\ref{s:SH}. In section~\ref{section examples}, we use the invariants to show in examples when the completions of $\mathcal{W}_{F,c}$ and $\mathcal{W}_{F,c'}$ are not symplectomorphic for different choices of curves $c$ and $c'$.

\subsection{Topological invariants}\label{s:homology}

First we discuss how Weinstein handle diagrams can be used compute the homology and fundamental group of $\mathcal{W}_{F,c}$, as well as the homology of its boundary. Following~\cite{OzbagciStipsicz}, consider a Weinstein $4$-manifold $X$ with boundary $Y$ with a Weinstein handle diagram with one $0$-handle, $1$-handles $h_1, \ldots, h_s$, and $2$-handles attached along Legendrian knots $\Lambda_1, \ldots, \Lambda_r$ with framing $tb(\Lambda_i)-1$. Each $1$-handle is a generator in $\pi_1(X)$, while each $2$-handle gives a relation determined by the signed intersection of its attaching sphere with the belt spheres of the $1$-handles. Furthermore, if one considers the matrix $\phi$ given by the signed intersection of $2$-handles with $1$-handles, one finds that $H_1(X;\Z)\cong \coker \phi$ and $H_2(X;\Z)\cong \ker \phi$. Since there is only one $0$-handle and no $3$ or $4$-handles, we know that $H_0(X;\Z)=\Z$ and $H_i(X;\Z)=0$ for $i>2$. If $X$ has a Weinstein handle diagram with no $1$-handles, then the intersection form is the linking matrix of the attaching link of the $2$-handles. The $(i,j)$th entry of the linking matrix is defined as the signed linking number of $\Lambda_i$ and $\Lambda_j$, and the $(i,i)$th entry is the framing of $\Lambda_i$, ($tb(\Lambda_i)-1$).

In order to compute the homology of the boundary $Y$ we replace each $1$-handle with a $0$ framed $2$-handle attached along an unknot and connect the ends of the attaching spheres passing through the $1$-handles in a canonical closure as in Figure~\ref{fig:handleswap}. We now have a new $4$-manifold $Z$ which is simply connected and has boundary $Y$. Then, $H_2(Z;\Z)$ is generated by surfaces given by attaching the Seifert surfaces $\Sigma_i$ of each $\Lambda_i$ with the core of the corresponding $2$-handle for $i=1, \ldots, r$. Use the following long exact sequence to compute $H_*(Y;\Z)$:
\begin{equation}\label{eqn:hom_les}
0\rightarrow H_2(Y;\Z) \rightarrow H_2(Z;\Z)\xrightarrow{\phi_1} H_2(Z,Y;\Z)\xrightarrow{\phi_2} H_1(Y;\Z)\rightarrow 0
\end{equation}
where $\phi_1([\Sigma_i])=n_i [D_i]+\sum_{j\neq i} lk(\Lambda_i, \Lambda_j)[D_j]$, $\phi_2([D_i])=[\partial D_i]$, $D_i$ denotes the meridional disk of $\Lambda_i$ that intersects $\Lambda_i$ positively, and $n_i$ is the framing of the $2$-handle attached along $\Lambda_i$.

\begin{figure}
	\begin{center}
		\includegraphics[width=12cm]{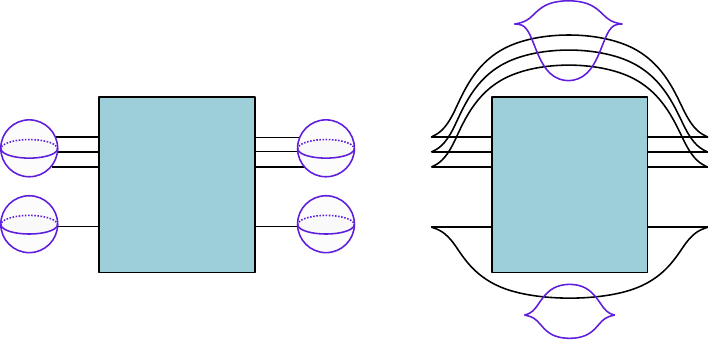}
		\caption{Diagrams whose boundaries correspond to the same contact 3-manifold \cite{DingGeiges}, we swap the purple 1-handles in the diagram on the left for purple $0$ framed $2$-handles on the right. The blue box denotes some arbitrary Legendrian tangle.}
		\label{fig:handleswap}
	\end{center}
\end{figure}

In the case of the complements of smoothed toric divisors and any Weinstein manifold $\mathcal{W}_{T^2,c}$ the homology computations can be obtained from their Weinstein diagrams.

\begin{lemma} \label{slopehoms} Suppose that $c=\{(a_1, b_1), \ldots, (a_n, b_n)\} \subset T^2$ is a set of co-oriented curves of the given slopes. Let $\phi=\begin{pmatrix}
  0 & a_1 & \cdots & a_n\\
    0 & b_1 & \cdots & b_n
  \end{pmatrix}$, then $H_1(\mathcal{W}_{T^2, c};\Z)=\coker \phi$
 and $H_2(\mathcal{W}_{T^2, c};\Z)= \ker \phi$.
\end{lemma}
\begin{proof}
By Theorem~\ref{thm:cotangentdiagram}, the co-normal lift of a curves $(a_i,b_i)\subset T^2$, $\Lambda_{(a_i,b_i)}\subset S^*T^2$ intersects the two one handles $a_i$ and $b_i$ times. Furthermore, there is one $2$-handle in the Weinstein diagram of $\mathcal{W}_{T^2,c}$ with signed intersection $0$ for each of the $1$-handles. 
\end{proof}

The homology of the boundary of $\mathcal{W}_{T^2, c}$ is computed using the  Thurston-Bennequin number of each Legendrian $\Lambda_{(a_i,b_i)}$ and their linking numbers in order to compute. Note that because these curves are not null-homologous in $S^*T^2$, they do not have a canonical $0$-framing abstractly, so the Thurston-Bennequin number is not strictly well-defined. However, once the Legendrian is drawn in a particular handlebody diagram in Gompf standard form, there is a convention provided in~\cite{Gompf} for assigning a Thurston-Bennequin number to Legendrians in the diagram. Note that under Gompf move $6$ the Thurston-Bennequin number of the link can change so these quantities are dependent on the diagram itself, although the canonical framing is preserved by contactomorphism (what changes is what is considered the $0$-framing). Here, we calculate these values for the curves arising in diagrams obtained as the output of the algorithm of Section~\ref{s:procedure}.

 \begin{lemma}\label{tb}
Let $(a,b)$ be coprime integers and $\Lambda_{(a,b)}$ be the image of the co-normal lift of $\gamma_{(a,b)}\subset T^2$ in the standard Gompf diagram for $S^*T^2$, where if $a<0$ we have also performed Legendrian isotopy and Gompf move $6$ {as shown in Figure~\ref{fig:tb_2} for the case $a<0,~b\geq 0,$ and $a>b$}. Then, the Thurston-Bennequin number of $\Lambda_{(a,b)}$, with respect to such diagrams, is given in the following table.

\begin{center}
\begin{tabular}{ |p{3cm}|p{3cm}|p{3cm}| }
 \hline
 $(a,b)$ & $tb(\Lambda_{(a,b)})$\\
 \hline
 $b\geq 0$ & $-2b^2-|ab|$ \\
 \hline
 $b < 0$   & $ab$ \\
 \hline
\end{tabular}
\end{center}
\end{lemma}
\begin{proof}

There are eight cases depending on the relative values of $a$ and $b$ to consider, each of which yields a co-normal lift with distinct Lagrangians: either $|a|<|b|$ or $|b|<|a|$, $a$ is either positive or negative, and $b$ is either non-negative or negative.. Here, we only discuss two cases in detail since the other cases follow by analogous arguments. 

\emph{Case $1$: consider a curve $\gamma_{(a,b)}$ such that $a,~b\geq 0$ and $a<b$.} Let $r=b-a$. Using the algorithm of section~\ref{s:procedure} and Legendrian isotopy,  $\mathcal{W}_{T^2, (a,b)}$ has a Weinstein diagram given by Figure~\ref{fig:tb_1}. Recall that the Thurston-Bennequin number of an oriented Legendrian knot can be directly computed from the front projection of a Legendrian: $tb(\Lambda)=\text{writhe}(\Pi(\Lambda))-(\text{the number of right cusps in }\Pi(\Lambda))$. {Note that in a front projection of a Legendrian a half twist in a band of $n$ parallel strands contributes $-\frac{n(n-1)}{2}$ to the writhe, see Figure~\ref{tb_contribution}. Then using Figure~\ref{fig:tb_1_b}, note that there are $4$ half twists of bands with $a+r=b$ strands, an additional $ab$ negative crossings, and $2(r+a)$ right cusps. Therefore,}
$$tb(\Lambda_{(a,b)})= \left(-4\frac{b(b-1)}{2}-ab\right)-(2a+2r)=-2(b^2-b)-ab-2b=-2b^2-ab.$$

\emph{Case $2$: let $\gamma_{(a,b)}$ be a curve such that $a<0,~b\geq 0,~a>b$ and $b-|a|=r$.} Again using the algorithm described in the previous section,  and Legendrian isotopy we obtain the front projection in  Figure~\ref{fig:tb_2_a_1} from which we {note that there are $4$ half twists of bands with $|a|+r=b$ strands, $4$ half twists of bands with $|a|$ strands, an additional $5|a|b$ negative crossings, $2|a|^2+2|a|b$ positive crossings, and $2|a|+2b$ right cusps.} Therefore,
$$tb((\Lambda_{(a,b)})= \left(-4\frac{b(b-1)}{2}-4\frac{|a|(|a|-1)}{2}-5|a|b+2|a|^2+2|a|b\right)-(2|a|+2b)=-2b^2-3b|a|.$$

If we also perform Legendrian Reidemeister moves and a single Gompf move $6$ we find that $\mathcal{W}_{T^2, (a,b)}$ has a Weinstein handlebody diagram as shown in Figure~\ref{fig:tb_2_b}. {For this Legendrian, there are $4$ half twists with $b$ strands, $|a|b$ negative crossings and $2b$ right cusps.} Therefore,
$$tb(\Lambda_{(a,b)})=\left(-4\frac{b(b-1)}{2}-|ab|\right)-(2b)=-2b^2-|ab|.$$

If we have a general link $\Lambda$, with a sublink $\Lambda_{(a_i,b_i)}$ such that $a_i<0$ then we can isotope the curves on the torus as discussed in Section~\ref{sec:curves} on each of the $\Lambda_{(a_i,b_i)}$. The Thurston-Bennequin values in the statement are given for the Legendrians $\Lambda_{(a,b)}$ after performing a sequence of Legendrian isotopies and a Gompf move $6$ if $a<0$.
\end{proof}

\begin{figure}
	\begin{center}
		\includegraphics[width=3 cm]{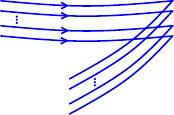}
		\caption{This image shows a half twist with $n$ strands where all of the strands are oriented in the same manner. There are $\frac{n(n-1)}{2}$ negative crossings and either $0$ or $n$ right cusps depending on the half twist is towards the left or right.}
		\label{tb_contribution}
	\end{center}
\end{figure}

\begin{figure}
	\centering
	\subfigure[$\gamma_{(a,b)}\subset T^2$]{\begin{tikzpicture}\node[inner sep=0] at (0,0) {\includegraphics[width=6 cm]{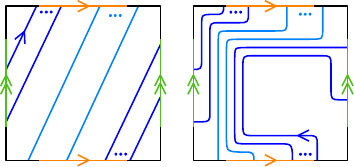}};
			\node at (-2.6,0.4){$a$};
			\node at (-1.6,0){$r$};
	\end{tikzpicture}} \hspace{1cm}
	\subfigure[$\gamma_{(a,b)}\subset \mathcal{J}^1(S^1) $]{\includegraphics[width =7cm]{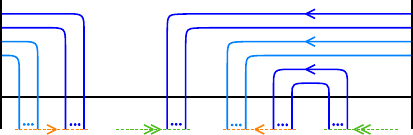}}  
	
	\vspace{1cm}
	
	\subfigure[$\Pi(\Lambda_{(a,b)})$]{\includegraphics[width=7 cm]{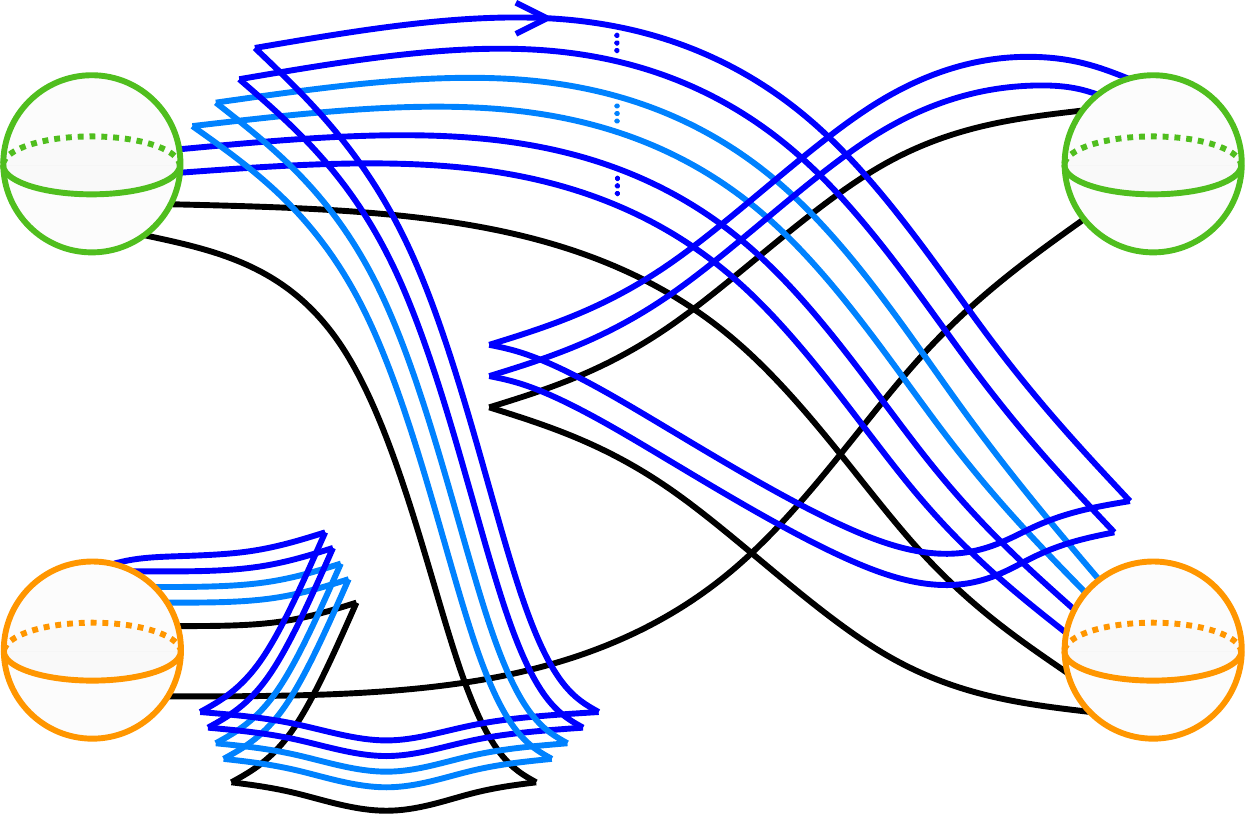}\label{tb_1_a} }
	\hspace{1cm}
	\subfigure[$\Pi(\Lambda_{(a,b)})$ after Legendrian isotopy]{\includegraphics[width=7cm]{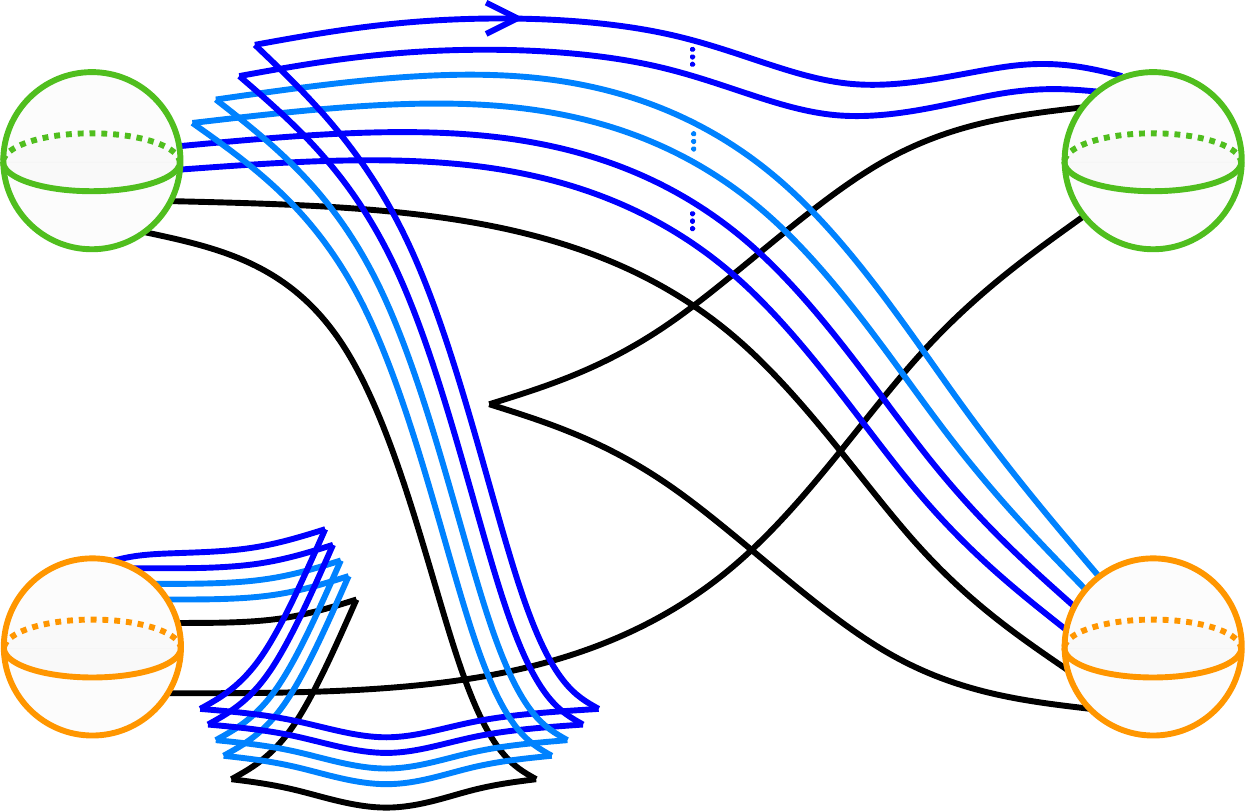}\label{fig:tb_1_b}}
	\caption{The steps to obtain front projection of $\Lambda_{(a,b)}$ for $a,b>0$ and $a<b$, where the darker strands represent $a$ parallel strands and the lighter strands represent $r=b-a$ parallel strands.}
	\label{fig:tb_1}
\end{figure}

\begin{figure}
	\centering
	\subfigure[$\gamma_{(a,b)}\subset T^2$]{\begin{tikzpicture}\node[inner sep=0] at (0,0) {\includegraphics[width=6 cm]{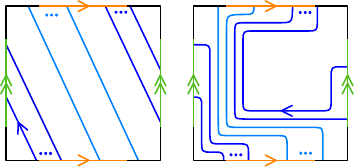}};
			\node at (-2.55,-0.5){$a$};
			\node at (-1.6,0){$r$};
	\end{tikzpicture}} 
	\hspace{2cm}
	\subfigure[$\gamma_{(a,b)}\subset \mathcal{J}^1(S^1)$]{\includegraphics[width=7 cm]{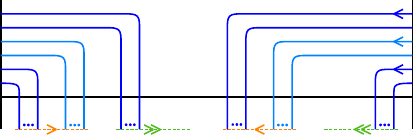}}  
	
	\vspace{1cm}
	\subfigure[$\Pi(\Lambda_{(a,b)})$]{\includegraphics[width=7cm]{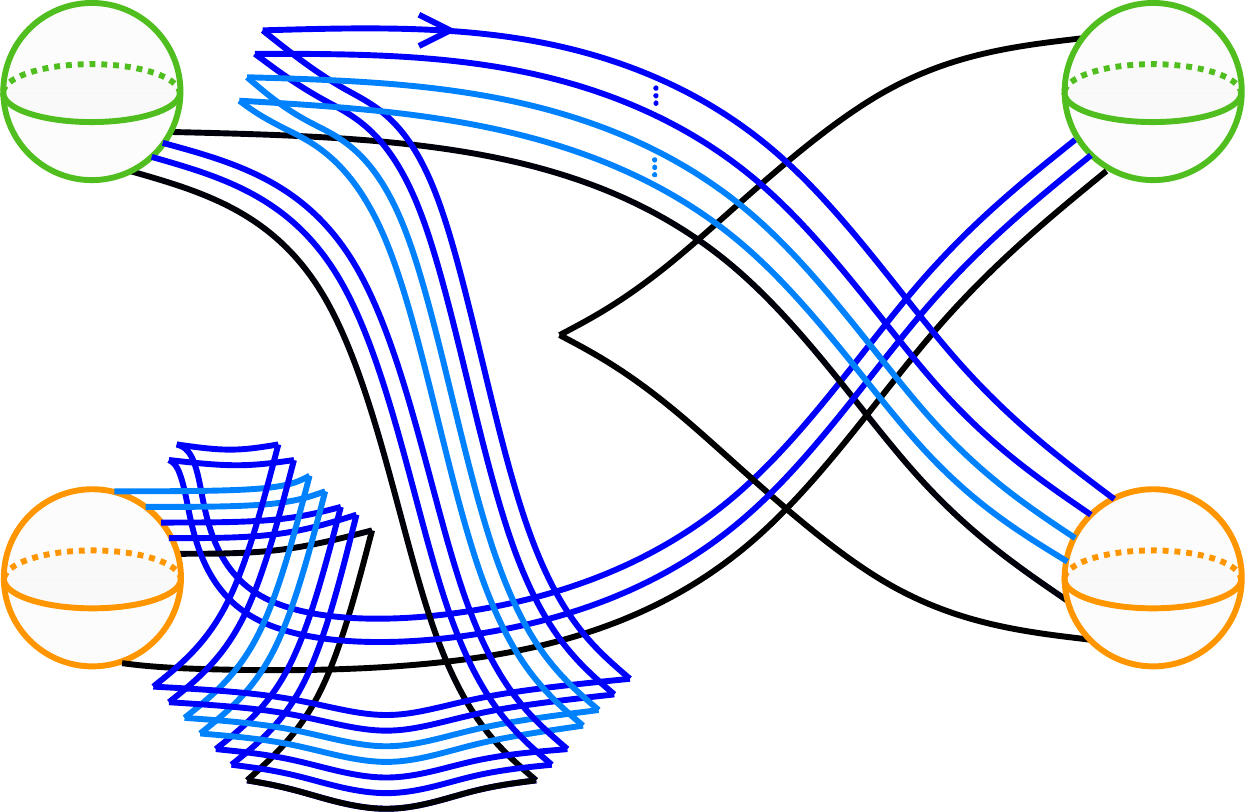} \label{fig:tb_2_a_1}}
	\hspace{1cm}
	\subfigure[$\Pi(\Lambda_{(a,b)})$ after Legendrian isotopy and Gompf move $6$]{\includegraphics[width=7cm]{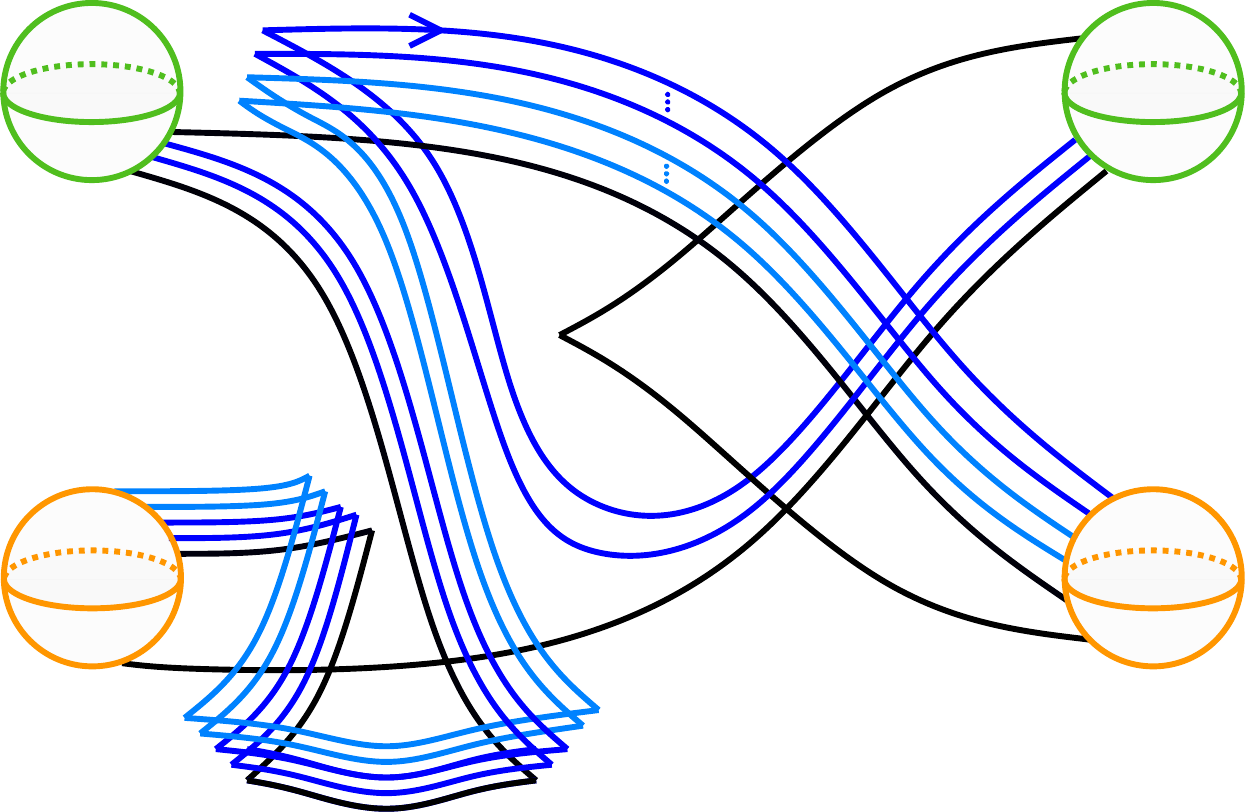}\label{fig:tb_2_b}}
	\caption{The steps to obtain the front projection of $\Lambda_{(a,b)}$ for $a<0$, $b>0$ and $a<b$, where the darker strands represent $a$ parallel strands and the lighter strands represent $r=b-a$ parallel strands.}
	\label{fig:tb_2}
\end{figure}

As an example, we can now compute the homology of the boundary of $\mathcal{W}_{T^2, c}$ for when $c$ is $\{(1,0), (a,b)\}$ or $\{(1,0), (a_1,b_1), (a_2, b_2)\}$. Note that for the case of complements of smoothed toric divisors we can always find an $SL(2, \Z)$ transformation taking one node to the node with slope $(1,0)$. The following two calculations follow directly from Lemma~\ref{tb} and the long exact sequence given in ~\ref{eqn:hom_les}.
\begin{prop}
 The homology of the boundary of the Weinstein $4$-manifold $\mathcal{W}_{T^2, \{(1,0), (a,b)\}}$ is given by: $$H_1(\partial \mathcal{W}_{T^2,\{(1,0),(a,b)\}}; \Z)=\coker(\phi)~\text{ and}$$ 
$$H_2(\partial \mathcal{W}_{T^2,\{(1,0),(a,b)}\}; \Z)=\ker(\phi) ~\text{where}$$ 
$$\phi=\begin{pmatrix}
0&0&0&1&a\\
0&0&0&0&b\\
0&0&0&0&-b\\
1&0&0&-1&-a\\
a&b&-b&-a&tb(\Lambda_{(a,b})-1\\

\end{pmatrix}.$$

\end{prop}

\begin{prop}\label{homology_boundary}
The homology of the boundary of the Weinstein $4$-manifold $\mathcal{W}_{T^2,\{(1,0),(a_1,b_1),(a_2,b_2)\}}$ is given by: 
 $$H_1(\partial \mathcal{W}_{T^2,\{(1,0),(a_1,b_1),(a_2,b_2)\}}; \Z)=\coker(\phi), ~\text{and}$$ $$H_2(\partial \mathcal{W}_{T^2,\{(1,0),(a_1,b_1),(a_2,b_2)\}}; \Z)=\ker(\phi), ~\text{ where}$$ 
$$\phi=\begin{pmatrix}
0&0&0&1&a_1&a_2\\
0&0&0&0&b_1&b_2\\
0&0&0&0&-b_1&-b_2\\
1&0&0&-1&-a_1&-a_2\\
a_1&b_1&-b_1&-a_1&tb(\Lambda_{(a_1,b_1})-1&lk(\Lambda_{(a_1,b_1)},\Lambda_{(a_2,b_2)})\\
a_2&b_2&-b_2&-a_2&lk(\Lambda_{(a_1,b_1)},\Lambda_{(a_2,b_2)})&tb(\Lambda_{(a_2,b_2})-1\\
\end{pmatrix}.$$
\end{prop}

\subsection{Symplectic invariants} \label{s:SH}
Having an explicit Weinstein diagram of a Weinstein $4$-manifold $X$ is also useful because one can compute invariants of the Weinstein manifold from the Legendrian attaching links $\Lambda$ of $X$. One such invariant is the symplectic homology of $X$, denoted by $S\mathbf{H}(X)$. {Symplectic homology is obtained from a chain complex whose generators are critical points of a Morse function on $X$ that is increasing towards the boundary $\partial(X)$ or positively parametrized closed Reeb orbits of $\partial X$, and the differential consists of certain pseudoholomorphic cylinders.} Symplectic homology is an important invariant with applications to the Weinstein conjecture, and homological mirror symmetry. The symplectic homology of log Calabi Yau surfaces, and of complements of symplectic divisors has been studied in~\cite{nguyen2015complement, ganatra2018log,Pascaleff_2019, diogo_lisi, ganatra_pomerleano_2020}. It is has been conjectured that the symplectic homology for complements of smoothed total toric divisors of toric $4$-manifolds is non-vanishing. In particular, Gross, Hacking, and Keel conjecture that for a log Calabi Yau manifold $U$ with mirror variety $U^{v}$, the ring of regular functions of $U$ is contained in the symplectic cohomology of $U$~\cite{GrossHackingKeel}. In Proposition~\ref{prop:vanishingSH} we prove that symplectic homology is non-vanishing on the complements of smoothed total toric divisors in toric $4$-manifolds. 

We now provide the relevant definitions and background on the Chekanov Eliashberg dga and graded normal rulings of Legendrians, as they are tools that we use in the proof of Proposition~\ref{prop:vanishingSH}. Let $\Lambda$ denote a Legendrian link of $l$ components in $\#^k(S^1\times S^2)$. The \emph{Chekanov Eliashberg dga}, denoted by $(\mathcal{A}(\Lambda), \partial_{\Lambda})$, is a Legendrian link invariant and can be defined over $\Z [H(\Lambda)]=\Z[t_1^{\pm1}, \ldots, t_l^{\pm1}]$, see~\cite{EkholmNg} for more details. The generators of the Chekanov Eliashberg dga are Reeb chords, and the differential counts pseudoholomorphic disks asymptotic to the Reeb chords in the symplectization of $\#^k(S^1\times S^2)$. A \emph{graded augmentation} $\epsilon$ of $(\mathcal{A}(\Lambda), \partial_{\Lambda})$ is a dga map $\epsilon: \mathcal{A}(\Lambda) \rightarrow \Z$ such that, $\epsilon \circ \partial_{\Lambda}=0$ and $\epsilon(t_1\cdots t_l)=(-1)^{l}$.  Leverson in Theorem $1.3$ of~\cite{Leverson} proved that the existence of a graded augmentation of a Legendrian link $\Lambda \sse (\#^k(S^1\times S^2), \xi_{std})$ corresponds to the existence of a \emph{graded normal ruling} of the front projection of $\Lambda$, providing an easier method to check for the existence of a graded augmentation.

\begin{definition}
Let $\Lambda \subset  (\#^k(S^1\times S^2), \xi_{std})$ be a Legendrian link. A \emph{ruling} of $\Lambda$ is a decomposition of the front projection of $\Lambda$ into pairs of paths such that any given pair must satisfy the following conditions. Any two paired paths must start at a common left cusp, or common $1$-handle, and end at a common right cups or common $1$-handle; the two paths can have no other intersections, and must bound a topological disk whose boundary is smooth except for the points where the pairs of paths meet at cusps or $1$-handles, or at crossings where the two paths can switch (such crossings are then reffered to as switches). A \emph{normal ruling} is a ruling such that near a switch the rulings are as shown in Figure~\ref{fig:normal_ruling_front}. A \emph{graded ruling} is one such that all switches occur at crossings where the two strands have the same Maslov potential. The \emph{Maslov potential} of a Legendrian link is a locally constant map that assigns to each strand in the front of $\Lambda$ an integer such that near a cusp the Maslov potential increases or decreases by $1$.
\end{definition}

\begin{figure}
	\centering
	\includegraphics[width=8cm]{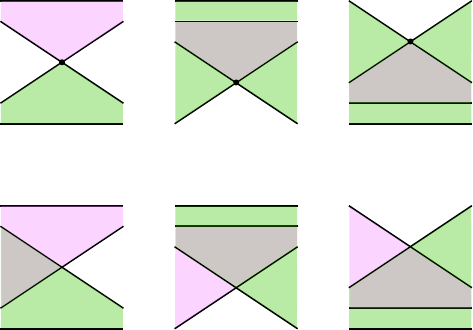}
	\caption{All possible configurations of a normal ruling near a crossing are given by $(a)-(f)$ and the vertical reflections of $(d)-(f)$.}
	\label{fig:normal_ruling_front}
\end{figure}     

\begin{prop}\label{prop:vanishingSH}
Let $F$ be an orientable surface and let $X$ be any Weinstein $4$-manifold constructed by attaching $1$ or $2$-handles to $D^*F$ and taking its cylindrical completion. Then $X$ has nonvanishing symplectic homology.
\end{prop}

\begin{proof}
Let $X_{\Lambda}$ be a Weinstein $4$-manifold constructed by attaching $2$-handles along a Legendrian link $\Lambda$ to the subcritical Weinstein domain with boundary $\#^k(S^1\times S^2)$ and taking its cylindrical completion. The Legendrian surgery formula states that $S\mathbf{H}(X_{\Lambda})=LC\mathbf{H}^{H_0}(\Lambda)$ (Corollary $5.7$ in~\cite{BEE}) where $LC\mathbf{H}^{H_0}(\Lambda)$ is the homology of the Hochschild complex generated by cyclically composable non-empty words of Reeb chords. See Section $4$ of~\cite{BEE} for the definition of $LC\mathbf{H}^{H_0}(\Lambda)$, and sections $5-6$ of \cite{BEE} as well as the proofs of Theorems $1.2$ and $1.3$ of~\cite{Ekholm} for the proof of this surgery formula. By Theorem $1.6$ of~\cite{orsola_thesis} (a generalization of Corollary $1.4$ in~\cite{Leverson}), if any sublink of $\Lambda$ has a graded augmentation then the complex $LCC^{H_0}(\Lambda)$ has at least one cycle that is not in the image of the differential. Therefore, the symplectic homology of $X_{\Lambda}$ is non-vanishing. 

Let $X$ denote any Weinstein $4$-manifold $X$ constructed by attaching $1$ or $2$-handles to $D^*F$ and taking its cylindrical completion. The Gompf handlebody of $D^*F$ has a single $2$-handle attached along a Legendrian that we denote by $\Lambda_0$. Note that by construction $\Lambda_0$ is a sublink of the Weinstein handlebody diagram of $X$. Figure~\ref{fig:normal_ruling} illustrates a graded normal ruling on the front of $\Lambda_0$ which pairs all strands that go through the same $1$-handle. Therefore, $\Lambda_0$ has a graded augmentation by~\cite{Leverson}. Then, by Theorem $1.6$ of~\cite{orsola_thesis}, we know that the symplectic homology of $X$ is non-vanishing.
\end{proof}

\begin{figure}
	\centering
	\includegraphics[width=6cm]{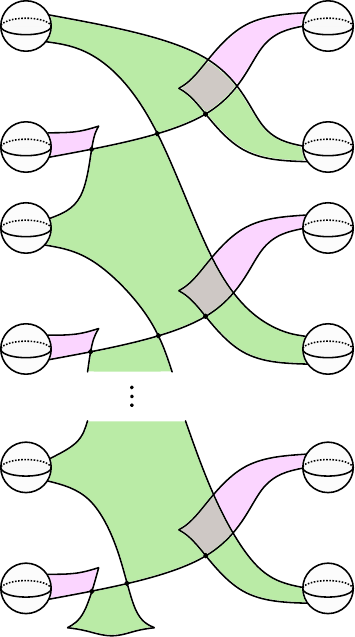}
	\caption{A graded normal ruling on a front projection of the Weinstein handlebody diagram of $D^*F$ where $F$ is an orientable surface of genus $g$.}
	\label{fig:normal_ruling}
\end{figure} 

A flexible Weinstein manifold has vanishing symplectic homology by~\cite{Cielibak} so we obtain the following corollary.

\begin{cor}\label{cor:flexible}
Any Weinstein $4$-manifold $X$ constructed by attaching $1$- or $2$-handles to $T^*F$ for $i=1,\ldots, k$ for any orientable closed surface $F$, is not a flexible Weinstein manifold.
\end{cor}

%% file: parts/examples.tex
\section{Examples from the Algorithm}\label{section examples}

We now apply the algorithm to obtain the Weinstein handlebody diagrams of $\mathcal{W}_{F,c}$ for different sets of co-oriented curves $c$ on surfaces $F$. The majority of the Weinstein $4$-manifolds we consider are Weinstein homotopic to complements of smoothed toric divisors.

\subsection{Smoothing one node}

We begin with the simplest smoothing. Consider the case where we smooth the total toric divisor of any toric 4-manifold in one node. A detailed treatment of this example is found in section 4 of \cite{ACGMMSW1}. By Theorem \ref{thm:toricWein} the complement of the smoothed total toric divisor is Weinstein homotopic to $\mathcal{W}_{T^2, (a,b)}$. By Propositions~\ref{propknodes} or \ref{prop:1handles}, the symplectomorphism type of the completion and Weinstein homotopy type of $\mathcal{W}_{T^2,(a,b)}$ is independent of the choice of slope $(a,b)$. Therefore, it suffices to consider $\mathcal{W}_{(1,0)}$. Following the algorithm of section~\ref{sec:curves}, its Weinstein handle diagram is shown in Figure \ref{1CP2start}. Through a sequence of Reidemeister moves and handle slides we obtain the simplified diagram in Figure \ref{1CP2simple}. {The Weinstein handle diagram of $\mathcal{W}_{T^2, (1,0)}$ shown in Figure~\ref{1CP2simple} shows that this is the self-plumbing of $D^*S^2$.} It is also a handlebody diagram of the affine variety $X=\{ (x,y,z)\in \C^3~ |~ x(xy^2-1)+z^2=0\}$ according to~\cite[Figure 56]{CM}. The homology of $\mathcal{W}_{T^2,(1,0)}$ can be computed from the Weinstein handle diagram or the slopes of the attaching spheres, as in Lemma \ref{slopehoms}. In particular,
\begin{align*}
\pi_1(\mathcal{W}_{T^2,(1,0)}); \Z) &= \Z,\\
H_1(\mathcal{W}_{T^2,(1,0)}); \Z) &= \Z, ~\text{and}\\
H_2(\mathcal{W}_{T^2,(1,0)}); \Z) &= \Z.
\end{align*}

\begin{figure}
	\begin{center}
		\includegraphics[width=7cm]{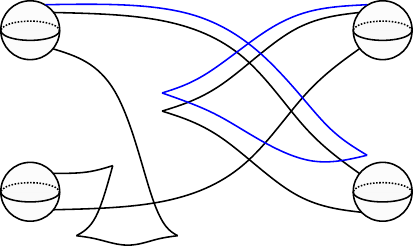}
		\caption{The Weinstein handlebody diagram of the complement of the total toric divisor smoothed in one node of any  toric manifold, that is, $\mathcal{W}_{T^2,(1,0)}$.}
		\label{1CP2start}
	\end{center}
\end{figure}

\begin{figure}
	\begin{center}
		\includegraphics[width=7 cm]{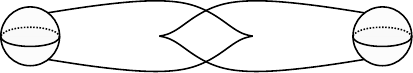}
		\caption{A simplified diagram of the complement of the toric divisor smoothed in one node of any  toric manifold.}
		\label{1CP2simple}
	\end{center}
\end{figure}

\subsection{Smoothing adjacent nodes in a blow up}

\begin{theorem}\label{thm:blowup_example}
  The Weinstein handlebody diagram of the complement of the total toric divisor smoothed in adjacent nodes of a blow up of any toric manifold is a standard max-tb positive Legendrian trefoil.
\end{theorem}

\begin{proof}
As explained in Section \ref{section toric} any corner in a Delzant polytope is $SL(2,\mathbb{Z})$-equivalent to the standard cone (with inward normals $(0,1)$ and $(1,0)$). If we blow up any toric manifold at the fixed point that maps to the vertex of that cone, then we obtain two new fixed points and in the moment map image we obtain two new vertices $V_1$ and $V_2$. The two new corners have inward normals $(0,1),(1,1)$ and $(1,0)$. The size of the blow up does not change the inward normals. 
However, the size of the blow up is relevant in order to obtain $\{V_1,V_2\}$-centered polytope (see Proposition \ref{prop centered blow}). When the size of the blow up is small enough we can take the complement of the total toric divisor of the blown up manifold smoothed in these two  nodes and apply Theorem~\ref{thm:toricWein} to show that this complement is Weinstein homotopic to $\mathcal{W}_{T^2,\{(1,0), (0,-1)\}}$. We now use Theorem \ref{thm:cotangentdiagram} and apply our algorithm to obtain the Weinstein handlebody diagram shown in Figure \ref{2CP2bstart}. The first step is to isotope the curves with slopes $(1,0)$ and $(0,-1)$ in $T^2$ as shown in Figure~\ref{fig:blowup_torus} so that they lie in the region $A\cup B\cup C$ of Figure~\ref{torussquare}. Next, we map this annulus to the front projection of $J^1(S^1)$ and the corresponding front projections of the Legendrian curves $\Lambda_{(1,0)}$ and $\Lambda_{(0,-1)}$ are shown in Figure~\ref{blowupjet}. We satellite the $J^1(S^1)$ picture as described in Step $3$ of the algorithm in Section \ref{sec:curves}, and obtain the Weinstein handlebody shown in Figure~\ref{2CP2bstart}. After a sequence of Legendrian isotopies, handle slides, and cancellations shown in Figure~\ref{blowupkcalc}, we obtain the right handed trefoil as a Weinstein handle diagram of the complement of the smoothed total toric divisor.
\end{proof}

Note that the simplified Weinstein handlebody diagram of $\mathcal{W}_{T^2,\{(1,0), (0,-1)\}}$ coincides the simplified Weinstein handle diagram found to represent the affine variety $\{(x,y,z)\in \C^3~|~ xyz+x+z=1\}$ in~\cite[Section 4.1, Figure 36]{CM}. The homology of $\mathcal{W}_{T^2,\{(1,0), (0,-1)\}}$ can be computed from the Weinstein handle diagram. In particular,
\begin{align*}
\pi_1(\mathcal{W}_{T^2,\{(1,0), (0,-1)\}}; \Z) &= 0, \\
H_1(\mathcal{W}_{T^2,\{(1,0), (0,-1)\}}; \Z) &= 0,~\text{and} \\
H_2(\mathcal{W}_{T^2,\{(1,0), (0,-1)\}}; \Z) &= \Z \text{ with intersection form } [0].
\end{align*}

\begin{figure}
	\begin{center}
		\includegraphics[width=7 cm]{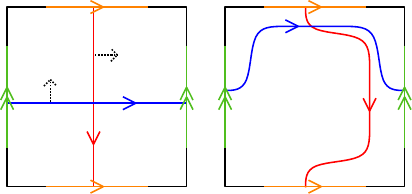}
		\caption{The curves with slopes $(1,0)$ and $(0,-1)$ in $T^2$ before and after smooth isotopy.}
		\label{fig:blowup_torus}
	\end{center}
\end{figure}

\begin{figure}
	\begin{center}
		\includegraphics[width=10cm]{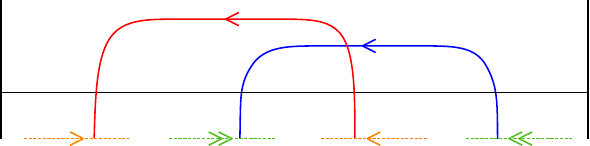}
		\caption{The Legendrian curves $\Lambda_{(1,0)}$ and $\Lambda_{(0,-1)}$ in $J^1(S^1)$.}
		\label{blowupjet}
	\end{center}
\end{figure}

\begin{figure}
	\begin{center}
		\includegraphics[width=7 cm]{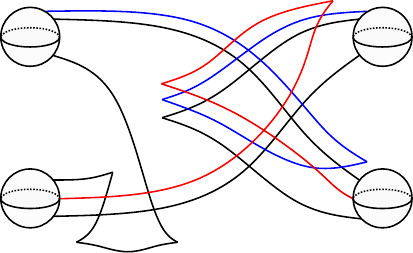}
		\caption{The Weinstein handlebody diagram of the complement of the total toric divisor smoothed in the two nodes that correspond to a blow up of any toric manifold, that is $\mathcal{W}_{T^2,\{(1,0),(0,-1)\}}$.}
		\label{2CP2bstart}
	\end{center}
\end{figure}

\begin{figure}
	\begin{center}
		\includegraphics[width=13 cm]{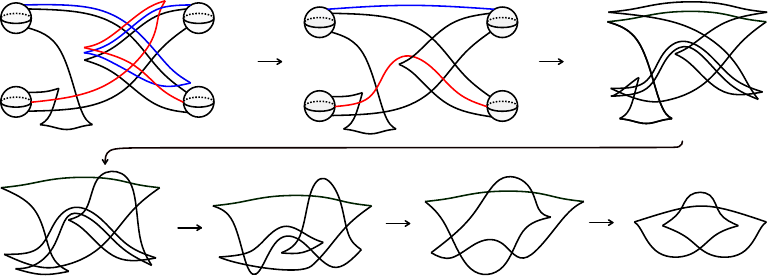}
		\caption{A series of Reidemeister moves, handle slides, and handle cancellations simplifying the Weinstein handle diagram of the complement of the total toric divisor of smoothed in the two nodes that correspond to a blow up of any toric manifold.}
		\label{blowupkcalc}
	\end{center}
\end{figure}

\subsection{Smoothing nodes in $\cpone\times \cpone$}

\begin{theorem}\label{thm:cp1_example}
  The complement of the total toric divisor in $(\cpone\times\cpone,\omega_{a,a})$ smoothed in opposite nodes is Weinstein homotopic to the cyclic plumbing of two disk cotangent bundles of spheres, see Figure \ref{CP1xCP1simple}.
\end{theorem}

\begin{proof}
Consider $(\cpone \times \cpone, \omega_{a,b})$ with the toric action described in Example \ref{example cp1xcp1}. As explained in Example \ref{example rectangle}
we need the monotonicity condition $a=b$ in order to obtain a rectangle that is centered with respect to opposite nodes.
Using the $SL(2,\mathbb{Z})$-transformation $\begin{pmatrix}
   1 & 1 \\
    0 & 1
  \end{pmatrix}$, the moment polytope of $(\cpone \times \cpone,\omega_{a,a})$ is mapped to the polytope in Figure \ref{figure4}. 
Let us smooth the total toric divisor in the singular nodes that map under the moment map to the opposite vertices $V_1$ and $V_2$ depicted in Figure \ref{figure4}. The slopes of the vertices are $s(V_1)=(1,0)$ and $s(V_2)=(-1,0)$, being the difference of the two adjacent inward normals. Thus, by Theorem \ref{thm:toricWein}, the complement of the smoothed divisor is Weinstein homotopic to $\mathcal{W}_{T^2,\{(1,0),(-1,0)\}}$. We now use Theorem \ref{thm:cotangentdiagram} and apply our algorithm. We begin with the curves of slope $(1,0)$ and $(-1,0)$ in $T^2$ shown in Figure \ref{CP1xCP1leg} and isotoped so that we can then draw the corresponding Legendrians $\Lambda_{(1,0)}$ and $\Lambda_{(-1,0)}$ in $J^1(S^1)$ which are shown in Figure~\ref{CP1xCP1jet}. Them by satelling the $J^1(S^1)$ picture around as described in step $3$ of the algorithm, we obtain the Weinstein handlebody diagram shown in Figure \ref{CP1xCP1start}. We perform a series of Reidemeister moves, handle slides and, handle cancellations, as shown in Figure \ref{CP1xCP1moves}, and simplify the Weinstein handle diagram to Figure \ref{CP1xCP1simple}. The final diagram shows that the manifold we obtain by taking the complement of this total toric divisor is the cyclic plumbing of two disk cotangent bundles of spheres. The core Lagrangian spheres are formed from the unique Lagrangian disk filling of each Legendrian unknot attaching sphere together with the core of the corresponding $2$-handle. That this is the cyclic plumbing can be seen from the linking of the two attaching spheres of the $2$-handles, together with the position of the $1$-handle which is needed when the plumbing graph contains a cycle.
\end{proof}

\begin{figure}
	\begin{overpic}[width=12 cm]{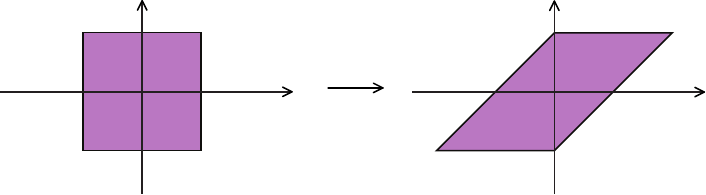}
		\put(45, 20) {$\begin{pmatrix}
				1 & 1 \\
				0 & 1
			\end{pmatrix}$}
		\put(74, 24){$V_2$}
		\put(80, 4){$V_1$} 
	\end{overpic}
	\caption{$SL(2,\mathbb{Z})$-transformation of the moment polytope of $(\cpone \times \cpone, \omega_{a,a}).$}
	\label{figure4}
\end{figure}

\begin{figure}
	\begin{center}
		\includegraphics[width=7 cm]{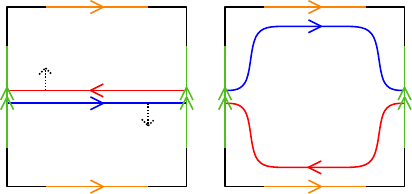}
		\caption{The curves of slope $(1,0)$ and $(-1,0)$ in $T^2$ before and after a smooth isotopy that corresponds to a Legendrian isotopy of their co-normal lifts.}
		\label{CP1xCP1leg}
	\end{center}
\end{figure}

\begin{figure}
	\begin{center}
		\includegraphics[width=10cm]{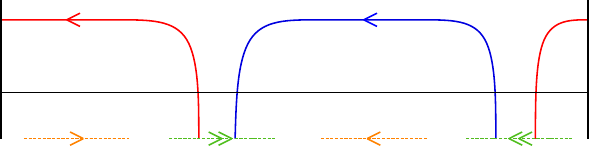}
		\caption{The correponding curves of $\Lambda_{(1,0)}$ and $\Lambda_{(-1,0)}$ in $J^1(S^1)$.}
		\label{CP1xCP1jet}
	\end{center}
\end{figure}

\begin{figure}
	\begin{center}
		\includegraphics[width=7cm]{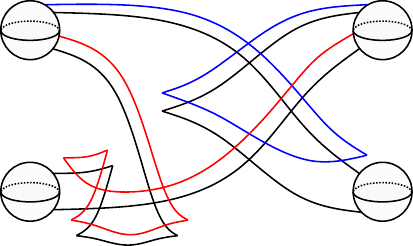}
		\caption{The Weinstein handle diagram of the complement of the total toric divisor of $(\C\P^1\times\C\P^1,\omega_{a,a})$ smoothed in two opposite nodes, $\mathcal{W}_{T^2,\{(1,0),(-1,0)\}}$.}
		\label{CP1xCP1start}
	\end{center}
\end{figure}

\begin{figure}
	\begin{center}
		\includegraphics[width=13 cm]{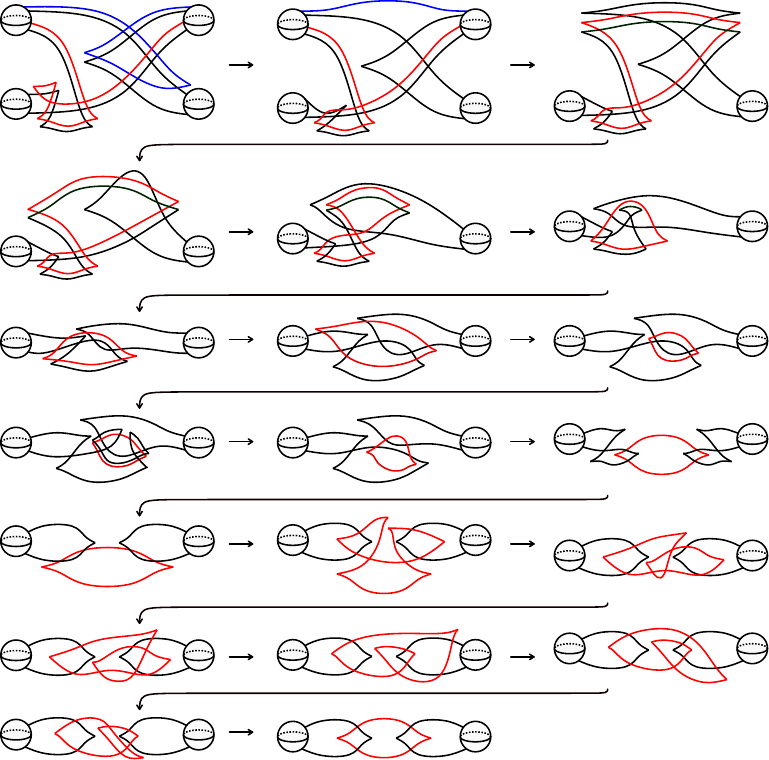}
		\caption{A series of Reidemeister moves and handle slides simplifying the Weinstein handle diagram of the complement of the total toric divisor of $(\C\P^1\times\C\P^1,\omega_{a,a})$ smoothed in two opposite nodes.}
		\label{CP1xCP1moves}
	\end{center}
\end{figure}

\begin{figure}
	\begin{center}
		\includegraphics[width=7cm]{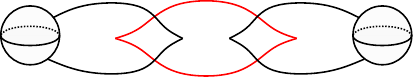}
		\caption{A simplified diagram of the complement of the total toric divisor of $(\C\P^1\times \C\P^1,\omega_{a,a})$ smoothed in two opposite nodes.}
		\label{CP1xCP1simple}
	\end{center}
\end{figure}

The homology of $\mathcal{W}_{T^2, \{(1,0), (-1,0)\}}$ which is Weinstein homotopic to the complement of the toric divisor in $\cpone \times \cpone$ smoothed in two opposite nodes 
can be computed from the Weinstein handle diagram or the slopes of the attaching spheres, as in Lemma \ref{slopehoms}. In particular,
\begin{align*}
\pi_1(\mathcal{W}_{T^2, \{(1,0), (-1,0)\}}); \Z) &=\Z, \\
H_1(\mathcal{W}_{T^2, \{(1,0), (-1,0)\}}); \Z) &=\Z, ~\text{and}\\
H_2(\mathcal{W}_{T^2, \{(1,0), (-1,0)\}}); \Z) &=\Z\oplus \Z.
\end{align*}

It is also possible to compute the Chekanov Eliashberg dga of the attaching Legendrian link in the Legendrian handle. Lekili and Etg\"u provide mirror symmetry results for this manifold. Specifically, they consider general plumbings of genus $g$ surfaces, of which this manifold is a particular case~\cite{EtguLekili}.

One application of having explicit Weinstein handlebody diagrams of a Weinstein $4$-manifold $X$ is easily finding closed exact Lagrangian surfaces in $X$ by considering the union of the Lagrangian core of a critical handle and the exact Lagrangian filling of the attaching Legendrian. Since the complement of the total toric divisor of $(\cpone \times \cpone,\omega_{a,a})$ smoothed in two opposite nodes is Weinstein homotopic $\mathcal{W}_{T^2, \{(1,0), (-1,0)\}}$, we see it contains two Lagrangian spheres $L_1$ and $L_2$ in different homology classes in $\mathcal{W}_{T^2, \{(1,0), (-1,0)\}}$ which therefore cannot be Hamiltonian isotopic in $\mathcal{W}_{T^2, \{(1,0), (-1,0)\}}$. By contrast, a result of Hind~\cite{Hind} shows that there is only one exact Lagrangian sphere up to Hamiltonian isotopy in $\cpone \times \cpone$. Thus, we see that the Hamiltonian isotopy taking $L_1$ to $L_2$ in $\cpone \times \cpone$ must pass through the smoothed total toric divisor.

\begin{remark} \label{remark two nodes}
According to Example \ref{example rectangle}, smoothing the total toric divisor of $(\cpone \times \cpone,\omega_{a,a})$ in any two opposite nodes produces symplectomorphic Weinstein manifolds. Similarly, smoothing the total toric divisor of $(\cpone \times \cpone,\omega_{a,b}),$ $b/2<a<2b,$ in any two adjacent nodes also produces symplectomorphic Weinstein manifolds. We now compare these two Weinstein manifolds. The slopes of the vertices of the moment polytope of $\cpone\times \cpone$ given on the right in Figure \ref{figure4} are $(1,0)$, $(-1,0)$, $(1,-2)$, and $(-1,2).$ By Theorem~\ref{thm:toricWein}, we are comparing the Weinstein manifolds $\mathcal{W}_{T^2, \{(1,0),(-1,0)\}}$ and $\mathcal{W}_{T^2, \{(1,0), (1,-2)\}}$. 
It is enough to compare their first homology. According to Lemma \ref{slopehoms} we obtain $H_1(
\mathcal{W}_{T^2, \{(1,0),(-1,0)\}} ; \Z) =\Z$ and $H_1(\mathcal{W}_{T^2, \{(1,0), (1,-2)\}} ; \Z) =\Z/2\Z.$ Thus, these Weinstein manifolds are not even homeomorphic. 
\end{remark}

\begin{theorem}\label{thm:cp1_example2}
  The Weinstein handlebody diagram of the complement of the total toric divisor of $(\cpone\times\cpone,\omega_{a,a})$ smoothed in 3 nodes is shown in Figure \ref{P1P13simp}.
\end{theorem}

\begin{proof}
Let us now smooth the total toric divisor of $(\cpone\times \cpone,\omega_{a,a})$ in three nodes. According to Example \ref{example rectangle}, we can choose any three nodes and obtain symplectomorphic Weinsten manifolds. By Theorem \ref{thm:toricWein}, the complement of the smoothed divisor is Weinstein homotopic to $\mathcal{W}_{T^2, c_1}$ where $c_1=\{(1,0), (-1,0), (1,-2)\}$. We now use Theorem \ref{thm:cotangentdiagram} and apply our algorithm. The resulting Weinstein handle diagram is shown in Figure \ref{P1P13}. Through a series of Reidemeister moves, handle slides and handle cancellations, we can simplify the diagram as shown in Figure \ref{P1P13simp}. 
\end{proof}

The homology of $\mathcal{W}_{T^2, c_1}$, where  $c_1=\{(1,0), (-1,0), (1,-2)\}$ can be computed from the Weinstein handle diagram or the slopes of the attaching spheres, as in Lemma \ref{slopehoms}. In particular,
\begin{align*}
\pi_1(\mathcal{W}_{T^2, c_1}; \Z) &=\Z/2\Z, \\
H_1(\mathcal{W}_{T^2, c_1}; \Z) &=\Z/2\Z,~\text{and} \\
H_2(\mathcal{W}_{T^2, c_1}; \Z) &=\Z\oplus\Z.
\end{align*}

\begin{figure}
	\begin{center}
		\includegraphics[width=7cm]{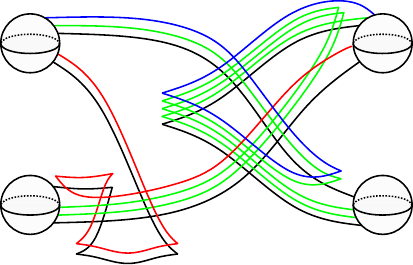}
		\caption{A Weinstein handlebody diagram of the complement of the total toric divisor of $(\C\P^1\times\C\P^1,\omega_{a,a})$ smoothed in three nodes,$\mathcal{W}_{T^2,(1,0),(-1,0),(1,-2)}$.}
		\label{P1P13}
	\end{center}
\end{figure}

\begin{figure}
	\begin{center}
		\includegraphics[width=7cm]{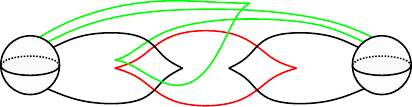}
		\caption{A simplified Weinstein handlebody diagram of the complement of the total toric divisor of  $(\C\P^1\times\C\P^1,\omega_{a,a})$ smoothed in three nodes.}
		\label{P1P13simp}
	\end{center}
\end{figure}

\begin{theorem}\label{thm:cp1_example3}
 The Weinstein handlebody diagram of the complement of the total toric divisor of  $(\cpone\times\cpone,\omega_{a,a})$ smoothed in all 4 nodes is shown in Figure \ref{P1P14}.
\end{theorem}

\begin{proof}
Smoothing all four nodes of the total toric divisor of $(\C\P^1\times\C\P^1,\omega_{a,a})$, the smoothed divisor becomes the torus $T^2$ and the complement of the smoothed divisor is Weinstein homotopic to $\mathcal{W}_{T^2, c_2}$ where $c_2=\{(1,0), (-1,0), (1,-2), (-1,2)\}$ by Theorem \ref{thm:toricWein}. We now use Theorem \ref{thm:cotangentdiagram} and apply our algorithm. The resulting Weinstein handlebody diagram is shown in Figure \ref{P1P14}. 
\end{proof}

The homology of $\mathcal{W}_{T^2, c_2}$ where $c_2=\{(1,0), (-1,0), (1,-2), (-1,2)\}$ can be computed from the Weinstein handle diagram or the slopes of the attaching spheres, as in Lemma \ref{slopehoms}. In particular,
\begin{align*}
\pi_1(\mathcal{W}_{T^2, c_2}; \Z) &=\Z/2\Z,\\
H_1(\mathcal{W}_{T^2, c_2}; \Z) &=\Z/2\Z, ~\text{and}\\
H_2(\mathcal{W}_{T^2, c_2}; \Z) &=\Z\oplus\Z\oplus\Z.
\end{align*}

\begin{figure}
	\begin{center}
		\includegraphics[width=7 cm]{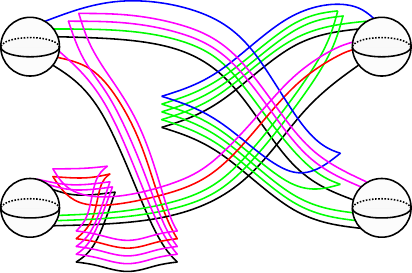}
		\caption{A diagram of the complement of the total toric divisor of monotone $\C\P^1\times\C\P^1$ smoothed in four nodes, $\mathcal{W}_{T^2,\{(1,0),(-1,0),(1,-2),(-1,2)\}}$.}
		\label{P1P14}
	\end{center}
\end{figure}

\subsection{Weinstein homotopic manifolds}

We now use our algorithm to find a few examples of unexpected Weinstein homotopies. For example we can find two manifolds $\mathcal{W}_{T^2,c}$ and $\mathcal{W}_{T^2,c'}$ such that $c$ can be realized as the slopes of vertices $\{V_1,V_2\}$ in a $\{V_1,V_2\}$-centered toric manifold, but $c'$ cannot.

\begin{example} \label{remark counterexample}
Recall from Proposition~\ref{prop:notexact}, that there is no $\{V_1,V_2\}$-centered toric manifold where $s(V_1)=s(V_2)=(1,0)$. However, we can still apply our algorithm to find a Weinstein handlebody for $\mathcal{W}_{(1,0),(1,0)}$. Surprisingly, this Weinstein domain is Weinstein homotopic to $\mathcal{W}_{T^2,\{(1,0), (-1,0)\}}$, which can be realized as a complement of a smoothed total toric divisor of $\cpone \times \cpone$ as we saw in the proof of Theorem~\ref{thm:cp1_example}.

The Weinstein handle diagram obtained from our algorithm of $\mathcal{W}_{T^2,\{(1,0), (1,0)\}}$ is shown on the top left of Figure \ref{1010}. After the series of Reidemeister moves and handle slides shown in Figure~\ref{1010}, it can be simplified to the same diagram as the one we obtained in Figure~\ref{CP1xCP1simple} after simplifying the diagram for $\mathcal{W}_{T^2,\{(1,0), (-1,0)\}}$ in Figure~\ref{CP1xCP1moves}. Thus, these Weinstein manifolds are Weinstein homotopic.

\begin{figure}
	\centering
	\includegraphics[width=16cm]{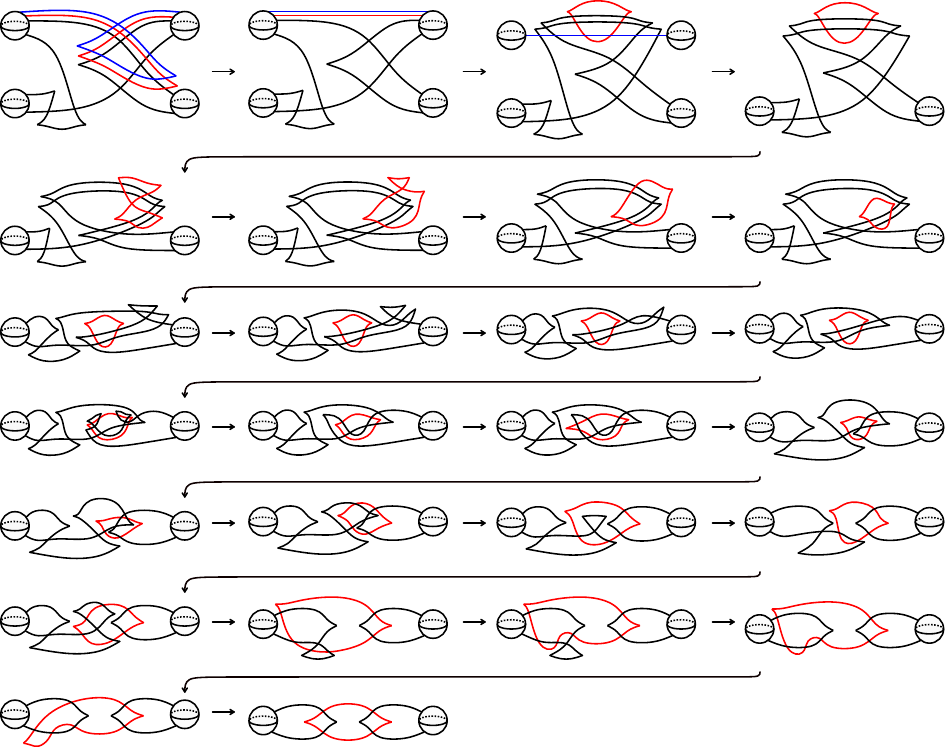}
	\caption{A series of Reidemeister moves and handle slides simplifying the Weinstein handle diagram of $\mathcal{W}_{T^2,\{(1,0), (1,0)\}}.$ }
	\label{1010}
\end{figure}    

In fact, if we consider these Weinstein manifolds as the complement of almost toric divisors in almost toric manifolds, their equivalence can be explained through a mutation.
The Weinstein manifold $\mathcal{W}_{T^2,\{(1,0), (1,0)\}}$ can be realized as the complement of an almost toric divisor in $(\mathbb{CP}^1\times\mathbb{CP}^1,\omega_{a,a})$ with a particular almost toric structure.
To find this almost toric structure, we start with the standard toric moment polytope for $\mathbb{CP}^1\times\mathbb{CP}^1$ shown on the right of Figure~\ref{figure4}. We perform a nodal trade at each of the vertices $V_1$ and $V_2$ to get the almost toric base diagram shown on the left of Figure~\ref{ATfigure-CP1xCP1-parallel_slopes}. Next, perform a mutation along the $(0,-1)$-eigenray. This maps the marked point and corresponding node lying on the $(0,-1)$-eigenray emanating from $V_2$ to a node and marked point on the $(0,1)$ eigenray emanating from $V_1$, so that both nodes lie on the same eigenray.  Nothing happens to the right half of the polytope (shaded in the picture) while the left half is transformed by applying the clockwise monodromy through the $(0,-1)$-eigenray. The clockwise monodromy matrix is $\begin{pmatrix}
   1 & 0 \\
   -1 & 1
  \end{pmatrix}$ and it maps the $(-1,0)$-edge of the polytope  to an $(-1,1)$-edge, and the $(-1,-1)$-edge to the $(-1,0)$-edge
so that we get the almost toric base polytope on the right of Figure~\ref{ATfigure-CP1xCP1-parallel_slopes}. Now the polytope is centered with the marked points on the same $(0,1)$-eigenray. The complement of the almost toric divisor is symplectomorphic to $\mathcal{W}_{T^2, \{(1,0), (1,0)\}}$ while the complement of the almost toric divisor for the polytope on the left of Figure~\ref{ATfigure-CP1xCP1-parallel_slopes} is symplectomorphic to $\mathcal{W}_{T^2, \{(1,0), (-1,0)\}}$. 

  \end{example}
\begin{figure}
	\centering
	\includegraphics[width=4cm]{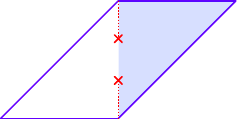}
	\hspace{1.5cm}
	\includegraphics[width=4cm]{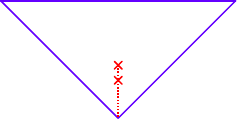}
	\caption{The two Weinstein manifolds $\mathcal{W}_{T^2,\{(1,0), (-1,0)\}}$ and $\mathcal{W}_{T^2,\{(1,0), (1,0)\}}$ are the complements of two almost toric divisors in $\mathbb{CP}^1\times\mathbb{CP}^1$ that are related by almost toric mutation.}
	\label{ATfigure-CP1xCP1-parallel_slopes}
\end{figure}

\begin{figure}
	\begin{center}
		\includegraphics[width=7 cm]{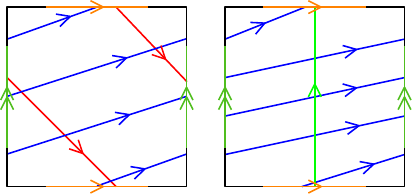}
		\caption{Curves with slopes $(1,-1)$, and $(3,1)$ on $T^2$ are shown on the left. Curves with slopes $(0,1),$ and $(4,1)$ on $T^2$ are shown on the right.}
		\label{fig:sl2q_slopes}
	\end{center}
\end{figure}

\begin{theorem} \label{remark:sl2q} The existence of an $SL(2, \Z)$ transformation between sets of slopes $c$ and $c'$ is a sufficient condition to guarantee that $\mathcal{W}_{T^2, c}$ and $\mathcal{W}_{T^2, c'}$ are Weinstein homotopic, but it is not necessary (even when $\mathcal{W}_{T^2,c}$ and $\mathcal{W}_{T^2,c'}$ are both realizable as complements of partially smoothed total toric divisors).
\end{theorem}

\begin{proof}
Proposition~\ref{prop:1handles} shows that an $SL(2, \Z)$ transformation between sets of slopes is a sufficient condition for Weinstein homotopy. To show that this condition is not necessary we provide the following example.
Consider the set of slopes $\{(1,-1),(3,1)\}$ and $\{(4,1),(0,1)\}$ {shown in Figure~\ref{fig:sl2q_slopes} on $T^2$}, which are related by an $SL(2,\mathbb{Q})$ transformation and not an $SL(2,\mathbb{Z})$ transformation (see Remark \ref{remark counterexample0}). We want to consider the manifolds  $\mathcal{W}_{T^2,\{(1,-1),(3,1)\}}$ and  $\mathcal{W}_{T^2,\{(0,1),(4,1)\}}$.
Using Theorem \ref{thm:toricWein}, both Weinstein manifolds can be obtained by our procedure of smoothing the nodes of a total toric divisor, since there are $\{V_1,V_2\}$-centered Delzant polytopes with these slopes
(see Figure \ref{two possible}). However,
we cannot compare them as in Proposition~\ref{propknodes}  since the sets of slopes are not related by an $SL(2,\mathbb{Z})$ transformation. We now apply the algorithm to both of these sets of slopes and obtain the Weinstein handlebody diagrams pictured in the top left of Figures \ref{fig:sl2q_1} and \ref{fig:sl2q_2}. After a series of simplifications via Reidemeister moves, handle slides and handle cancellations to both handlebody diagrams, we obtain the same simplified diagram, seen in the bottom right of Figures \ref{fig:sl2q_1} and \ref{fig:sl2q_2}. Thus these manifolds are Weinstein homotopic.
\end{proof}


\begin{figure}
	\begin{center}
		\includegraphics[width=12 cm]{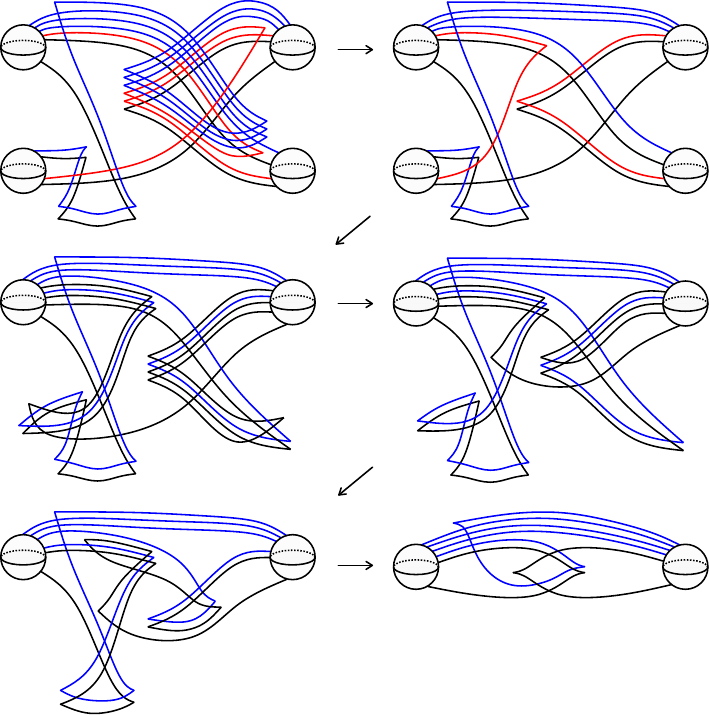}
		\caption{The algorithm generated Weinstein diagram of $\mathcal{W}_{T^2,\{(1,-1),(3,1)\}}$ and a series of simplifications.}
		\label{fig:sl2q_1}
	\end{center}
\end{figure}

\begin{figure}
	\begin{center}
		\includegraphics[width=12 cm]{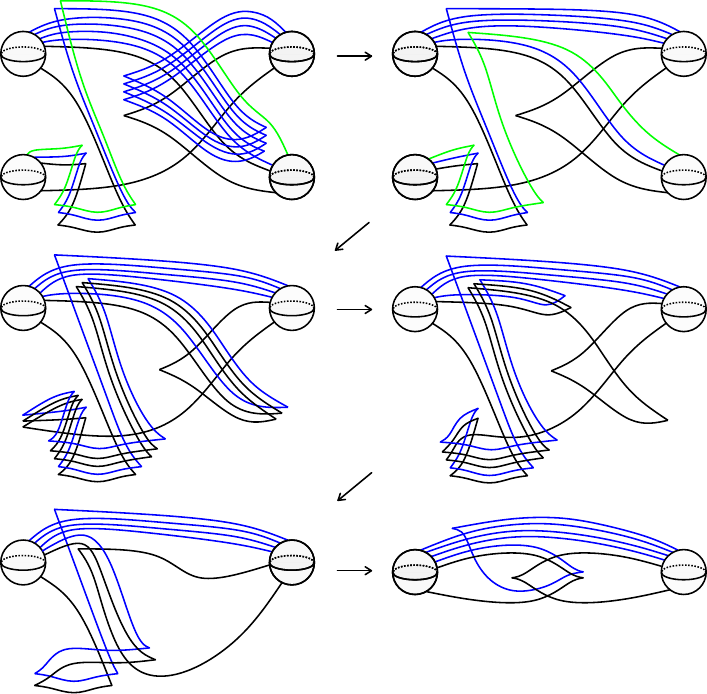}
		\caption{The algorithm generated Weinstein diagram of $\mathcal{W}_{T^2,\{(0,1),(4,1)\}}$ and a series of simplifications.}
		\label{fig:sl2q_2}
	\end{center}
\end{figure}

\subsection{Smoothing multiple nodes in $\cptwo$}

Consider the toric symplectic 4-manifold $(\mathbb{CP}^2, \omega_{FS})$. The standard toric structure has moment map image a right triangle as shown on the left of Figure~\ref{figure5}. The total toric divisor consists of three complex projective lines intersecting at three nodes. Smoothing one of these nodes joins two lines into a smooth conic which then meets the remaining line at two nodes. Smoothing two of the nodes joins all three reducible components into a single component of degree three (a cubic) which has a single remaining node. The nodal cubic in $\cptwo$ is known to have a unique symplectic isotopy class~\cite{Shev}, but it is not obvious how to present the Weinstein handle structure on the complement without our techniques. Similarly, smoothing all three nodes results in a single smooth cubic (again there is a unique symplectic isotopy class of such curves~\cite{Sikorav}). By embedding $\cptwo$ into $\C\P^N$ such that the intersection with $\C\P^{N-1}$ is the algebraic curve, and thus realizing the complement as an affine variety in $\C^N$ we see that the complement of these algebraic curves must have a Stein structure, and thus a Weinstein structure. However there was not previously a strategy to present a handle diagram encoding this Weinstein manifold.

We can smooth all singularities of the total toric divisor of $\mathbb{CP}^2$ and obtain a complement with a Weinstein structure because $\cptwo$ is a monotone toric manifold. We modify the toric structure from the standard one using the $SL(2,\mathbb{Z})$-transformation $\begin{pmatrix}
   0 & 1 \\
    -1 & 1
  \end{pmatrix}$ so that the moment polytope of $\mathbb{CP}^2$ is mapped to the polytope on the right in Figure \ref{figure5}. 
The difference of the three adjacent inward normal vectors, i.e. the slopes of the vertices, are $(0,-1)$, $(3,-1)$ and $(-3,2)$.
By Theorem \ref{thm:toricWein}, the complement of the total toric divisor smoothed in any subset of these three singularities is $\mathcal{W}_{T^2,c}$ where $c\subseteq\{(0,-1),(3,-1), (-3,2)\}$. {Since we already discussed the case of a toric manifold with a single node smoothed, we begin with the case where two nodes are smoothed and one node remains, so the divisor is an irreducible cubic curve with a single node.}

\begin{figure}
	\begin{overpic}[width=12 cm]{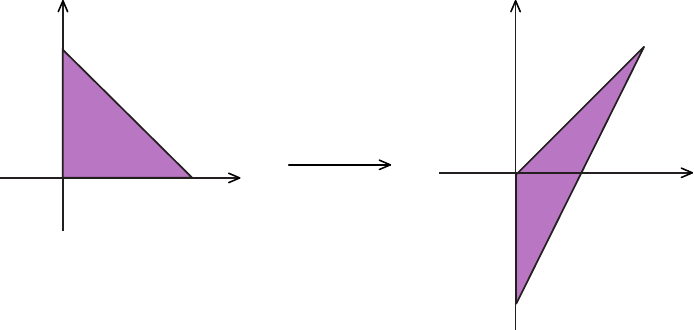}
		\put(43, 29){$\begin{pmatrix}
				0 & 1 \\
				-1 & 1
			\end{pmatrix}$}
	\end{overpic}
	\caption{An $SL(2,\mathbb{Z})$-transformation of the moment polytope of $(\mathbb{CP}^2, \omega_{FS}).$}
	\label{figure5}
\end{figure}

\begin{figure}
	\begin{center}
		\includegraphics[width=7 cm]{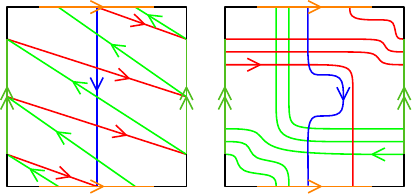}
		\caption{The curves $(0,-1)$, $(3,-1)$ and $(-3,2)$ in $T^2$ and their resulting isotopies which are push-offs of the boundary of the square in the positive Reeb direction.}
		\label{CP2leg}
	\end{center}
\end{figure}

\begin{figure}
	\begin{center}
		\includegraphics[width=10 cm]{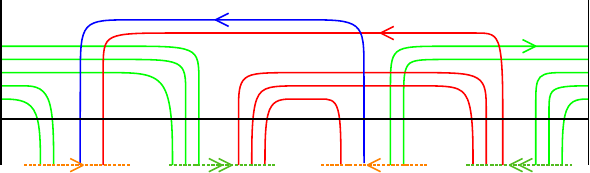}
		\caption{The curves in $J^1(S^1)$ which are identified with the curves $(0,-1)$, $(3,-1)$ and $(-3,2)$ on $T^2$, as in Figure \ref{CP2leg}.}
		\label{CP2jet}
	\end{center}
\end{figure}

\begin{theorem}\label{thm:nodalcubic}
  The Weinstein handlebody diagram of the complement of a nodal cubic is shown in Figure \ref{CP22final}.
\end{theorem}

\begin{proof}
Consider the case where we smooth the total toric divisor in two of these singularities. By Theorem \ref{thm:toricWein} the complement of the smoothed total toric divisor is Weinstein homotopic to $\mathcal{W}_{T^2, \{(0,-1), (-3,2)\}}$. We now use Theorem \ref{thm:cotangentdiagram} and apply our algorithm. The Weinstein handle diagram of $\mathcal{W}_{T^2, \{(0,-1), (-3,2)\}}$ is shown in Figure \ref{CP22start}. Through a sequence of Reidemeister moves, handle and handles lides we obtain the simplified diagram in Figure \ref{CP22final}. 
\end{proof}

The homology of $\mathcal{W}_{T^2, \{(0,-1), (-3,2)\}}$, can be computed from the Weinstein handle diagram or the slopes of the attaching spheres, as in Lemma \ref{slopehoms}. In particular,
\begin{align*}
\pi_1(\mathcal{W}_{T^2, \{(0,-1), (-3,2)\}}); \Z) &=\Z/3\Z,\\
H_1(\mathcal{W}_{T^2, \{(0,-1), (-3,2)\}}); \Z) &=\Z/3\Z, ~\text{and}\\
H_2(\mathcal{W}_{T^2, \{(0,-1), (-3,2)\}}); \Z) &=\Z.
\end{align*}

\begin{figure}
	\begin{center}
		\includegraphics[width=7cm]{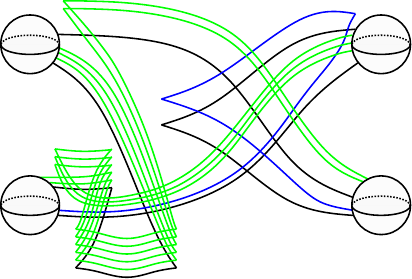}
		\caption{The Weinstein handlebody diagram of the complement of the total toric divisor of $\cptwo$ smoothed in two nodes.}
		\label{CP22start}
	\end{center}
\end{figure}

\begin{figure}
	\begin{center}
		\includegraphics[width=7 cm]{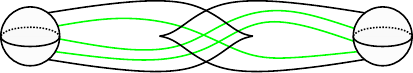}
		\caption{A simplified diagram of the complement of the total toric divisor of $\cptwo$ smoothed in two nodes, $\mathcal{W}_{T^2, \{(0,-1), (-3,2)\}}$.}
		\label{CP22final}
	\end{center}
\end{figure}

\begin{theorem}\label{thm:cubic}
  The Weinstein handlebody diagram of the complement of a smooth cubic in $\cptwo$ is shown in Figure~\ref{CP23final}.
\end{theorem}

\begin{proof}
Consider the case where we smooth the total toric divisor of $\cptwo$ in all three singularities. By Theorem \ref{thm:toricWein} the complement of the smoothed total toric divisor is Weinstein homotopic to $\mathcal{W}_{T^2, \{(0,-1), (-3,2), (3,-1)\}}$. We now use Theorem \ref{thm:cotangentdiagram} and apply our algorithm. Figure~\ref{CP2leg} shows the three curves $(0,-1), (-3,2),$ and $(3,-1)$ in $T^2$ isotoped so that they are push-offs of the boundary of the square in the positive Reeb direction. Figure~\ref{CP2jet} shows the identification of these three curves with curves in $J^1(S^1)$. Finally, the Weinstein handlebody diagram of $\mathcal{W}_{T^2, \{(0,-1), (-3,2), (3,-1)\}}$ is shown in Figure \ref{CP23start}. Through a sequence of Reidemeister moves and handle slides we obtain the simplified diagram in Figure \ref{CP23final}. 
\end{proof}

The fundamental group and homology of $\mathcal{W}_{T^2, \{(0,-1), (-3,2), (3,-1)\}}$, can be computed from the Weinstein handle diagram or the slopes of the attaching spheres, as in Lemma \ref{slopehoms}. In particular,
\begin{align*}
\pi_1(\mathcal{W}_{T^2, \{(0,-1), (-3,2), (3,-1)\}}; \Z) &=\Z/3\Z,\\
H_1(\mathcal{W}_{T^2, \{(0,-1), (-3,2), (3,-1)\}}; \Z) &=\Z/3\Z,~\text{and} \\
H_2(\mathcal{W}_{T^2, \{(0,-1), (-3,2), (3,-1)\}}; \Z) &=\Z\oplus\Z.\\
\end{align*}

\begin{figure}
	\begin{center}
		\includegraphics[width=7 cm]{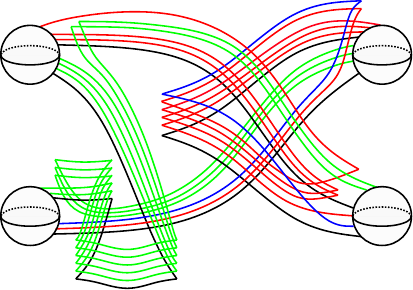}
		\caption{The Weinstein handlebody diagram of the complement of the total toric divisor of $\cptwo$ smoothed in all three nodes, $\mathcal{W}_{T^2, \{(0,-1), (-3,2), (3,-1)\}}$.}
		\label{CP23start}
	\end{center}
\end{figure}

\begin{figure}
	\begin{center}
		\includegraphics[width=7 cm]{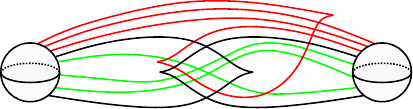}
		\caption{A simplified diagram of the complement of the total toric divisor of $\cptwo$ smoothed in all three nodes.}
		\label{CP23final}
	\end{center}
\end{figure}

\begin{remark} The complement of a smooth cubic in $\cptwo$, $\mathcal{W}_{T^2, \{(0,-1), (-3,2), (3,-1)\}}$ is a complex affine variety. The Weinstein handlebody diagram shown in Figure~\ref{CP23final} provides a new perspective on the symplectic topology of this complex affine variety. There are exact Lagrangian spheres, and tori in $\mathcal{W}_{T^2, \{(0,-1), (-3,2), (3,-1)\}}$ explicitly using decomposable $0$- and $1$-handle moves as defined in \cite{EtnNg18}. {From the Weinstein handlebody diagram we can also compute the intersection form. Let $A$ denote the red knot, $B$, the green knot and $C$ the black knot in Figure~\ref{CP23final}. Since $A$ and $B$ both pass through the $1$-handle $3$ times algebraically while $C$ passes through the $1$-handle $0$ times algebraically, $H_2(\mathcal{W}_{T^2, \{(0,-1), (-3,2), (3,-1)\}}; \Z)$ is generated by the union of the Seifert surfaces and the core of the $2$-handles of the links $C$ and $A-B$. The intersection form is then:
$$\begin{pmatrix}
tb(C)-1&lk(C, A)-lk(C,B)\\
lk(C,A)-lk(C,B)&tb(A-B)-1
\end{pmatrix}=\begin{pmatrix}
1-1&0\\
0&2-1
\end{pmatrix}=\begin{pmatrix}
0&0\\
0&1
\end{pmatrix}.$$}
{Note, to compute $tb(A-B)$, we orient $A$ and $B$ with opposite orientations, and use the signed number of crossings (of $A$ with $A$, $B$ with $B$, or $A$ with $B$) as a substitute for the writhe in the combinatorial formula. Equivalently this can be computed as $tb(A)+tb(B)-2lk(A,B)$ where we compute linking number by orienting $A$ and $B$ so that both go through the $1$-handle in the same direction.}
\end{remark}

\subsection{Smoothing nodes in $\cptwo\# 5\overline{\cptwo}$} \label{section:CP2bu5}

Consider $\cptwo \# 5\cptwobar$, whose Delzant polytope is illustrated in Figure~\ref{octagon}. If we want to smooth all eight singularities that map under the moment map to the vertices of the octagon, we first note the rays of the polytope are parallel. By Proposition \ref{prop:notexact}, the complement of the total toric divisor smoothed in all eight nodes is not an exact symplectic manifold and cannot support a Weinstein structure.

\begin{figure}
	\centering
	\includegraphics[width=4cm]{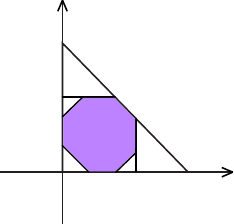}
	\caption{The Delzant polytope of  $\cptwo \# 5\cptwobar$.}
	\label{octagon}
\end{figure}

\begin{figure}
	\centering
	\includegraphics[width=4cm]{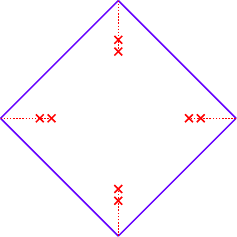}
	\caption{An almost toric base of  $\cptwo \# 5\cptwobar$.}
	\label{BU5}
\end{figure}
However, we can resolve this issue if we look at $\cptwo \# 5 \cptwobar$ as an almost toric manifold, and correspondingly deform the divisor. $\cptwo \# 5 \cptwobar$ admits an almost toric structure whose base diagram is shown in Figure~\ref{BU5}. We obtain this almost toric description by starting with the standard monotone toric structure on $(\cpone \times \cpone, \omega_{a,a})$ and performing a monotone blow up so that the chopped off corner cuts the original edges in half. We now have a toric manifold corresponding to $\cptwo \# 2 \cptwobar$. A nodal trade at the new corners gives an almost toric structure with base diagram shown on the left of Figure~\ref{BU2-3}.
  \begin{figure}
  	\centering
  	\includegraphics[width=4cm]{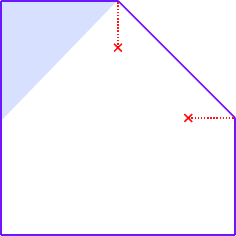}
  	\hspace{1cm}
  	\includegraphics[width=4cm]{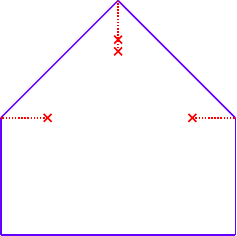}
  	\caption{An almost toric base of  $\cptwo \# 2\cptwobar$ on the left, $\cptwo \# 3\cptwobar$ on the right.}
  	\label{BU2-3}
  \end{figure}
Because we performed the nodal trade, the vertex at the top right of the shaded blue triangle is not a node of the almost toric divisor. Instead, its preimage is a circle just like points in the interior of an edge in the toric setting. Therefore, the preimage of the shaded blue triangle is a $4$-dimensional ball and we can perform a symplectic blow up by replacing this ball with a $\cpone$. This is an (almost-) toric blow up (see~\cite{Symington, LeungSymington} and also~\cite[Section 2.4]{Vianna}). After blowing up, the fiber of the top central corner is again one point and we can perform another nodal trade with the same eigenray (to prove this we do a mutation of the existing node to get a toric Delzant corner and perform the nodal trade as in Section~\ref{section:almosttoric}). After another nodal trade at the top left corner we get the polytope with nodes on the right of Figure~\ref{BU2-3} as the base diagram of an almost toric fibration on $\cptwo \# 3 \cptwobar$. We then perform this procedure of two nodal trades and a blow up two more times at each original corner to get the desired base of the almost toric tructure on $\cptwo \# 5 \cptwobar$. This almost toric structure can also be obtained from the one of Figure 18 $(B_1)$ in~\cite[Section 3.3]{Vianna} after performing a mutation and an $SL(2,\mathbb{Z})$ transformation.

The eigenlines for each singularity intersect in the barycenter of the polytope, thus our almost toric polytope is centered. By Section~\ref{section:almosttoric}, the complement in $\cptwo \# 5\cptwobar$ of the neighborhood of the almost toric divisor is a Weinstein domain $\mathcal{W}_{T^2, c}$ given by attaching 8 Weinstein $2$-handles to $D^* T^2$ along the following Legendrian attaching spheres (repeated slopes are parallel push-offs of each other):
$$\Lambda_{(1,0)},\Lambda_{(1,0)},~\Lambda_{(0,-1)},~\Lambda_{(0,-1)},~\Lambda_{(-1,0)},~\Lambda_{(-1,0)},~\Lambda_{(0,1)},~\Lambda_{(0,1)}.$$

\begin{theorem}\label{thm:octagonexample}
  The Weinstein handlebody diagram for the complement of the almost toric divisor 
  in $\cptwo \# 5\cptwobar$ corresponding to the almost toric fibration with base shown in Figure~\ref{BU5} is shown in Figure \ref{CP25}.
\end{theorem}

\begin{proof}
Figure \ref{CP25torus} shows the curves $c=\{(1,0),(1,0),(0,-1),(0,-1),(-1,0),(-1,0),(0,1),(0,1)\}$ on $T^2$. We can now apply steps (\ref{step:isotopy}) to (\ref{step:satellite}) of our algorithm. This identifies all our curves in $T^2$ with curves in $J^1(S^1)$ as in Figure \ref{CP25jet}. We then satellite the image of the curves in $J^1(S^1)$ onto the image of $S^1$ in the Gompf diagram of $T^*T^2$. The result is Figure~\ref{CP25}. 
\end{proof}

\begin{figure}
	\begin{center}
		\includegraphics[width=7cm]{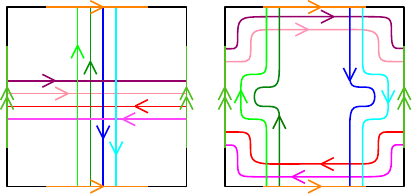}
		\caption{The curves $(1,0),(1,0),(0,-1),(0,-1),(-1,0),(-1,0),(0,1),(0,1)$ on $T^2$, and their resulting isotopies which are push-offs of the boundary of the square in the positive Reeb direction.}
		\label{CP25torus}
	\end{center}
\end{figure}

\begin{figure}
	\begin{center}
		\includegraphics[width=10 cm]{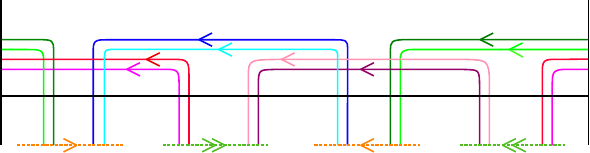}
		\caption{The curves in $J^1(S^1)$ which are identified with the curves $(1,0),(1,0),(0,-1),(0,-1),(-1,0),(-1,0),(0,1),(0,1)$ on $T^2$ as in Figure \ref{CP25torus}.}
		\label{CP25jet}
	\end{center}
\end{figure}

As in the previous example we perform a series of Reidemeister moves, Gompf moves, handle cancellations and handle slides, and obtain the leftmost diagram illustrated in Figure~\ref{CP25slide}. Note that both $1$-handles were canceled, and that if we slide the black trefoil under the red Legendrian unknot we obtain a $7$-component link of Legendrian unknots with maximal Thurston-Bennequin number equal to $-1$. The homology of 
$\mathcal{W}_{T^2, c},$ is easily computed from the Weinstein handle decomposition. In particular, 
\begin{align*}
H_1(\mathcal{W}_{T^2, c}; \Z) &=0 ~\text{and}\\
H_2(\mathcal{W}_{T^2, c};\Z) &=\Z^7.\\
 \end{align*}

Note that the second homology is generated by exact Lagrangian spheres built from the Lagrangian core of the $2$-handles and the Lagrangian disk fillings of the Legendrian attaching spheres.

\begin{figure}
	\begin{center}
		\includegraphics[width=7 cm]{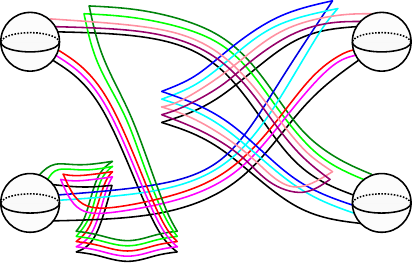}
		\caption{The Weinstein handle diagram of $\mathcal{W}_{T^2, c}$, where $c=\{(1,0),(1,0),(0,-1),(0,-1),(-1,0),(-1,0),(0,1),(0,1)\}.$}
		\label{CP25}
	\end{center}
\end{figure}

\begin{figure}
	\begin{center}
		\includegraphics[width=14 cm]{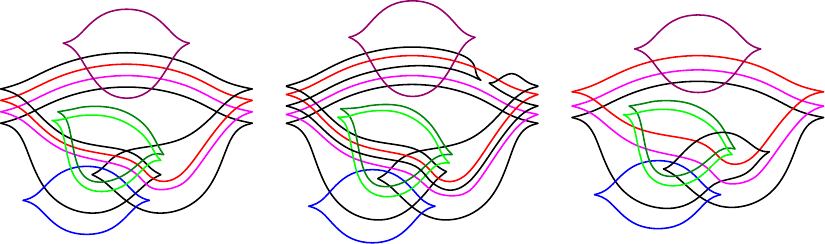}
		\caption{The leftmost diagram is obtained via a simplification of Figure \ref{CP25}. In the middle diagram, we perform an additional 2-handle slide of the black trefoil under the red unknot to obtain the link of Legendrian unknots in the rightmost picture.}
		\label{CP25slide}
	\end{center}
\end{figure}

\subsection{Higher genus}

We study a basic example of a Weinstein $4$-manifold $\mathcal{W}_{F, c}$ where $F$ is an orientable surface of genus $2$. This provides an example of the algorithm from Theorem~\ref{thm:cotangentdiagram} applied to a orientable surface of genus greater than one. 

\begin{theorem}\label{thm:genus2}
  The Weinstein handlebody diagram of the Weinstein manifold $\mathcal{W}_{F, c}$ where $F$ is an orientable surface of genus $2$, and $c$ is $\{\gamma_i\}_{i=1}^4$ as pictured in Figure \ref{g2octo} is given in Figure \ref{fig:g2pinch}.
\end{theorem}

\begin{proof}
We apply the algorithm from Theorem~\ref{thm:cotangentdiagram}. We use a polygonal representation of $F$ as an octagon with sides identified and where an annular region in the octagon is given by taking the complement of a closed disk in the center of the octagon, see Figure~\ref{g2octo}. Our first step is to isotope the co-oriented curves so that they lie in this region. Then, we identify this annular region with $J^1(S^1)$. Figure~\ref{fig:g2jet} shows the corresponding Legendrian curves. As in the case for the torus, we consider $L_1$, the attaching sphere of the $2$-handle in the Gompf handlebody diagram (see Figure~\ref{fig:cotorus}), and $L_0$, a Legendrian co-normal lift of the boundary of the $2$-dimensional zero handle of $F$ in the zero section of $D^*F$. We satellite the Legendrian curves in $J^1(S^1)$ along a neighborhood of $L_0\cup L_1$ and obtain Figure~\ref{fig:g2egstart}, which can be simplified to the Weinstein handlebody diagram in Figure~\ref{fig:g2pinch}.
\end{proof}

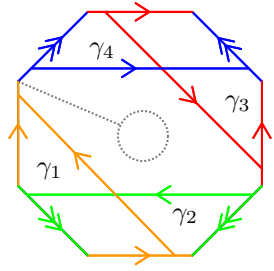
\begin{figure}[h!]
\centering
	\begin{tikzpicture}
	\node[inner sep=0] at (0,0) {\includegraphics[width=3.5cm]{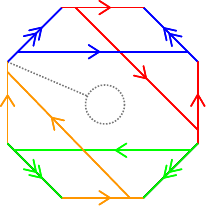}};
\node at (-1.2,-.5) {$\gamma_1$};
\node at (0.6,-1.1) {$\gamma_2$};
\node at (1.3, 0.4){$\gamma_3$};
\node at (-.5, 1.1){$\gamma_4$};
 \end{tikzpicture}
    \caption{The curves $\gamma_1$, $\gamma_2$, $\gamma_3$ and $\gamma_4$ in orange, green, red and blue respectively on a genus 2 surface. Removing the disk in the middle gives an annular region. Cut along the dotted lines an ``unfold" to obtain the rectangle corresponding to Figure~\ref{fig:g2jet}.}
    \label{g2octo}
\end{figure}

\begin{figure}
	\begin{center}
	\begin{tikzpicture}
	\node[inner sep=0] at (0,0) {\includegraphics[width=10cm]{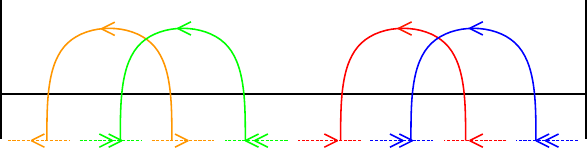}};
\node at (-3.8,1) {$\Lambda_{\gamma_1}$};
\node at (-1.2,1) {$\Lambda_{\gamma_2}$};
\node at (1.2, 1){$\Lambda_{\gamma_3}$};
\node at (3.8, 1){$\Lambda_{\gamma_4}$};
 \end{tikzpicture}
		\caption{The corresponding curves $\Lambda_{\gamma_i}$ $i=1, \ldots, 4$ in $J^1(S^1)$.}
		\label{fig:g2jet}
	\end{center}
\end{figure}
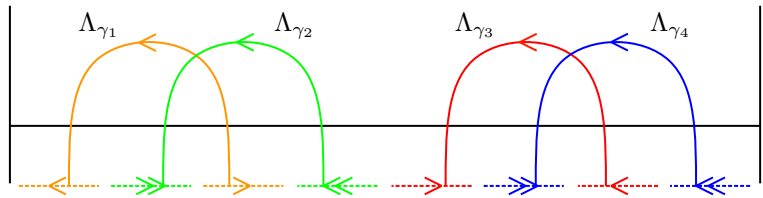

\begin{figure}
	\begin{center}
		\includegraphics[width=7cm]{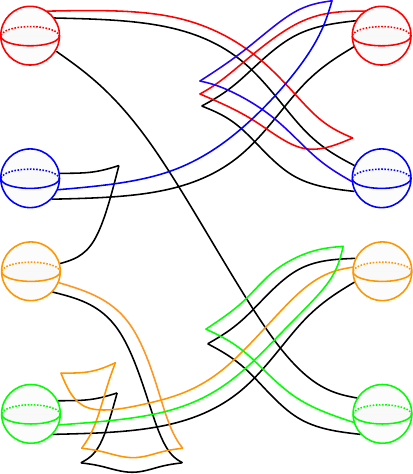}
		\caption{The Weinstein handlebody diagram corresponding to the manifold  $\mathcal{W}_{F, c}$ where $F$ is a genus $2$ surface and $c= \{\gamma_i\}^4_{i=1}$ as pictured in Figure~\ref{g2octo}.}
		\label{fig:g2egstart}
	\end{center}
\end{figure}

\begin{figure}
	\begin{center}
		\includegraphics[width=5 cm]{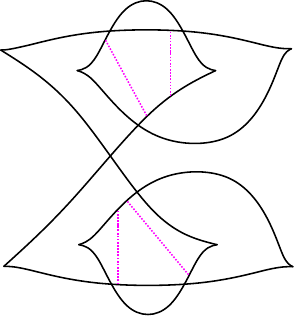}
		\caption{The Weinstein handlebody diagram from Figure \ref{fig:g2egstart} after a series of handle slides and cancellations. If we perform pinch moves along the pink dotted lines, we obtain a Legendrian isotopic to a standard unknot. See Remark~\ref{rem:pinch}.}
		\label{fig:g2pinch}
	\end{center}
\end{figure}

\subsection{Non orientable surfaces}
Another basic example to consider is when $F$ is a real projective plane. We use the proof of Theorem~\ref{thm:cotangentdiagram} to construct a Weinstein handlebody diagram of the Weinstein $4$-manifold $\mathcal{W}_{F, \gamma}$ when $\gamma \subset F$ the co-oriented curve in Figure~\ref{fig:crosscapcurve}.  

\begin{theorem}\label{thm:nonorientable}
  The Weinstein handlebody diagram of the Weinstein manifold $\mathcal{W}_{F,\gamma}$ is shown in Figure \ref{fig:crosscapfinal}, where $F$ is the real projective plane and $\gamma$ is the co-oriented curve shown in Figure~\ref{fig:crosscapcurve}.
\end{theorem}

\begin{proof}
We apply the algorithm from Theorem~\ref{thm:cotangentdiagram}. We use a polygonal representation of $F$ as a disk with sides identified. There is an annular region in the disk given by taking the complement of a closed disk in the center of the disk, see Figure~\ref{fig:crosscapcurve}. Our first step is to isotope the co-oriented curve $\gamma$ so that it lies in this region. Then, we identify the annular region with $J^1(S^1)$. Figure~\ref{fig:crosscapjet} then shows the corresponding Legendrian curve $\Lambda_{\gamma}$. As in the case for the torus, we consider $L_1$, the attaching sphere of the $2$-handle in the Gompf handlebody diagram, (see Figure~\ref{fig:cotorus}), and $L_0$, the co-normal lift of the boundary of the $2$-dimensional zero handle in $F$. We then satellite the Legendrian curves in $J^1(S^1)$ along a neighborhood of $L_0\cup L_1$ and obtain Figure~\ref{fig:crosscapfinal}. Note that in this case, the curve is exactly a parallel copy of the attaching sphere of the $2$-handle.
\end{proof}

\begin{figure}
	\begin{center}
		\includegraphics[width=3.5 cm]{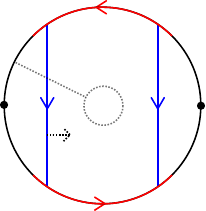}
		\caption{A co-oriented curve $\gamma$ in the real projective plane.  Removing the disk in the middle gives an annular region. Cut along the dotted lines an ``unfold" to obtain the rectangle corresponding to Figure~\ref{fig:crosscapjet}.}
		\label{fig:crosscapcurve}
	\end{center}
\end{figure}

\begin{figure}
	\begin{center}
		\includegraphics[width=10cm]{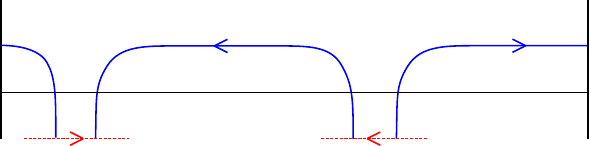}
		\caption{The correponding curve $\Lambda_{\gamma}$ in $J^1(S^1)$.}
		\label{fig:crosscapjet}
	\end{center}
\end{figure}

\begin{figure}
	\begin{center}
		\includegraphics[width=7 cm]{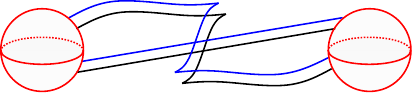}
		\caption{The Weinstein handlebody diagram of $\mathcal{W}_{F, \gamma}$ where $F$ is the real projective plane and $\gamma$ is the co-oriented curve shown in Figure~\ref{fig:crosscapcurve}.}
		\label{fig:crosscapfinal}
	\end{center}
\end{figure}

\begin{remark}
From Weinstein diagrams, we can build closed exact Lagrangians using decomposable handle moves as defined in \cite{EtnNg18}. For instance if we perform pinch moves (ie. 1-handle attachments) along the pink curves of Figure \ref{fig:g2pinch}, the Legendrian knot can be simplified to a Legendrian unknot of maximal tb, giving us a genus $2$ exact Lagrangian surface. Therefore this Weinstein manifold contains a closed exact Lagrangian surface of genus $2$. A pinch move between the black and blue components of the Legendrian link in Figure~\ref{fig:crosscapfinal} show that it bounds a genus $0$ exact Lagrangian surface. Therefore, this second Weinstein manifold contains a closed exact Lagrangian sphere.
\end{remark}\label{rem:pinch}

%% file: toricdivisors_for_memoirs.bbl
\newcommand{\etalchar}[1]{$^{#1}$}
\providecommand{\bysame}{\leavevmode\hbox to3em{\hrulefill}\thinspace}
\providecommand{\MR}{\relax\ifhmode\unskip\space\fi MR }
\providecommand{\MRhref}[2]{%
  \href{http://www.ams.org/mathscinet-getitem?mr=#1}{#2}
}
\providecommand{\href}[2]{#2}
\begin{thebibliography}{ACSG{\etalchar{+}}21}

\bibitem[Abo09]{Abouzaid}
Mohammed Abouzaid, \emph{Morse homology, tropical geometry, and homological
  mirror symmetry for toric varieties}, Selecta Math. (N.S.) \textbf{15}
  (2009), no.~2, 189--270.

\bibitem[ACSG{\etalchar{+}}21]{ACGMMSW1}
Bahar Acu, Orsola Capovilla-Searle, Agn{\`e}s Gadbled, Aleksandra
  Marinkovi{\'{c}}, Emmy Murphy, Laura Starkston, and Angela Wu, \emph{An
  introduction to {W}einstein handlebodies for complements of smoothed toric
  divisors}, Research Directions in Symplectic and Contact Geometry and
  Topology, Association for Women in Mathematics Series, vol.~7, Springer,
  Cham, 2021.

\bibitem[Ati82]{Ati}
Michael~F. Atiyah, \emph{Convexity and commuting {H}amiltonians}, Bull. London
  Math. Soc. \textbf{14} (1982), no.~1, 1--15.

\bibitem[Aud91]{Audin}
Mich\`{e}le Audin, \emph{The topology of torus actions on symplectic
  manifolds}, Progress in Mathematics, vol.~93, Birkh\"{a}user Verlag, Basel,
  1991, Translated from the French by the author.

\bibitem[Aur07]{AurouxTduality}
Denis Auroux, \emph{Mirror symmetry and {$T$}-duality in the complement of an
  anticanonical divisor}, J. G\"{o}kova Geom. Topol. GGT \textbf{1} (2007),
  51--91.

\bibitem[BEE12]{BEE}
Fr\'{e}d\'{e}ric Bourgeois, Tobias Ekholm, and Yasha Eliashberg, \emph{Effect
  of {L}egendrian surgery}, Geom. Topol. \textbf{16} (2012), no.~1, 301--389,
  With an appendix by Sheel Ganatra and Maksim Maydanskiy.

\bibitem[BT82]{BottTu}
Raoul Bott and Loring~W. Tu, \emph{Differential forms in algebraic topology},
  Graduate Texts in Mathematics, vol.~82, Springer-Verlag, New York-Berlin,
  1982.

\bibitem[CdS03]{CdS-toric}
Ana Cannas~da Silva, \emph{Symplectic toric manifolds}, Symplectic geometry of
  integrable {H}amiltonian systems ({B}arcelona, 2001), Adv. Courses Math. CRM
  Barcelona, Birkh\"{a}user, Basel, 2003, pp.~85--173.

\bibitem[CE12]{CE}
Kai Cieliebak and Yakov Eliashberg, \emph{From {S}tein to {W}einstein and
  back}, American Mathematical Society Colloquium Publications, vol.~59,
  American Mathematical Society, Providence, RI, 2012, Symplectic geometry of
  affine complex manifolds.

\bibitem[Cie02]{Cielibak}
Kai Cielibak, \emph{Handle attaching in symplectic homology and the {C}hord
  {C}onjecture}, J. Eur. Math. Soc. \textbf{4} (2002), no.~2, 115--141.

\bibitem[CLS11]{Cox}
David~A. Cox, John~B. Little, and Henry~K. Schenck, \emph{Toric varieties},
  Graduate Studies in Mathematics, vol. 124, American Mathematical Society,
  Providence, RI, 2011.

\bibitem[CM19]{CM}
Roger Casals and Emmy Murphy, \emph{Legendrian fronts for affine varieties},
  Duke Math. J. \textbf{168} (2019), no.~2, 225--323.

\bibitem[CS22]{orsola_thesis}
Orsola Capovilla-Searle, \emph{Infinitely many planar {L}agrangian fillings and
  symplectic milnor fibers}, arXiv preprint, arXiv:2201.03081 (2022).

\bibitem[DG09]{DingGeiges}
Fan Ding and Hansj\"{o}rg Geiges, \emph{Handle moves in contact surgery
  diagrams}, J. Topol. \textbf{2} (2009), no.~1, 105--122.

\bibitem[DL19]{diogo_lisi}
Lu\'{\i}s Diogo and Samuel~T. Lisi, \emph{Symplectic homology of complements of
  smooth divisors}, J. Topol. \textbf{12} (2019), no.~3, 967--1030.

\bibitem[Don96]{Donaldson}
Simon~K. Donaldson, \emph{Symplectic submanifolds and almost-complex geometry},
  J. Differential Geom. \textbf{44} (1996), no.~4, 666--705.

\bibitem[Ekh19]{Ekholm}
Tobias Ekholm, \emph{Holomorphic curves for {L}egendrian surgery}, arXiv
  preprint arXiv:1906.07228 (2019).

\bibitem[EL17]{EkholmLekili}
Tobias Ekholm and Yank{\i} Lekili, \emph{Duality between {L}agrangian and
  {L}egendrian invariants}, arXiv preprint arXiv:1701.01284 (2017).

\bibitem[EL19]{EtguLekili}
Tolga Etg\"{u} and Yank{\i} Lekili, \emph{Fukaya categories of plumbings and
  multiplicative preprojective algebras}, Quantum Topol. \textbf{10} (2019),
  no.~4, 777--813.

\bibitem[Eli90]{Eliash}
Yakov Eliashberg, \emph{Topological characterization of {S}tein manifolds of
  dimension {$>2$}}, Internat. J. Math. \textbf{1} (1990), no.~1, 29--46.

\bibitem[EN15]{EkholmNg}
Tobias Ekholm and Lenhard Ng, \emph{Legendrian contact homology in the boundary
  of a subcritical {W}einstein 4-manifold}, J. Differential Geom. \textbf{101}
  (2015), no.~1, 67--157.

\bibitem[EN18]{EtnNg18}
John~B. Etnyre and Lenhard~L. Ng, \emph{Legendrian contact homology in
  $\mathbb{R}^3$}, arXiv preprint arXiv:1811.10966 (2018).

\bibitem[Eva21]{Evans}
Jonathan~D. Evans, \emph{Lectures on {L}agrangian torus fibrations}, arXiv
  preprint arXiv:2110.08643 (2021).

\bibitem[GHK15]{GrossHackingKeel}
Mark Gross, Paul Hacking, and Sean Keel, \emph{Mirror symmetry for log
  {C}alabi-{Y}au surfaces {I}}, Publ. Math. Inst. Hautes \'{E}tudes Sci.
  \textbf{122} (2015), 65--168.

\bibitem[Gir02]{GirouxICM}
Emmanuel Giroux, \emph{G\'{e}om\'{e}trie de contact: de la dimension trois vers
  les dimensions sup\'{e}rieures}, Proceedings of the {I}nternational
  {C}ongress of {M}athematicians, {V}ol. {II} ({B}eijing, 2002), Higher Ed.
  Press, Beijing, 2002, pp.~405--414.

\bibitem[Gir17]{Giroux}
\bysame, \emph{Remarks on {D}onaldson's symplectic submanifolds}, Pure Appl.
  Math. Q. \textbf{13} (2017), no.~3, 369--388.

\bibitem[Gom98]{Gompf}
Robert~E. Gompf, \emph{Handlebody construction of {S}tein surfaces}, Ann. of
  Math. (2) \textbf{148} (1998), no.~2, 619--693.

\bibitem[GP20]{ganatra_pomerleano_2020}
Sheel Ganatra and Daniel Pomerleano, \emph{Symplectic cohomology rings of
  affine varieties in the topological limit}, Geometric and Functional Analysis
  \textbf{30} (2020), no.~2, 334--456.

\bibitem[GP21]{ganatra2018log}
\bysame, \emph{A log {PSS} morphism with applications to {L}agrangian
  embeddings}, J. Topol. \textbf{14} (2021), no.~1, 291--368.

\bibitem[GS82]{GS}
Victor~W. Guillemin and Shlomo Sternberg, \emph{Convexity properties of the
  moment mapping}, Invent. Math. \textbf{67} (1982), no.~3, 491--513.

\bibitem[GS99]{GompfStipsicz}
Robert~E. Gompf and Andr\'{a}s~I. Stipsicz, \emph{{$4$}-manifolds and {K}irby
  calculus}, Graduate Studies in Mathematics, vol.~20, American Mathematical
  Society, Providence, RI, 1999.

\bibitem[GS09]{GayStipsicz}
David~T. Gay and Andr\'{a}s~I. Stipsicz, \emph{Symplectic surgeries and normal
  surface singularities}, Algebr. Geom. Topol. \textbf{9} (2009), no.~4,
  2203--2223.

\bibitem[Hin04]{Hind}
Richard~K. Hind, \emph{Lagrangian spheres in {$S^2\times S^2$}}, Geom. Funct.
  Anal. \textbf{14} (2004), no.~2, 303--318.

\bibitem[HK20]{HackingKeating}
Paul Hacking and Ailsa Keating, \emph{Homological mirror symmetry for log
  {C}alabi-{Y}au surfaces}, arXiv preprint arXiv:2005.05010, to appear in
  Geometry and Topology (2020).

\bibitem[Lev17]{Leverson}
Caitlin Leverson, \emph{Augmentations and rulings of {L}egendrian links in
  {$\#^k(S^1\times S^2)$}}, Pacific J. Math. \textbf{288} (2017), no.~2,
  381--423.

\bibitem[LM19]{LiMak}
Tian-Jun Li and Cheuk~Yu Mak, \emph{Symplectic divisorial capping in dimension
  4}, J. Symplectic Geom. \textbf{17} (2019), no.~6, 1835--1852.

\bibitem[LS10]{LeungSymington}
Naichung~Conan Leung and Margaret Symington, \emph{Almost toric symplectic
  four-manifolds}, J. Symplectic Geom. \textbf{8} (2010), no.~2, 143--187.

\bibitem[Ngu15]{nguyen2015complement}
Khoa~Lu Nguyen, \emph{On the complement of a positive normal crossing divisor
  with no triple intersection in a projective variety}, arXiv preprint
  arXiv:1512.08537 (2015).

\bibitem[OS04]{OzbagciStipsicz}
Burak Ozbagci and Andra{\'a}s Stipsicz, \emph{Surgery on contact 3-manifolds
  and {S}tein surfaces}, Bolyai society mathematical studies, vol.~13,
  Springer-Verlag, Berlin, 2004.

\bibitem[Ozb19a]{Ozbagci}
Burak Ozbagci, \emph{Correction to: {S}tein and {W}einstein structures on disk
  cotangent bundles of surfaces}, Arch. Math. (Basel) \textbf{113} (2019),
  no.~6, 671--672.

\bibitem[Ozb19b]{OzbagciC}
\bysame, \emph{Stein and {W}einstein structures on disk cotangent bundles of
  surfaces}, Arch. Math. (Basel) \textbf{113} (2019), no.~6, 661--670.

\bibitem[Pas19]{Pascaleff_2019}
James Pascaleff, \emph{On the symplectic cohomology of log {C}alabi-{Y}au
  surfaces}, Geom. Topol. \textbf{23} (2019), no.~6, 2701--2792.

\bibitem[Pic99]{Pick}
Georg Pick, \emph{Geometrisches zur zahlenlehre}, Sitzungber. Lotos,
  Naturwissen Zeitschrift \textbf{19} (1899), 311--319.

\bibitem[She04]{Shev}
Vsevolod~V. Shevchishin, \emph{On the local {S}everi problem}, Int. Math. Res.
  Not. (2004), no.~5, 211--237.

\bibitem[Sik03]{Sikorav}
Jean-Claude Sikorav, \emph{The gluing construction for normally generic
  {$J$}-holomorphic curves}, Symplectic and contact topology: interactions and
  perspectives ({T}oronto, {ON}/{M}ontreal, {QC}, 2001), Fields Inst. Commun.,
  vol.~35, Amer. Math. Soc., Providence, RI, 2003, pp.~175--199.

\bibitem[STW16]{STW}
Vivek~V. Shende, David Treumann, and Harold Williams, \emph{On the
  combinatorics of exact {L}agrangian surfaces}, arXiv preprint
  arXiv:1603.07449 (2016).

\bibitem[Sym03]{Symington}
Margaret Symington, \emph{Four dimensions from two in symplectic topology},
  Topology and geometry of manifolds ({A}thens, {GA}, 2001), Proc. Sympos. Pure
  Math., vol.~71, Amer. Math. Soc., Providence, RI, 2003, pp.~153--208.

\bibitem[Via17]{Vianna}
Renato Vianna, \emph{Infinitely many monotone {L}agrangian tori in del {P}ezzo
  surfaces}, Selecta Math. (N.S.) \textbf{23} (2017), no.~3, 1955--1996.

\bibitem[Wei91]{Weinstein}
Alan~D. Weinstein, \emph{Contact surgery and symplectic handlebodies}, Hokkaido
  Math. J. \textbf{20} (1991), no.~2, 241--251.

\bibitem[Wen10]{Wendl}
Chris Wendl, \emph{Strongly fillable contact manifolds and {$J$}-holomorphic
  foliations}, Duke Math. J. \textbf{151} (2010), no.~3, 337--384.

\end{thebibliography}
